\documentclass[10pt,fleqn]{article}

\usepackage[a4paper,textwidth=178mm,textheight=224mm,centering]{geometry}

\usepackage[T1]{fontenc}
\usepackage[utf8]{inputenc}

\usepackage{amssymb,amsmath,amsthm,amsfonts,mathrsfs}
\usepackage{mathtools}
\usepackage{bm}

\usepackage{graphicx}
\usepackage{subcaption}
\usepackage{float,wrapfig}
\usepackage{tikz}

\usepackage{multirow,multicol}
\usepackage{booktabs}

\usepackage[linesnumbered,ruled,algosection]{algorithm2e}
\usepackage{enumitem}

\usepackage[numbers,sort&compress]{natbib}

\usepackage[colorlinks,linktocpage,citecolor=black,linkcolor=black,
            pdftitle={A volume-controlled level-set method for topology optimization},
            pdfauthor={Wei Gong, Zhenmin Zhang, Shengfeng Zhu}]{hyperref}

\allowdisplaybreaks

\numberwithin{equation}{section}

\theoremstyle{plain}
\newtheorem{theorem}{Theorem}[section]
\newtheorem{lemma}[theorem]{Lemma}
\newtheorem{corollary}[theorem]{Corollary}

\newtheorem{example}[theorem]{Example}
\newtheorem{remark}[theorem]{Remark}

\theoremstyle{definition}
\newtheorem{definition}[theorem]{Definition}

\renewenvironment{proof}[1][]{%
  \par\trivlist
  \item[\hskip\labelsep{\scshape Proof\ifx\relax#1\relax.\else\ (#1).\fi}]\rmfamily\ignorespaces
}{%
  \endtrivlist
}
\renewcommand*{\qed}{\hbox{}\hfill$\Box$}

\newcommand{\dx}{\,\mathrm{d}\bm{x}}
\newcommand{\ds}{\,\mathrm{d}s}
\newcommand{\dd}{\,\mathrm{d}}
\newcommand{\jump}[1]{[\![ #1 ]\!]}
\newcommand{\avg}[1]{\{\!\!\{ #1 \}\!\!\}}

\begin{document}
\let\WriteBookmarks\relax
\def\floatpagepagefraction{1}
\def\textpagefraction{.001}

\title{\Large A volume-controlled level-set method for topology optimization\thanks{This work is supported in part by the National Key Basic Research Program under Grant 2022YFA1004400, the Strategic Priority Research Program of the Chinese Academy of Sciences XDB0640000 \& XDB0640200, the National Natural Science Foundation of China under Grant (Nos. 12471377, 12471393 and 12494543), and the Science and Technology Commission of Shanghai Municipality (Nos. 23JC1400503 and 22DZ2229014).}}

\author{%
  Wei Gong\textsuperscript{1}, \quad
  Zhenmin Zhang\textsuperscript{2}, \quad
  Shengfeng Zhu\textsuperscript{2,3,*}\\[1.2em]
  \parbox{0.90\textwidth}{\centering\normalsize
    \textsuperscript{1}The State Key Laboratory of Mathematical Sciences, National Center for
    Mathematics and Interdisciplinary Sciences and Institute of Computational Mathematics,
    Academy of Mathematics and Systems Science, Chinese Academy of Sciences,
    Beijing 100190, China\\[0.4em]
    \textsuperscript{2}School of Mathematical Sciences, East China Normal University,
    Shanghai 200241, China\\[0.4em]
    \textsuperscript{3}Key Laboratory of MEA (Ministry of Education) \& Shanghai Key
    Laboratory of Pure Mathematics and Mathematical Practice, Shanghai 200241, China\\[0.7em]
    \textsuperscript{*}Corresponding author.
    E-mail addresses: \texttt{wgong@lsec.cc.ac.cn} (W.~Gong),
    \texttt{52290155013@stu.ecnu.edu.cn} (Z.~Zhang),
    \texttt{sfzhu@math.ecnu.edu.cn} (S.~Zhu)%
  }%
}
\date{}
\maketitle

\begin{abstract}
Topology optimization typically involves a geometric volume/mass constraint due to the limited use of material. As a popular technique for topology optimization, the level-set method typically employs the Lagrange method, the penalty method, or the augmented Lagrangian approach to address the volume constraint. This requires careful parameter tuning and may lead to non-monotonic changes in the objective. In this paper, we present a new volume-controlled level-set method for topology optimization. By imposing a scalar weighted volume-compatibility constraint in a modified Stokes-type extension-regularization problem, the velocity is made volume-compatible at the beginning of each frozen transport interval. At the continuous level, we prove that the initial rate of change of the regularized volume vanishes and that the subsequent frozen-velocity drift is of order $T^2$; under a narrow-band assumption, the corresponding constant is independent of the smoothing parameter. For the discontinuous Galerkin spatial semi-discretization of the level-set transport equation, we derive a quantitative regularized-volume drift estimate that displays the dependence on the smoothing parameter, spatial mesh size, transport time, and discrete velocity. The effectiveness of the present algorithm is demonstrated through the minimum compliance problem for linear elasticity, the eigenfrequency maximization for structural vibration, and the minimum energy dissipation problem for Stokes flows.
\end{abstract}


\section{Introduction}\label{sect:intro}

Topology optimization constitutes a fundamental class of partial differential equation (PDE)-constrained optimization problems in which the design variables describe the geometry and topology of the computational domain. Given a governing PDE, the goal is to determine an optimal material distribution that minimizes (or maximizes) a performance functional subject to physical and geometric constraints, while allowing topological changes such as the nucleation, merging, or disappearance of holes. Representative applications include minimum compliance and eigenfrequency design in structural mechanics, drag and dissipation reduction in fluid mechanics, and thermal, acoustic, and electromagnetic design; see, e.g., \cite{bendsoe2013topology,allaire2021shape} and the references therein.

A variety of numerical approaches have been developed for shape and topology optimization, including the homogenization method \cite{Allaire}, the density or SIMP method \cite{bendsoe1989optimal,bendsoe1999material,bourdin2003design,borrvall2001topology}, the level-set method \cite{osher1988fronts,allaire2004structural,burger2003framework,wang2003level,Yamada}, the phase-field method \cite{bourdin2003design,garcke2015numerical,takezawa2010shape,wang2004phase,JLXZ,xie2023phasefield}, the moving morphable component method \cite{GuoZhangZhong}, and the threshold dynamics method \cite{chen2022efficient}. Among these, the level-set method is a widely used implicit interface-tracking technique that provides sharp interface descriptions and naturally accommodates topological changes without explicit hole nucleation rules. It has been successfully applied to structural optimization \cite{allaire2004structural,wang2003level}, fluid topology design \cite{BorrvallPetersson:2003,pironneau1973optimum_1,pironneau1974optimum_2}, structural vibration \cite{allaire2005levelvibration,liu2014eigenvalue}, spectral optimization \cite{LQZ,JQZ}, and multi-physics design \cite{FADJ,DGHK}.

A persistent difficulty in level-set topology optimization is the treatment of the volume (or mass) constraint. Since only a limited amount of material is typically available, admissible designs must satisfy a prescribed volume fraction. Classical formulations usually enforce this constraint through a Lagrange multiplier, a penalty term, or an augmented Lagrangian strategy. Although standard, these approaches introduce algorithmic parameters that must be selected and updated carefully; inappropriate updates may cause oscillatory behavior, slow convergence, or noticeable volume errors. Moreover, because the level-set function is usually evolved by a Hamilton--Jacobi-type equation with a velocity derived from shape sensitivity analysis (see, e.g., \cite{allaire2004structural,wang2003level,WLKW,allaire2005levelvibration}), volume is not directly controlled during the evolution step and is instead corrected through the outer constrained optimization loop.

Several strategies have been proposed to improve volume control. In \cite{WLKW}, a bisection method was used to compute the volume Lagrange multiplier in a meshless level-set method, while \cite{CWW} developed a volume-preserving phase-field method using a secant-type update. These approaches improve robustness but still treat the volume constraint through an external algebraic mechanism rather than as a property of the geometric evolution. In a related direction, conservative level-set methods for two-phase flows \cite{olsson2005conservative_1,olsson2007conservative_2} combine divergence-free velocities with DG or finite-volume discretizations to preserve phase volume. By contrast, direct DG discretization of the level-set transport equation is not volume-conservative in shape and topology optimization when the velocity is not divergence-free \cite{TYZ,TZ,FGP}. Stokes gradient flows have also been used to construct volume-preserving interface evolution laws in two-phase flows \cite{DuanLiYang} and shape optimization \cite{gong2024convergent,FYZ}. These observations motivate embedding an appropriate volume constraint directly in the velocity extension problem. DG discretization is attractive for the transport step because of its local conservation principle and numerical fluxes across element interfaces \cite{DE}.

Motivated by these observations, we develop a new volume-controlled level-set method for general topology optimization problems. A scalar weighted volume-compatibility constraint is embedded in a Stokes-type extension-regularization problem, and the resulting descent velocity drives the level-set transport equation. Thus, volume control is incorporated directly into the geometric evolution law rather than enforced through outer multiplier or penalty updates.

The framework is analyzed at both the continuous and discrete levels. At the continuous level, the Stokes-type extension provides a descent direction, while the weighted constraint cancels the first-order change of the regularized material volume at the beginning of each frozen transport interval and yields an explicit short-time drift bound of order $T^2$; under a narrow-band assumption, the corresponding constant is independent of the smoothing parameter. At the discrete level, we consider a DG spatial semi-discretization of the level-set advection equation. Because the weighted constraint is imposed only when the velocity is computed at the beginning of a transport interval, we derive a regularized-volume drift estimate rather than claim exact semi-discrete conservation. The estimate explicitly resolves the effects of the smoothing parameter, mesh size, transport time, and discrete velocity.

The method is demonstrated on three representative PDE-constrained topology optimization problems: minimum compliance in linear elasticity, fundamental-eigenfrequency maximization in structural vibration, and minimum dissipation in Stokes flow. The numerical results show that the material volume remains close to its prescribed target while competitive objective values and stable convergence behavior are retained.

The main contributions are summarized as follows:
\begin{itemize}
    \item We propose a volume-controlled level-set method in which a scalar weighted volume-compatibility constraint is incorporated directly into a Stokes-type extension-regularization problem coupled with level-set transport.
    \item We establish the descent property and an $O(T^2)$ frozen-velocity regularized-volume drift at the continuous level, including the narrow-band uniformity, and derive a DG semi-discrete drift estimate displaying the dependence on the smoothing parameter, mesh size, transport time, and discrete velocity.
    \item We validate the framework on structural compliance, eigenfrequency optimization, and Stokes-flow topology optimization, demonstrating effective volume control and stable objective evolution.
\end{itemize}
The remainder of the paper is organized as follows. Section~\ref{sect:method} introduces the volume-controlled level-set formulation for topology optimization in structural and fluid design problems and derives a short-time regularized-volume drift estimate for each frozen continuous transport interval. Section~\ref{sect:NM} presents the numerical algorithm and the DG discretization of the level-set transport equation, and derives quantitative regularized-volume drift estimates for the spatially semi-discrete scheme. Section~\ref{sect:NR} provides numerical experiments for linear elasticity, eigenfrequency optimization, and Stokes flow design. Concluding remarks are given in Section~\ref{sect:Conclusion}.

\section{Volume-controlled level-set method for topology optimization}\label{sect:method} 
Let us first introduce the level-set method for a model problem in topology optimization. Let $D\subset\mathbb{R}^d$ $(d=2,3)$ be an open bounded design domain. Let $\Omega\subset D$ be the open domain to be optimized. We introduce a level-set function $\phi:D\rightarrow\mathbb{R}$ for implicit representation:
	\begin{equation}
		\phi(\bm{x}) 
		\begin{cases}
			<0 & \bm{x} \in \Omega, \\
			=0 & \bm{x} \in \Gamma := \partial \Omega, \\
			>0 & \bm{x} \in D \setminus \Omega.
		\end{cases}
	\end{equation}
Consider the following topology optimization problem using the level-set function
	\begin{eqnarray}
    \begin{aligned}
		&\min_{\phi }&& J(\phi,u)  \\
        &\texttt{s.t.}  &&E(\phi,u)=0,\\ 
        & && C(\phi) = 0,
        \end{aligned}
	\end{eqnarray}
where $J:\mathcal{H} \times \mathcal{X} \rightarrow \mathbb{R}$ for some Banach spaces $\mathcal{H}$ and $\mathcal{X}$, $E:\mathcal{H} \times \mathcal{X}\rightarrow \mathbb{R}$ and $C:\mathcal{H}\rightarrow \mathbb{R}$ correspond to the physical constraint of certain PDE and the geometric volume constraint of $\Omega$, respectively. More specifically, 	
\begin{equation}
C(\phi):=\int_D (1-H(\phi))dx-V_0,
\end{equation}
where $V_0>0$ is a given target volume, and $H(\phi)$ is the Heaviside function defined by 
	\begin{equation}
		H(\phi) = 
		\begin{cases}
			1 & \phi > 0, \\
			0 & \phi < 0.
		\end{cases}
	\end{equation}
	In numerical computations, we introduce a $C^1$ and a piecewise quadratic function $H_\gamma(\phi)$ as a regularized Heaviside function of $H(\phi)$:
	\begin{equation} \label{eq:HeaviSide}
		H_\gamma(\phi)=
		\begin{cases}
			0 & \phi<-\gamma,\\[4pt]
			\displaystyle \frac12+\frac{\phi}{\gamma}+\frac{\phi^2}{2\gamma^2} & -\gamma\le \phi\le 0,\\[8pt]
			\displaystyle \frac12+\frac{\phi}{\gamma}-\frac{\phi^2}{2\gamma^2} & 0< \phi\le \gamma,\\[8pt]
			1 & \phi>\gamma,
		\end{cases}
	\end{equation}
	where a small number $\gamma>0$ corresponds to the thickness of a tube that surrounds the interface $\Gamma$. The derivative reads as follows
	\begin{equation}\label{Hprime}
	H_\gamma^\prime(\phi)=
	\begin{cases}
		0 & |\phi|>\gamma,\\[4pt]
		\displaystyle \frac{1}{\gamma}+\frac{\phi}{\gamma^2} & -\gamma\le \phi\le 0,\\[8pt]
		\displaystyle \frac{1}{\gamma}-\frac{\phi}{\gamma^2} & 0<\phi\le \gamma.
	\end{cases}
	\end{equation}
	Let $M(\phi)=1-H(\phi)$ and
	\[
	M_\gamma(\phi)=1-H_\gamma(\phi),
	\]
	the latter of which will be used to approximate the characteristic function $\chi_{\Omega}=M(\phi)$ of $\Omega$, while the volume $|\Omega|=\int_D \chi_{\Omega} {\rm d}x=\int_D M(\phi) {\rm d}x$ of $\Omega$ is approximated by $\int_D M_{\gamma}(\phi) {\rm d}x$. We have $M_\gamma^\prime(\phi)=-H_\gamma^\prime(\phi)$. To illustrate that $\int_D M_{\gamma}(\phi) {\rm d}x$ is a good approximation of $|\Omega|$, we define the following signed distance function
		\begin{equation}
			\psi(\bm{x}) =
			\begin{cases}
				-\min\limits_{\bm{x}_I \in \partial \Omega}  \| \bm{x} - \bm{x}_I \| & \texttt{if } \bm{x} \in \Omega, \\
				\min\limits_{\bm{x}_I \in \partial \Omega}  \| \bm{x} - \bm{x}_I \| & \texttt{if } \bm{x} \in D \setminus \Omega,
			\end{cases}
		\end{equation}
        where $\| \cdot \|$ denotes the Euclidean norm. It has a desirable property $\|\nabla\psi(\bm{x})\|=1$ in $D$. Denoting by $\bm{n}$ an outward unit normal vector on a certain part of the boundary of $D$, we have the following theorem:

    \begin{theorem}
    \label{prop:volume-approximation}
		Assume that the interface $\Gamma$ is of class $C^2$, has no boundary, $\Gamma\Subset D$, and $\gamma<{\rm reach}(\Gamma)$, where
\begin{equation*}
    {\rm reach}(\Gamma):=\sup\{r>0: {\rm every point }\ \bm{x}\ {\rm with }\ {\rm dist}(\bm{x},\Gamma)<r\ {\rm has\ a\ unique\ nearest\ point\ on }\ \Gamma\}.
\end{equation*}
        Then there exists a constant $a_0$ that satisfies $\gamma<a_0<\min\{{\rm reach}(\Gamma),{\rm dist}(\Gamma,\partial D)\}$, such that the signed distance function $\psi(\bm{x})$ is smooth in the tubular neighborhood $ \lbrace \bm{x}: -a_0 \leqslant \psi(\bm{x}) \leqslant a_0 \rbrace \subset D$; see, e.g., \cite{delfour2011shapes}. If, in addition, the magnitude of the curvature is uniformly bounded by a positive constant $\kappa_{\infty}$, where $\kappa(t)$ denotes the signed curvature of $\Gamma$ parameterized by arc length $t$ for $d=2$, and $\kappa_1$ and $\kappa_2$ denote the principal curvatures of $\Gamma$ for $d=3$, then we have the following estimate
		\begin{equation}
			\Big| \int_D M_{\gamma}(\psi) {\rm d}x - |\Omega| \Big| \leqslant L_{\Gamma} \kappa_{\infty} \gamma^2,
		\end{equation}
        where $L_{\Gamma}$ denotes the length ($d=2$) or the area ($d=3$) of $\Gamma$. 
			
			
    \end{theorem}
	\begin{proof}
    It is obvious that $0\leqslant M_\gamma(\phi(\bm{x}))\leqslant 1$ for any $\bm{x}\in D$.
		Notice that
		\begin{align}
			\Big| \int_D M_{\gamma}(\psi) {\rm d}x - |\Omega| \Big| &= \Big| \int_{D \setminus \Omega} M_{\gamma}(\psi) {\rm d}x - \int_{\Omega} (1 - M_{\gamma}(\psi)) {\rm d}x \Big|.  
		\end{align}
		Then we can obtain by the co-area formula and the property $\|\nabla\psi\|=1$ for $|\psi|\leqslant a_0$ that
		\begin{equation} \label{3_I1_eq}
			0 < \int_{D \setminus \Omega} M_{\gamma}(\psi) {\rm d}x = \int_0^{\gamma} M_{\gamma}(\ell)L(\ell) d\ell,
		\end{equation}
		where $L(\ell)$ is the length or area of the interface corresponding to $\psi = \ell$, so for any $|\ell| \leqslant a_0$
		\begin{equation*}
			L(\ell) = \begin{cases}
			    \int_0^{L_{\Gamma}} (1 + \kappa(t) \ell) dt & d = 2, \\
                    \int_{\Gamma} (1+\kappa_1 \ell)(1+\kappa_2 \ell) dS & d=3.
			\end{cases} 
		\end{equation*}
		Here $L_{\Gamma} = L(0)$. Similarly, because of the symmetry $M_{\gamma}(t) + M_{\gamma}(-t) = 1$ for any $t \in \mathbb{R}$,  
		\begin{align} \label{3_I2_eq}
			0 < \int_{\Omega} (1 - M_{\gamma}(\psi)) {\rm d}x = \int_{-\gamma}^0 (1-M_{\gamma}(\ell)) L(\ell) d\ell 
			= \int_{-\gamma}^0 M_{\gamma}(-\ell) L(\ell) d\ell 
			= \int_{0}^{\gamma} M_{\gamma}(\ell) L(-\ell) d\ell.
		\end{align}
		From \eqref{3_I1_eq}-\eqref{3_I2_eq} and $M_{\gamma}(t) \leqslant 1/2$ for any $t \geqslant 0$, we can get for $d = 2$
		\begin{align}
			\Big| \int_D M_{\gamma}(\psi) {\rm d}x - |\Omega| \Big| &= \Big| \int_{0}^{\gamma} M_{\gamma}(\ell) \left[ L(\ell) - L(-\ell) \right] d\ell \Big| \notag \\
			&= \Big| \int_{0}^{\gamma} 2M_{\gamma}(\ell) \int_{0}^{L_{\Gamma}} \kappa(t) \ell dt \ d\ell \Big| \notag \\
			& \leqslant L_{\Gamma} \kappa_{\infty} \Big|  \int_{0}^{\gamma} 2\ell \ M_{\gamma}(\ell) d\ell \Big| \notag \\
			& \leqslant L_{\Gamma} \kappa_{\infty} \frac{\gamma^2}{2} \leqslant L_{\Gamma} \kappa_{\infty} \gamma^2,
		\end{align}
        and for $d = 3$
            \begin{align}
			\Big| \int_D M_{\gamma}(\psi) {\rm d}x - |\Omega| \Big| &= \Big| \int_{0}^{\gamma} M_{\gamma}(\ell) \left[ L(\ell) - L(-\ell) \right] d\ell \Big| \notag \\
			&= \Big| \int_{0}^{\gamma} 2 \ell M_{\gamma}(\ell) \int_{\Gamma} (\kappa_1 + \kappa_2) dS \ d\ell \Big| \notag \\
			& \leqslant 2L_{\Gamma} \kappa_{\infty} \Big|  \int_{0}^{\gamma} 2\ell \ M_{\gamma}(\ell) d\ell \Big| \notag \\
			& \leqslant L_{\Gamma} \kappa_{\infty} \gamma^2,
		\end{align}
		both of which give the bound. \qed
	\end{proof}

From the perspective of numerical accuracy, Theorem \ref{prop:volume-approximation} establishes that the discrepancy between the regularized volume and its geometric counterpart is of order $\mathcal{O}(\gamma^2)$. Here, $\gamma>0$ denotes the regularization parameter, which governs the thickness of the transition layer around the zero level-set interface.
    
	Throughout the paper, we describe our algorithm based on level-set methods for three topology optimization model problems: (1) structural compliance; (2) structural vibration; (3) Stokes flows.
	
	\subsection{Compliance minimization for linear elasticity system} \label{subsect:min_compliance}
	Let the boundary $\partial D$ of $D$ be decomposed into three disjoint pieces such that $\partial D = \partial D_D \cup \partial D_N \cup \partial D_0$ (see Fig. \ref{fig:two-model-examples} (left) for an illustration). 
    Let $ [H^1_{\partial D_D}(D)]^d := \left\lbrace \bm{v} \in [H^1(D)]^d: \ \bm{v} = 0 \texttt{ on } \partial D_D \right\rbrace$, where $H^1(D)$ denotes the space of square-integrable functions whose weak derivatives are also square-integrable in $D$.
One can use the ersatz material method to build the topology optimization model of compliance minimization:
	\begin{equation} \label{2_original_J1}
		\min_{C(\phi) = 0 } J(\phi,\bm{u}) = \int_D A_{\eta} \varepsilon(\bm{u}):\varepsilon(\bm{u}) {\rm d}x,
	\end{equation}
	where the displacement field $\bm{u}=\bm{u}(\phi)\in [H^1_{\partial D_D}(D)]^d$ satisfies
\begin{equation}\label{Model2_1_PDE}
		\begin{cases}
			-\bm{\nabla} \cdot (A_{\eta} \varepsilon(\bm{u})) = \bm{f}_s & \texttt{in} \ D, \\
			\bm{u} = \bm{0} & \texttt{on} \ \partial D_D, \\
			A_{\eta} \varepsilon(\bm{u}) \bm{n} = \bm{g}_s & \texttt{on} \ \partial D_N, \\
			A_{\eta} \varepsilon(\bm{u}) \bm{n}= \bm{0} & \texttt{on} \ \partial D_0
		\end{cases}
	\end{equation}
	 for a given body force $\bm{f}_s \in [L^2(D)]^d$ and boundary traction $\bm{g}_s \in [L^2(\Gamma_N)]^d$. Here, the stress tensor $             A_{\eta}\varepsilon(\bm{u}) =
             M_{\eta,\gamma}(\phi) A\varepsilon(\bm{u})$, where
	\begin{align*}
              M_{\eta,\gamma}(\phi)=M_{\gamma}(\phi)  + \eta(1-M_{\gamma}(\phi)), 
	\end{align*} 
  $\eta \ll 1$ is a positive constant, and the strain tensor $\varepsilon(\bm{u}) := \frac{1}{2} \left( \bm{\nabla} \bm{u} + \bm{\nabla} \bm{u}^T \right)$. The Hooke's law reads as
     \begin{align*}
		 A \varepsilon(\bm{u}) = 2 \mu_L \varepsilon(\bm{u}) + \lambda_L \texttt{tr}(\varepsilon(\bm{u})) \bm{I},
	\end{align*} 
     where $\bm{I} \in \mathbb{R}^{d \times d}$ is the identity tensor, ${\rm tr}(\cdot)$ denotes the trace operator, and the Lam\'e parameters are given by
	\begin{equation*}
		\mu_L = \frac{E_Y}{2(1+\nu)}, \texttt{ and } \lambda_L = \frac{E_Y\nu}{(1+\nu)(1+\nu(1-d))}
	\end{equation*}
	with $E_Y$ and $\nu$ being Young’s modulus and Poisson’s ratio, respectively.

    \subsection{Eigenfrequency maximization for structural vibration}
    Let the boundary of $D$ be decomposed into two disjoint pieces such that $\partial D = \partial D_D \cup \partial D_N$ (see Fig. \ref{fig:two-model-examples} (middle) for an illustration). 
    Similarly to the compliance minimization problem, one can use the ersatz material method to build the topology optimization model of eigenfrequency maximization
	\begin{equation} \label{2_original_J3}
		\min_{C(\phi) = 0 } J(\phi,\bm{u}) = -\lambda_1,
	\end{equation}
	where $\lambda_1$ is the first eigenvalue corresponding to the fundamental eigenfrequency $\omega_1$ with $\lambda_1=\omega_1^2$ and the eigenpair $(\omega^2, \bm{u})\in \mathbb{R}^+\times [H^1_{\partial D_D}(D)]^d$ satisfies the following eigenvalue problem
    \begin{equation}\label{Model2_3_PDE}
		\begin{cases}
			-\bm{\nabla} \cdot (A_{\eta_1} \varepsilon(\bm{u})) = \omega^2 \rho_M \bm{u} & \texttt{in} \ D, \\
			\bm{u} = 0 & \texttt{on} \ \partial D_D, \\
			A_{\eta_1} \varepsilon(\bm{u}) \bm{n} = \bm{0} & \texttt{on} \ \partial D_N,
		\end{cases}
	\end{equation}
    with a sequence of frequencies $0<\omega_1\leqslant \omega_2\leqslant\cdots<\infty$. The stress tensor $A_{\eta_1}\varepsilon(\bm u) =M_{\eta_1,\gamma}(\phi) A\varepsilon(\bm{u})$ with some constant $0 < \eta_1 \ll 1$, where
	\begin{align*}
              M_{\eta_1,\gamma}(\phi)=M_{\gamma}(\phi)  + \eta_1(1-M_{\gamma}(\phi)), 
	\end{align*} 
    the strain tensor $\varepsilon(\bm{u})$, the Young modulus $E$ and Poisson's ratio $\nu$ are similar to those in Section \ref{subsect:min_compliance}. Let $\chi_{\texttt{block}}$ be the characteristic function of a small square region (a black block illustrated in Fig. \ref{fig:two-model-examples}) in the design domain (cf. \cite{allaire2005levelvibration}) with the mass density $\rho_{\texttt{block}}$, which mimics the load of the structure and is not subject to optimization; we set the effective mass density $\rho_M:D\rightarrow\mathbb{R}$ as
    \begin{equation}
        \rho_M = \rho_{\texttt{block}} \chi_{\texttt{block}} + \left( M_{\gamma}(\phi) + \eta_2(1-M_{\gamma}(\phi)\right) (1-\chi_{\texttt{block}})
    \end{equation}
    for some positive constant $\eta_2 \ll 1$, which is the mass density assigned to the non-material region.
    \begin{figure}[htbp]
        \centering
        \begin{subfigure}[htbp]{0.3\textwidth}
            \centering
            \begin{tikzpicture}[scale = 1.8]
                \fill[white] (0,0) rectangle (2,1);
                \fill[red!75] (0,{1/5}) rectangle (2,{4/5});

                \draw[thick] (0,0) rectangle (2,1);
                \node at (1,{1/2}) {$D$};

                \draw[line width=2.4pt,blue] (0,0) -- (0,1);
                \node[left=6pt,blue] at (0,{1/2}) {$\partial D_D$};
                \foreach \y in {{1/25},{3/25},{5/25},{7/25},{9/25},{11/25},{13/25},{15/25},{17/25},{19/25},{21/25},{23/25}}
                {
                    \draw[blue,thick] (0,\y) -- ++(-0.05,-0.05);
                }

                \draw[line width=2.4pt,blue] (2,{9/20}) -- (2,{11/20});
                \node[right=3pt,blue] at (2,{27/50}) {$\partial D_N$};

                \draw[->,thick] ({43/20},{21/50}) -- ({43/20},{29/100});
                \node[right=1pt] at ({43/20},{71/200}) {$g$};
                \node[above=1pt,black] at (1,1) {$\partial D_0$};
                \node[below=1pt,black] at (1,0) {$\partial D_0$};
            \end{tikzpicture}
        \end{subfigure}
        \hfill
        \begin{subfigure}[htbp]{0.3\textwidth}
            \centering
            \begin{tikzpicture}[scale=1.8]
        \fill[white] (0,0) rectangle (2,1);
        \fill[red!75] (0,{1/5}) rectangle (2,{4/5});

        \fill[black] ({19/10},{9/20}) rectangle (2,{11/20});

        \draw[thick] (0,0) rectangle (2,1);
        \node at (1,{1/2}) {$D$};

        \draw[line width=2.4pt,blue] (0,0) -- (0,1);
        \node[left=6pt,blue] at (0,{1/2}) {$\partial D_D$};
        \foreach \y in {{1/25},{3/25},{5/25},{7/25},{9/25},{11/25},{13/25},{15/25},{17/25},{19/25},{21/25},{23/25}}
        {
            \draw[blue,thick] (0,\y) -- ++(-0.05,-0.05);
        }

        \node[above=1pt,black] at (1,1) {$\partial D_N$};
        \node[below=1pt,black] at (1,0) {$\partial D_N$};
        \node[right=4pt,black] at (2,{1/2}) {$\partial D_N$};
    \end{tikzpicture}
            
        \end{subfigure}
        \hfill
        \begin{subfigure}[htbp]{0.3\textwidth}
 \begin{tikzpicture}[scale = 1.8]
                \fill[white] (0,0) rectangle ({3/2},1);

                \fill[red!75]
                (0,{2/3}) -- (0,{5/6}) -- ({3/4},{11/20}) --
                ({3/2},{5/6}) -- ({3/2},{2/3}) -- ({3/4},{9/20}) -- cycle;
                \fill[red!75]
                (0,{1/6}) -- (0,{1/3}) -- ({3/4},{11/20}) --
                ({3/2},{1/3}) -- ({3/2},{1/6}) -- ({3/4},{9/20}) -- cycle;

                \draw[thick] (0,0) rectangle ({3/2},1);
                \node at ({3/4},{1/2}) {$D$};

                \draw[line width=2.2pt,blue] (0,{1/6}) -- (0,{1/3});
                \draw[line width=2.2pt,blue] (0,{2/3}) -- (0,{5/6});
                \node[left=8pt,blue] at (0,{1/4}) {$\Gamma_{\texttt{in}}$};
                \node[left=8pt,blue] at (0,{3/4}) {$\Gamma_{\texttt{in}}$};
                \foreach \y in {{19/100},{23/100},{27/100},{31/100}}
                {
                    \draw[->,blue,thick] (-0.06,\y) -- (0.01,\y);
                }
                \foreach \y in {{69/100},{73/100},{77/100},{81/100}}
                {
                    \draw[->,blue,thick] (-0.06,\y) -- (0.01,\y);
                }

                \draw[line width=2.2pt,blue] ({3/2},{1/6}) -- ({3/2},{1/3});
                \draw[line width=2.2pt,blue] ({3/2},{2/3}) -- ({3/2},{5/6});
                \node[right=8pt,blue] at ({3/2},{1/4}) {$\Gamma_{\texttt{out}}$};
                \node[right=8pt,blue] at ({3/2},{3/4}) {$\Gamma_{\texttt{out}}$};
                \foreach \y in {{19/100},{23/100},{27/100},{31/100}}
                {
                    \draw[<-,blue,thick] ({39/25},\y) -- ({149/100},\y);
                }
                \foreach \y in {{69/100},{73/100},{77/100},{81/100}}
                {
                    \draw[<-,blue,thick] ({39/25},\y) -- ({149/100},\y);
                }

                \node[above=1pt,black] at (0.75,1) {$\partial D_{\texttt{wall}}$};
                \node[below=1pt,black] at (0.75,0) {$\partial D_{\texttt{wall}}$};
            \end{tikzpicture}
        \end{subfigure}

        \caption{Topology optimization examples: structural compliance (left), eigenfrequency (middle), and Stokes flow (right).}
        \label{fig:two-model-examples}
        \end{figure}
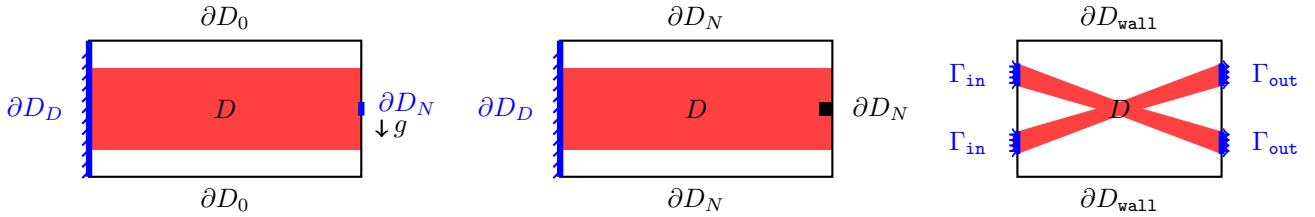

    \subsection{Energy dissipation minimization for Stokes flows}
	Consider the Brinkman approximation to the design problem related to fluid flows \cite{BorrvallPetersson:2003}. Let the boundary $\partial D$ of the design domain $D$ be decomposed into three disjoint pieces such that $\partial D = \partial D_{\texttt{wall}} \cup \Gamma_{\texttt{in}} \cup \Gamma_{\texttt{out}}$. Here, $\Gamma_{\texttt{in}}$ and $\Gamma_{\texttt{out}}$ are fixed as prescribed inflow and outflow boundaries, respectively (see Fig. \ref{fig:two-model-examples} for an illustration).
	 Let $M_{\gamma}(\phi)$ approximate the density function $\rho \in [0,1]$ so that $M_{\gamma}(\phi)=1$ gives the fluid region and $M_{\gamma}(\phi)=0$ gives the solid region, while $\alpha(M_{\gamma}(\phi))$ penalizes the intermediate values. We consider the optimization problem
	\begin{equation} \label{2_original_J2}
		\min_{C(\phi) = 0} J(\phi,\bm{u}) = \frac{1}{2} \int_D \big( \alpha(M_{\gamma}(\phi)) \left| \bm{u} \right|^2 + \nu_f \left| \bm{\nabla} \bm{u} \right|^2 \big) {\rm d}x - \int_{D} \bm{f}_f \cdot \bm{u} {\rm d}x
	\end{equation}
	Let $\Gamma_D^f:=\Gamma_{\rm in}\cup\partial D_{\texttt{wall}}$ and $\Gamma_N^f:=\Gamma_{\rm out}$, so that $\partial D=\Gamma_D^f\cup\Gamma_N^f$ and $\Gamma_D^f\cap\Gamma_N^f=\emptyset$. Define
	$\bm V_{\bm g_f}^f:=\{\bm v\in[H^1(D)]^d:\bm v=\bm g_f\text{ on }\Gamma_{\rm in},\ \bm v=\bm0\text{ on }\partial D_{\texttt{wall}}\}$ and $\bm V_0^f:=\{\bm v\in[H^1(D)]^d:\bm v=\bm0\text{ on }\Gamma_D^f\}$. We seek $\bm u\in\bm V_{\bm g_f}^f$ and $p_f\in L_0^2(D)$ such that
    \begin{equation}\label{weakStokes}
	\begin{aligned}
		\int_D \alpha(M_{\gamma}(\phi)) \bm{u} \cdot \bm{v} {\rm d}x + \int_D \nu_f \bm{\nabla u} : \bm{\nabla v} {\rm d}x - \int_D p_f\,\bm{\nabla} \cdot \bm{v} {\rm d}x = \langle\bm f_f,\bm v\rangle & \quad \forall \bm{v} \in \bm V_0^f, \\
		\int_D q_f\,\bm{\nabla} \cdot \bm{u} {\rm d}x = 0 & \quad \forall q_f \in L_0^2(D),
	\end{aligned}
    \end{equation}
	where $\bm{f}_f\in [H^{-1}(D)]^d$, $\bm{g}_f\in [H^{1/2}(\Gamma_{\rm in})]^d$, $L_0^2(D)=\{q_f\in L^2(D)\vert \int_D q_f{\rm d}x=0\}$, $\nu_f>0$ is a viscosity constant, the natural traction condition is imposed on $\Gamma_N^f$, and the penalization $\alpha(\rho)$ satisfies $\lim\limits_{\rho \rightarrow 0} \alpha(\rho) = +\infty$ and $\lim\limits_{\rho \rightarrow 1} \alpha(\rho) = 0$, so it can force $\bm{u} \rightarrow 0$ as $\rho \rightarrow 0$ and $\bm{u}$ large but finite as $\rho \rightarrow 1$, respectively. Note that \eqref{weakStokes} is the weak form of the Brinkman system for velocity $\bm{u}$ and pressure $p_f$:
	\begin{equation} \label{Model2_2_PDE}
		\begin{cases}
			-\nu_f \bm{\Delta} \bm{u} + \alpha(M_{\gamma}(\phi)) \bm{u} + \bm{\nabla} p_f = \bm{f}_f & \texttt{ in } D, \\
			\bm{\nabla} \cdot \bm{u} = 0 & \texttt{ in } D, \\
			\bm{u} = \bm{g}_f & \texttt{ on } \Gamma_{\texttt{in}}, \\
			\bm{u} = \bm{0} & \texttt{ on } \partial D_{\texttt{wall}}, \\
			\nu_f \bm{\nabla} \bm{u} \bm{n} - p_f \bm{n} = 0 & \texttt{ on } \Gamma_{\texttt{out}}.
		\end{cases}
	\end{equation}

    \subsection{Derivation of directional derivatives}
	Consider the velocity method and assume that the domain $\Omega$ changes along the given velocity field $\bm{V}$. For a sufficiently small time parameter $\delta > 0$, let $\bm{V} \in C \left( [0, \delta); C_0^1(\mathbb{R}^d;\mathbb{R}^d) \right)$, where $C_0^1(\mathbb{R}^d;\mathbb{R}^d)$ denotes the space of all continuously differentiable functions with compact support contained in $\mathbb{R}^d$. Let $\bm{x}(t, \bm{X})\in D$ ($t \in [0, \delta)$) satisfy
	\begin{equation} \label{3_deform_eq}
		\begin{cases}
			\dfrac{d\bm{x}(t, \bm{X})}{dt} = \bm{V}(t, \bm{x}(t, \bm{X})), \\
			\bm{x}(0, \bm{X}) = \bm{X}.
		\end{cases}
	\end{equation}
	The transformed domain $\Omega_t$ is denoted as the set $\lbrace \bm{x}(t, \bm{X}): \bm{X} \in \Omega \rbrace$ and $\Omega_0 := \Omega$. For any $t \in [0, \delta)$, let $\phi_t := \phi(t, \bm{x})$ be the level-set function that describes the domain $\Omega_t$. Assume that 
	$\phi(t, \bm{x}(t)) = \phi(0, \bm{x}(0)) = \phi(0, \bm{X})$ for $t \in [0, \delta)$. In view of  \eqref{3_deform_eq}, we can obtain the following convection equation of $\phi(t, \bm{x})$:
	\begin{equation} \label{3_levelset_eq}
		0 = \frac{d\phi}{dt} = \partial_t \phi\left(t, \bm{x} \right) + \bm{V}\left(t, \bm{x} \right) \cdot \nabla \phi \left(t, \bm{x} \right) = 0 \quad \forall (t, \bm{x}) \in [0, T) \times {D},
	\end{equation}
    where $T>0$ is a long-time instance.
        Although the deformation flow \eqref{3_deform_eq} is defined only for sufficiently small $t \in [0,\delta)$ to ensure the existence of a diffeomorphism, the time variable $t \in [0,T)$ in the level-set equation is an optimization pseudo-time. The long-time evolution is realized by repeatedly applying local deformations and recomputing the velocity field at each iteration \cite{PS}. Unlike the Hamilton-Jacobi equation considered in the level-set method for topology optimization \cite{allaire2004structural,wang2003level,burger2003framework}, the transport equation \eqref{3_levelset_eq} based on continuous Galerkin finite element discretizations \cite{YOYINP,ZHW} and discontinuous-Galerkin finite element discretizations \cite{TZ,TYZ} is considered here for topology optimization.
    
    Let $\phi_0$ be the initial level-set function that describes the domain $\Omega$ and evolved by \eqref{3_levelset_eq}.
	Next, we introduce the variational derivative at $\phi_0$ in the direction $\bm{V}$:
	\begin{definition}
		For a shape functional $\mathcal{J}:\mathcal{H}  \rightarrow \mathbb{R}$, the directional derivative of $\mathcal{J}$ at $\phi_0$ in the direction $\bm{V}\in C \left( [0, \delta); C_0^1(\mathbb{R}^d;\mathbb{R}^d) \right)$ is defined by the following limit if it exists:
		\begin{equation}
			d\mathcal{J}(\phi_0;\bm{V}) := \lim\limits_{t \searrow 0} \frac{\mathcal{J}(\phi_t) - \mathcal{J}(\phi_0)}{t},
		\end{equation}
        where $\mathcal{J}(\phi_t):=J(\phi_t,\bm{u}_t)$ and $\mathcal{J}(\phi_0):=J(\phi_0,\bm{u}_0)$ with  $\bm{u}_t=\bm{u}(\phi_t)$ and $\bm{u}_0=\bm{u}(\phi_0)$.
	\end{definition} 
       For simplicity, we write  $d\mathcal{J}(\phi;\bm{V})$ when $\bm{V}(t,   \bm{x}) \equiv \bm{V}(0, \bm{x}) =: \bm{V}(\bm{x}) \in W^{1,\infty}(D;\mathbb{R}^d)$ appears in, e.g., the perturbation of the identity method, which is equivalent to the velocity method in the sense of first-order expansion. We present below the directional derivatives of three types of model problems (see the Appendix for formal derivations). 
	
	For the compliance minimization problem, 
    the derivative of objective in \eqref{2_original_J1} is given by 
		\begin{align} \label{3_Modified_dJ1}
			d\mathcal{J}(\phi; \bm{V}) = \int_D A \varepsilon(\bm{u}):\varepsilon(\bm{u}) \ M_{\eta,\gamma}^\prime(\phi_0) \ \bm{V} \cdot \nabla \phi_0 {\rm d}x.
		\end{align}

    For the eigenfrequency problem, 
     assume that $\lambda_1=\lambda_1(\phi)$ is simple and thus it is differentiable. The directional derivative of the objective $-\lambda_1(\phi)$ is given by 
		\begin{align} \label{3_Modified_dJ3}
			d\mathcal{J}(\phi; \bm{V}) = \int_D \left[ A \varepsilon(\bm{u}_1):\varepsilon(\bm{u}_1) - \lambda_1 \frac{1-\eta_2}{1-\eta_1}(1-\chi_{\texttt{block}}) |\bm{u}_1|^2\right] \ M_{\eta_1,\gamma}^\prime(\phi_0) \ \bm{V} \cdot \nabla \phi_0 {\rm d}x,
		\end{align}
    where $\bm{u}_1$ is the unit eigenvector normalized by $\int_D\rho_M |\bm{u}_1|^2dx=1$ associated with $\lambda_1$ to \eqref{Model2_3_PDE}. 
    		
	For the Stokes flow energy dissipation minimization problem, 
		let $(\bm{u}(\phi),p(\phi))$ be the weak solution to the Stokes equation \eqref{Model2_2_PDE}. Then, the variational derivative of the objective in \eqref{2_original_J2} is given by 
		\begin{align} \label{3_Modified_dJ2}
			d\mathcal{J}(\phi; \bm{V}) = -\frac{1}{2} \int_D \alpha^\prime(M_{\gamma}(\phi_0)) \left| \bm{u} \right|^2 \ M'_{\gamma}(\phi_0) \ \bm{V} \cdot \nabla \phi_0 {\rm d}x.
		\end{align}
	
	\subsection{Volume-controlled level-set method}
	In numerical computations, the level-set equation \eqref{3_levelset_eq} is usually evolved for a short time $T > 0$, along with an explicit discretization of the velocity field $\bm{V}\left(t, \bm{x} \right) = \bm{V}\left(0, \bm{x} \right)$. So, it suffices to take $\bm{V}\left(t, \bm{x} \right) = \bm{V}(\bm{x})$.
	
	To evolve the level-set equation \eqref{3_levelset_eq}, we need to find an appropriate velocity field $\bm{V}(x)$, this is usually realized through the so-called Hilbertian extension-regularization approach: Find $\bm{V} \in \bm{H}$ such that:
	\begin{equation} \nonumber
			a(\bm{V},\bm{W}) = -d\mathcal{J}(\phi; \bm{W})\quad \forall\bm{W} \in \bm{H}, 		
	\end{equation}
    where $a(\cdot,\cdot):\bm{H}\times \bm{H}\rightarrow\mathbb{R}$ is a bilinear form associated with some Hilbert space $\bm{H}$. We refer to \cite[Sec. 5.2]{allaire2021shape} for a comprehensive overview of different choices of $\bm{H}$ and $a(\cdot,\cdot)$, and we refer to  \cite{Muller} for velocities in a general Banach space.

    In the current paper, we impose a scalar volume-compatibility constraint within the Hilbertian extension-regularization framework and solve the following Stokes-type saddle-point problem: Find $(\bm{V}, p) \in \bm{H} \times \mathbb{R}$ such that 
	\begin{equation} \label{3_CLS_eq}
		\begin{cases}
			a(\bm{V},\bm{W}) + b(\bm{W}, p) = -d\mathcal{J}(\phi; \bm{W})\quad \forall\bm{W} \in \bm{H},  \\
			b(\bm{V},r) = 0\quad \forall r \in \mathbb{R},
		\end{cases}
	\end{equation}
	where
    \begin{equation}
    \bm{H}:= \left\{
        \begin{aligned}
           & \lbrace \bm{V} \in [H^1(D)]^d \, : \, \bm{V} \cdot \bm{n} = 0 \ \text{on} \ \partial D_D \cup \partial D_0, \, \bm{V} = 0 \ \text{on} \ \partial D_N \rbrace\quad\text{for structural compliance}\\
          &  \lbrace \bm{V} \in [H^1(D)]^d \, : \, \bm{V} \cdot \bm{n} = 0 \ \text{on} \ \partial D  \rbrace\quad\text{for structural eigenfrequency}\\
          &  \lbrace \bm{V} \in [H^1(D)]^d \, : \, \bm{V} = 0 \ \text{on} \ \partial D  \rbrace\quad\text{for Stokes flows}.
        \end{aligned}
        \right.
    \end{equation}
    Then 
	\begin{align}
		a(\bm{V},\bm{W}) := \int_D \left[ \bm{\nabla} \bm{V} : \bm{\nabla} \bm{W} + \bm{V} \cdot \bm{W} \right] {\rm d}x, \quad
		b(\bm{V}, r) := r \int_D M'_{\gamma}(\phi) \ \bm{V} \cdot \nabla \phi {\rm d}x,
	\end{align}
	and $d\mathcal{J}(\phi;\bm{V})$ is defined in the previous subsection. Since the scalar multiplier space is one-dimensional, the inf-sup condition reduces to the existence of a velocity field \(\bm{V}\in \bm{H}\) such that \(b(\bm{V},1)\neq 0\). Assuming this nondegeneracy condition holds, the saddle-point problem satisfies the LBB condition and is well posed.
	
	Since $a(\cdot,\cdot)$ is coercive and $\phi_t = \phi(t, \bm{x})$, for $t\in (0,T)$ with some $T > 0$ being small enough, we can derive
	\begin{align}
		\mathcal{J}(\phi_t) &= \mathcal{J}(\phi_0) + t d\mathcal{J}(\phi;\bm{V}) + o(t) \notag \\
		&= \mathcal{J}(\phi_0) - t a(\bm{V},\bm{V}) - tb(\bm{V}, p) + o(t) \notag \\
		&= \mathcal{J}(\phi_0) - t a(\bm{V},\bm{V}) + o(t) \notag \\
		&< \mathcal{J}(\phi_0),	
	\end{align}
	for $0\neq\bm{V}\in \bm{H}$, which implies that the velocity given by \eqref{3_CLS_eq}  is indeed a descent direction.
    Because $r\in\mathbb R$, the second equation in this Stokes-type extension problem imposes one scalar weighted constraint rather than the pointwise condition $\nabla\cdot\bm V=0$. Under the boundary conditions defining $\bm H$, one has $\bm V\cdot\bm n=0$ on $\partial D$, and hence, at the beginning of a transport interval,
    \begin{equation}\label{eq:weighted-compatibility}
        b(\bm V,1)=\int_D \bm V\cdot\nabla M_\gamma(\phi_0)\,\mathrm d\bm x
        =-\int_D M_\gamma(\phi_0)\nabla\cdot\bm V\,\mathrm d\bm x.
    \end{equation}
    Therefore, the second equation enforces a weighted compatibility condition at the current level-set function; it does not make the velocity divergence-free in $D$. For the short-time transport estimate below, we additionally assume that the frozen velocity satisfies $\bm V\in W^{1,\infty}(D;\mathbb R^d)$ and that its flow maps $D$ onto itself. This is an analytical regularity assumption for the transport step and is not inferred solely from the $H^1$ well-posedness of the saddle-point problem.

    \begin{theorem}[Regularized-volume drift under a frozen velocity]\label{thm:frozen-volume-drift}
    Let $\bm V\in W^{1,\infty}(D;\mathbb R^d)$ be held fixed on $[0,T]$. Assume that $\bm V\cdot\bm n=0$ on $\partial D$ and that the flow generated by $\bm V$ maps $D$ onto itself. Let $\phi$ solve
    \begin{equation}
        \partial_t\phi+\bm V\cdot\nabla\phi=0,
        \qquad \phi(0)=\phi_0,
    \end{equation}
    and define
    \begin{equation*}
        \rho_\gamma(t):=M_\gamma(\phi(t)),
        \qquad \mathcal V_\gamma(t):=\int_D\rho_\gamma(t)\,\mathrm d\bm x.
    \end{equation*}
    Set
    \begin{equation*}
        \beta:=\|\nabla\cdot\bm V\|_{L^\infty(D)},\qquad
        \vartheta:=\|\bm V\|_{L^\infty(D)},\qquad
        G_\gamma:=\|\nabla M_\gamma(\phi_0)\|_{L^1(D)}.
    \end{equation*}
    Assume \(M_\gamma(\phi_0)\in W^{1,1}(D)\). If the weighted compatibility condition is imposed at the beginning of the transport interval,
    \begin{equation}
        \int_D M_\gamma(\phi_0)\nabla\cdot\bm V\,\mathrm d\bm x=0,
    \end{equation}
    then $\mathcal V_\gamma'(0)=0$ and
    \begin{equation}\label{eq:frozen-volume-drift}
        |\mathcal V_\gamma(T)-\mathcal V_\gamma(0)|
        \le \frac12\beta\vartheta G_\gamma e^{\beta T}T^2.
    \end{equation}
    Assume also that, for some $r_0>0$ and all $0<\gamma\le r_0$,
    \begin{equation*}
        |M_\gamma'(s)|\le\gamma^{-1}\mathbf 1_{\{|s|\le\gamma\}},
        \qquad \|\nabla\phi_0\|_{L^\infty(D)}\le C_M,
    \end{equation*}
    and
    \begin{equation*}
        |\{\bm x\in D:|\phi_0(\bm x)|\le r\}|\le C_{\rm band}r,
        \qquad 0<r\le r_0.
    \end{equation*}
    Then $G_\gamma\le C_MC_{\rm band}$ and consequently
    \begin{equation}
        |\mathcal V_\gamma(T)-\mathcal V_\gamma(0)|
        \le C\beta\vartheta e^{\beta T}T^2,
    \end{equation}
    where $C$ is independent of $\gamma$ and $T$. If, moreover, $\phi_0$ is the signed-distance initialization covered by Theorem~\ref{prop:volume-approximation} and the initial geometric volume is $V_0=|\Omega_0|$, then
    \begin{equation}
        |\mathcal V_\gamma(T)-V_0|
        \le C_{\rm geom}\gamma^2+C\beta\vartheta e^{\beta T}T^2.
    \end{equation}
    \end{theorem}

    \begin{proof}
    By the chain rule,
    \begin{equation*}
        \partial_t\rho_\gamma+\bm V\cdot\nabla\rho_\gamma=0.
    \end{equation*}
    Set $w(t):=\rho_\gamma(t)-\rho_\gamma(0)$. Since $\bm V$ is frozen in time,
    \begin{equation*}
        \partial_t w+\bm V\cdot\nabla w=-\bm V\cdot\nabla\rho_\gamma(0),
        \qquad w(0)=0.
    \end{equation*}
    The standard $L^1$ estimate for the transport equation \cite[Sec.~2]{ambrosio2014continuity}, together with $\bm V\cdot\bm n=0$ on $\partial D$, yields
    \begin{equation*}
        \frac{\mathrm d}{\mathrm dt}\|w(t)\|_{L^1(D)}
        \le\beta\|w(t)\|_{L^1(D)}+\vartheta G_\gamma.
    \end{equation*}
    Hence, by Gronwall's inequality,
    \begin{equation*}
        \|w(t)\|_{L^1(D)}\le\vartheta G_\gamma t e^{\beta t}.
    \end{equation*}
    On the other hand, using the transport equation and integration by parts,
    \begin{equation*}
        \mathcal V_\gamma'(t)=\int_D\rho_\gamma(t)\nabla\cdot\bm V\,\mathrm d\bm x.
    \end{equation*}
    The initial weighted compatibility condition, therefore, gives
    \begin{equation*}
        \mathcal V_\gamma'(t)=\int_D[\rho_\gamma(t)-\rho_\gamma(0)]
        \nabla\cdot\bm V\,\mathrm d\bm x,
    \end{equation*}
    so that
    \begin{equation*}
        |\mathcal V_\gamma'(t)|\le\beta\vartheta G_\gamma t e^{\beta t}.
    \end{equation*}
    In particular, $\mathcal V_\gamma'(0)=0$. Integrating over $[0,T]$ and using $e^{\beta t}\le e^{\beta T}$ proves \eqref{eq:frozen-volume-drift}. For the narrow-band estimate,
    \begin{align*}
        G_\gamma
        &=\int_D|M_\gamma'(\phi_0)|\,|\nabla\phi_0|\,\mathrm d\bm x\le\frac{C_M}{\gamma}|\{\bm x\in D:|\phi_0(\bm x)|\le\gamma\}|\le C_MC_{\rm band}.
    \end{align*}
    The estimate relative to $V_0$ follows from the triangle inequality and the $O(\gamma^2)$ geometric consistency estimate of Theorem~\ref{prop:volume-approximation}. \qed
    \end{proof}

    \begin{remark}
    The weighted constraint removes the first-order change of the regularized volume at the beginning of each frozen transport interval. The subsequent drift therefore starts at second order in the transport time. The result concerns one continuous transport interval only. Spatial DG errors, time-discretization errors, and the jump caused by reinitialization are separate numerical effects and are treated independently below.
    \end{remark}
    

	\section{Numerical Methods}\label{sect:NM}
	\subsection{Optimization algorithm}
	We use a standard gradient-descent-based level-set method to solve the topology optimization problem. The Algorithm \ref{4_mainAlgorithm} shows the main computational procedure. All numerical experiments are performed using FreeFEM (version 4.15) \cite{Hecht}. 
	
	\begin{algorithm} 
		\caption{Volume-controlled level-set method for topology optimization}  \label{4_mainAlgorithm}
		\KwIn
		{
			Level set function $\phi_0$ with $\int_D M_\gamma(\phi_0)\,\mathrm{d}\bm{x} = V_0$. 
		}
		
		\For{$n = 0,1, \cdots$, until convergence}
		{
			\begin{enumerate}
				\item[(1)] Solve the state and possible adjoint state equations.
				\item[(2)] Use the derivative $d\mathcal{J}(\phi_n; \bm{V})$ to obtain a descent direction $\bm{V}$ for $J(\phi_n)$ by solving the \\ volume-compatible extension-regularization problem \eqref{3_CLS_eq}.
				\item[(3)] Solve the level-set advection equation \eqref{3_levelset_eq} over time $(0,T_n)$ with $T_n=5\tau_n$, \\ where $\tau_n$ is the inner time step. 
				\item[(4)] Reinitialize $\phi^*_{n+1}(\bm{x}) = \phi(\tau_n, \bm{x})$ to a signed distance function $\phi_{n+1}(\bm{x})$ if necessary. 
				\item[(5)] New level-set function $\phi_{n+1}(\bm{x})$ and new shape $\Omega_{n+1}:=\lbrace \bm{x} \in D: \ \phi_{n+1}(\bm{x}) < 0 \rbrace$.
			\end{enumerate}
		}  
	\end{algorithm}

    In the reinitialization process of $\phi^*_{n+1}$, the FreeFEM's built-in ``distance'' command based on the fast marching method (see, e.g., \cite{Sethian}) is used first to reinitialize an arbitrary function $\phi^*_{n+1}$ into a continuous $\mathbb{P}_1$ finite element function $\phi^{\text{Re}}_{n+1}$. Then, $\phi^{\text{Re}}_{n+1}$ is directly interpolated into the corresponding DG finite element function $\phi_{n+1}$. The fast marching reinitialization is designed to preserve the zero level set and therefore the represented geometric interface up to numerical accuracy. It does not, however, preserve the regularized volume $\int_D M_\gamma(\phi)\,\mathrm d\bm x$ exactly, because the level-set profile away from the zero contour is modified. We therefore regard the reinitialization-induced change of the regularized volume as a separate numerical jump.
    
    \begin{remark}\label{rem:reinitialization}
	Accordingly, the transport estimates in Sections~2.5 and 3.3 apply before reinitialization. The additional regularized-volume jump caused by reinitialization is not included in those estimates and is monitored separately in the numerical experiments.
    \end{remark}

	\subsection{Finite element discretization}
	Firstly, we consider the finite element method for discretizing the volume-compatible extension-regularization problem \eqref{3_CLS_eq}.
	Let $\mathbb{P}_k$ ($k \in \mathbb{N}_{+}$) be the set of polynomials with degrees less than or equal to $k$. Let $\mathcal{T}_h$ be a family of shape regular simplicial meshes of the design domain $D$ with mesh size $h:=\max\limits_{K \in \mathcal{T}_h} h_K $, where $h_K$ denotes the diameter of the element $K$. Define the finite element space $\bm{S}_h := \left\lbrace \bm{v}_h \in [C^0(\overline{D})]^d  : \ \bm{v}_h \big|_K \in \mathbb{P}_1(K), \ \forall K \in \mathcal{T}_h \right\rbrace$ for the velocity.
	Set $\bm{S}_{0,h} = \bm{S}_h \bigcap \bm{H}$. The numerical scheme for \eqref{3_CLS_eq} is to find $(\bm{v}_h, p_h) \in \bm{S}_{0,h} \times \mathbb{R}$ such that 
	\begin{equation} \label{4_CLS_FEM}
		\begin{cases}
			a(\bm{v}_h,\bm{w}_h) + b(\bm{w}_h, p_h) = -d\mathcal{J}(\phi_h; \bm{w}_h)\quad \forall \bm{w}_h\in \bm{S}_{0,h}, \\
			b(\bm{v}_h,q_h) = 0\quad \forall q_h \in \mathbb{R}.
		\end{cases}
	\end{equation} 

    In the subsequent DG discretization of the level-set equation, the bilinear form $b(\bm{V},q)$ is modified as follows to control the regularized-volume drift:
    \begin{equation*}
        b(\bm{V},q):=q \sum_{K \in \mathcal{T}_h} \int_K M_\gamma(\phi_h)\,\nabla\cdot \bm V\,{\rm d}x,
    \end{equation*}
    where $\phi_h$ is the discrete level-set function.
    
	Assume that $\lbrace \bm{\psi}_1, \bm{\psi}_2, \cdots, \bm{\psi}_N \rbrace$ is the set of basis functions associated to the interior nodes. Set $\bm{v}_h = \sum\limits_{i=1}^N \bm{v}_i \bm{\psi}_i$. So we can obtain the linear system for \eqref{4_CLS_FEM}:
	\begin{equation}
		\begin{bmatrix}
			A^* & b \\
			b^T & 0
		\end{bmatrix}
		\begin{bmatrix}
			V \\
			p_h
		\end{bmatrix} = 
		\begin{bmatrix}
			f \\
			0
		\end{bmatrix},
	\end{equation}
	where
	\begin{align*}
		A^*=(a(\bm{\psi}_j , \bm{\psi}_i))_{i,j}, \quad
		b = \begin{bmatrix}
			b(\bm{\psi}_1, 1) & b(\bm{\psi}_2, 1) & \cdots & b(\bm{\psi}_N, 1)	
		\end{bmatrix}^T, \\
		V = \begin{bmatrix}
			\bm{v}_1 & \bm{v}_2 & \cdots & \bm{v}_N 
		\end{bmatrix}^T, \quad 
		f = \begin{bmatrix}
			-d\mathcal{J}(\phi_h;\bm{\psi}_1) & -d\mathcal{J}(\phi_h;\bm{\psi}_2) & \cdots & -d\mathcal{J}(\phi_h;\bm{\psi}_N)
		\end{bmatrix}^T.
	\end{align*}
    
	We use the DG method with the so-called dual-$\mathbb{P}_k^{\text{DC}}$ formulation \cite{ern2006discontinuous} for spatial discretizations  of the level-set equation \eqref{3_levelset_eq}. The DG finite element method has good conservative properties for discretizing transport equations due to its use of the  local conservation principle and numerical fluxes across element interfaces. Introduce a discontinuous finite element space 
	\begin{equation*}
		U_h := \left\lbrace v_h \in L^2(D)  : \ v_h \big|_K \in \mathbb{P}_k(K), \ \forall K \in \mathcal{T}_h \right\rbrace.
	\end{equation*}
	Given $\phi_h^0=\Pi_h\phi_0$ with $\Pi_h:L^2({D})\rightarrow U_h$ being an $L^2$-projection operator, the spatial semi-discrete formulation of the level-set equation reads: Find $\phi_h=\phi_h(t)(\cdot)\in U_h$ such that
	\begin{equation} \label{semidiscretize}
		\sum_{K\in \mathcal{T}_h}\int_{K} (\partial_t\phi_h + \bm{V}_h \cdot \nabla \phi_h ) v_h {\rm d}x + \sum\limits_{K \in \mathcal{T}_h} \int_{\partial K \cap E} \frac{1}{2} \left( \left| \bm{V}_h \cdot \bm{n} \right| - \bm{V}_h \cdot \bm{n} \right) \jump{\phi_h} v_h \dd s = 0\quad \forall v_h\in U_h,
	\end{equation}
	where $E$ is the set of interior edges and $\jump{\phi_h} := \phi_h^{\text{interior}} - \phi_h^{\text{exterior}}$ is the jump of $\phi_h$ across the edge, $\bm{n}$ is the unit outward  normal vector to the edge in 2D or to the surface in 3D. We note that there is no boundary contribution from the inflow boundary because we impose $\bm{V}_h = 0$ or $\bm{V}_h \cdot \bm{n} = 0$ on the boundary  for the velocity field $\bm{V}_h$. Given $\phi_h^n\in U_h$, the  fully discrete space-time weak formulation reads: Find $\phi_h^{n+1}\in U_h$ such that
	\begin{equation} \label{4_DG_equation}
		\sum_{K\in \mathcal{T}_h}\int_{K} \Big( \frac{\phi_h^{n+1} - \phi_h^n}{\tau} + \bm{V}_h \cdot \nabla \phi^{n+1}_h \Big) v_h {\rm d}x + \sum\limits_{K \in \mathcal{T}_h} \int_{\partial K \cap E} \frac{1}{2} \left( \left| \bm{V}_h \cdot \bm{n} \right| - \bm{V}_h \cdot \bm{n} \right) \jump{\phi_h^{n+1}} v_h \dd s = 0\quad \forall v_h\in U_h,
	\end{equation}
	where $\tau>0$ is a suitable pseudo-time step size. 
        
        Finally, we use $\mathbb{P}_1$ Lagrange finite elements to discretize \eqref{Model2_1_PDE}, \eqref{Model2_3_PDE} and Taylor--Hood $(\mathbb{P}_2$--$\mathbb{P}_1)$ finite elements to discretize \eqref{Model2_2_PDE}.
        
\subsection{Regularized-volume error for the spatial semi-discretization}
The regularized volume is a nonlinear functional of the DG solution, so the usual conservation argument cannot be applied directly: $M_\gamma'(\phi_h)$ is generally not an admissible test function in $U_h$, and the arithmetic average of its two traces is not, in general, the secant slope required by the chain rule across an interior face.  We therefore estimate the regularized-volume drift rather than assert an exact cancellation of numerical fluxes.

For this estimate, during one level-set update let $\bm V_h$ denote the velocity obtained from \eqref{4_CLS_FEM} and held fixed in time.  Write
\[
 \vartheta_h:=\|\bm V_h\|_{L^\infty(D)},\qquad
 \beta_h:=\|\nabla\cdot\bm V_h\|_{L^\infty(D)},\qquad
 L_h:=\|\nabla\bm V_h\|_{L^\infty(D)}.
\]
Let $X_h$ be the flow of $\bm V_h$.  We assume that $X_h(t,\cdot)$ is a bi-Lipschitz map of $D$ onto itself for $0\le t\le T$.

Set $\phi_0:=\phi(0,\cdot)$, and let $\Gamma_0=\{\bm x\in\overline D:\phi_0(\bm x)=0\}\Subset D$ be a compact $C^2$ hypersurface.  Assume $\phi_0\in W^{1,\infty}(D)\cap C(\overline D)$ and that $\phi_0$ agrees with the signed distance to $\Gamma_0$ in a fixed tubular neighborhood.  The continuous reference solution is characterized by
\begin{equation}\label{eq:flow-identity}
 \phi(t,X_h(t,\bm x))=\phi_0(\bm x),\qquad 0\le t\le T.
\end{equation}

\begin{lemma}[Narrow-band estimate]\label{lem:narrow-band}
There is a constant $C>0$, depending on $D$, the fixed interface $\Gamma_0$, the chosen radius of the tubular-neighborhood, and the positive separation of $|\phi_0|$ from zero outside that neighborhood, but independent of $h$, $\gamma$, and $t$, such that for every $r>0$ and $0\le t\le T$:
\begin{equation}\label{eq:narrow-band}
 |\{\bm x\in D: |\phi_0(\bm x)|\le r\}|\le Cr,
 \qquad
 |\{ \bm x\in D: |\phi(t,\bm x)|\le r\}|\le C e^{\beta_h t}r.
\end{equation}
\end{lemma}

\begin{proof}
Let $\rho_0>0$ be such that $\phi_0$ is the signed distance in the tubular neighborhood $N_{\rho_0}(\Gamma_0)$, and choose
\[
 0<\rho_1<\min\{\rho_0,{\rm reach}(\Gamma_0),{\rm dist}(\Gamma_0,\partial D)\}.
\]
By the positive-separation assumption in the statement, set
\[
 m_1:=\inf_{\bm x\in D\setminus N_{\rho_1/2}(\Gamma_0)}|\phi_0(\bm x)|>0.
\]
Set $r_*:=\frac12\min\{\rho_1,m_1\}$.  If $0<r\le r_*$, then $|\phi_0(\bm x)|\le r$ excludes $\bm x\in D\setminus N_{\rho_1/2}(\Gamma_0)$, and hence $\{\bm x\in D:|\phi_0(\bm x)|\le r\}\subset N_{\rho_1/2}(\Gamma_0)$.  In this neighborhood $\phi_0$ is the signed distance, so
\[
 \{\bm x\in D: |\phi_0(\bm x)|\le r\}
 =\{\bm y+s\bm n(\bm y):\bm y\in\Gamma_0,\ |s|\le r\}.
\]
The normal-coordinate map \cite{delfour2011shapes} $F(\bm y,s)=\bm y+s\bm n(\bm y)$ is one-to-one for $|s|<\rho_1$.  If $\kappa_1,\ldots,\kappa_{d-1}$ are the principal curvatures of $\Gamma_0$, its surface-volume Jacobian satisfies
\[
 J_F(\bm y,s)=\left|\prod_{i=1}^{d-1}(1+s\kappa_i(\bm y))\right|,
 \qquad
 C_J:=\sup_{\bm y\in\Gamma_0,\ |s|\le\rho_1/2}J_F(\bm y,s)<\infty.
\]
Consequently,
\[
 |\{\bm x\in D: |\phi_0(\bm x)|\le r\}|
 =\int_{\Gamma_0}\int_{-r}^{r}J_F(\bm y,s)\,\mathrm ds\,\mathrm dS(\bm y)
 \le 2C_J|\Gamma_0|r,
 \qquad 0<r\le r_*.
\]
For $r>r_*$, the elementary estimate
\[
 |\{\bm x\in D: |\phi_0(\bm x)|\le r\}|\le |D|\le \frac{|D|}{r_*}r
\]
extends the first assertion of this lemma to all $r>0$.

By \eqref{eq:flow-identity},
\[
 \{\bm x\in D:|\phi(t,\bm x)|\le r\}=X_h(t,\{\bm x\in D: |\phi_0(\bm x)|\le r\}).
\]
Indeed, the two inclusions of the set are derived from the characteristic identity and  invertibility of $X_h(t,\cdot)$.  Writing $J_X(t,\bm x):=\det D X_h(t,\bm x)$, the Jacobian equation is, for almost every $\bm x$,
\[
 \partial_tJ_X(t,\bm x)
 =(\nabla\cdot\bm V_h)(X_h(t,\bm x))J_X(t,\bm x),
 \qquad J_X(0,\bm x)=1.
\]
Thus, Liouville's formula \cite[Sec.~2]{ambrosio2014continuity} gives
\[
 0<J_X(t,\bm x)
 =\exp\!\left(\int_0^t(\nabla\cdot\bm V_h)(X_h(s,\bm x))\,\mathrm ds\right)
 \le e^{\beta_h t}.
\]
Applying the change-of-variables formula and the first assertion now yields
\[
 |\{\bm x\in D:|\phi(t,\bm x)|\le r\}|
 =\int_{\{\bm x\in D:|\phi_0(\bm x)|\le r\}}J_X(t,\bm x)\,\mathrm d\bm x
 \le e^{\beta_ht}|\{\bm x\in D:|\phi_0(\bm x)|\le r\}|
 \le Ce^{\beta_ht}r,
\]
which proves the second assertion. \qed
\end{proof}

Differentiating \eqref{eq:flow-identity}, using the inverse-flow estimate and the Jacobian bound from Liouville's formula, and applying Gronwall's inequality \cite[Sec.~2]{ambrosio2014continuity} yield
\begin{equation}\label{eq:gradient-bound}
 \sup_{0\le t\le T}\|\nabla\phi(t)\|_{L^2(D)}\le C e^{(L_h+\beta_h/2)T}.
\end{equation}
In particular, below no $H^2$ regularity of the signed-distance function is required.  We shall also use the following elementary consequence of the saturation of $M_\gamma$ outside $[-\gamma,\gamma]$.

\begin{lemma}[Material-function comparison]\label{lem:material-comparison}
For $z,\varphi\in L^2(D)$ and for all $\delta>0$,
\begin{equation}\label{eq:material-comparison}
 \int_D|M_\gamma(z)-M_\gamma(\varphi)|\,\mathrm d\bm x
 \le |\{\bm x\in D:|\varphi(\bm x)|\le\gamma+\delta\}|
      +\delta^{-2}\|z-\varphi\|_{L^2(D)}^2.
\end{equation}
\end{lemma}

\begin{proof}
Set
\[
 A_\delta:=\{\bm x\in D:|\varphi(\bm x)|\le\gamma+\delta\},\qquad
 B_\delta:=\{\bm x\in D:|z(\bm x)-\varphi(\bm x)|>\delta\}.
\]
For almost every $\bm x\notin A_\delta\cup B_\delta$, we have
\begin{align*}
 \varphi(\bm x)>\gamma+\delta
 &\ \Longrightarrow\ z(\bm x)\ge\varphi(\bm x)-\delta>\gamma,\\
 \varphi(\bm x)<-\gamma-\delta
 &\ \Longrightarrow\ z(\bm x)\le\varphi(\bm x)+\delta<-\gamma.
\end{align*}
Thus, $z(\bm x)$ and $\varphi(\bm x)$ lie in the same saturated region of $M_\gamma$, so their material values coincide.  Since $0\le M_\gamma\le1$, it follows that
\[
 |M_\gamma(z)-M_\gamma(\varphi)|
 \le \mathbf{1}_{A_\delta}+\mathbf{1}_{B_\delta}
 \quad\text{almost everywhere in }D.
\]
Moreover, Chebyshev's inequality gives
\[
 |B_\delta|
 \le\frac{1}{\delta^2}\int_{B_\delta}|z-\varphi|^2\,\mathrm d\bm x
 \le\frac{1}{\delta^2}\|z-\varphi\|_{L^2(D)}^2.
\]
Consequently,
\begin{align*}
 \int_D|M_\gamma(z)-M_\gamma(\varphi)|\,\mathrm d\bm x
 &\le |A_\delta|+|B_\delta|\le |\{\bm x\in D: |\varphi(\bm x)|\le\gamma+\delta\}|
      +\delta^{-2}\|z-\varphi\|_{L^2(D)}^2,
\end{align*}
which proves \eqref{eq:material-comparison}. \qed
\end{proof}

We next turn to the spatial approximation.  For the analysis, take the discontinuous affine space
\[
 U_h:=\{v_h\in L^2(D):v_h|_K\in\mathbb P_1(K)\quad \forall K\in\mathcal T_h\},
\]
and let $\Pi_h$ be the elementwise $L^2$ projection onto $U_h$.  Fix an orientation $\bm n_e$ on every interior face $e$, put $a_e=\bm V_h\cdot\bm n_e$, and use the corresponding jumps and averages.  Define
\begin{align}
 B_h(z_h,v_h)
 := {}&\sum_{K\in\mathcal T_h}\int_K(\bm V_h\cdot\nabla z_h)v_h\,\mathrm d\bm x
 -\sum_e\int_e a_e\jump{z_h}\avg{v_h}\,\mathrm ds +\frac12\sum_e\int_e|a_e|\jump{z_h}\jump{v_h}\,\mathrm ds.
 \label{eq:Bh}
\end{align}
Separately, introduce the upwind jump seminorm
\begin{equation}\label{eq:Jh}
 J_h(t;z_h):=\sum_e\int_e|a_e|\,|\jump{z_h(t)}|^2\,\mathrm ds.
\end{equation}
The semi-discrete method reads: Find $\phi_h(t)\in U_h$ such that 
\begin{equation}\label{eq:semidiscrete}
 (\partial_t\phi_h,v_h)_D+B_h(\phi_h,v_h)=0
 \quad\forall v_h\in U_h,
 \qquad
 \phi_h(0)=\Pi_h\phi_0.
\end{equation}

\begin{lemma}[Semi-discrete level-set error]\label{lem:dg-error}
Let
\[
 \varepsilon_h(T):=
 \sup_{0\le t\le T}\|\phi_h(t)-\phi(t)\|_{L^2(D)}^2
 +\int_0^T J_h(t;\phi_h)\,\mathrm dt.
\]
Then
\begin{equation}\label{eq:epsilon-bound}
 \varepsilon_h(T)
 \le C e^{(\beta_h+2L_h)T}\bigl(h^2+h\vartheta_h T\bigr).
\end{equation}
\end{lemma}

\begin{proof}
Set $\rho_h=\Pi_h\phi-\phi$ and $\xi_h=\phi_h-\Pi_h\phi$.  Standard local projection and trace estimates \cite[Lemmas~1.58--1.59]{DE} give
\begin{equation}\label{eq:projection-estimates}
 \|v-\Pi_hv\|_{L^2(K)}^2\le Ch_K^2\|\nabla v\|_{L^2(K)}^2,
 \qquad
 \|v-\Pi_hv\|_{L^2(\partial K)}^2\le Ch_K\|\nabla v\|_{L^2(K)}^2.
\end{equation}
Because $\Pi_h$ is time independent, $(\partial_t\rho_h,v_h)_D=0$ for $v_h\in U_h$.  Consistency of \eqref{eq:Bh} with the continuous transport equation therefore gives
\begin{equation}\label{eq:error-equation}
 (\partial_t\xi_h,v_h)_D+B_h(\xi_h,v_h)=-B_h(\rho_h,v_h).
\end{equation}
Elementwise integration by parts and the upwind flux imply
\begin{equation}\label{eq:energy-identity}
 B_h(\xi_h,\xi_h)
 =\frac12J_h(t;\xi_h)-\frac12\int_D(\nabla\cdot\bm V_h)\xi_h^2\,\mathrm d\bm x.
\end{equation}
For the residual, integrate the volume term in $B_h(\rho_h,\xi_h)$ by parts.  Since $\bm V_h$ and $\xi_h$ are elementwise affine,
$\nabla\cdot(\bm V_h\xi_h)\in\mathbb P_1(K)$, so its product with $\rho_h$ integrates to zero by $L^2$ orthogonality.  The remaining face terms are
\[
 B_h(\rho_h,\xi_h)
 =\sum_e\int_e\left(a_e\avg{\rho_h}+\frac12|a_e|\jump{\rho_h}\right)\jump{\xi_h}\,\mathrm ds.
\]
To estimate these face terms, set
\begin{align*}
 R_h(t)&:=\sum_e\int_e|a_e|
 \left(|\avg{\rho_h}|^2+|\jump{\rho_h}|^2\right)\,\mathrm ds\le C\vartheta_h\sum_{K\in\mathcal T_h}
 \|\rho_h\|_{L^2(\partial K)}^2
 \le Ch\vartheta_h\|\nabla\phi(t)\|_{L^2(D)}^2.
\end{align*}
Here the first inequality uses shape regularity and the finite overlap of the faces, while the second follows from the trace estimate in \eqref{eq:projection-estimates}.  Weighted Cauchy--Schwarz gives
\[
 |B_h(\rho_h,\xi_h)|
 \le C R_h(t)^{1/2}J_h(t;\xi_h)^{1/2}.
\]
Young's inequality therefore yields
\begin{equation}\label{eq:residual-bound}
 |B_h(\rho_h,\xi_h)|
 \le\frac18J_h(t;\xi_h)
 +Ch\vartheta_h\|\nabla\phi(t)\|_{L^2(D)}^2.
\end{equation}
Taking $v_h=\xi_h$ in \eqref{eq:error-equation} and using \eqref{eq:energy-identity}--\eqref{eq:residual-bound} gives
\[
 \frac12\frac{\mathrm d}{\mathrm dt}\|\xi_h(t)\|_{L^2(D)}^2
 +\frac38J_h(t;\xi_h)
 \le \frac{\beta_h}{2}\|\xi_h(t)\|_{L^2(D)}^2
 +Ch\vartheta_h\|\nabla\phi(t)\|_{L^2(D)}^2.
\]
By \eqref{eq:gradient-bound}, the last term is bounded by $Ch\vartheta_h e^{(2L_h+\beta_h)T}$ (and hence by its value at $t=T$).  Multiplying by $e^{-\beta_ht}$, integrating from $0$ to $t$, and using $\xi_h(0)=0$ give
\[
 \sup_{0\le t\le T}\|\xi_h(t)\|_{L^2(D)}^2
 +\int_0^TJ_h(t;\xi_h)\,\mathrm dt
 \le C e^{(\beta_h+2L_h)T}h\vartheta_h T.
\]
For the projection component, \eqref{eq:projection-estimates} gives
\[
 \sup_{0\le t\le T}\|\rho_h(t)\|_{L^2(D)}^2\le Ce^{(2L_h+\beta_h)T}h^2.
\]
Moreover,
\[
 J_h(t;\rho_h)
 \le C\vartheta_h\sum_{K\in\mathcal T_h}
 \|\rho_h(t)\|_{L^2(\partial K)}^2
 \le Ch\vartheta_h\|\nabla\phi(t)\|_{L^2(D)}^2,
\]
and hence
\[
 \int_0^TJ_h(t;\rho_h)\,\mathrm dt
 \le Ce^{(2L_h+\beta_h)T}h\vartheta_h T.
\]
Since $\phi$ is continuous across faces, $\jump{\Pi_h\phi}=\jump{\rho_h}$.  Combining the last three estimates proves \eqref{eq:epsilon-bound}. \qed
\end{proof}

It remains to convert the level-set error into a bound for the regularized volume.  Assume that the velocity is computed at the beginning of the interval from
\begin{equation}\label{eq:initial-constraint}
 \sum_{K\in\mathcal T_h}\int_K
 M_\gamma(\phi_h(0))\,\nabla\cdot\bm V_h\,\mathrm d\bm x=0.
\end{equation}
No corresponding identity is assumed at later times.  Define
\[
 \mathcal V_{\gamma,h}(t):=\int_D M_\gamma(\phi_h(t))\,\mathrm d\bm x.
\]

\begin{theorem}[Semi-discrete regularized-volume drift]\label{thm:volume-drift}
Under the preceding assumptions, for $0<h\le1$,
\begin{align}
 |\mathcal V_{\gamma,h}(T)-\mathcal V_{\gamma,h}(0)|
 \le {}& C(1+\beta_hT)e^{\beta_hT}\gamma+C e^{\beta_hT+\frac23L_hT}
       \bigl(h^2+h\vartheta_h T\bigr)^{1/3}\nonumber\\
       &+C\beta_hT e^{\frac23\beta_hT}
\bigl(h^2+e^{(2L_h+\beta_h)T}\vartheta_h^2T^2\bigr)^{1/3}.
 \label{eq:volume-drift}
\end{align}
\end{theorem}

\begin{proof}
Insert the continuous endpoint values and write
\begin{align*}
 |\mathcal V_{\gamma,h}(T)-\mathcal V_{\gamma,h}(0)|
 \le A_0+A_c+A_T,
\end{align*}
where
\begin{align*}
 A_0&:=\left|\int_D[M_\gamma(\phi_0)-M_\gamma(\phi_h(0))]\,\mathrm d\bm x\right|,\quad  A_c:=\left|\int_D[M_\gamma(\phi(T))-M_\gamma(\phi_0)]\,\mathrm d\bm x\right|,\\
 A_T&:=\left|\int_D[M_\gamma(\phi_h(T))-M_\gamma(\phi(T))]\,\mathrm d\bm x\right|.
\end{align*}

For $A_0$, apply Lemmas~\ref{lem:narrow-band} and \ref{lem:material-comparison}, together with
$\|\Pi_h\phi_0-\phi_0\|_{L^2(D)}^2\le Ch^2$.  Choosing $\delta=h^{2/3}$ gives
\begin{equation}\label{eq:A0}
 A_0\le C\gamma+Ch^{2/3}.
\end{equation}

For the continuous term, the transport equation and the boundary condition imply
\[
 \frac{\mathrm d}{\mathrm dt}\int_DM_\gamma(\phi(t))\,\mathrm d\bm x
 =\int_DM_\gamma(\phi(t))\nabla\cdot\bm V_h\,\mathrm d\bm x.
\]
Indeed, this identity follows by multiplying the transport equation by $M_\gamma'(\phi)$, applying the chain rule, and integrating by parts; the boundary term vanishes because $\bm V_h\cdot\bm n=0$ on $\partial D$.  Subtracting the fixed initial constraint \eqref{eq:initial-constraint} gives, for each $t$,
\[
 \left|\frac{\mathrm d}{\mathrm dt}\int_DM_\gamma(\phi(t))\,\mathrm d\bm x\right|
 \le\beta_h\int_D
 |M_\gamma(\phi(t))-M_\gamma(\phi_h(0))|\,\mathrm d\bm x.
\]
From the definition of $A_c$ and the above inequality we conclude that
\begin{equation}\label{eq:Ac-start}
 A_c\le\beta_h\int_0^T\int_D
 |M_\gamma(\phi(t))-M_\gamma(\phi_h(0))|\,\mathrm d\bm x\,\mathrm dt.
\end{equation}
To estimate the integrand, first use $\phi_h(0)=\Pi_h\phi_0$, the projection estimate, and $\partial_t\phi=-\bm V_h\cdot\nabla\phi$ to obtain
\begin{align*}
 \|\phi(t)-\phi_h(0)\|_{L^2(D)}
 &\le \|\phi_0-\Pi_h\phi_0\|_{L^2(D)}
 +\int_0^t\|\bm V_h\cdot\nabla\phi(s)\|_{L^2(D)}\,\mathrm ds\le Ch+Ce^{(L_h+\beta_h/2)T}\vartheta_h t.
\end{align*}
Consequently,
\[
 \|\phi(t)-\phi_h(0)\|_{L^2(D)}^2
 \le C\bigl(h^2+e^{(2L_h+\beta_h)T}\vartheta_h^2t^2\bigr).
\]
Apply Lemma~\ref{lem:material-comparison} with $z=\phi_h(0)$ and $\varphi=\phi(t)$, and then use Lemma~\ref{lem:narrow-band}.  With
\[
 \delta_c(t)=e^{-\beta_ht/3}
 \bigl(h^2+e^{(2L_h+\beta_h)T}\vartheta_h^2t^2\bigr)^{1/3},
\]
the narrow-band and Chebyshev contributions are balanced, so
\[
 \int_D|M_\gamma(\phi(t))-M_\gamma(\phi_h(0))|\,\mathrm d\bm x
 \le Ce^{\beta_ht}\gamma
 +Ce^{\frac23\beta_ht}
 \bigl(h^2+e^{(2L_h+\beta_h)T}\vartheta_h^2t^2\bigr)^{1/3}.
\]
Substitution into \eqref{eq:Ac-start}, followed by bounding the increasing factors at $t=T$, gives
\begin{equation}\label{eq:Ac}
 A_c\le C\beta_hTe^{\beta_hT}\gamma
 +C\beta_hTe^{\frac23\beta_hT}
  \bigl(h^2+e^{(2L_h+\beta_h)T}\vartheta_h^2T^2\bigr)^{1/3}.
\end{equation}

Finally, apply the same comparison lemma at $t=T$ with
$\delta_T=e^{-\beta_hT/3}\varepsilon_h(T)^{1/3}$.  Lemma~\ref{lem:dg-error} gives
\begin{equation}\label{eq:AT}
 A_T\le Ce^{\beta_hT}\gamma
 +Ce^{\beta_hT+\frac23L_hT}
   \bigl(h^2+h\vartheta_h T\bigr)^{1/3}.
\end{equation}
Combining \eqref{eq:A0}, \eqref{eq:Ac}, and \eqref{eq:AT}, and absorbing $h^{2/3}$ into the second term, proves \eqref{eq:volume-drift}. \qed
\end{proof}

\begin{corollary}[Error relative to the target volume]\label{cor:target-volume}
Under the assumptions of Theorem~\ref{thm:volume-drift}, suppose $0<\gamma\le1$, the initial geometric volume is $V_0$, and
\[
 \left|\int_DM_\gamma(\phi_0)\,\mathrm d\bm x-V_0\right|\le C\gamma^2.
\]
If, in addition, $\vartheta_h$, $\beta_h$, and $L_h$ are uniformly bounded in $h$, then
\begin{equation}\label{eq:target-volume-general}
 |\mathcal V_{\gamma,h}(T)-V_0|
 \le C e^{CT}\Bigl[\gamma+(h^2+hT)^{1/3}
                    +T(h^2+T^2)^{1/3}\Bigr].
\end{equation}
In particular, if $c_-h\le\gamma\le c_+h$ with constants $c_-,c_+>0$ independent of $h$, then
\begin{equation}\label{eq:target-volume-mesh}
 |\mathcal V_{\gamma,h}(T)-V_0|
 \le C e^{CT}\Bigl[h+(h^2+hT)^{1/3}
                    +T(h^2+T^2)^{1/3}\Bigr].
\end{equation}
Here and below, $C$ is independent of $h$, $\gamma$, and $T$.
\end{corollary}

\begin{proof}
The initial projection estimate \eqref{eq:A0} and the assumed geometric consistency yield
\begin{align*}
 |\mathcal V_{\gamma,h}(0)-V_0|
 &\le\left|\int_D[M_\gamma(\phi_h(0))-M_\gamma(\phi_0)]\,\mathrm d\bm x\right|
      +\left|\int_DM_\gamma(\phi_0)\,\mathrm d\bm x-V_0\right|\\
 &\le C\gamma+Ch^{2/3}+C\gamma^2
 \le C\gamma+Ch^{2/3}.
\end{align*}
Adding this bound to \eqref{eq:volume-drift}, using the uniform bounds on
$\vartheta_h$, $\beta_h$, and $L_h$, and absorbing the polynomial factors in
$T$ into $e^{CT}$ gives \eqref{eq:target-volume-general}.  Substituting
$\gamma\le c_+h$ proves \eqref{eq:target-volume-mesh}.  If the initial
discrete regularized volume is calibrated to equal $V_0$, no initial mismatch
term is needed and Theorem~\ref{thm:volume-drift} applies directly. \qed
\end{proof}

\begin{remark}
The estimate above is a short-transport-interval regularized-volume drift bound rather than a fixed-\(T\) mesh-convergence theorem. In particular, if \(T>0\) is held fixed while \(h\to0\), the last contribution in the displayed estimate does not necessarily vanish. A vanishing bound follows under a joint refinement/short-time regime. For example, if \(\gamma\simeq h\) and \(T=O(h)\), then the right-hand side of \eqref{eq:target-volume-mesh} is \(O(h^{2/3})\) up to the displayed exponential factor.
\end{remark}

Theorem~\ref{thm:volume-drift} concerns one complete frozen-velocity update, not one inner time step. At outer iteration $n$, $\bm V_h^n$ is kept fixed on $[0,T_n]$, while \eqref{4_DG_equation} is applied $m$ times with $T_n=m\tau_n$ (here $m=5$). Writing
\[
V_h^j:=\int_D M_\gamma(\phi_h^j)\,dx,
\]
the fully discrete interval drift satisfies
\[
|V_h^m-V_h^0|
\le
|\mathcal V_{\gamma,h}(T_n)-\mathcal V_{\gamma,h}(0)|
+
|V_h^m-\mathcal V_{\gamma,h}(T_n)|.
\]
Combining the standard first-order backward-Euler error with a narrow-band argument analogous to Lemma~\ref{lem:material-comparison}, the second term is expected to be of order $O(\gamma+\tau_n^{2/3})$; reinitialization is performed afterwards and remains a separate numerical effect, as stated in Remark~\ref{rem:reinitialization}.

    \section{Numerical results}\label{sect:NR}
	In this section, we provide numerical results to verify the effectiveness of the level-set method. Algorithm \ref{4_mainAlgorithm} is terminated after a fixed number of iterations, the level-set equation is solved over five inner pseudo-time steps, i.e., $T=5\tau$, while the level-set function is reinitialized periodically for a fixed number of outer iterations. 
    To evaluate the volume-control performance of the algorithm, we quantify the numerical relative error $V^e:=\frac{1}{N}\sum_{n=1}^N|V_0-V_n|/V_0$ of the volume during iterations, where $V_n:=\int_{D} M_{\gamma}(\phi_h^n) {\rm d}x$.
    Some parameter values and numerical results of volume errors are provided in \autoref{tab:1}. The nonzero relative volume errors reported in \autoref{tab:1} are obtained after fully space-time discretizations and depend on the smoothing parameter $\gamma$, the spatial mesh size $h$, and the pseudo-time step size $\tau$, the latter of which may accumulate over successive level-set updates. We set $k=1$ for the discretization of the level-set function.
   
	\subsection{Optimization of structural compliance}
    Set the force $\bm{f}_s = \bm{0}$. Set the L\'ame constants $\lambda_L = 1, \mu_L = 8$ and $\lambda_L = \frac{15}{26}$, $\mu_L = \frac{5}{13}$ for 2D and 3D, respectively. Choose $\eta = 10^{-4}$ and $\eta = 5\times 10^{-3}$ for 2D and 3D, respectively. Choose the width of the interface $\gamma = 5\times 10^{-3}$ and $\gamma = 0.03$ for 2D and 3D, respectively.
	\begin{example} 
		\label{5_M1E1}
		(Cantilever: 2D \cite{wang2003level}) Set $D = (0,2) \times (0,1)$, Dirichlet boundary $\partial D_D = \lbrace (0,y): y \in [0,1] \rbrace$, and Neumann boundary $\partial D_N = \lbrace (2,y): y \in [0.45, 0.55] \rbrace$ (see the left of \autoref{5_M1_Ex1}). Set $\bm{g}_s = (0, -5)$. 
	\end{example}
    
In \autoref{Figure_M1E1_All}, we observe that the area remains close to its prescribed target value and the compliance decays monotonically during topological evolution.

Using the same two-dimensional cantilever setting as in Example~\ref{5_M1E1}, we now quantify how well the regularized volume is preserved by the transport step. For outer iteration $n$, let $V_n$ and $\widetilde V_n$ denote the regularized volumes immediately before and after transport, respectively, and define
\begin{equation}
    e_n:=\frac{|\widetilde V_n-V_n|}{V_n}.
\end{equation}
For a sampled sequence $\{e_n\}_{n\in\mathcal I}$ we report
\begin{equation}
    E_{\max}=\max_{n\in\mathcal I}e_n,
    \qquad
    E_{\mathrm{RMS}}=\left(\frac{1}{|\mathcal I|}\sum_{n\in\mathcal I}e_n^2\right)^{1/2},
    \qquad
    E_{\mathrm{end}}=e_{n_{\mathrm{last}}}.
\end{equation}
For adjacent refinement levels, the observed order is
\begin{equation}
    p=\frac{\log(E_c/E_f)}{\log(q_c/q_f)},
\end{equation}
where $E_c$ and $E_f$ are the values of the statistic under consideration on the coarser and finer levels, respectively, and $q_c$ and $q_f$ are the corresponding transport times; since $T$ is halved at each step, $q_c/q_f=2$.

Table~\ref{tab:transport-time-refinement} reports the pure transport error under time refinement, for which the mesh and $\gamma=0.005$ are kept fixed while $T=0.075$. All four tests start from the same stored states, and the listed orders $p$ are those between adjacent transport times.
\begin{table}[htbp]
  \centering
  \caption{Pure transport error under time refinement.}
  \label{tab:transport-time-refinement}
  \begin{tabular}{lcccc}
    \toprule
    Statistic & $T$ & $T/2$ & $T/4$ & $T/8$ \\
    \midrule
    $E_{\max}$          & $4.164\times10^{-2}$ & $1.195\times10^{-2}$ & $4.138\times10^{-3}$ & $2.034\times10^{-3}$ \\
    $E_{\mathrm{RMS}}$ & $5.116\times10^{-3}$ & $1.480\times10^{-3}$ & $5.098\times10^{-4}$ & $2.580\times10^{-4}$ \\
    $E_{\mathrm{end}}$ & $4.785\times10^{-5}$ & $2.417\times10^{-5}$ & $1.214\times10^{-5}$ & $6.087\times10^{-6}$ \\
    \midrule
    $p_{\max}$          & -- & $1.8012$ & $1.5298$ & $1.0246$ \\
    $p_{\mathrm{RMS}}$ & -- & $1.7900$ & $1.5371$ & $0.9827$ \\
    $p_{\mathrm{end}}$ & -- & $0.9856$ & $0.9928$ & $0.9964$ \\
    \bottomrule
  \end{tabular}
\end{table}
The drift is approximately first order in $T$. This is lower than the continuous $O(T^2)$ frozen-velocity rate, but is consistent with the fully discrete first-order time stepping and is more favorable than the fixed-mesh semi-discrete bound alone guarantees. 
	
	\vspace*{2ex}
	\begin{example} \label{5_M1E2}	
		(Bridge: 2D \cite{allaire2004structural})  Set $D = (0,2) \times (0,1)$, $\partial D_D = \lbrace (0,y), \ (2,y): y \in [0, 0.05] \rbrace \bigcup \lbrace (x,0): x \in [0, 0.05] \cup [1.95, 2] \rbrace$, and $\partial D_N = \lbrace (x,0):  x \in [0.95, 1.05] \rbrace$ (see the right plot of \autoref{5_M1_Ex1}). Set $\bm{g}_s = (0, -20)$. 
	\end{example}
    
	\autoref{Figure_M1E2_All} illustrates the topological evolution, changes in the approximate volume, and a decrease in the objective.
	
	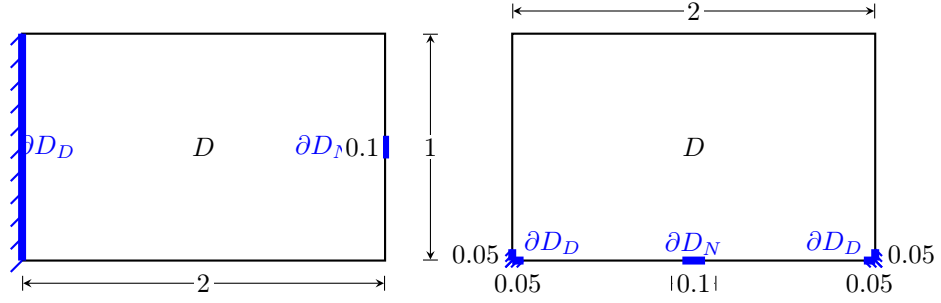
\begin{figure}[htbp] 
		\centering
		\begin{tikzpicture}[scale=3] 
			\draw[thick] (0,0) rectangle (1.6,1);
			\node at (0.8,0.5) {$D$};  
			
			\draw[line width=3pt, blue] (0,0) -- (0, 1);
            \node[left=2pt, blue] at (0.3, 0.5) {$\partial D_D$};
			\def\N{10}
			\def\dx{0.05}
			\def\dy{0.05}
			\foreach \i in {0,1,...,\N}
			{
				\pgfmathsetmacro{\ypos}{1 - \i/\N}
				\draw[thick, blue] (0, \ypos) -- ++(-\dx, -\dy);
			}
			
			\draw[line width=3pt, blue] (1.6,0.45) -- (1.6,0.55);
			\node[right=1pt, blue] at (1.15,0.5) {$\partial D_N$};
            
			\tikzset{
				dim/.style={black, line width=0.4pt, |<->|, >=stealth},
				dim label/.style={midway, fill=white, inner sep=1.2pt}
			}
			
			\draw[<->, dim] (0,-0.1) -- (1.6,-0.1) 
			node[dim label] {2};
			
			\draw[<->, dim] (1.8,0) -- (1.8,1) 
			node[dim label] {1};
			\draw[<->, dim] (1.5,0.45) -- (1.5,0.55) 
			node[dim label] {0.1};
		\end{tikzpicture}
        \begin{tikzpicture}[scale=3] 
			\draw[thick] (0,0) rectangle (1.6,1);
			\node at (0.8,0.5) {$D$};  
			
			\draw[line width=3pt, blue] (0,0) -- (0, 0.05);
			\draw[line width=3pt, blue] (0,0) -- (0.05, 0);
			\node[right=1pt, blue] at (0, 0.08) {$\partial D_D$};
			\def\N{2}
			\def\dx{0.03}
			\def\dy{0.03}
			\foreach \i in {0,1,...,\N}
			{
				\pgfmathsetmacro{\ypos}{0.05*(1 - \i/\N)}
				\draw[thick, blue] (0, \ypos) -- ++(-\dx, -\dy);
				\draw[thick, blue] (\ypos, 0) -- ++(-\dx, -\dy);
			}
			
			\draw[line width=3pt, blue] (1.6,0) -- (1.6, 0.05);
			\draw[line width=3pt, blue] (1.6,0) -- (1.55, 0);
			\node[left=1pt, blue] at (1.6, 0.08) {$\partial D_D$};
			\def\N{2}
			\def\dx{0.03}
			\def\dy{0.03}
			\foreach \i in {0,1,...,\N}
			{
				\pgfmathsetmacro{\ypos}{0.05*(1 - \i/\N)}
				\draw[thick, blue] (1.6, \ypos) -- ++(\dx, -\dy);
				\draw[thick, blue] (1.6-\ypos, 0) -- ++(\dx, -\dy);
			}
			\draw[line width=3pt, blue] (0.75,0) -- (0.85,0);
			\node[blue] at (0.8,0.075) {$\partial D_N$};
			
			\tikzset{
				dim/.style={black, line width=0.4pt, |<->|, >=stealth},
				dim label/.style={midway, fill=white, inner sep=1.2pt}
			}
			
			\draw[<->, dim] (0,-0.1) -- (0.05,-0.1) 
			node[dim label] {0.05};
			\draw[<->, dim] (0.7,-0.1) -- (0.9,-0.1) 
			node[dim label] {0.1};
			\draw[<->, dim] (1.55,-0.1) -- (1.6,-0.1) 
			node[dim label] {0.05};
			\draw[<->, dim] (0,1.1) -- (1.6,1.1) 
			node[dim label] {2};
			
			\draw[<->, dim] (-0.16,0) -- (-0.16,0.05) 
			node[dim label] {0.05};
			\draw[<->, dim] (1.76,0) -- (1.76,0.05) 
			node[dim label] {0.05};
		\end{tikzpicture}
        
		\caption{The design domain $D$ and its boundary for Example \ref{5_M1E1} (left) and Example \ref{5_M1E2} (right).}
		\label{5_M1_Ex1}
	\end{figure}

	\begin{figure}[htbp]
		\centering
		\begin{subfigure}[b]{0.32\textwidth}
			\includegraphics[width=\linewidth]{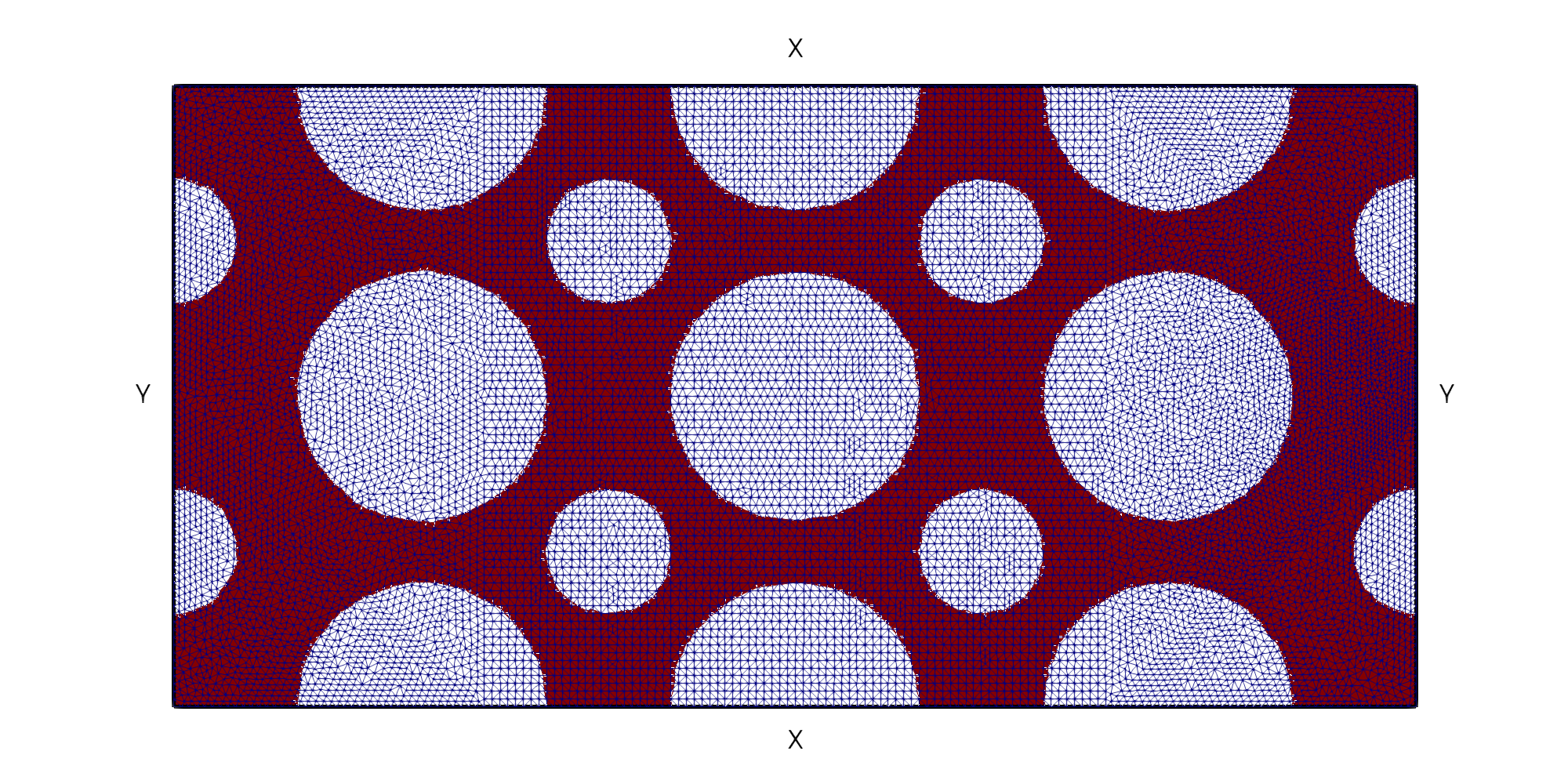} 
		\end{subfigure}
		\hfill 
		\begin{subfigure}[b]{0.32\textwidth}
			\includegraphics[width=\linewidth]{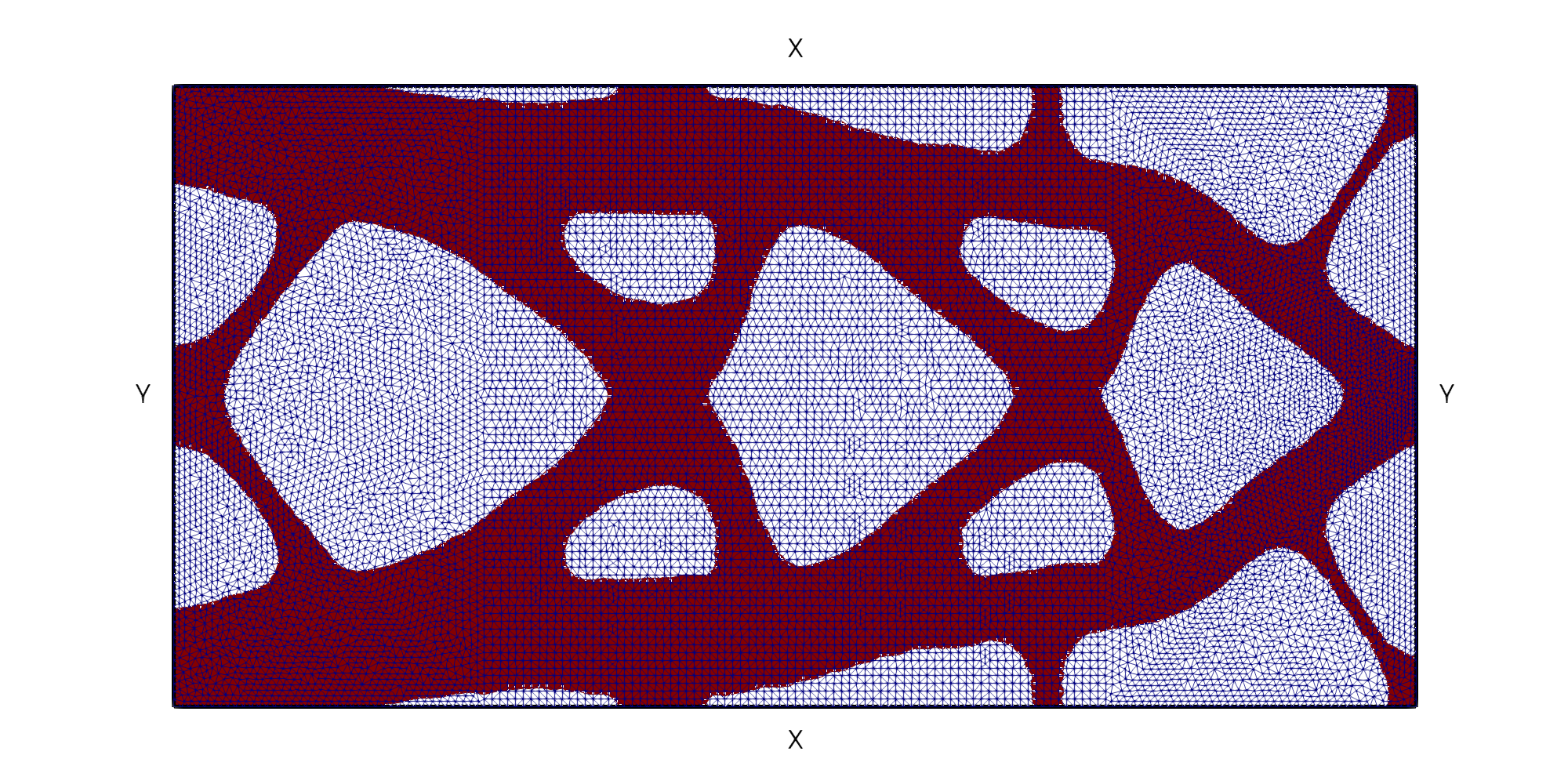}
		\end{subfigure}
        		\hfill
		\begin{subfigure}[b]{0.32\textwidth}
			\includegraphics[width=\linewidth]{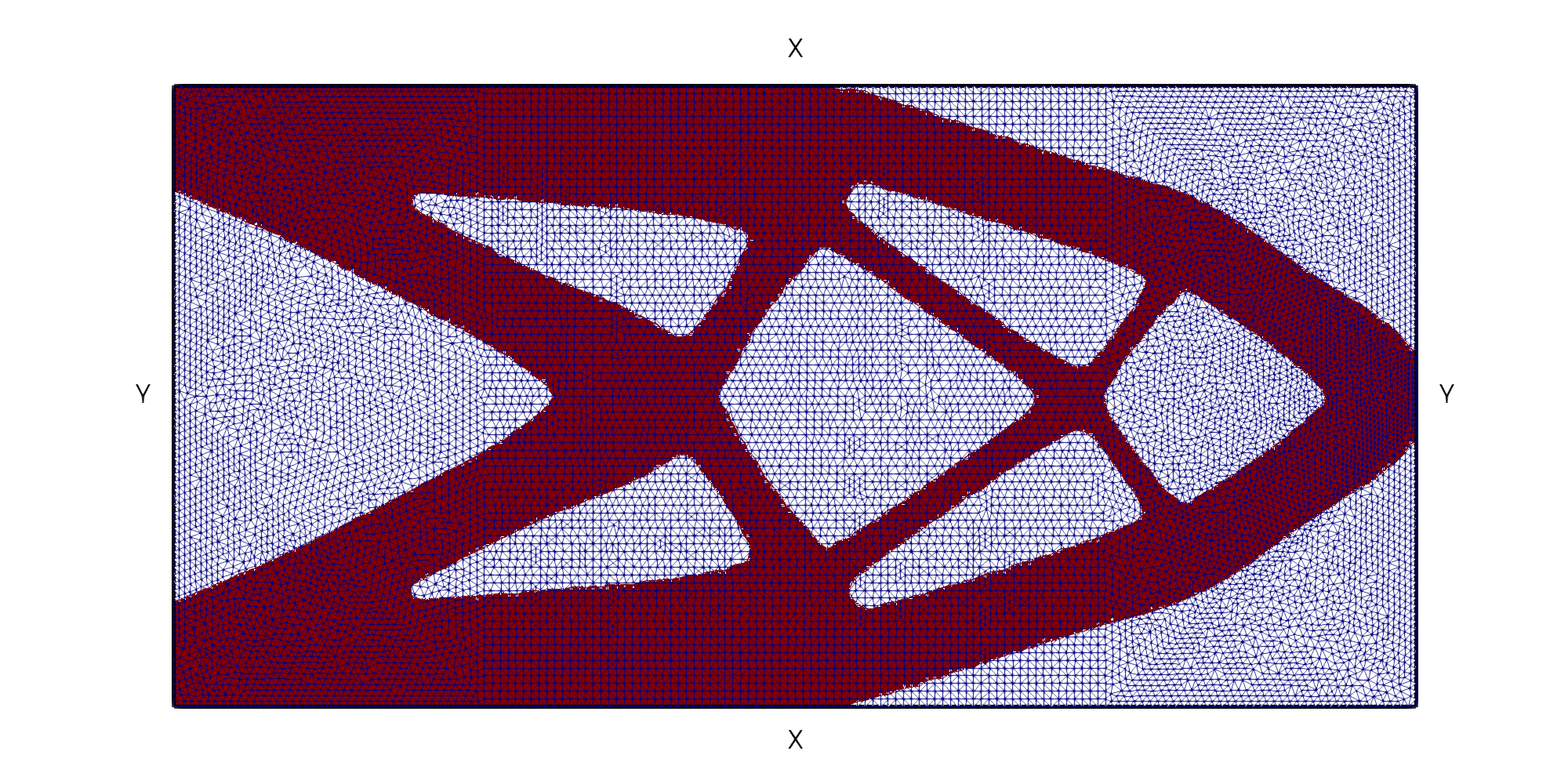}
		\end{subfigure}
        
		\begin{subfigure}[b]{0.350\textwidth}
                \centering			\includegraphics[width=\linewidth]{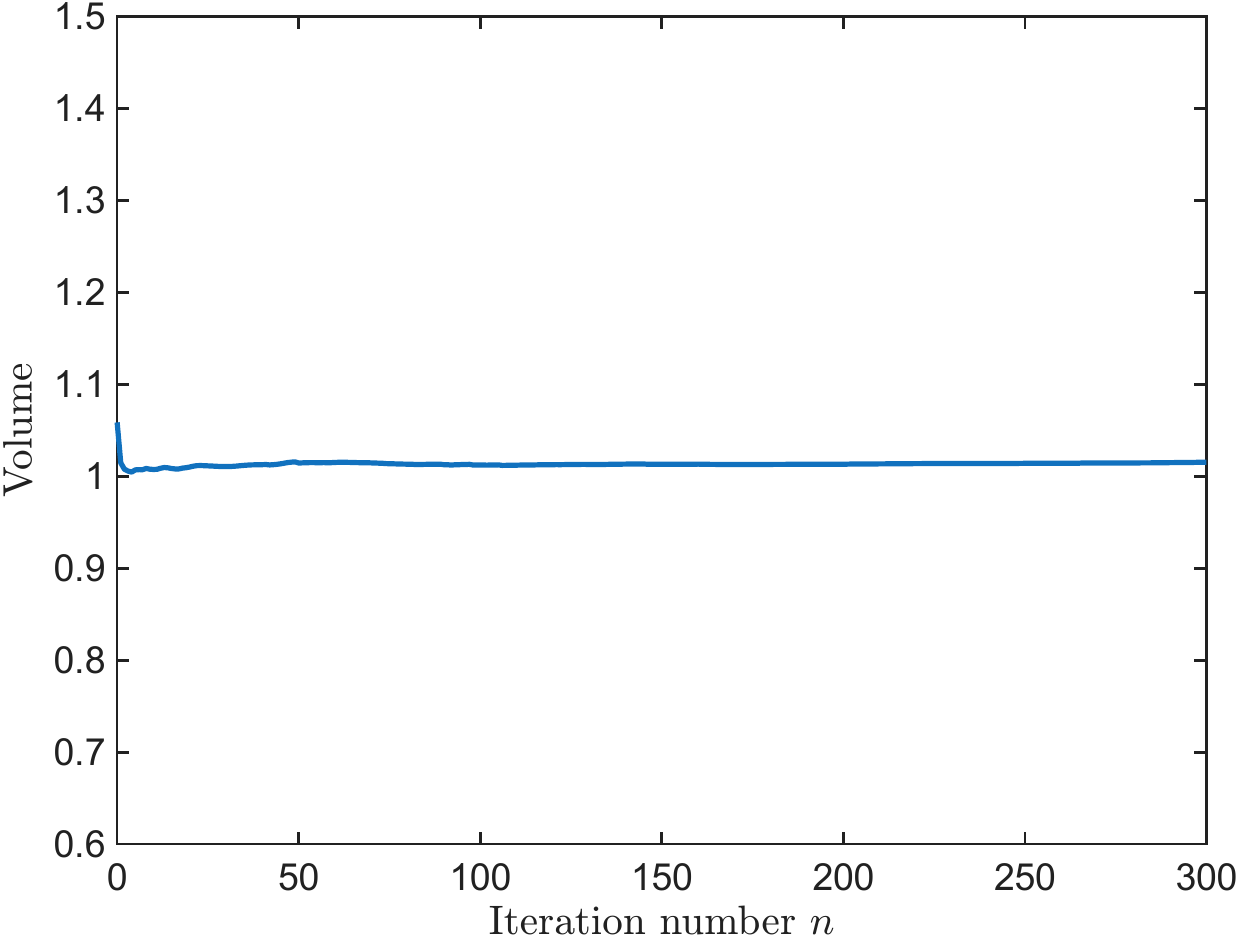} 
		\end{subfigure}
		\begin{subfigure}[b]{0.350\textwidth}			\includegraphics[width=\linewidth]{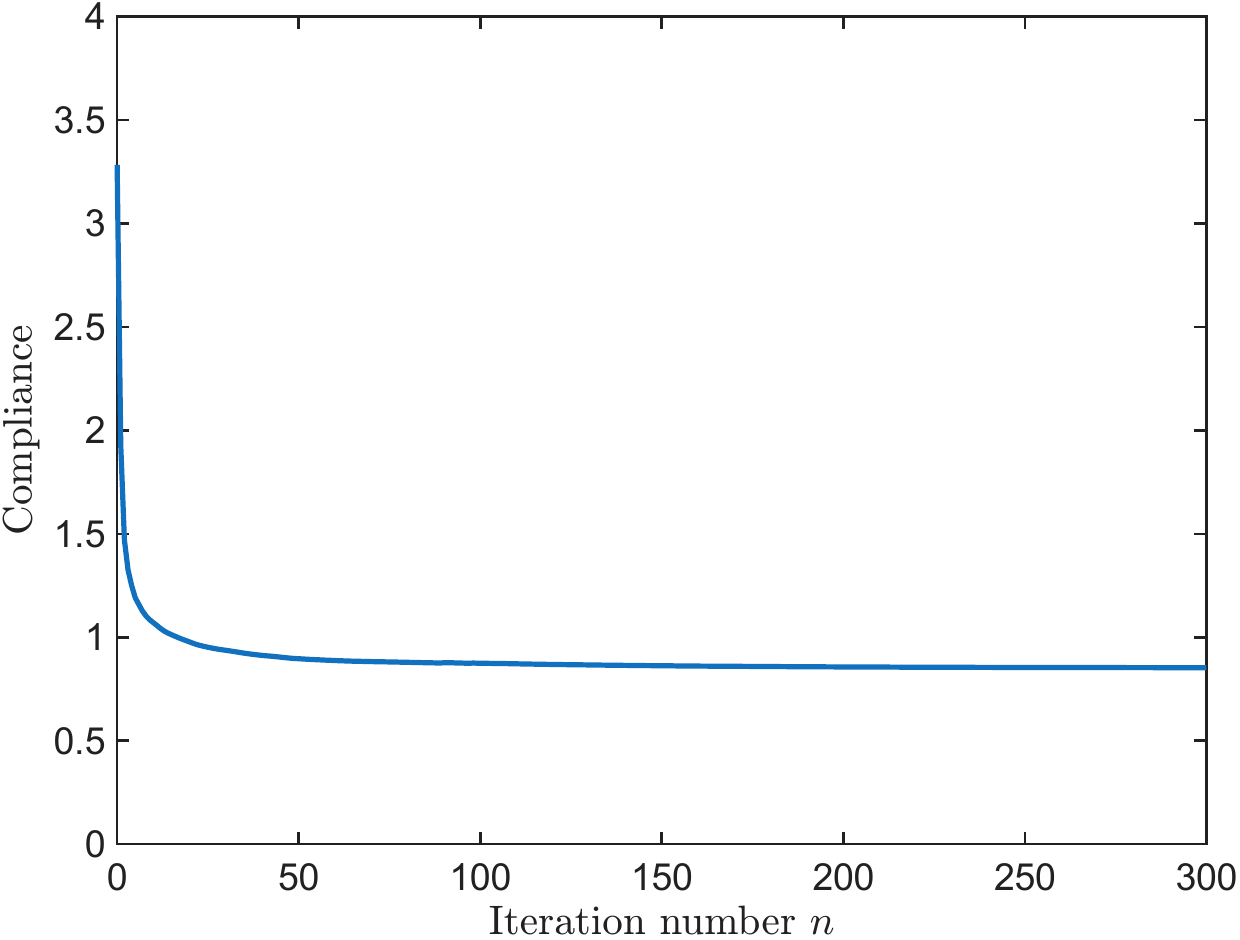}
		\end{subfigure}
		\caption{Evolution of intermediate designs for Example \ref{5_M1E1}: $n = 0,20,300$ from top left to top right and convergence histories of the volume (bottom left) and the compliance (bottom right).}
		\label{Figure_M1E1_All}
	\end{figure}

	\begin{figure}[htbp]
		\centering
		\begin{subfigure}[b]{0.32\textwidth}
			\centering
			\includegraphics[width=\linewidth]{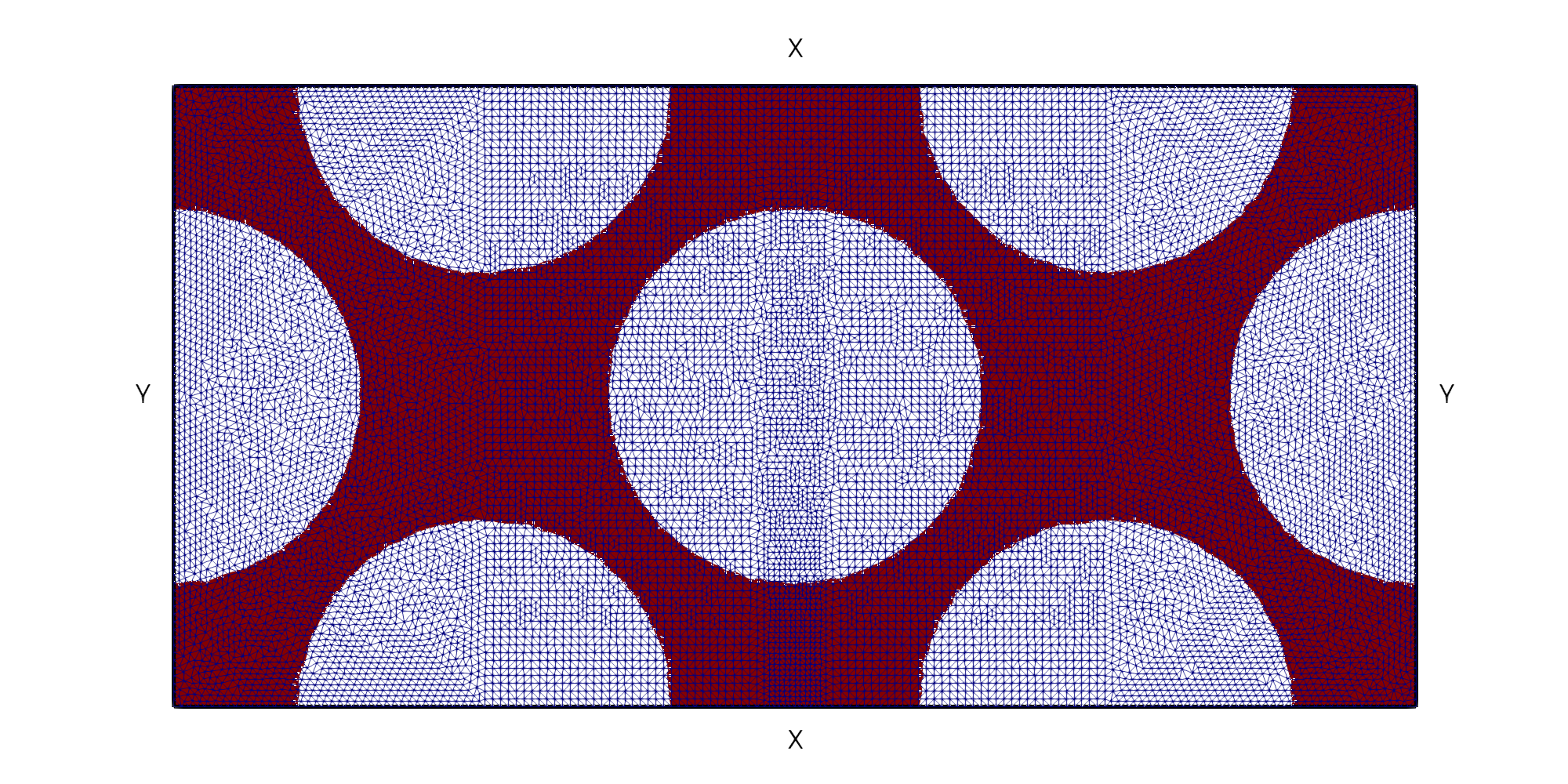} 
		\end{subfigure}
		\hfill 
		\begin{subfigure}[b]{0.32\textwidth}
			\centering
			\includegraphics[width=\linewidth]{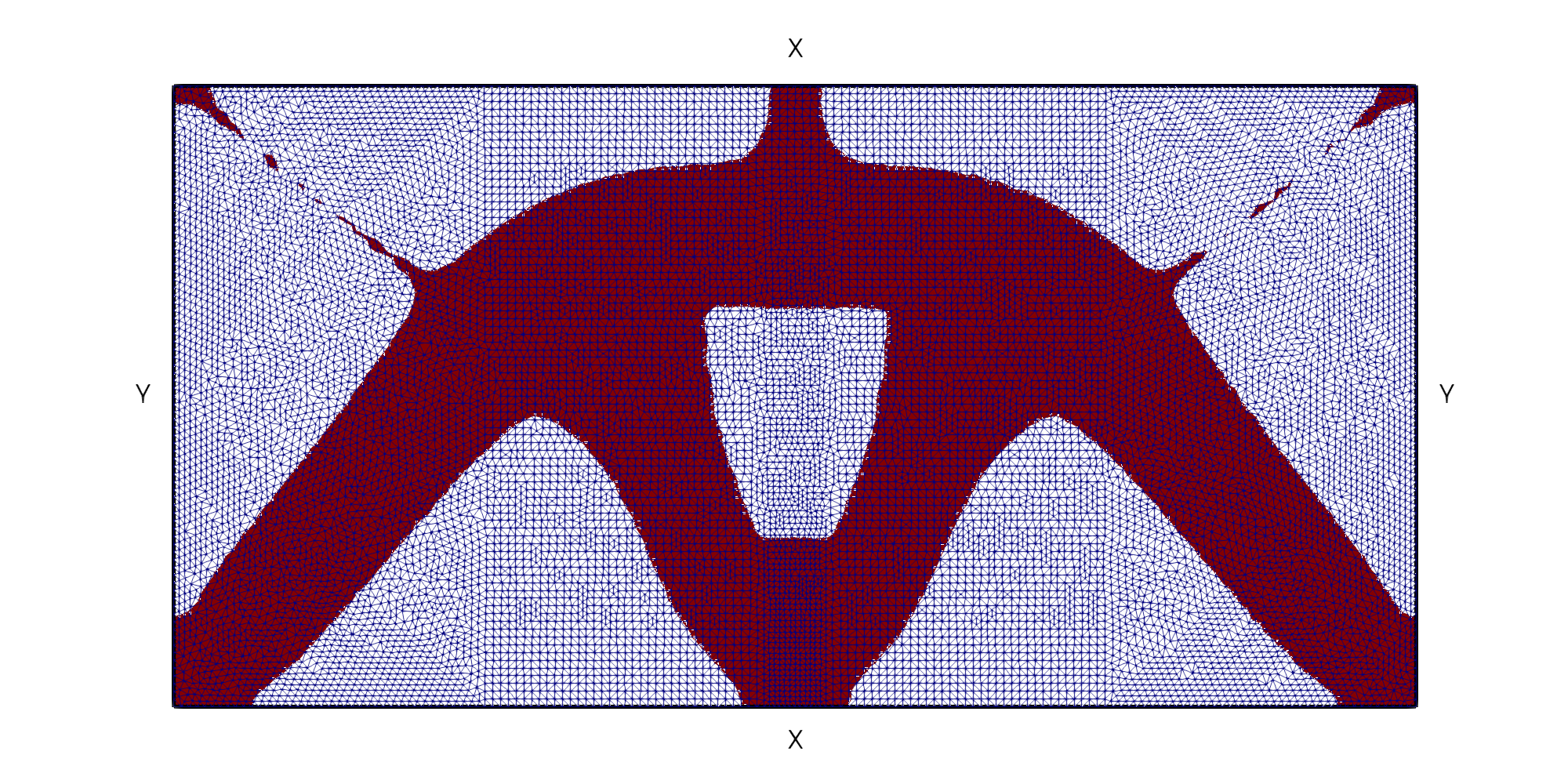}
		\end{subfigure}
		\hfill
		\begin{subfigure}[b]{0.32\textwidth}
			\centering
			\includegraphics[width=\linewidth]{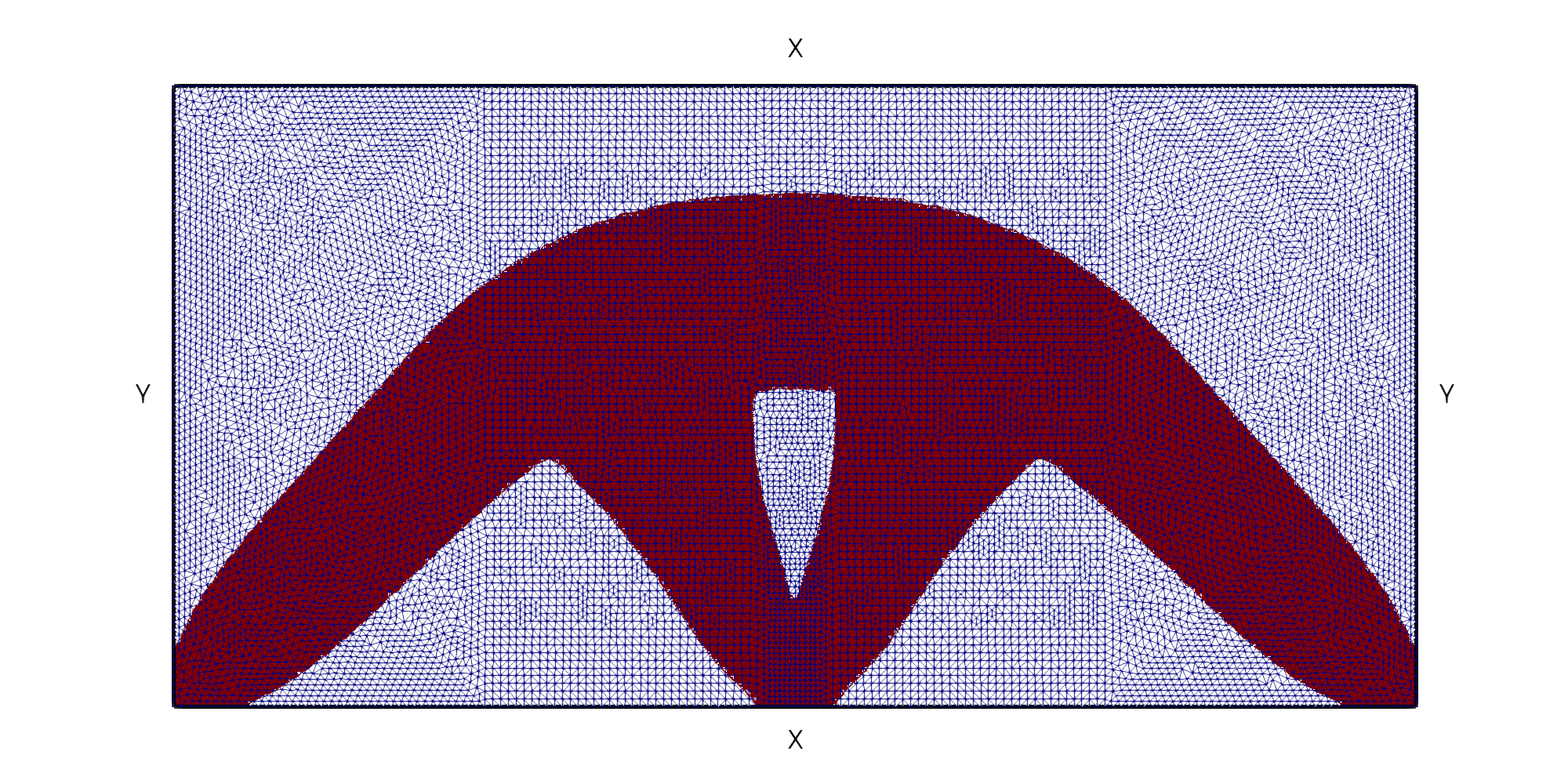}
		\end{subfigure}
		
		\begin{subfigure}[b]{0.36\textwidth}			\includegraphics[width=\linewidth]{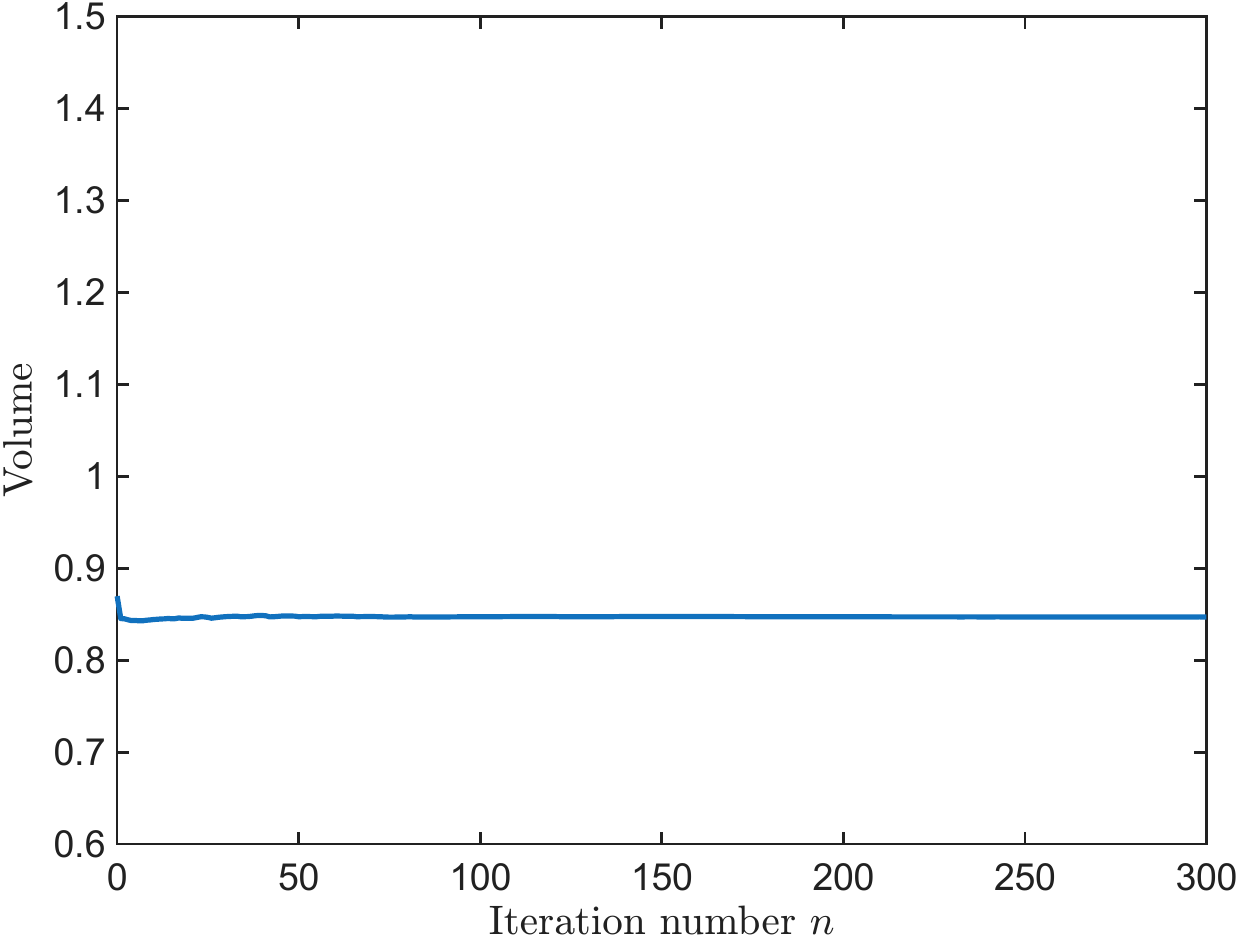} 
		\end{subfigure}
		\begin{subfigure}[b]{0.35\textwidth}			\includegraphics[width=\linewidth]{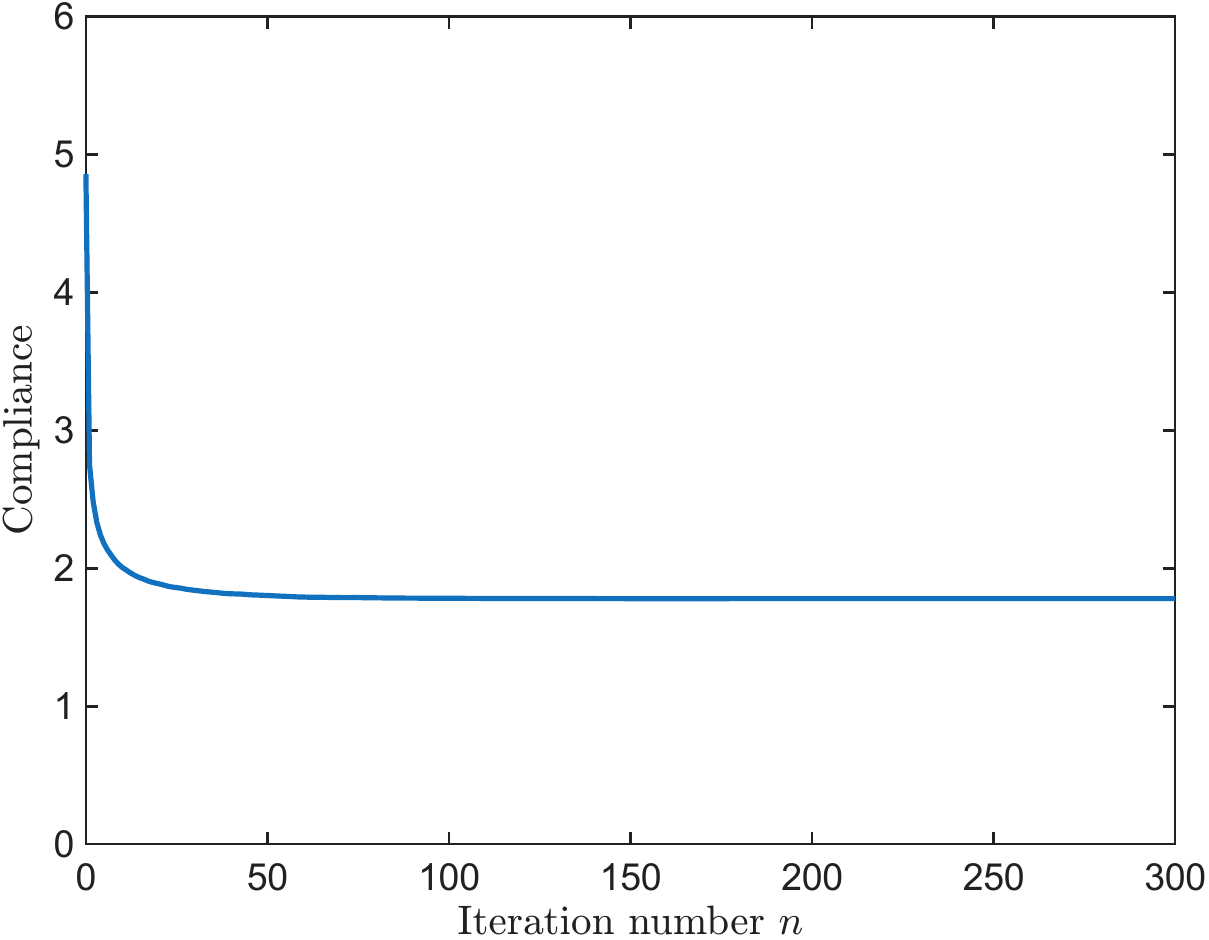}
		\end{subfigure}
		\caption{Evolution of intermediate designs for Example \ref{5_M1E2}: $n = 0,20,300$ from top left to top right and convergence histories of the volume (bottom left) and the compliance (bottom right).}
		\label{Figure_M1E2_All}
	\end{figure}
  
        \vspace*{2ex}
	\begin{example} \label{5_M1E3}	
		(MBB \cite{wang2003level})  Set $D = (0,2) \times (0,1)$, $\partial D_D = \lbrace (x,0): x \in [0, 0.05] \rbrace$, and $\partial D_N = \lbrace (x,0):  x \in [0.95, 1.05] \rbrace$. Set $\bm{u}_y = 0$ on $\partial D^*_D = \lbrace (x,0): x \in [1.95, 2] \rbrace$ (see top left plot of \autoref{Figure_M1E3_Rho}). Set $\bm{g}_s = (0, -20)$. 
	\end{example}

The topological evolution is shown in \autoref{Figure_M1E3_Rho}. \autoref{Figure_M1E3_VandJ} illustrates stable control of the volume around its prescribed target value and a monotonic decrease in the objective.
        
        \begin{figure}[htbp]
        \centering
        \begin{tikzpicture}[scale=2.5] 
            \draw[thick] (0,0) rectangle (2,1);
            \node at (1.0,0.5) {$D$};
            
		  \draw[line width=3pt, blue] (0,0) -- (0.05, 0);
		  \node[right=1pt, blue] at (0, 0.08) {$\partial D_D$};
		  \def\N{2}
		  \def\dx{0.03}
		  \def\dy{0.03}
		  \foreach \i in {0,1,...,\N}
		  {
				\pgfmathsetmacro{\ypos}{0.05*(1 - \i/\N)}
				\draw[thick, blue] (\ypos, 0) -- ++(-\dx, -\dy);
		  }
			
            \draw[line width=3pt, blue] (0,0) -- (0.05, 0);
            \node[right=1pt, blue] at (0, 0.08) {$\partial D_D$};

            \draw[line width=3pt, blue] (2.0,0) -- (1.95, 0);
            \node[left=1pt, blue] at (2.0, 0.08) {$\partial D^*_D$};

            \def\N{2}
            \foreach \i in {0,1,...,\N}
            {
                \pgfmathsetmacro{\ypos}{0.05*(1 - \i/\N)}
                \draw[blue, thick] (2.0-\ypos, -0.025) circle (0.01);
            }

            \draw[line width=3pt, blue] (0.95,0) -- (1.05,0);
            \node[blue] at (1.0,0.075) {$\partial D_N$};

            \tikzset{
                dim/.style={black, line width=0.4pt, |<->|, >=stealth},
                dim label/.style={midway, fill=white, inner sep=1.2pt}
            }

            \draw[<->, dim] (0,-0.1) -- (0.05,-0.1) 
            node[dim label] {0.05};

            \draw[<->, dim] (0.9,-0.1) -- (1.1,-0.1) 
            node[dim label] {0.1};

            \draw[<->, dim] (1.95,-0.1) -- (2.0,-0.1) 
            node[dim label] {0.05};

            \draw[<->, dim] (0,1.1) -- (2.0,1.1) 
            node[dim label] {2};

            \draw[<->, dim] (2.35,0) -- (2.35,1.0) 
            node[dim label] {1};

            \draw[<->, dim] (-0.16,0) -- (-0.16,0.05) 
            node[dim label] {0.05};

            \draw[<->, dim] (2.16,0) -- (2.16,0.05) 
            node[dim label] {0.05};
        \end{tikzpicture}
        \begin{subfigure}[b]{0.4\textwidth}
			\includegraphics[width=\linewidth]{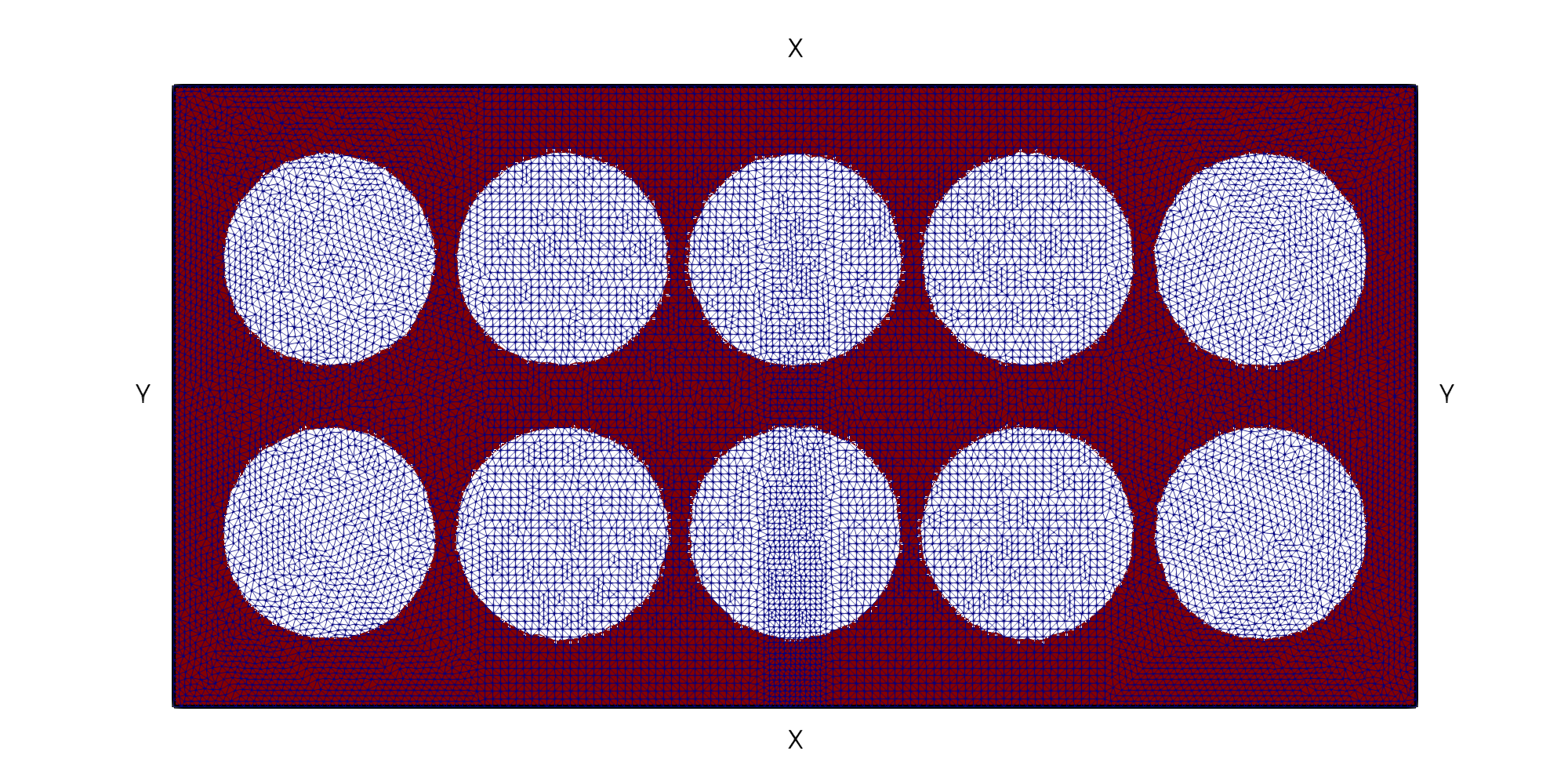}
		\end{subfigure}
        
        \begin{subfigure}[b]{0.4\textwidth}
			\includegraphics[width=\linewidth]{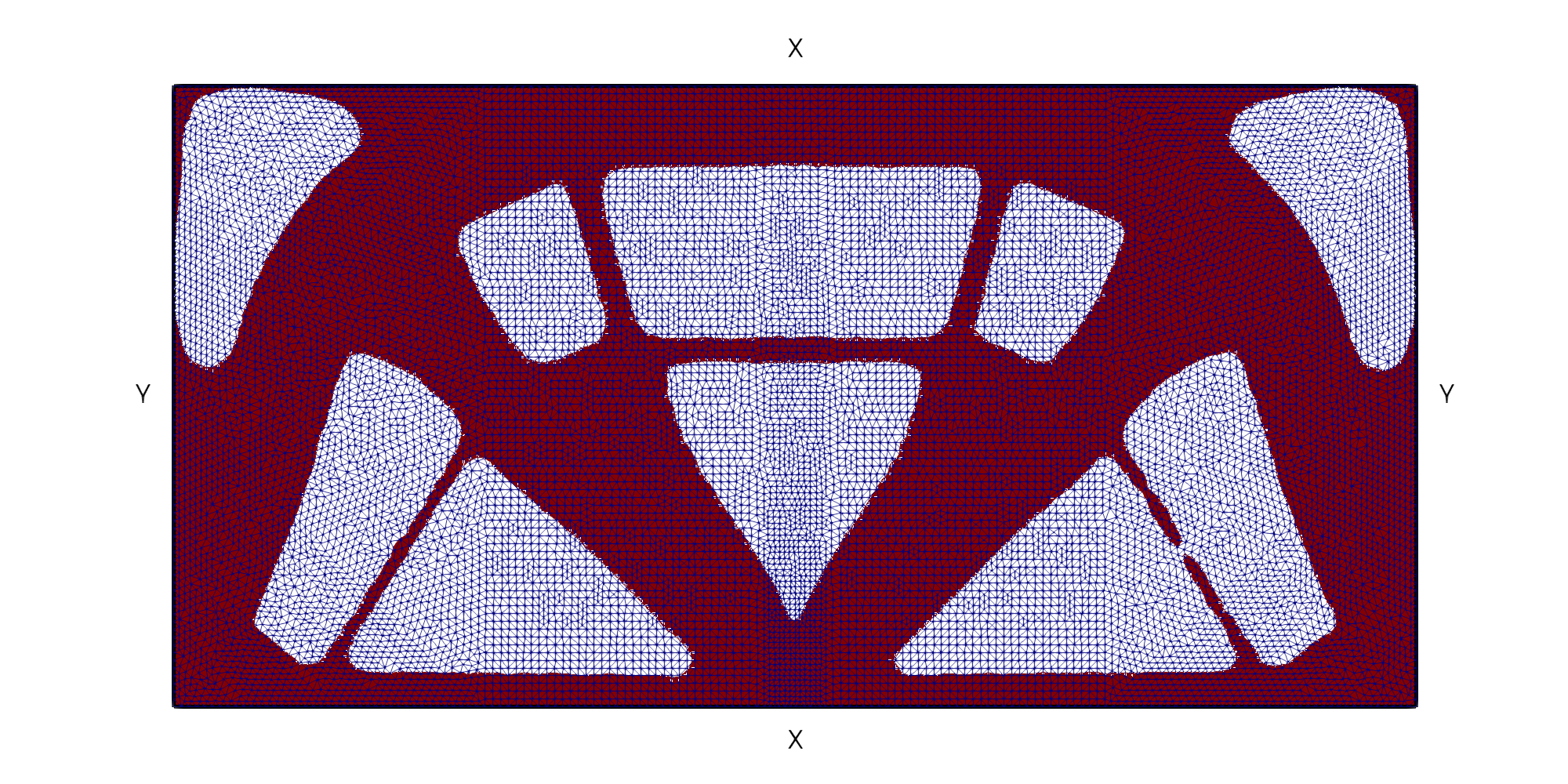}
		\end{subfigure}
        \hspace{0.04\textwidth}
        \begin{subfigure}[b]{0.4\textwidth}
			\includegraphics[width=\linewidth]{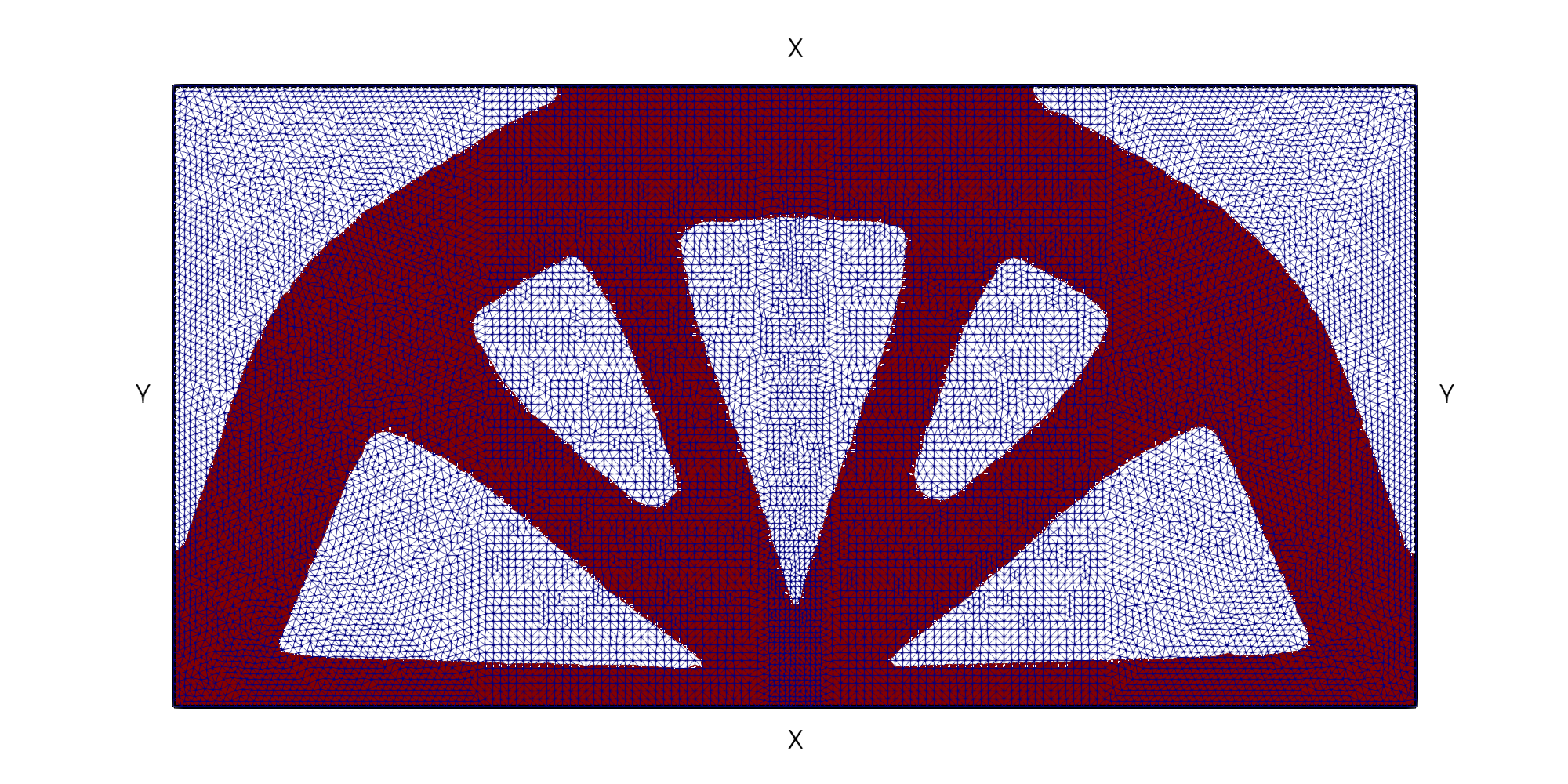}
		\end{subfigure}       
        \caption{Example \ref{5_M1E3}: The design domain $D$ and its boundary (top left), initial design (top right), intermediate design for $n = 250$ (bottom left) and final design (bottom right).}
        \label{Figure_M1E3_Rho}
        \end{figure}
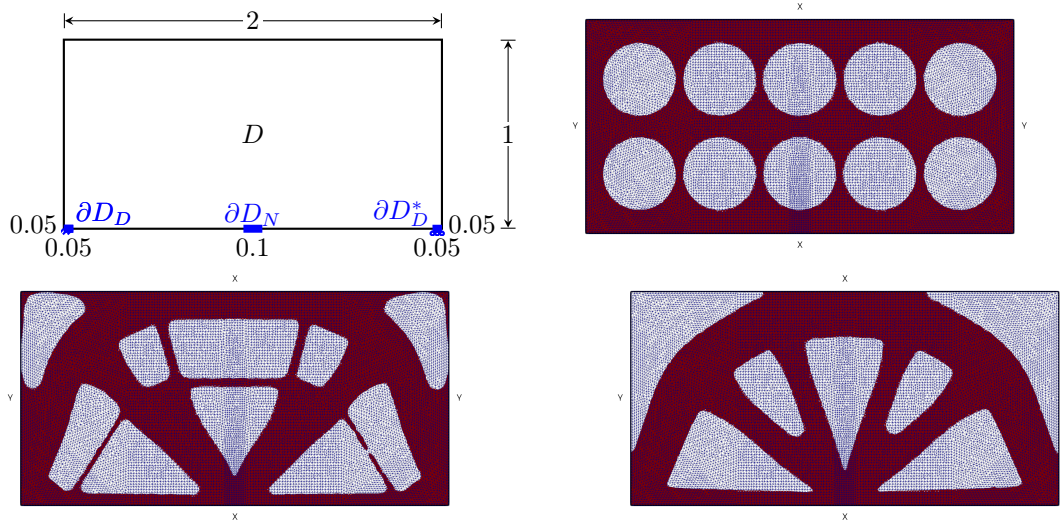

        \begin{figure}[htbp]
		\centering
		\begin{subfigure}[b]{0.363\textwidth}			\includegraphics[width=\linewidth]{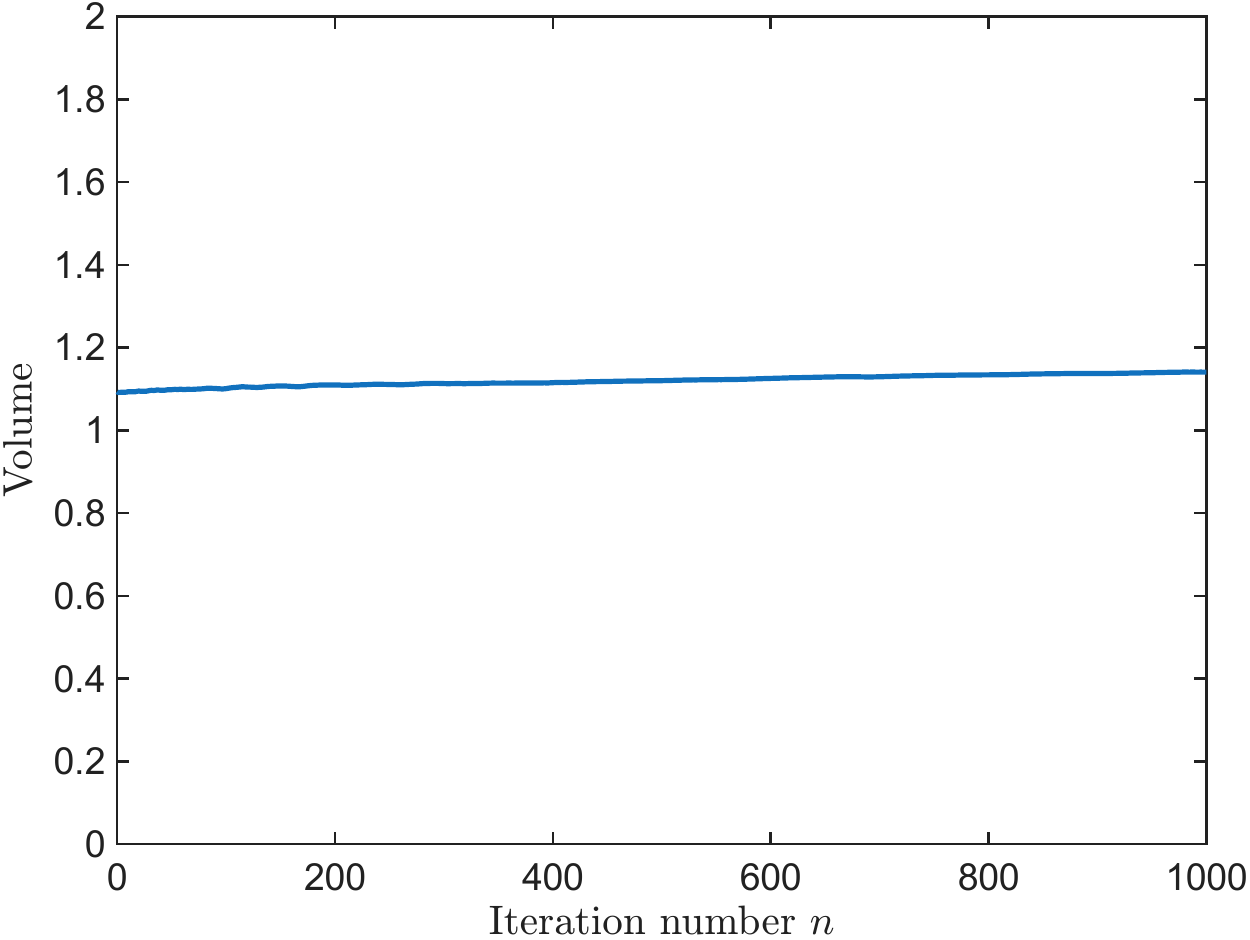} 
		\end{subfigure}
		\begin{subfigure}[b]{0.359\textwidth}			\includegraphics[width=\linewidth]{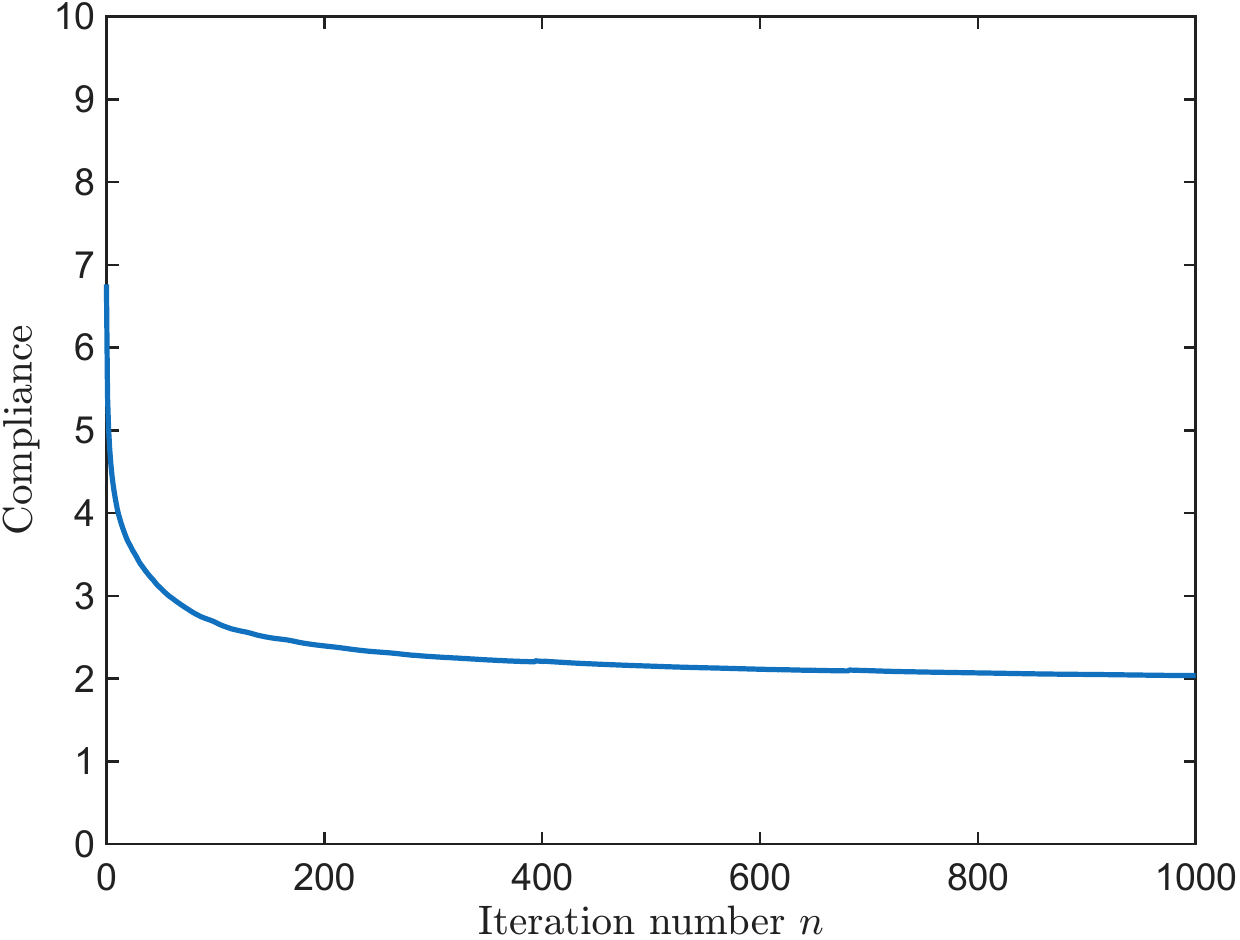}
		\end{subfigure}
		\caption{Convergence histories of the volume (left) and the compliance (right) for Example \ref{5_M1E3}.}
		\label{Figure_M1E3_VandJ}
	\end{figure}
        
	\begin{example}
		\label{5_M4E1}
		(Cantilever: 3D \cite{allaire2004structural}) Set $D = (0,1) \times (0,2) \times (0,1)$, $\partial D_D = \lbrace (x,0,z): x,z \in [0,1] \rbrace$ and $\partial D_N = \lbrace (x,1,z): (x-0.5)^2+(z-0.5)^2 \leqslant 0.01 \rbrace$. Set $\bm{g}_s = (0, 0, -100\pi)$ at the center of the right face of $D$. 
	\end{example}
    
	 \autoref{Figure_M4E1_Rho} shows the volume-control behavior and the decrease in energy during optimization for both cases. 
	
	\begin{figure}[htbp]
		\centering
		\begin{subfigure}[b]{0.325\textwidth}
			\includegraphics[width=\linewidth]{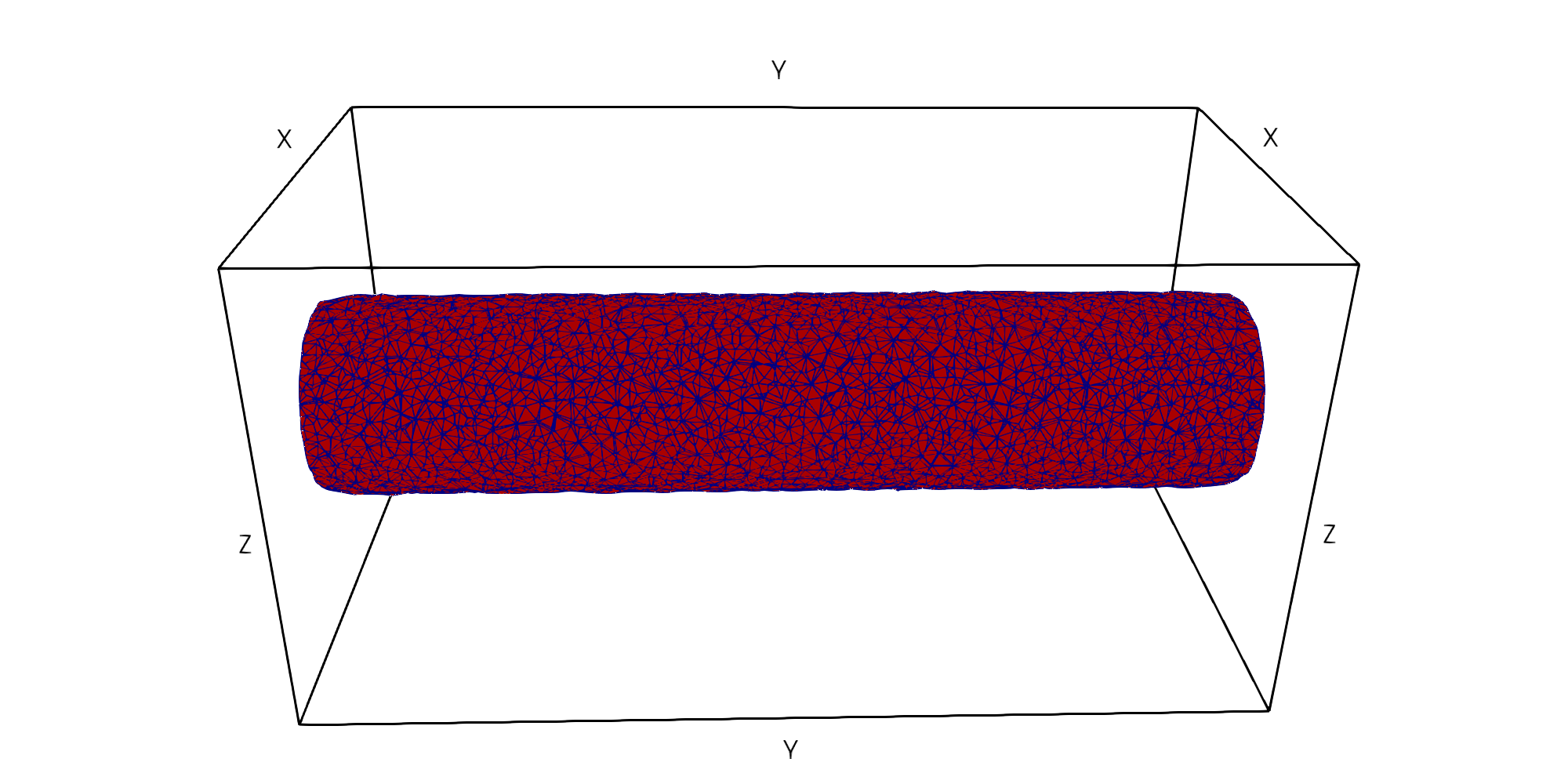} 
		\end{subfigure}
		\hfill 
		\begin{subfigure}[b]{0.325\textwidth}
			\includegraphics[width=\linewidth]{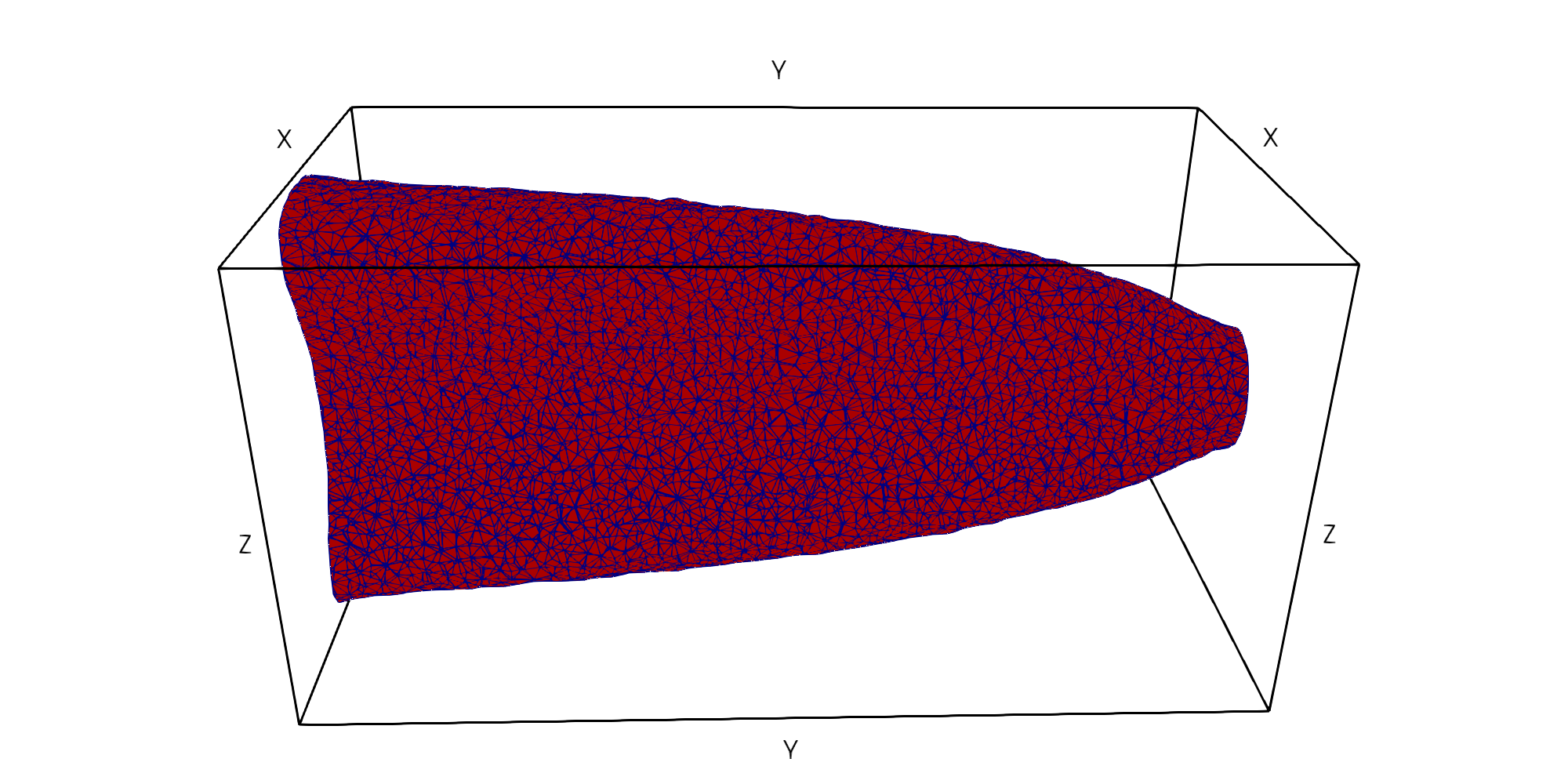}
		\end{subfigure}
		\hfill
		\begin{subfigure}[b]{0.325\textwidth}
			\includegraphics[width=\linewidth]{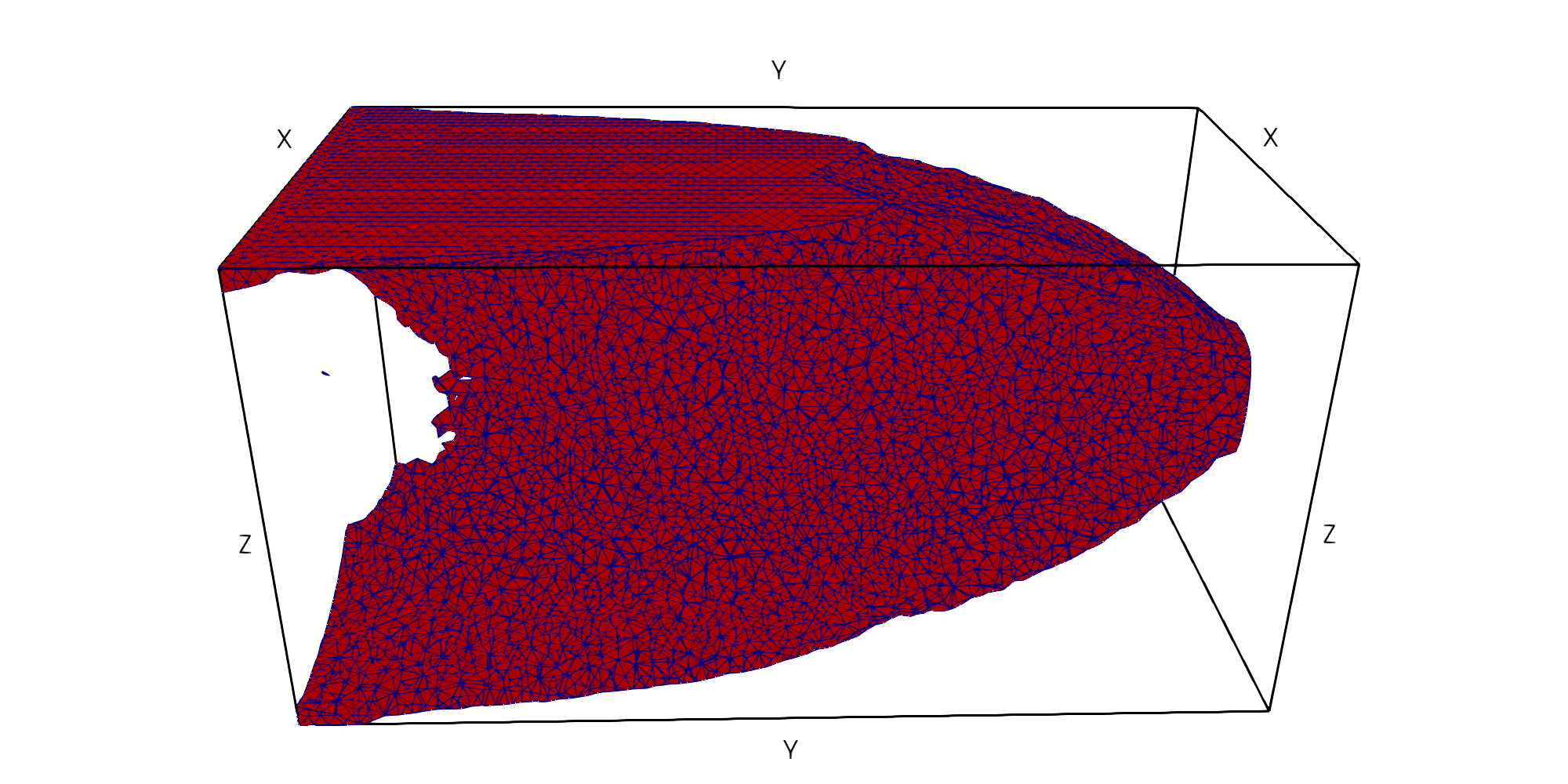}
		\end{subfigure}
       
		\begin{subfigure}[b]{0.325\textwidth}
			\includegraphics[width=\linewidth]{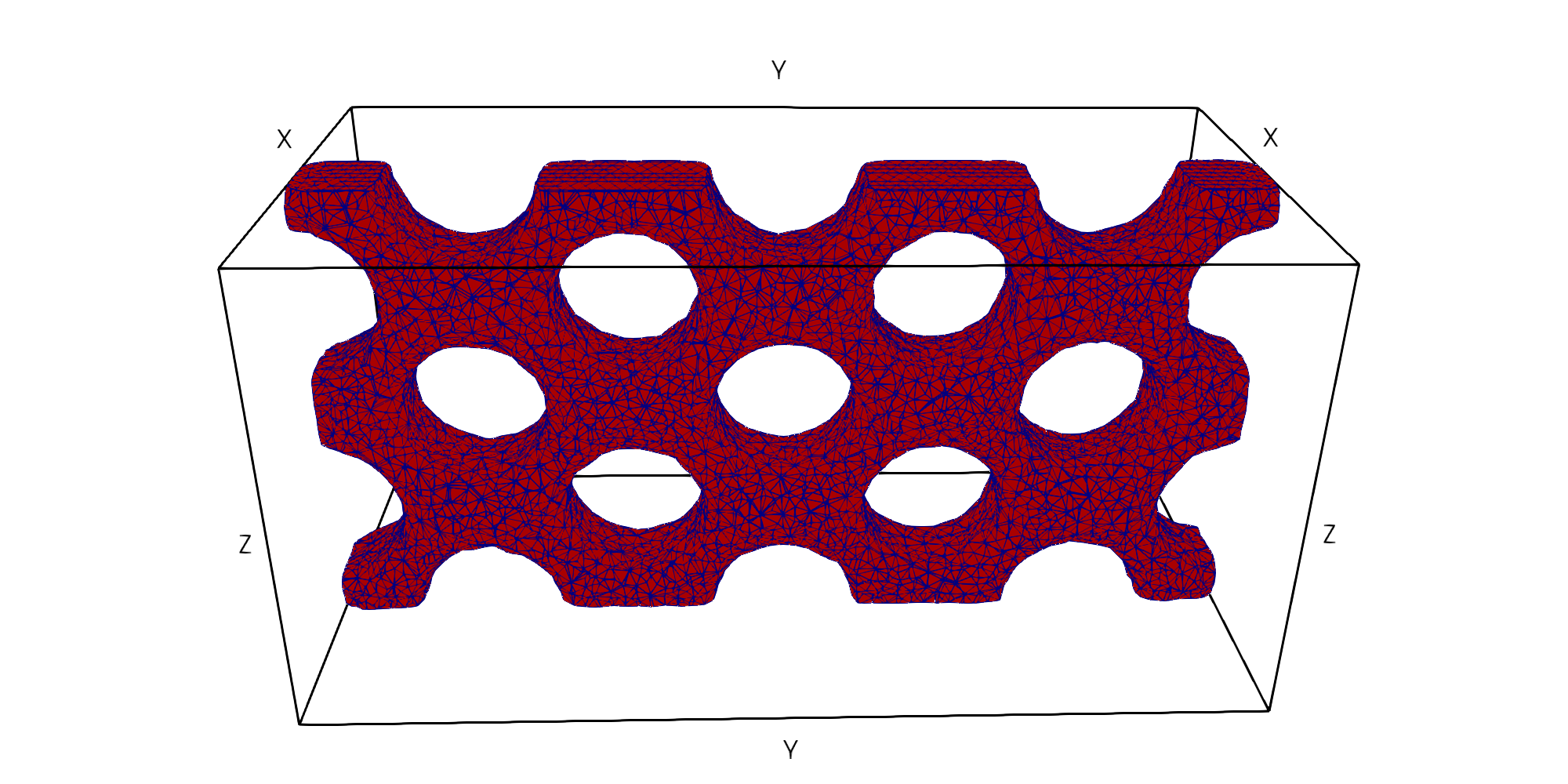} 
		\end{subfigure}
		\hfill 
		\begin{subfigure}[b]{0.325\textwidth}
			\includegraphics[width=\linewidth]{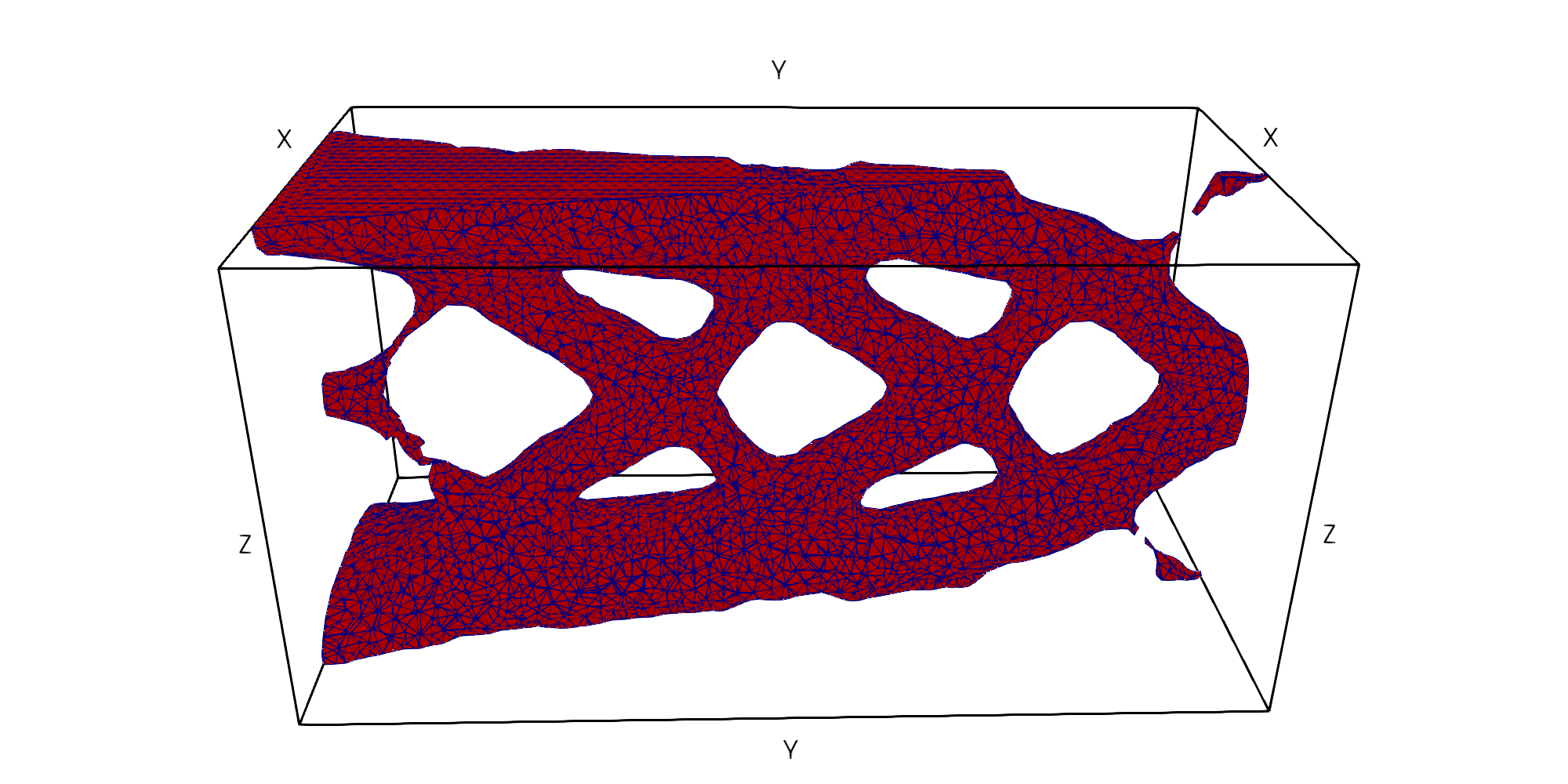}
		\end{subfigure}
		\hfill
		\begin{subfigure}[b]{0.325\textwidth}
			\includegraphics[width=\linewidth]{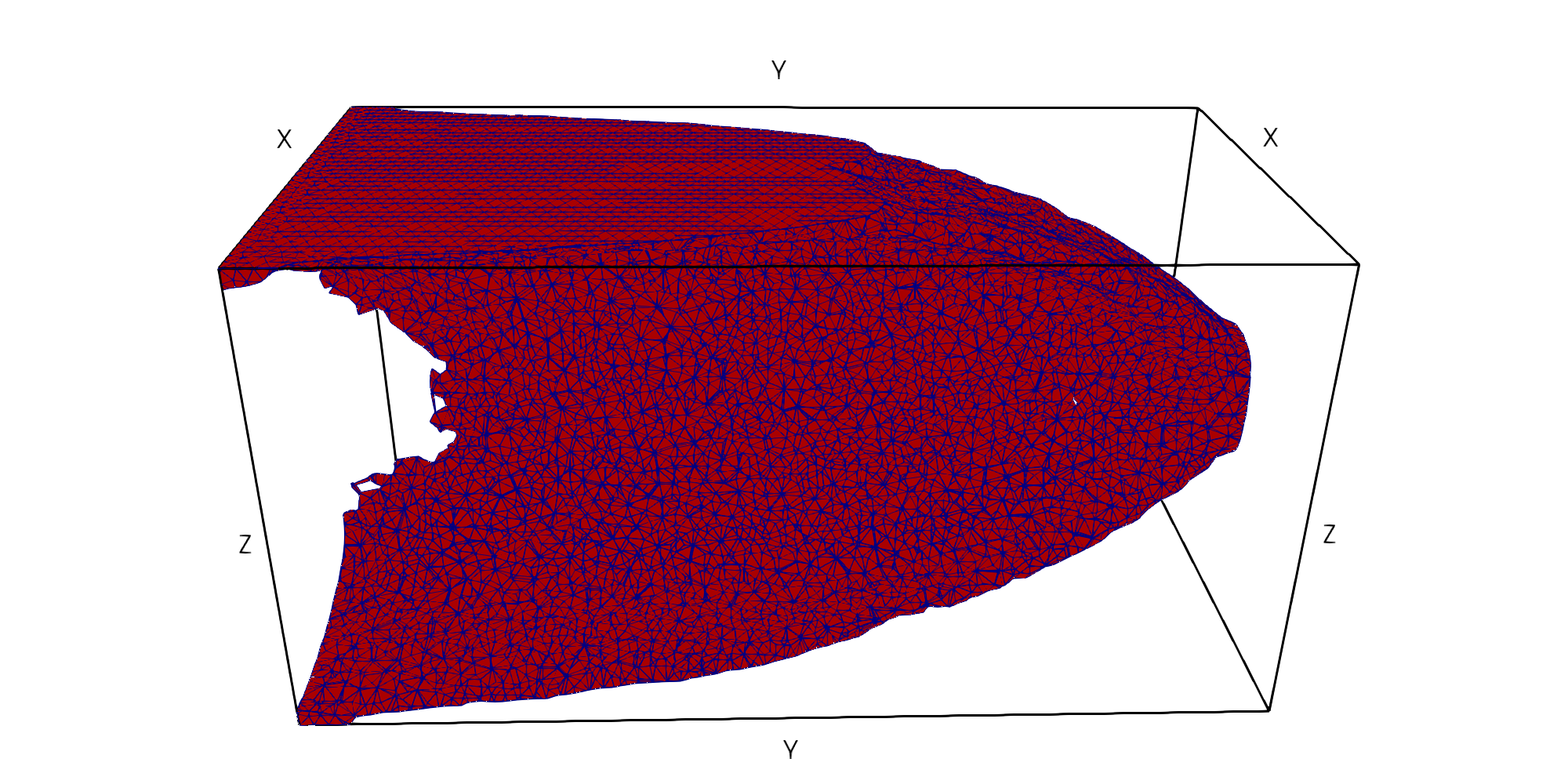}
		\end{subfigure}
     	\hfill   
        	\begin{subfigure}[b]{0.371\textwidth}			\includegraphics[width=\linewidth]{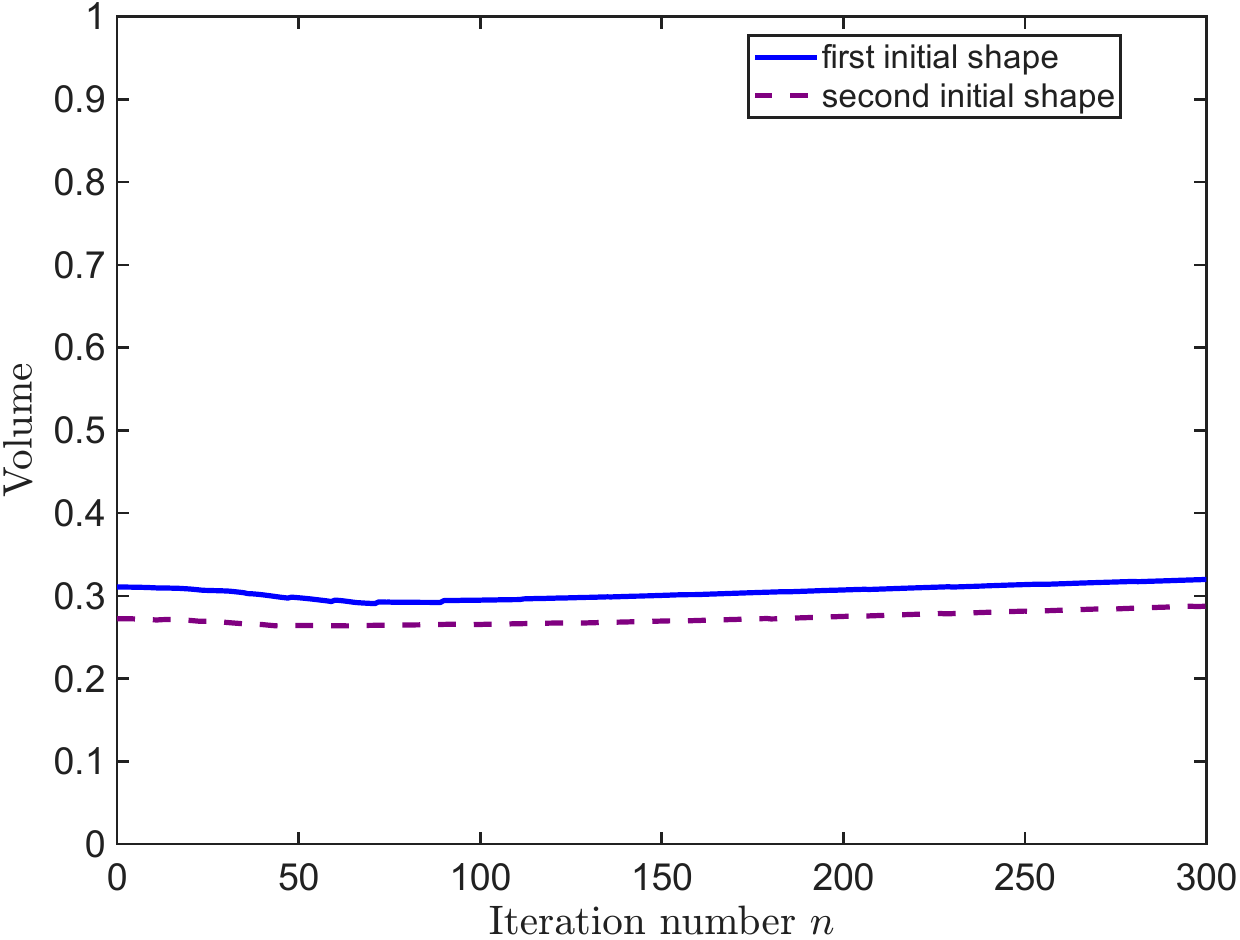} 
		\end{subfigure}
		\begin{subfigure}[b]{0.379\textwidth}			\includegraphics[width=\linewidth]{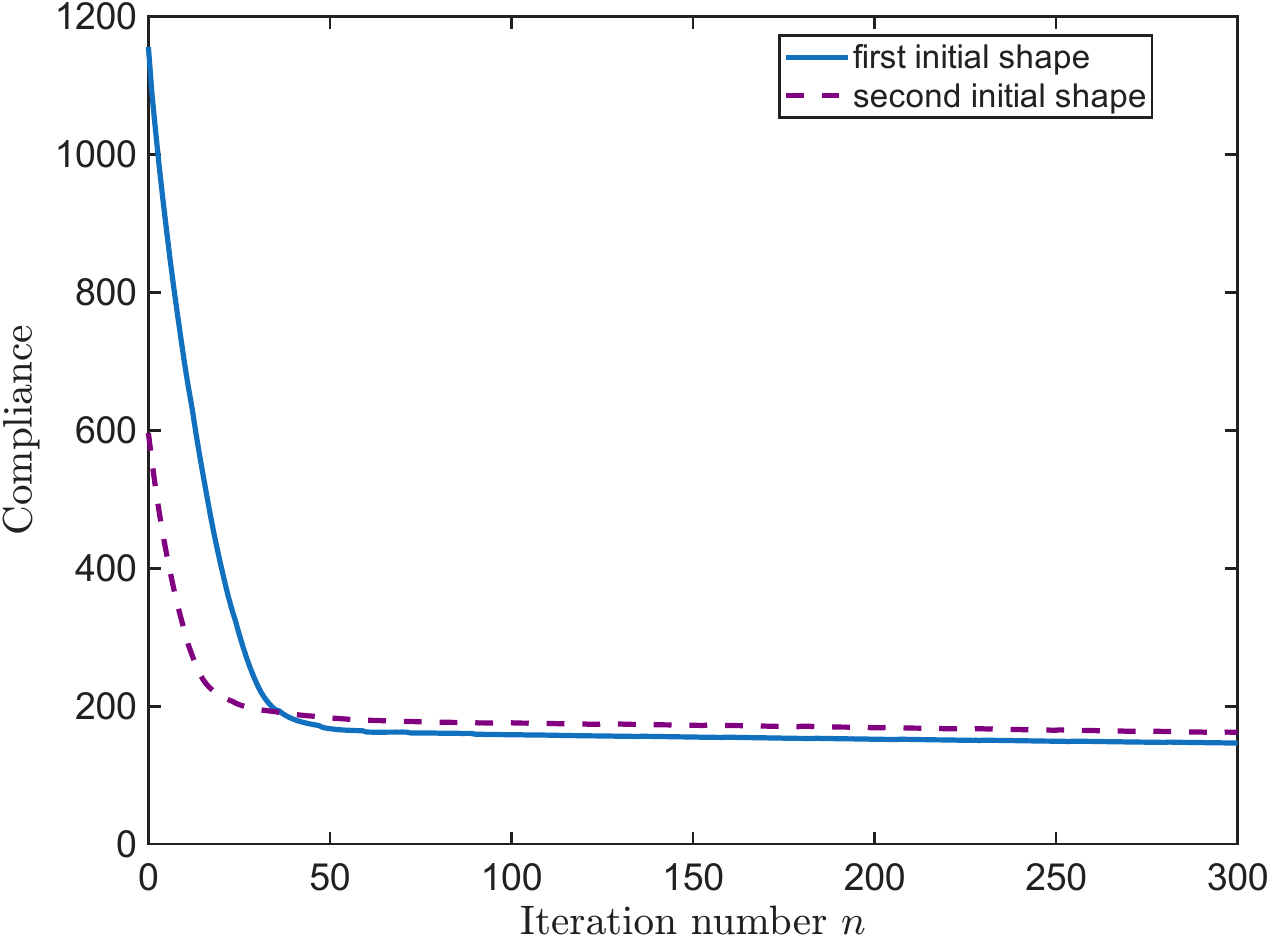}
		\end{subfigure}
		\caption{Evolutions of two initial designs for Example \ref{5_M4E1}: $n = 0,30,300$ (from left to right of the first two lines) and convergence histories of the volume (left) and the compliance (right) of line 3.}
		\label{Figure_M4E1_Rho}
	\end{figure}
	
    	
	\begin{example}
		\label{5_M4E2}
		(Bridge: 3D \cite{allaire2004structural}) Set $D = (0,1) \times (0,2) \times (0,1)$, $\partial D_D = \lbrace (x,0,0): x \in [0,1] \rbrace$ and  $\partial D_N = \lbrace (x,y,0): (x-0.5)^2+(y-1)^2 \leqslant 0.01 \rbrace$. Set $\bm{g}_s = (0, 0, -100\pi)$ at the center of the bottom face of $D$. 
	\end{example}
    
	\autoref{Figure_M4E2_Rho} shows small variations in the approximate volume and compliance, respectively, for two different initial shapes. In each case, the volume remains close to its prescribed target value, and the objective decays stably.
	
	\begin{figure}[htbp]
		\centering
		\begin{subfigure}[b]{0.325\textwidth}
			\includegraphics[width=\linewidth]{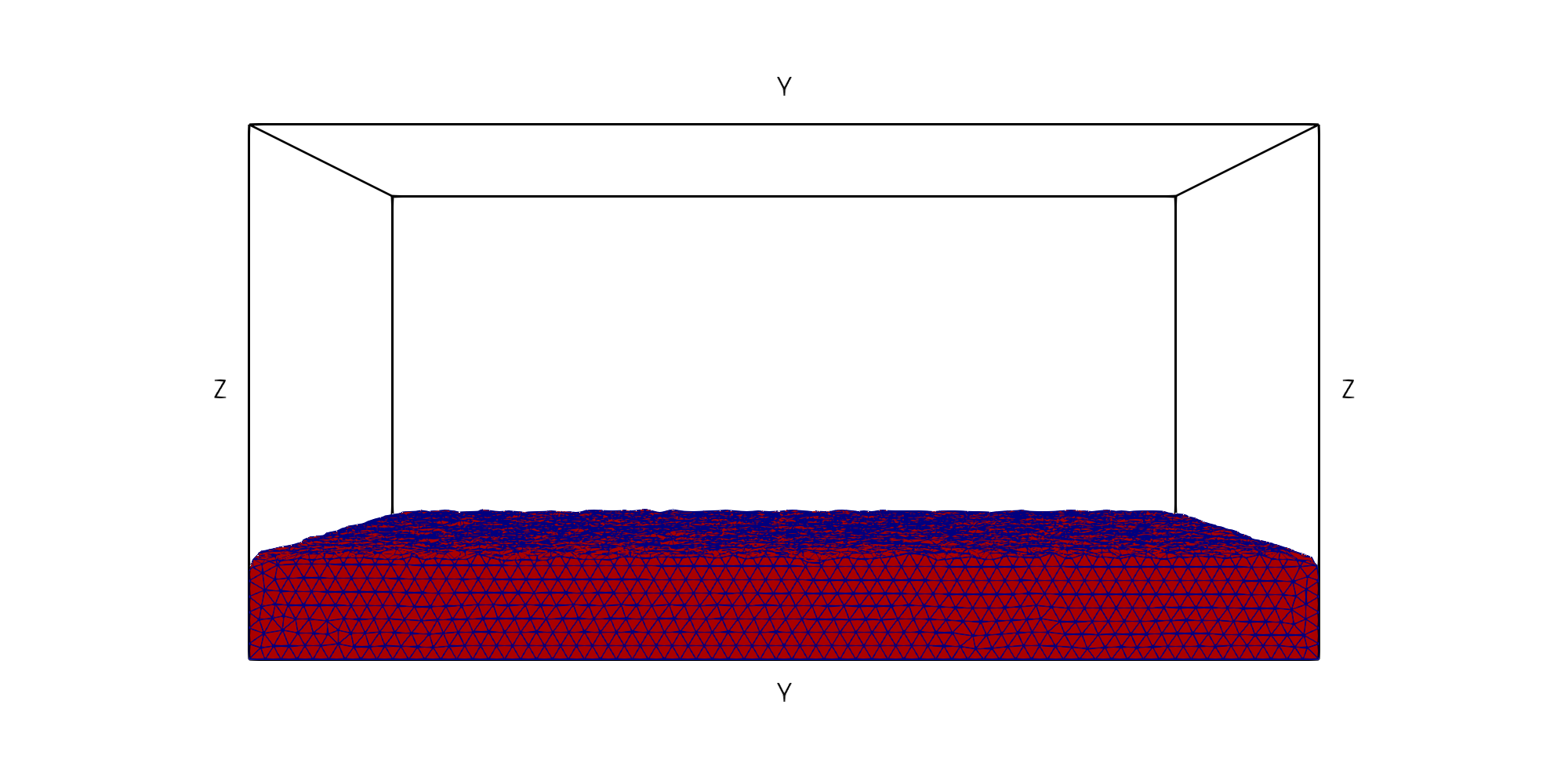} 
		\end{subfigure}
		\hfill 
		\begin{subfigure}[b]{0.325\textwidth}
			\includegraphics[width=\linewidth]{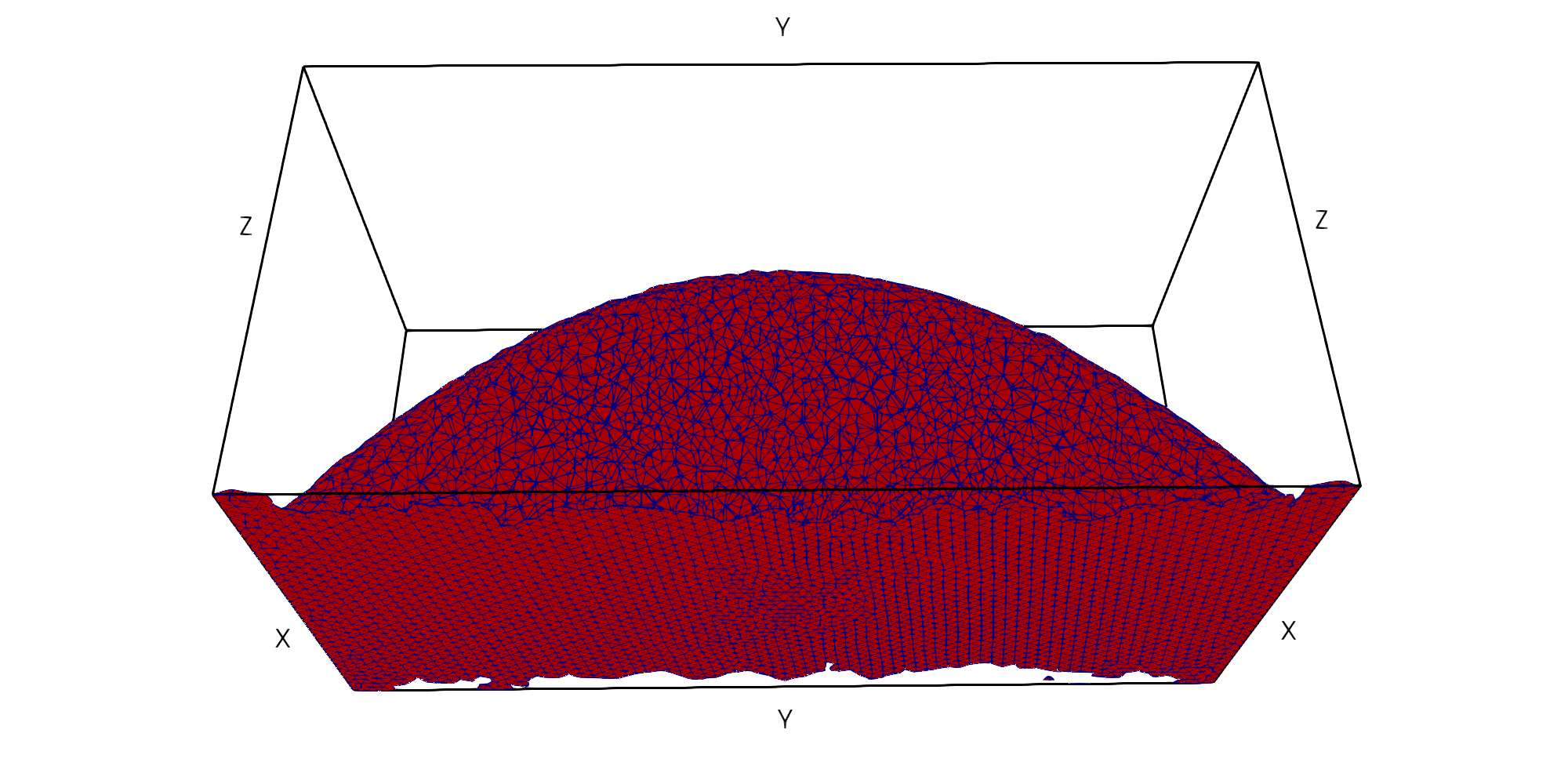}
		\end{subfigure}
		\hfill
		\begin{subfigure}[b]{0.325\textwidth}
			\includegraphics[width=\linewidth]{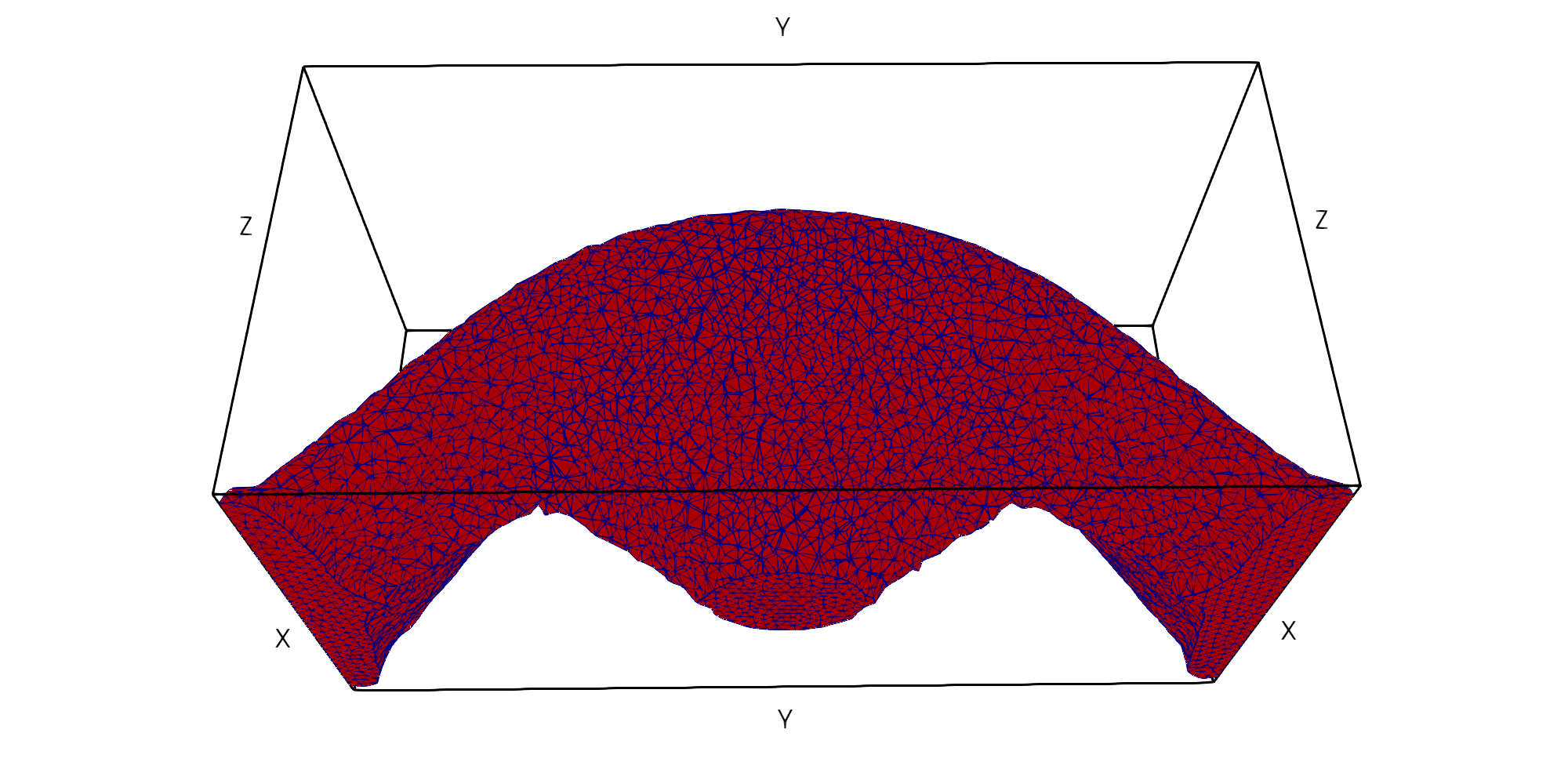}
		\end{subfigure}
		\hfill
				\begin{subfigure}[b]{0.325\textwidth}
			\includegraphics[width=\linewidth]{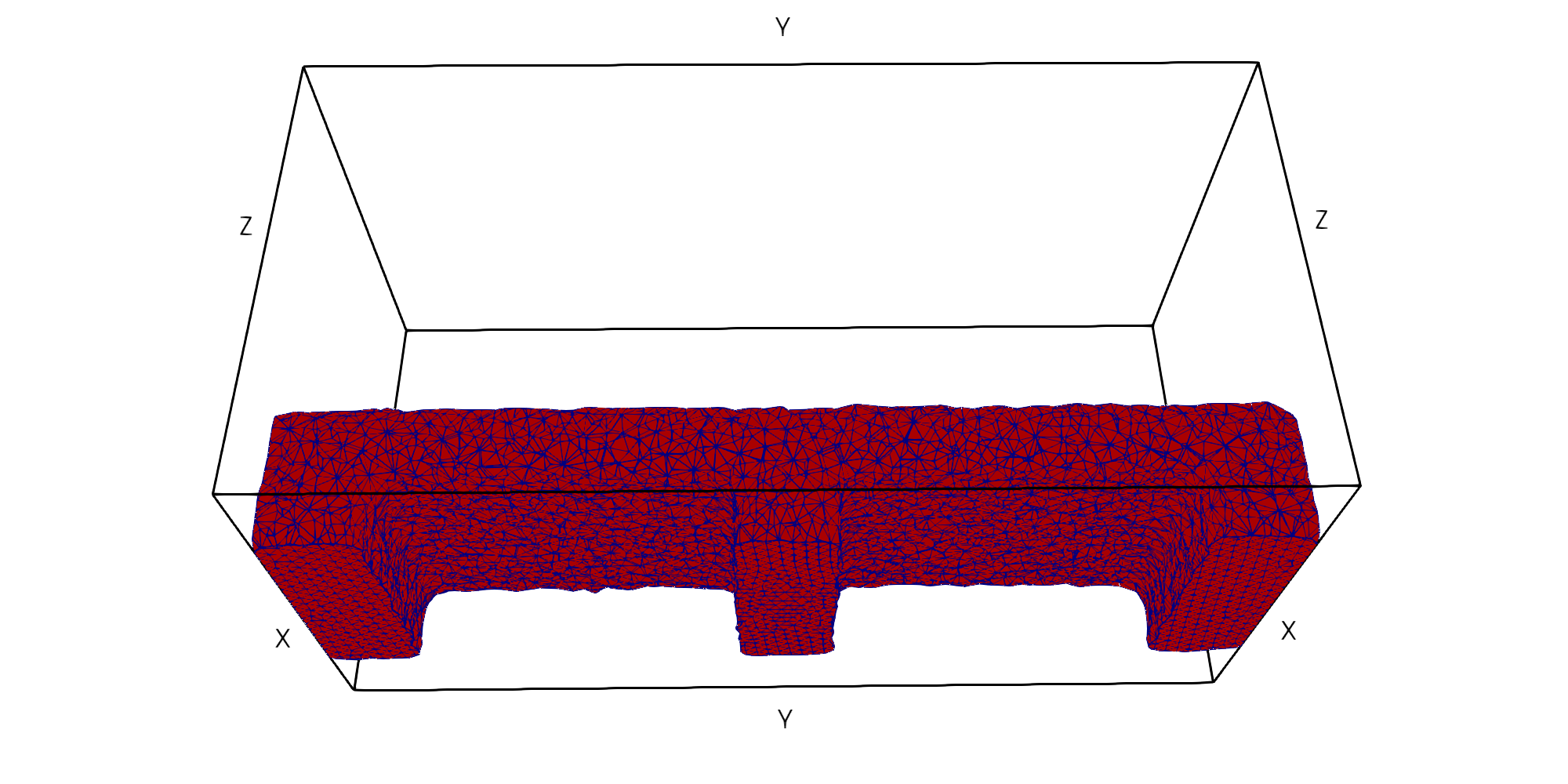} 
		\end{subfigure}
		\hfill 
		\begin{subfigure}[b]{0.325\textwidth}
			\includegraphics[width=\linewidth]{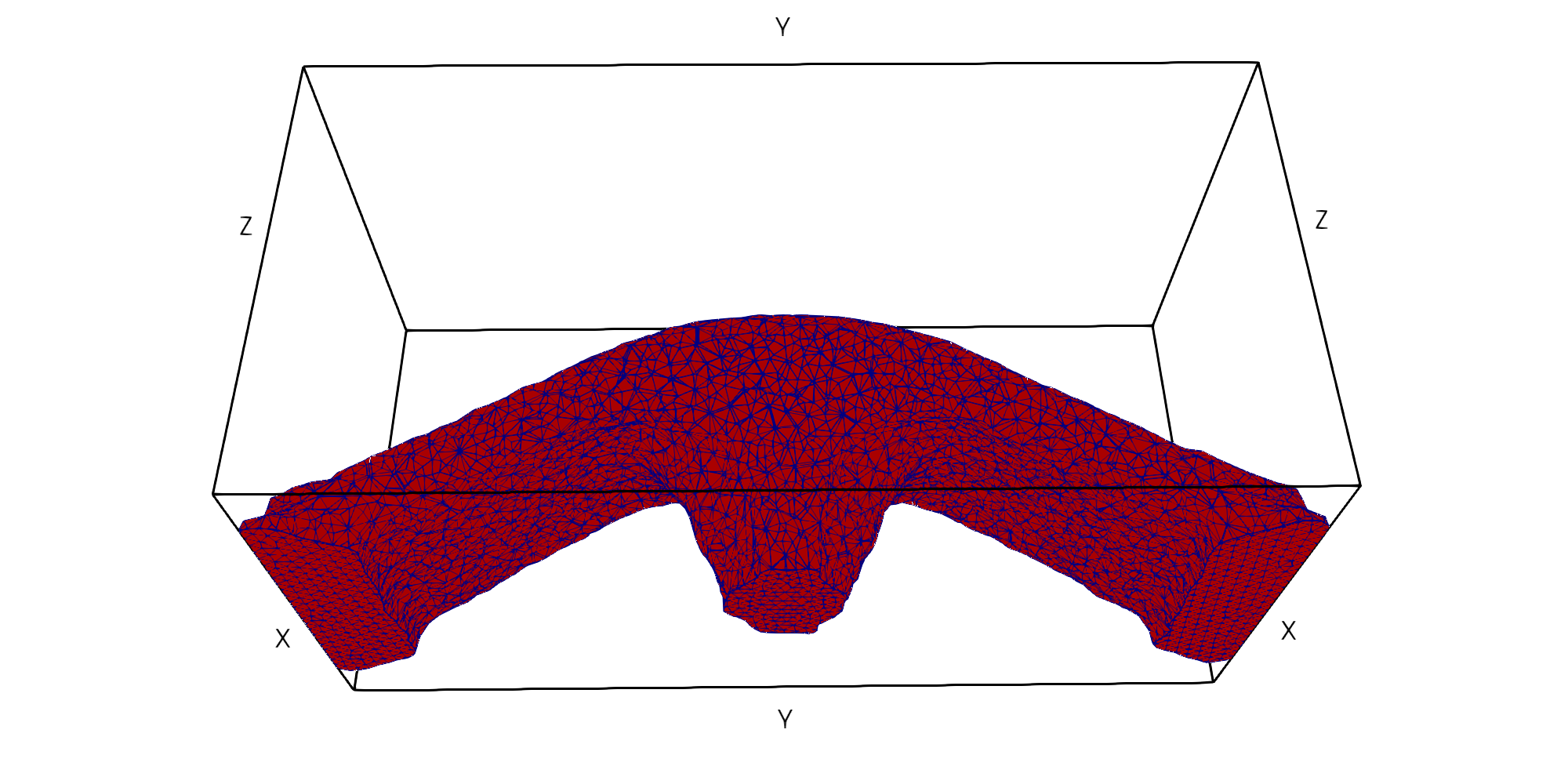}
		\end{subfigure}
		\hfill
		\begin{subfigure}[b]{0.325\textwidth}
			\includegraphics[width=\linewidth]{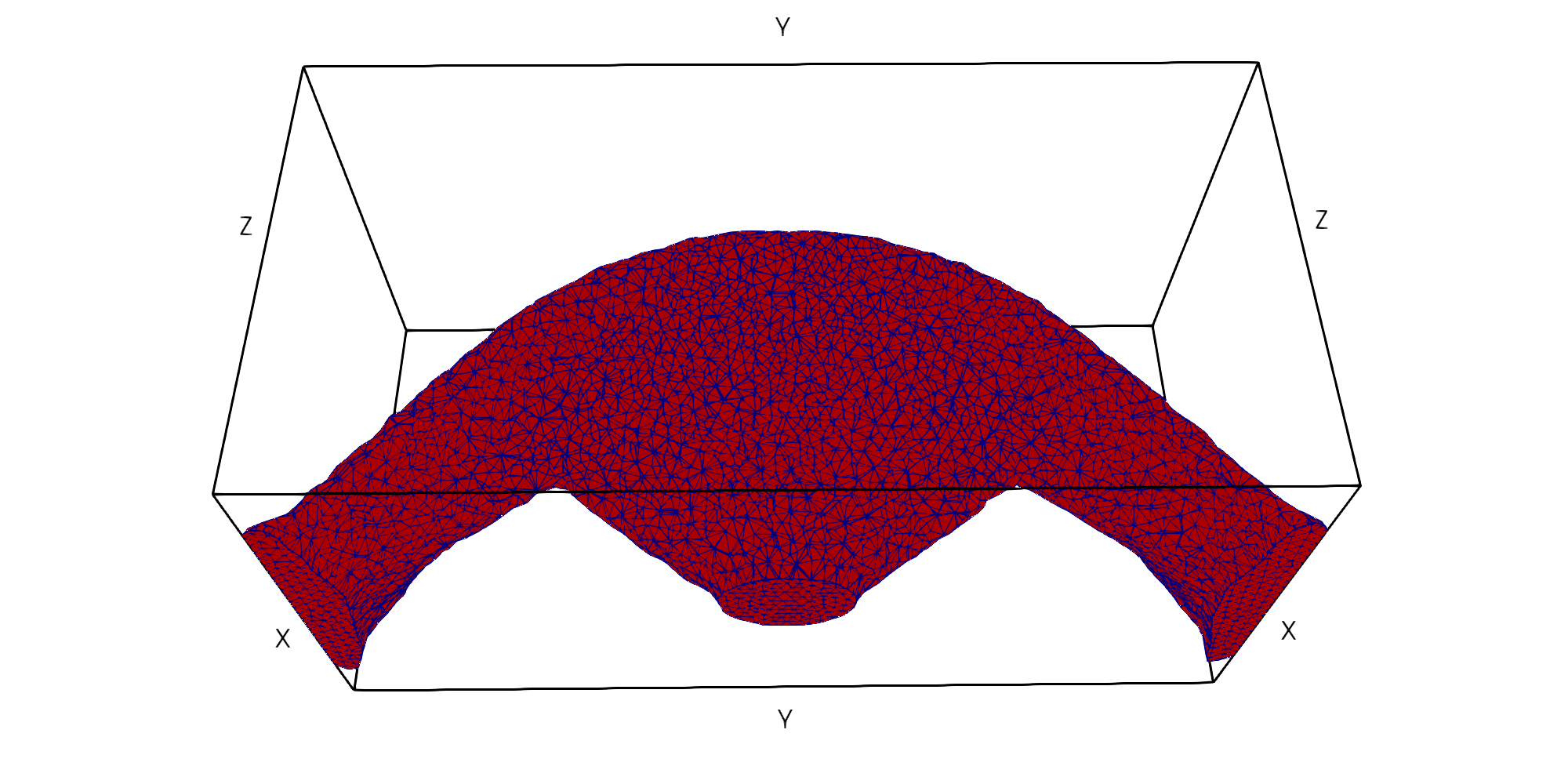}
		\end{subfigure}
        		\begin{subfigure}[b]{0.359\textwidth}			\includegraphics[width=\linewidth]{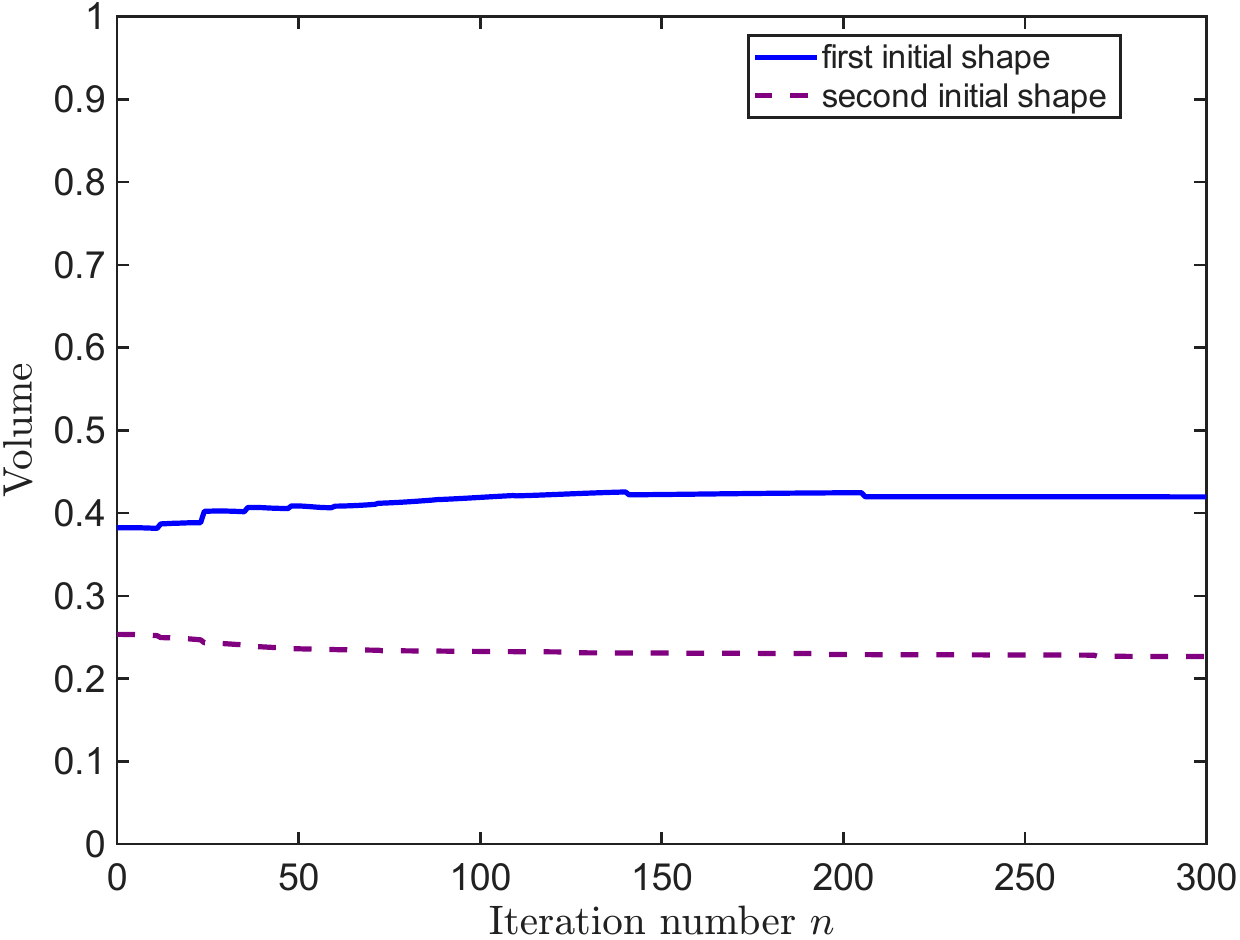} 
		\end{subfigure}
		\begin{subfigure}[b]{0.361\textwidth}			\includegraphics[width=\linewidth]{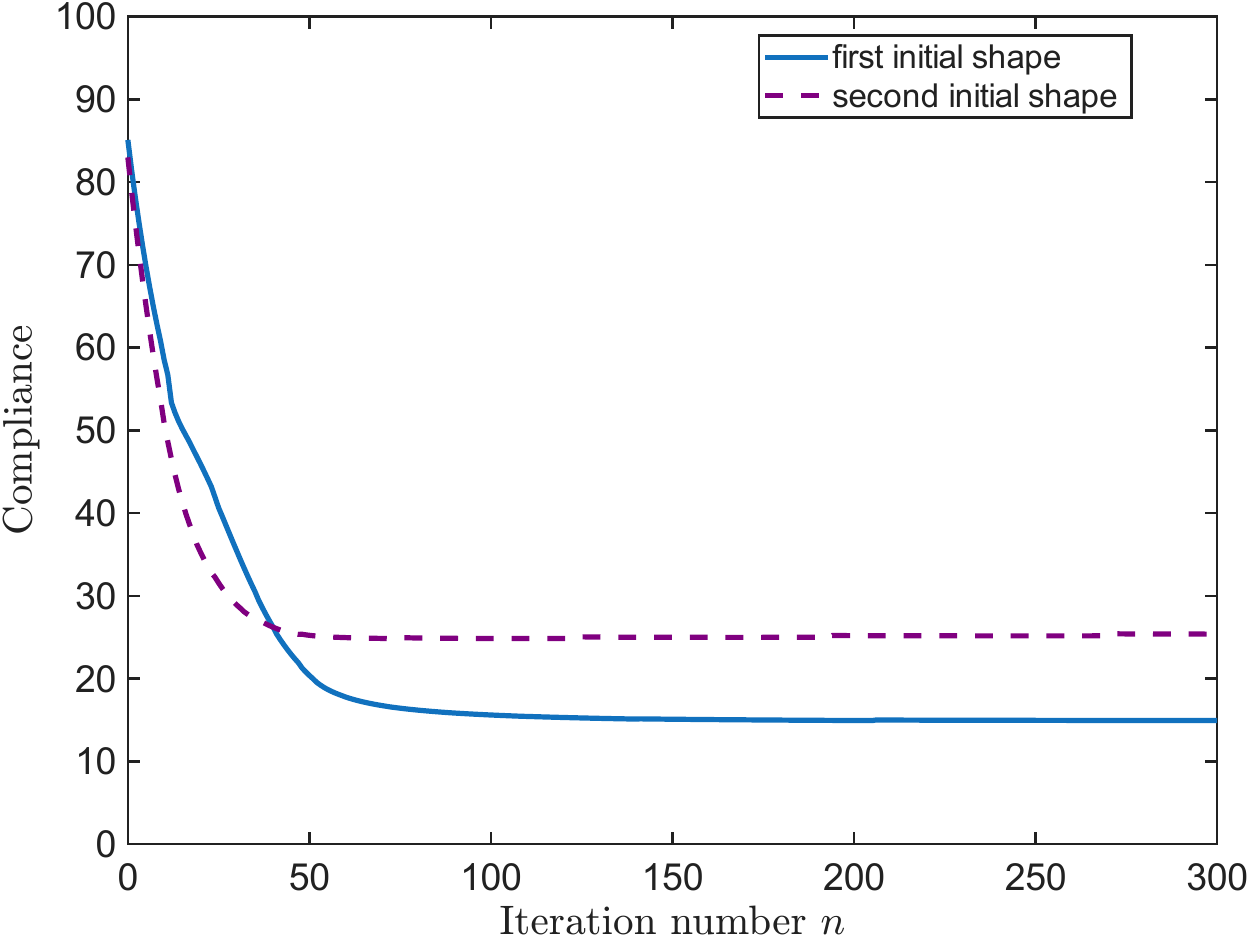}
		\end{subfigure}
		\caption{Evolution of two initial designs for Example \ref{5_M4E2}: $n = 0,60,300$ (from left to right of the first two lines) and convergence histories of the volume (left) and the compliance (right) of line 3.
        }
		\label{Figure_M4E2_Rho}
	\end{figure}
	
    
\subsection{Structural vibration}

Set the non-material stiffness $\eta_1 = 10^{-2}$ and density $\eta_2 = 10^{-4}$. The width of the interface is set to $\gamma = 5 \times 10^{-3}$.  Choose the L\'ame constants $\lambda_L = 1, \mu_L = 8$, and $\lambda_L = \frac{15}{26}$, $\mu_L = \frac{5}{13}$ for 2D and 3D, respectively. 
    \begin{figure}[htbp]
    \centering

    \begin{tikzpicture}[scale=2.]
        \draw[thick] (0,0) rectangle (1.6,1);
        \node at (0.8,0.5) {$D$};

        \draw[line width=3pt, blue] (0,0) -- (0,1);
        \node[left=2pt, blue] at (0.4,0.5) {$\partial D_D$};

        \def\N{10}
        \def\dx{0.05}
        \def\dy{0.05}
        \foreach \i in {0,1,...,\N}
        {
            \pgfmathsetmacro{\ypos}{1 - \i/\N}
            \draw[thick, blue] (0,\ypos) -- ++(-\dx,-\dy);
        }

        \fill[black] (1.5,0.45) rectangle (1.6,0.55);

        \tikzset{
            dim/.style={black, line width=0.4pt, |<->|, >=stealth},
            dim label/.style={midway, fill=white, inner sep=1.2pt}
        }

        \draw[<->, dim] (0,-0.1) -- (1.6,-0.1)
        node[dim label] {2};

        \draw[<->, dim] (1.8,0) -- (1.8,1)
        node[dim label] {1};

        \draw[<->, dim] (1.5,0.60) -- (1.6,0.60)
        node[dim label] {0.1};
    \end{tikzpicture}
    \hspace{0.6cm}
    \begin{tikzpicture}[scale=2.25]
        \draw[thick] (0,0) rectangle (3.2,0.8);
        \node at (1.8,0.4) {$D$};

        \fill[black] (1.52,0.32) rectangle (1.68,0.48);

        \tikzset{
            dim/.style={black, line width=0.4pt, |<->|, >=stealth},
            dim label/.style={midway, fill=white, inner sep=1.2pt}
        }

        \draw[line width=3pt, blue] (0,0) -- (0,0.8);
        \node[left=2pt, blue] at (0.4,0.4) {$\partial D_D$};

        \def\N{8}
        \def\dx{0.04}
        \def\dy{0.04}
        \foreach \i in {0,1,...,\N}
        {
            \pgfmathsetmacro{\ypos}{0.8*(1-\i/\N)}
            \draw[thick, blue] (0,\ypos) -- ++(-\dx,-\dy);
        }

        \draw[line width=3pt, blue] (3.2,0) -- (3.2,0.8);
        \node[right=2pt, blue] at (2.8,0.4) {$\partial D_D$};

        \foreach \i in {0,1,...,\N}
        {
            \pgfmathsetmacro{\ypos}{0.8*(1-\i/\N)}
            \draw[thick, blue] (3.2,\ypos) -- ++(\dx,-\dy);
        }

        \draw[<->, dim] (0,0.95) -- (3.2,0.95)
        node[dim label] {3.2};

        \draw[<->, dim] (3.35,0) -- (3.35,0.8)
        node[dim label] {0.8};

        \draw[<->, dim] (1.52,0.50) -- (1.68,0.50)
        node[dim label] {0.16};
    \end{tikzpicture}

    \caption{The design domain $D$ and its boundary for Example \ref{5_M5E1} (left) and Example \ref{5_M5E2} (right).}
    \label{5_M5_Ex1}
\end{figure}
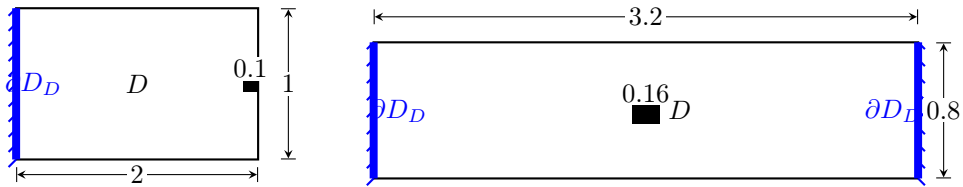

    \begin{example} 
		\label{5_M5E1}
		(Cantilever: 2D \cite{allaire2005levelvibration}) Set $D = (0,2) \times (0,1)$ and $\partial D_D = \lbrace (0,y): y \in [0,1] \rbrace$ (see the left plot of \autoref{5_M5_Ex1}).
    \end{example}
	
	Set the block's mass $\rho_{\texttt{block}} = 35$. From \autoref{Figure_M5E1_All}, we can clearly see the merging of holes during the iteration, and the final shape is comparable to that in the literature. On the other hand, \autoref{Figure_M5E1_All} shows the change in the approximate volume is visually small. A stable and monotonic increase of the objective, i.e., the first eigenvalue, is also observed.

    \begin{figure}[htbp]
		\centering
		\begin{subfigure}[b]{0.4\textwidth}
			\includegraphics[width=\linewidth]{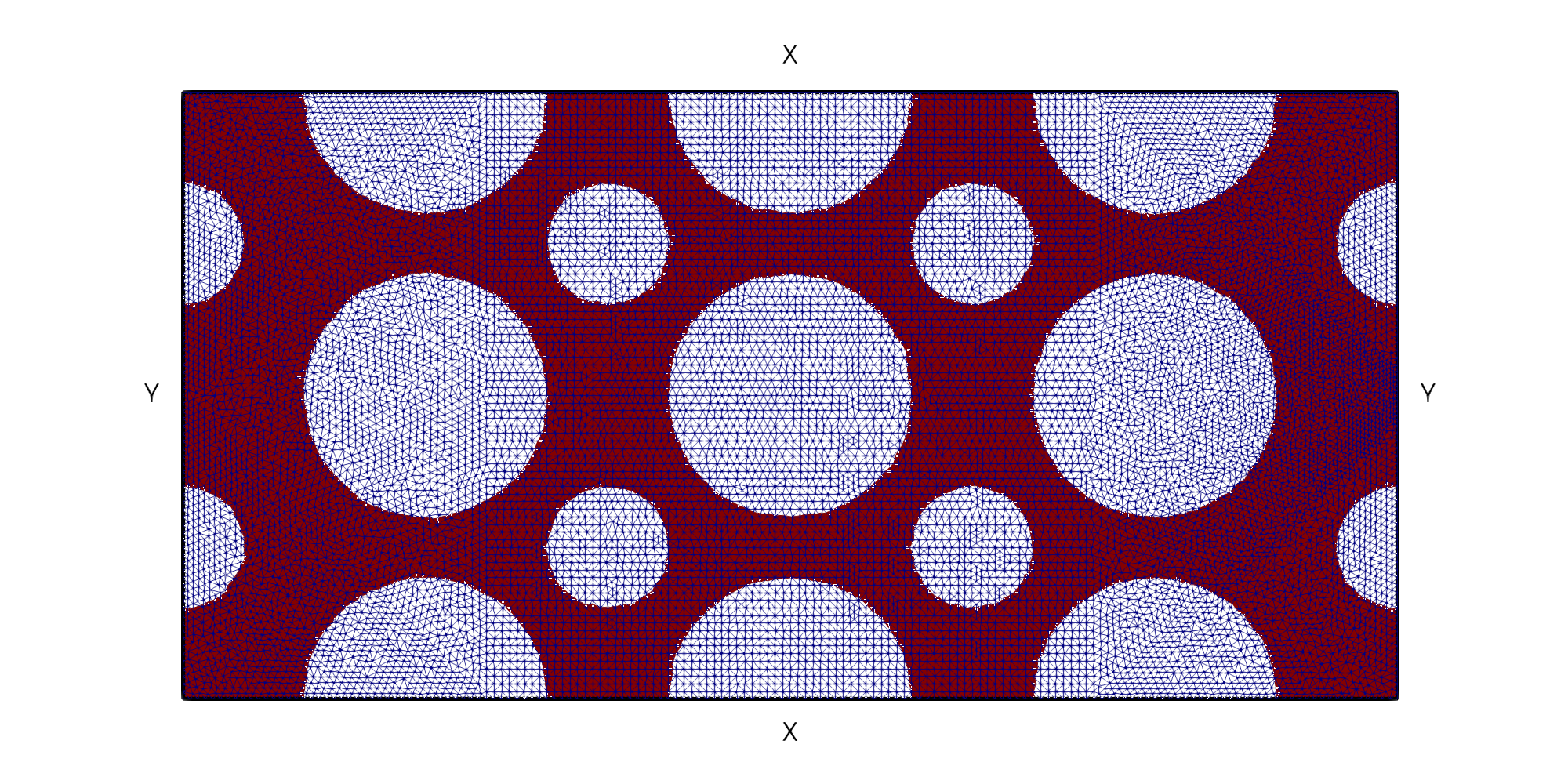} 
		\end{subfigure}		
            \hspace{0.06\textwidth}%
		\begin{subfigure}[b]{0.4\textwidth}
			\includegraphics[width=\linewidth]{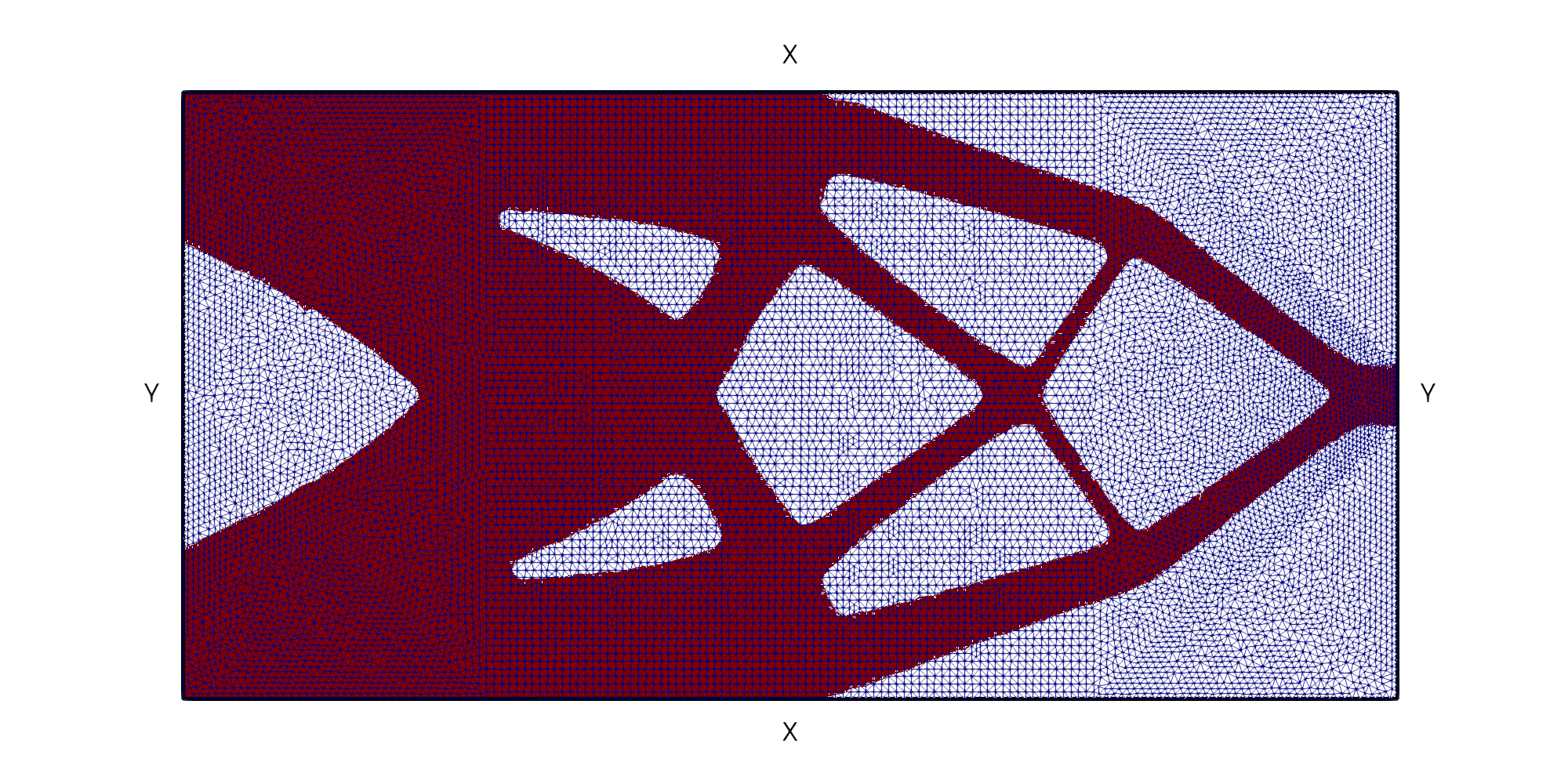}
		\end{subfigure}

		\begin{subfigure}[b]{0.351\textwidth}
                \centering			\includegraphics[width=\linewidth]{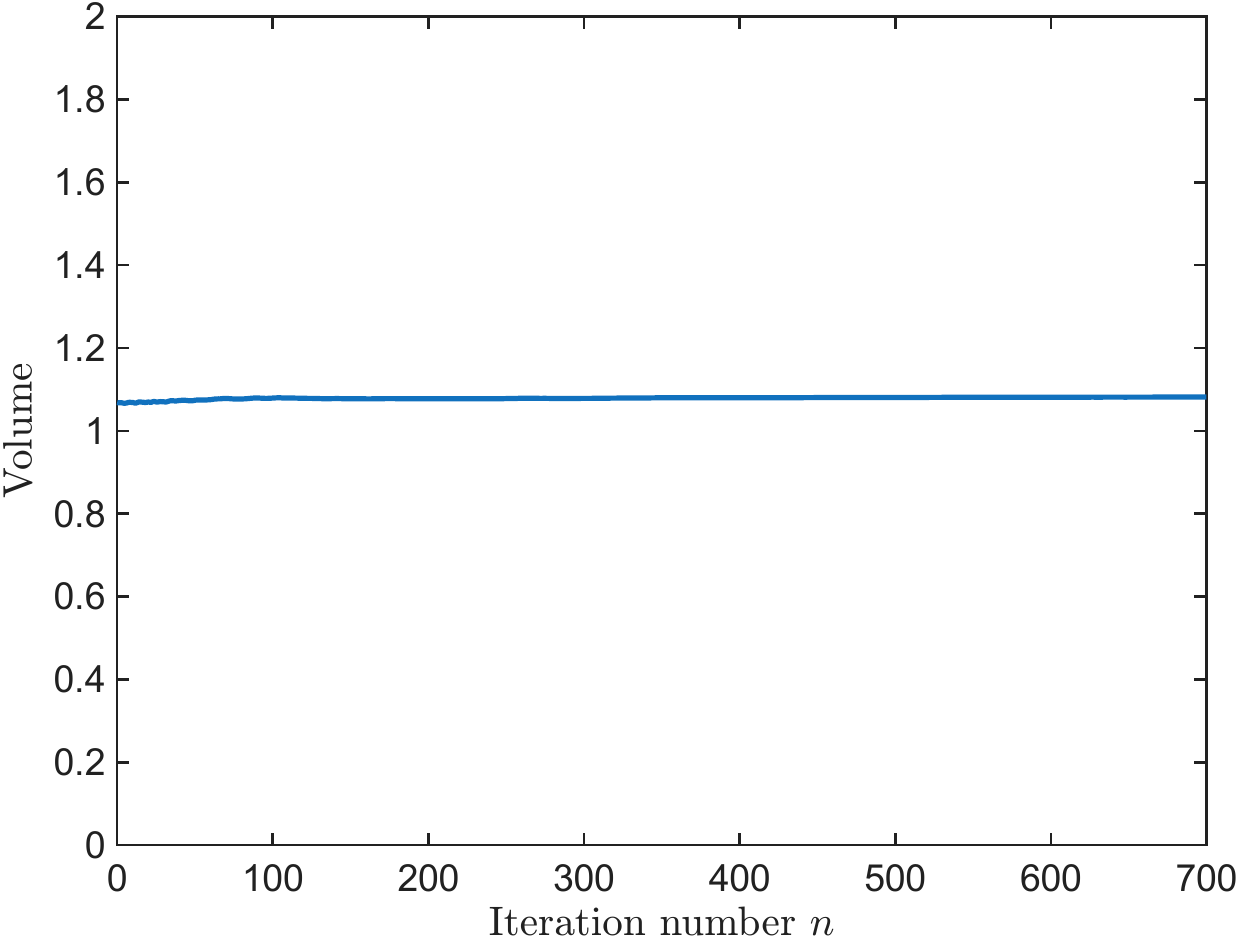} 
		\end{subfigure}
            \hspace{0.06\textwidth}%
		\begin{subfigure}[b]{0.349\textwidth}			\includegraphics[width=\linewidth]{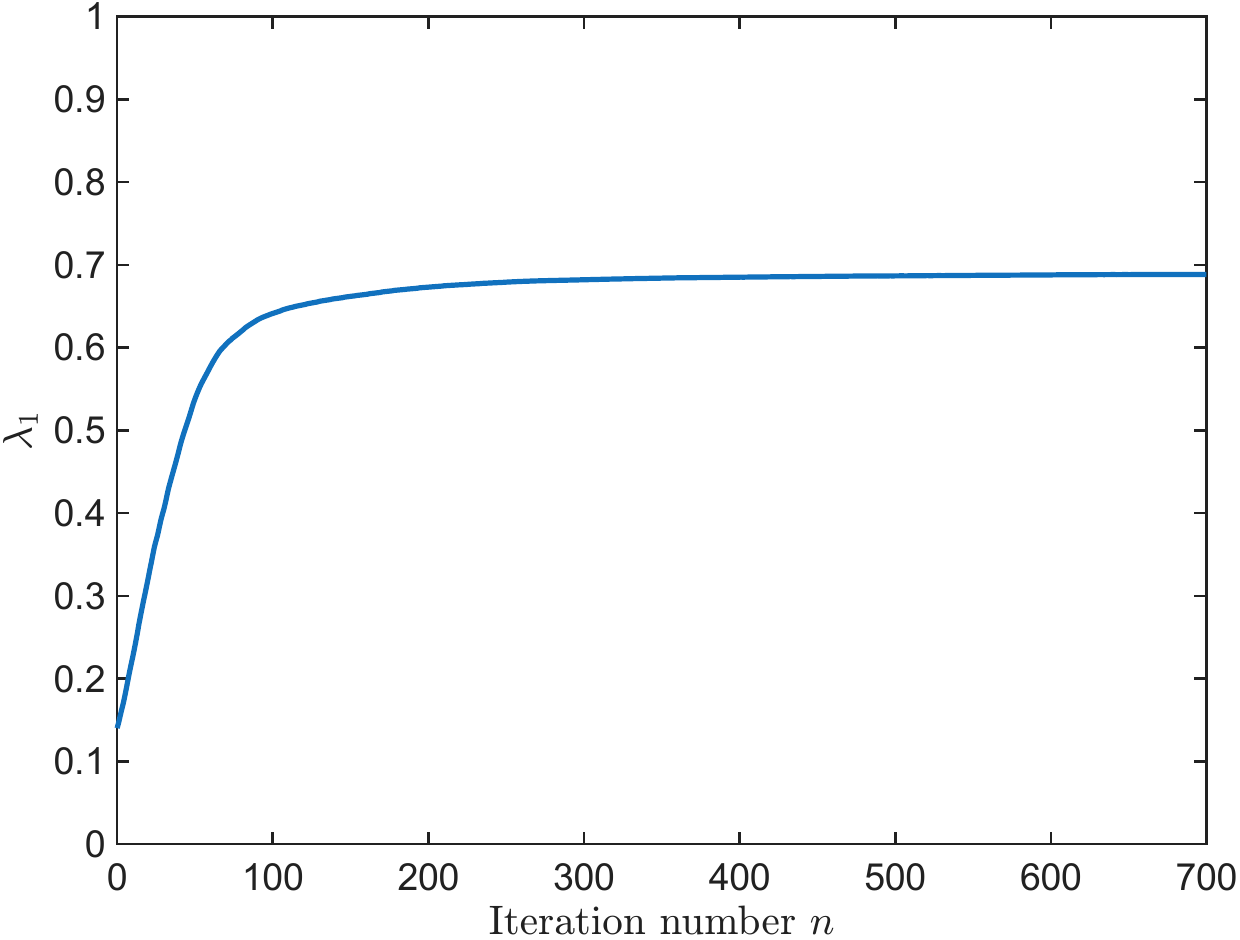}
		\end{subfigure}
		\caption{The initial design (top left), final design (top right), convergence histories of the volume (bottom left) and the first eigenvalue (bottom right) for Example \ref{5_M5E1}.}
		\label{Figure_M5E1_All}
    \end{figure}

    \begin{example} 
		\label{5_M5E2}
		(Double Cantilever: 2D \cite{allaire2005levelvibration}) Set $D = (0,3.2) \times (0,0.8)$ and $\partial D_D = \lbrace (0,y), \ (3.2,y): y \in [0,0.8] \rbrace$ (see the plot on the right of \autoref{5_M5_Ex1}).
    \end{example}
	
	Set $\rho_{\texttt{block}} = 25$. The initial and final shapes are shown in \autoref{Figure_M5E2_All}. \autoref{Figure_M5E2_All} moreover illustrates changes in the approximate volume and the first eigenvalue, respectively, showing stable volume control and a monotonic increase in the eigenvalue.

    \begin{figure}[htbp]
		\centering
		\begin{subfigure}[b]{0.48\textwidth}
			\includegraphics[width=\linewidth,
            trim=620bp 30bp 620bp 30bp,
            clip]{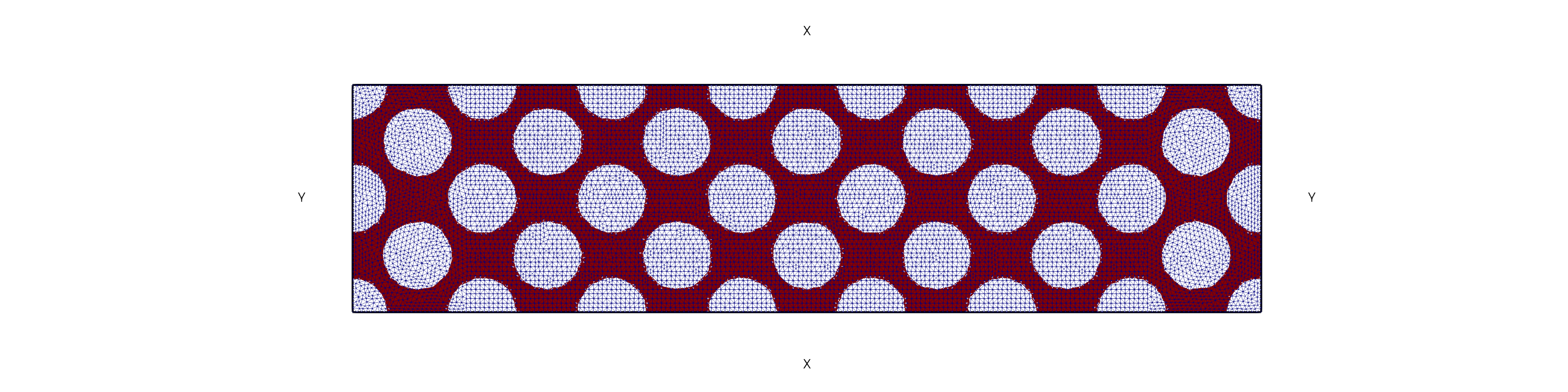} 
		\end{subfigure}		
            \hfill
		\begin{subfigure}[b]{0.48\textwidth}
			\includegraphics[width=\linewidth,
            trim=620bp 30bp 620bp 30bp,
            clip]{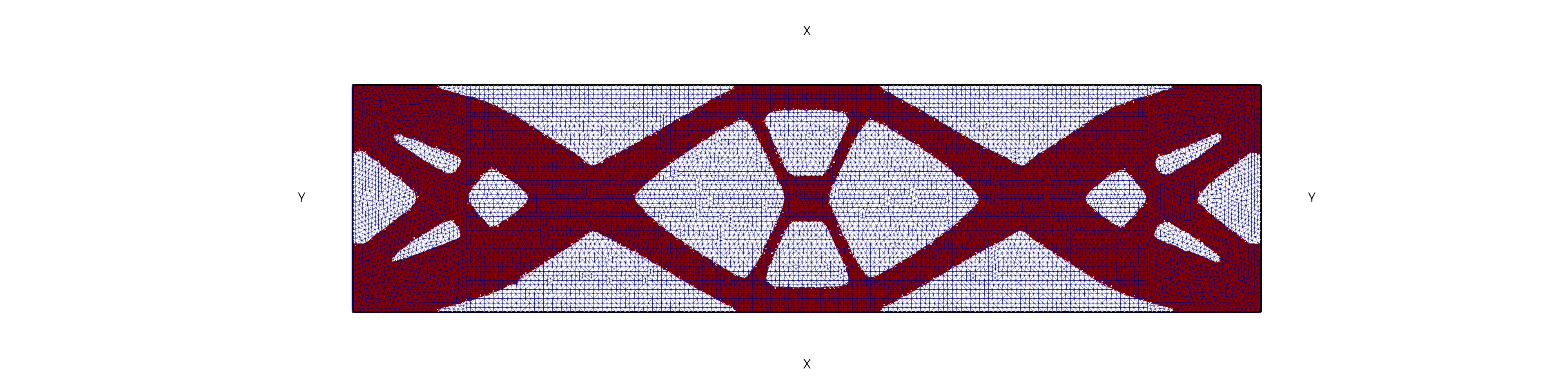}
		\end{subfigure}

		\begin{subfigure}[b]{0.351\textwidth}
                \centering			\includegraphics[width=\linewidth]{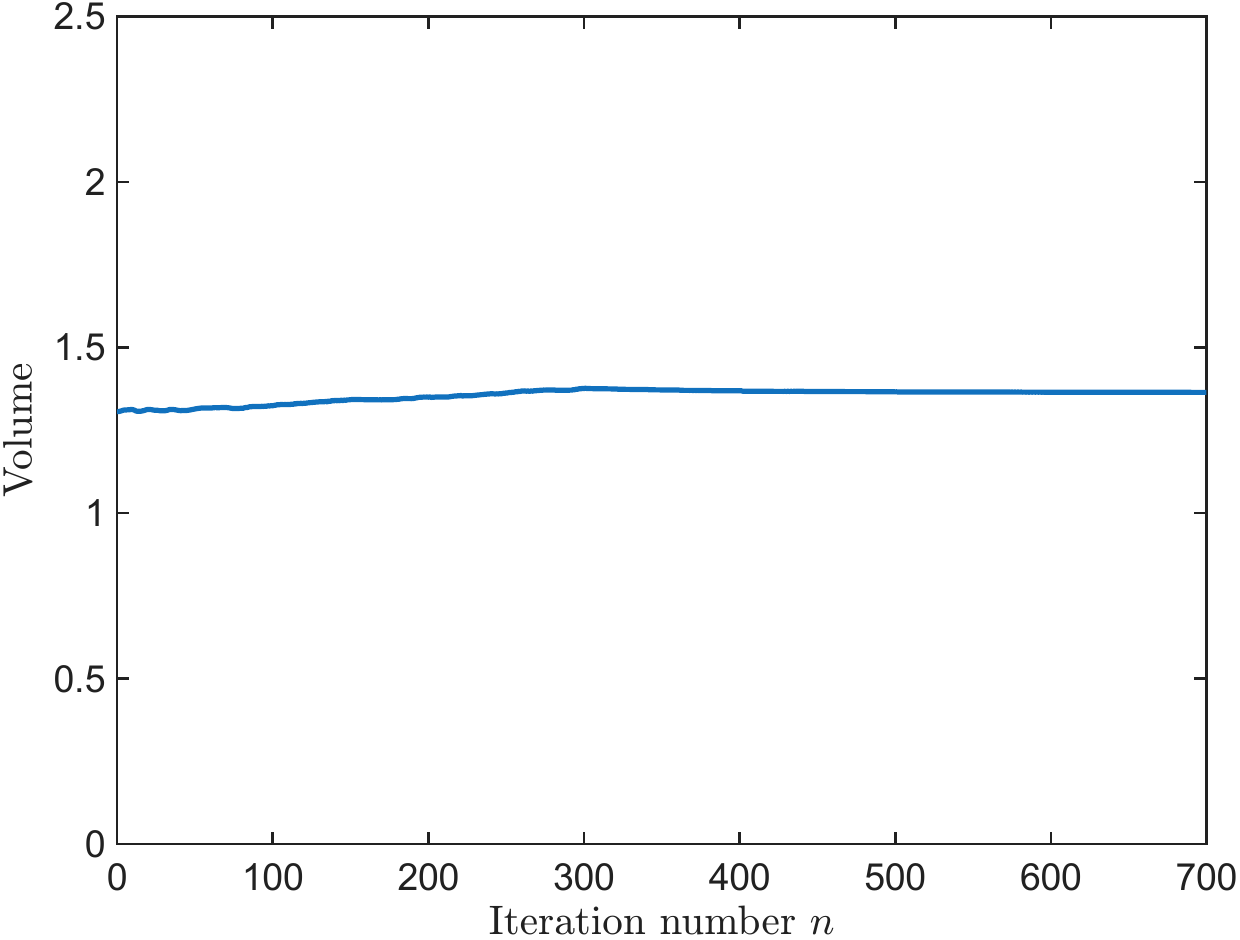} 
		\end{subfigure}
            \hspace{0.06\textwidth}%
		\begin{subfigure}[b]{0.349\textwidth}			\includegraphics[width=\linewidth]{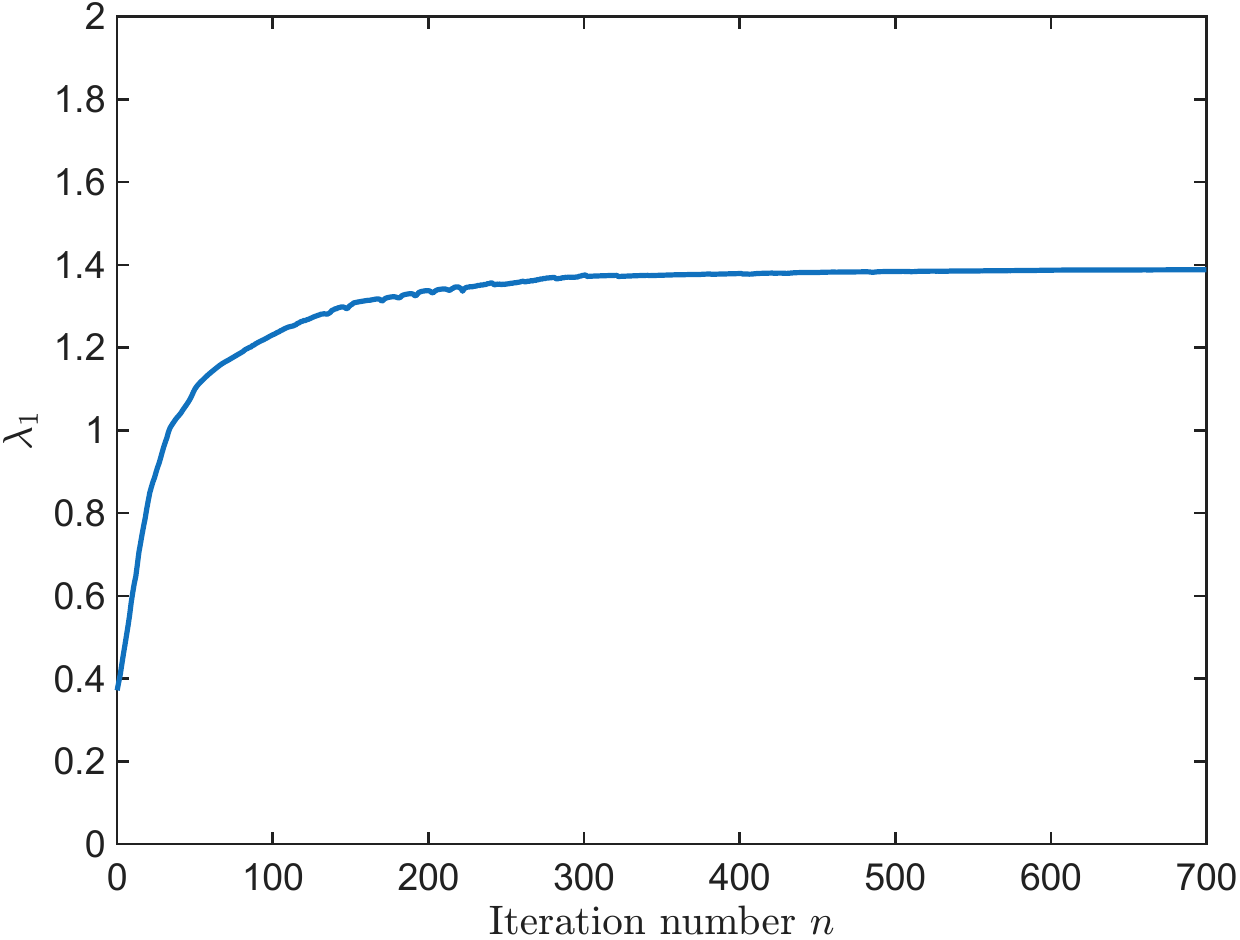}
		\end{subfigure}
		\caption{The initial design (top left), final design (top right), convergence histories of the volume (bottom left) and the first eigenvalue (bottom right) for Example \ref{5_M5E2}.}
		\label{Figure_M5E2_All}
    \end{figure}

    \begin{example} 
		\label{5_M5E3}
		(Cantilever: 3D \cite{allaire2005levelvibration}) Set $D = (0,1) \times (0,2) \times (0,1)$, $\partial D_D = \lbrace (0,y,z): y \in [0,2], \ z \in [0,1] \rbrace$ and block $D_{\texttt{block}} = (0.9,1) \times (0.9,1.1) \times (0.45,0.55)$.
    \end{example}
	
	Set $\rho_{\texttt{block}} = 333$.  \autoref{Figure_M5E3_All} illustrates topological changes and changes in the approximate volume and the first eigenvalue, showing stable volume control and a monotonic increase in the eigenvalue.

    \begin{figure}[htbp]
		\centering
		\begin{subfigure}[b]{0.3\textwidth}
			\includegraphics[width=\linewidth]{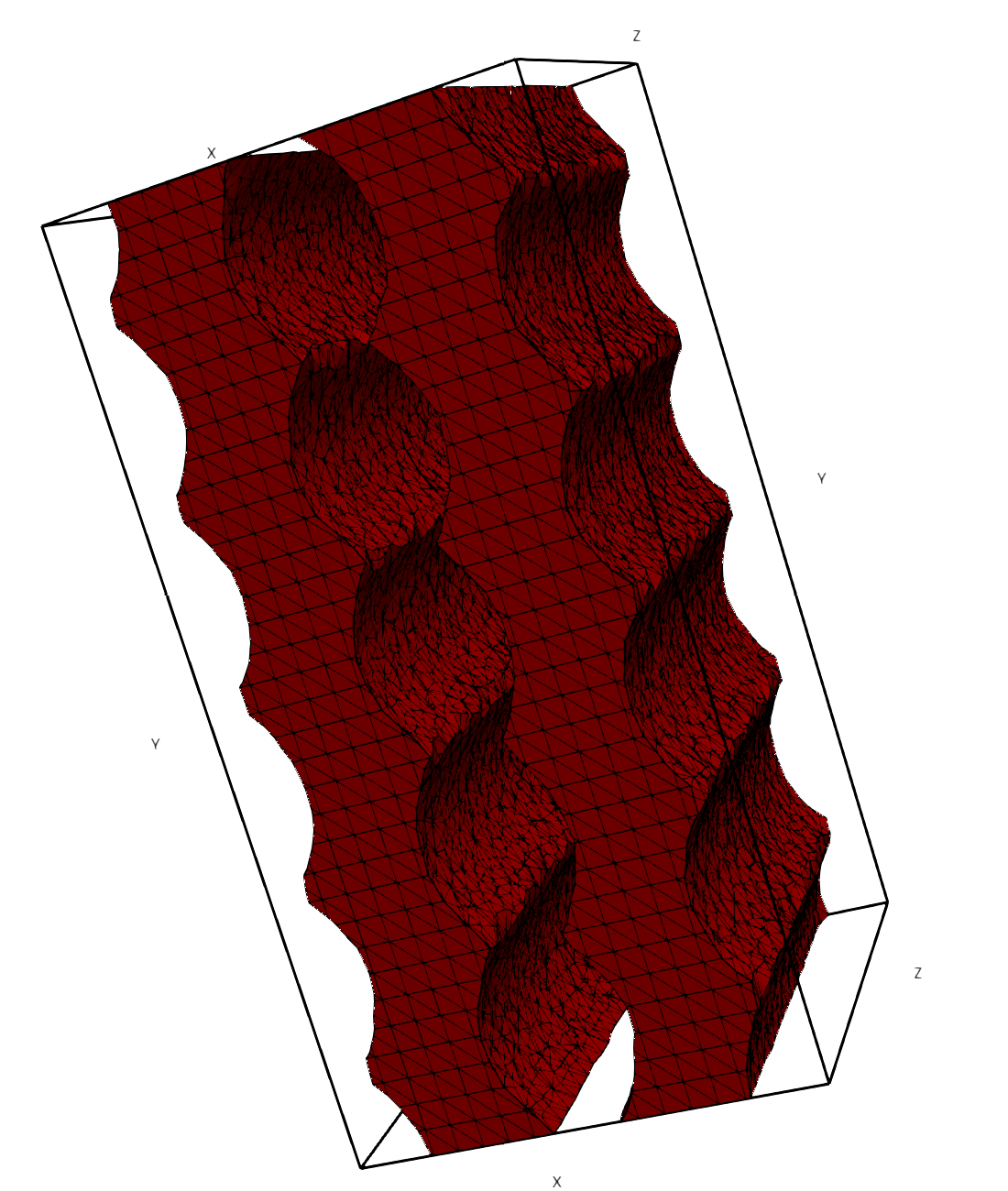} 
		\end{subfigure}		
            \hspace{0.06\textwidth}%
		\begin{subfigure}[b]{0.3\textwidth}
			\includegraphics[width=\linewidth]{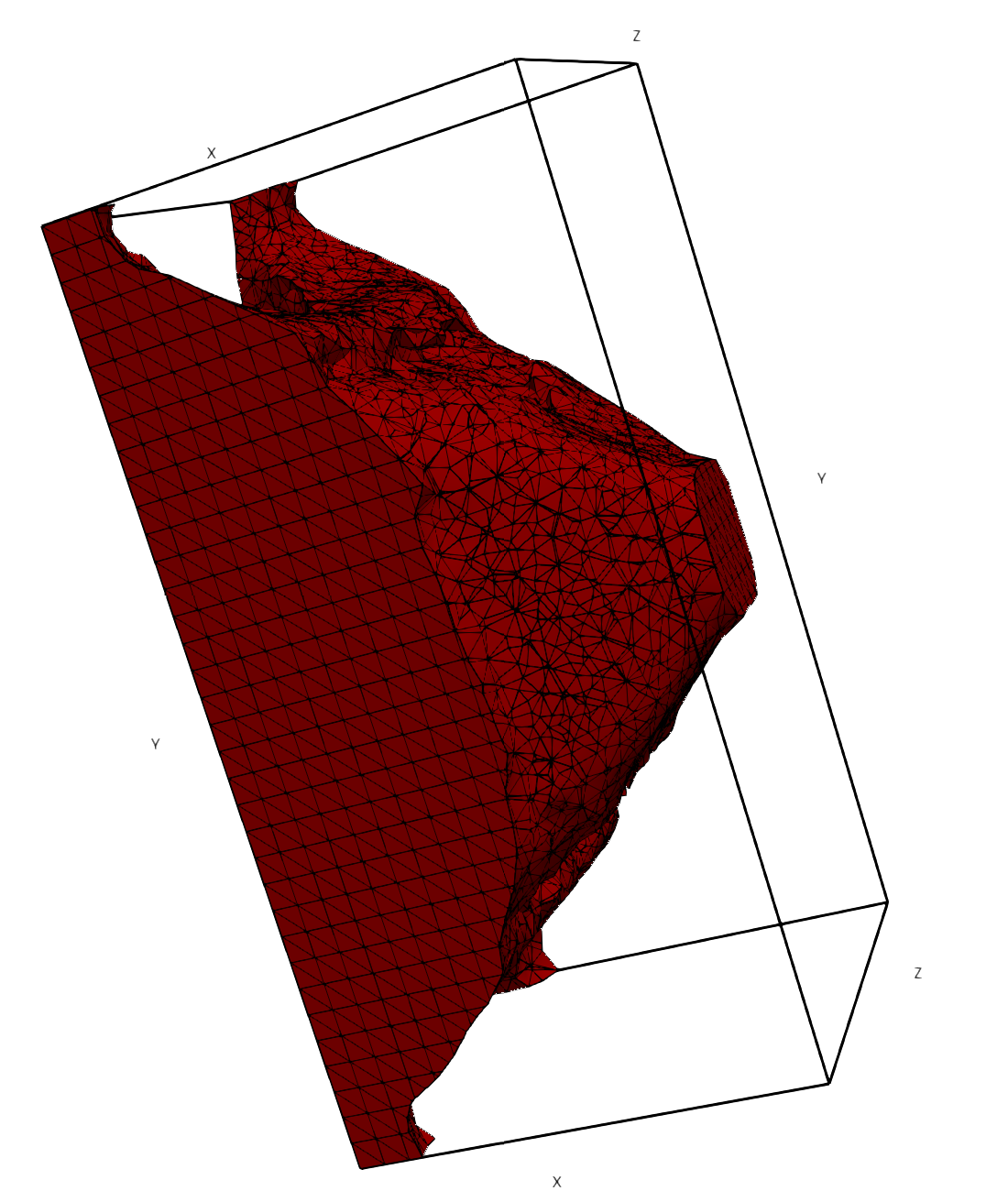}
		\end{subfigure}

		\begin{subfigure}[b]{0.351\textwidth}
                \centering			\includegraphics[width=\linewidth]{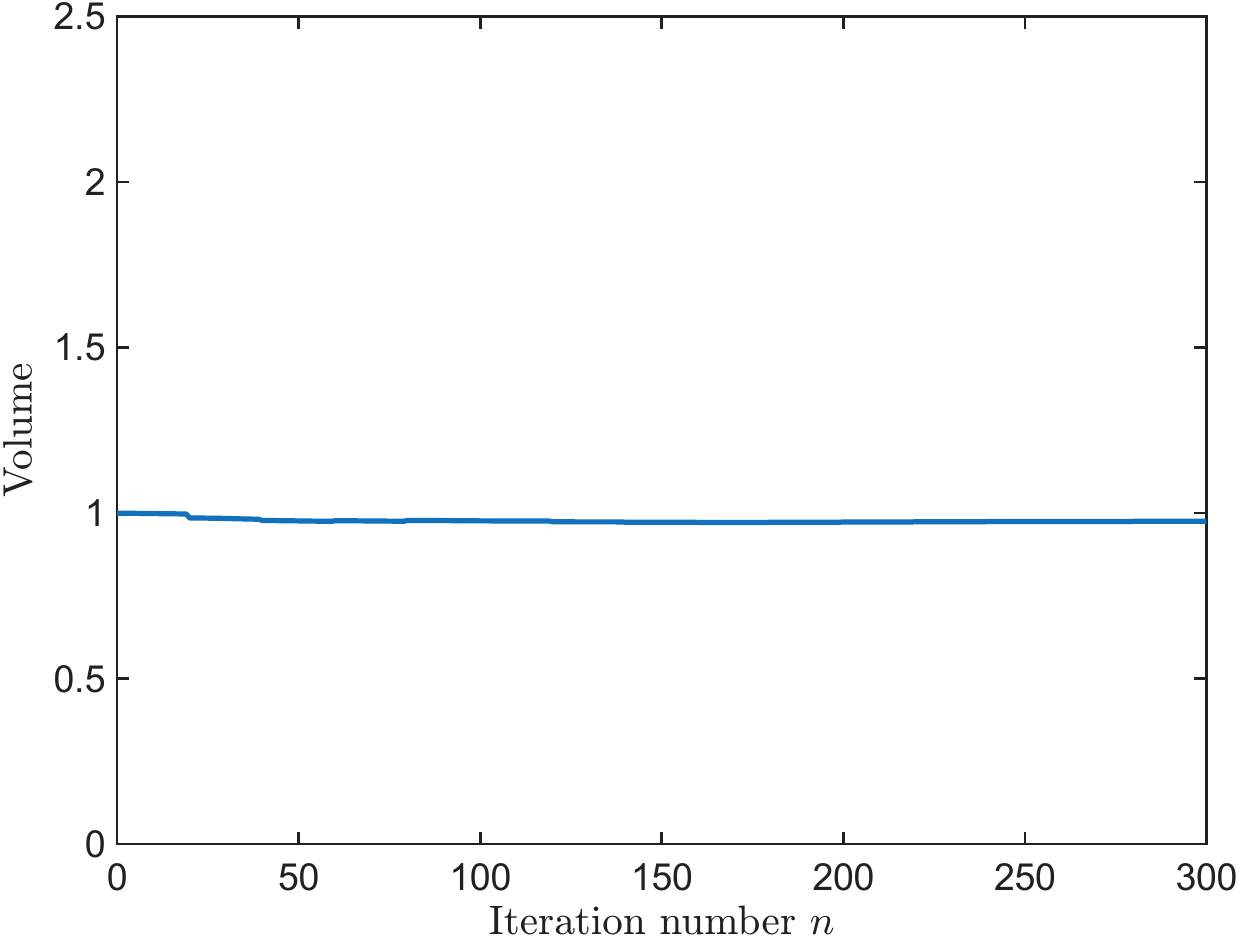} 
		\end{subfigure}
            \hspace{0.06\textwidth}%
		\begin{subfigure}[b]{0.349\textwidth}			\includegraphics[width=\linewidth]{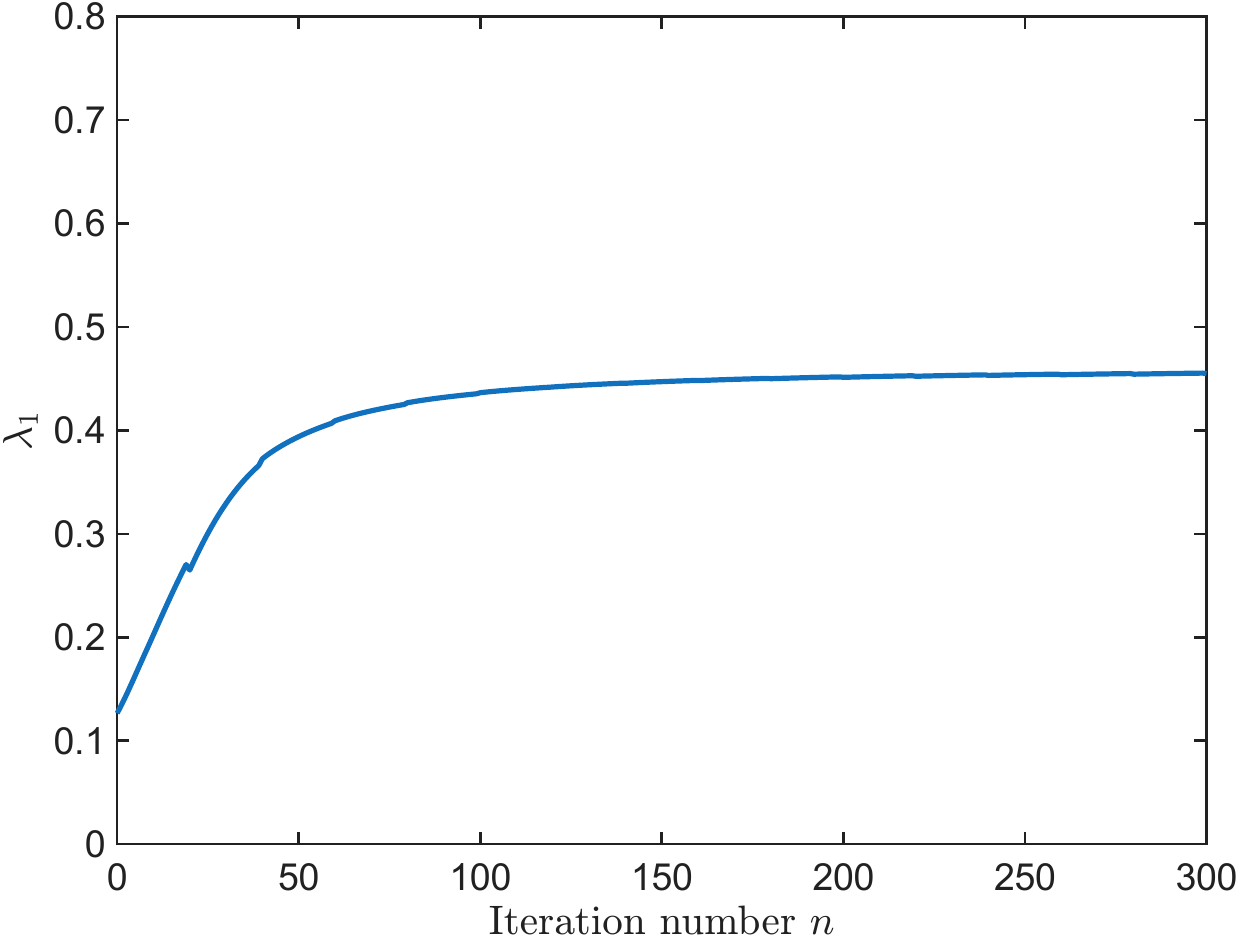}
		\end{subfigure}
		\caption{The initial design (top left), final design (top right), convergence histories of the volume (bottom left) and the first eigenvalue (bottom right) for Example \ref{5_M5E3}.}
		\label{Figure_M5E3_All}
    \end{figure}

\subsection{Design of Stokes flows}
	For the topology optimization of Stokes flows, we set $\nu_f = 1$ and use the following interpolation formula \cite{borrvall2001topology,BorrvallPetersson:2003}:
	\begin{equation}
		\alpha(\rho) = \overline{\alpha} + \left( \underline{\alpha} - \overline{\alpha} \right) \frac{\rho(q_{\alpha}+1)}{\rho + q_{\alpha}}.
	\end{equation}
    	Choose $\overline{\alpha} = 2.5 \times 10^{4}$, $\underline{\alpha} = 2.5 \times 10^{-4}$, $q_{\alpha} = 0.1$, and $\gamma = 5 \times 10^{-3}$. Set $\bm{f}_f = \bm{0}$ (except for \textbf{Example} \ref{5_M2E4}).
	\begin{example}
		\label{5_M2E1}
		(Double-inlet double-outlet channel \cite{BorrvallPetersson:2003}) Set $D = (0,1.5) \times (0,1)$ (see the plot on the left of \autoref{5_M2_Ex1}), 
        \begin{equation}
        \begin{aligned}
		   &\Gamma_{\rm{in}} = \left\lbrace (0,y): y \in \left[  \frac{1}{6}, \frac{1}{3} \right] \cup \left[  \frac{2}{3}, \frac{5}{6} \right] \right\rbrace \text{ (inflow boundary)},
           \\
           & \Gamma_{\rm{out}} = \left\lbrace (1.5,y): y \in \left[  \frac{1}{6}, \frac{1}{3} \right] \cup \left[  \frac{2}{3}, \frac{5}{6} \right] \right\rbrace \text{ (outflow boundary)}.
        \end{aligned}
		\end{equation}
        Set
	\begin{equation}
		\bm{g}_f = 
		\begin{cases}
			\Big(1 - 144 \Big( y - \dfrac{1}{4} \Big)^2, 0\Big)^T & x = 0 \text{ or } 1.5,\ y \in \left[ \frac{1}{6}, \frac{1}{3}\right], \\
			\Big(1 - 144 \Big( y - \dfrac{3}{4} \Big)^2, 0\Big)^T & x = 0 \text{ or } 1.5,\ y \in \left[ \frac{2}{3}, \frac{5}{6}\right], \\
			(0, 0)^T & \text{ otherwise on } \partial D.
		\end{cases}
	\end{equation}
    	\end{example}
        
	Shape changes and final velocity field are shown in \autoref{Figure_M2E1_VandJ}. The approximate volume remains close to its prescribed target value, and the convergence history of the objective is shown. 
	
	\vspace*{2ex}
	\begin{example}
		\label{5_M2E2}
		(2D diffuser \cite{BorrvallPetersson:2003})  Set $D = (0,1)^2$,  $\Gamma_{\text{in}} = \left\lbrace (0,y): y \in [0,1] \right\rbrace$ and  $\Gamma_{\text{out}} = \left\lbrace (1,y): y \in \left[  \frac{1}{3}, \frac{2}{3} \right]  \right\rbrace$ (see the right plot of  \autoref{5_M2_Ex1}). Set
	\begin{equation}
		\bm{g}_f = 
		\begin{cases}
			\Big(1 - \Big( 2y - 1 \Big)^2, 0\Big)^T & x = 0, \\
			\Big(3 - 3 \Big( 6y - 3 \Big)^2, 0\Big)^T & x = 1,\ y \in \left[ \frac{1}{3}, \frac{2}{3}\right], \\
			(0, 0)^T & \text{ otherwise on } \partial D.
		\end{cases}
	\end{equation}
    	\end{example}
        
Shape changes and final velocity field are shown in \autoref{Figure_M2E2_VandJ}. The approximate volume remains close to its prescribed target value, and the convergence history of the objective is shown. 
	
	\vspace*{2ex}
	\begin{example}
		\label{5_M2E3}
		(2D bend channel \cite{BorrvallPetersson:2003})  Set $D = (0,1)^2$, $\Gamma_{\text{in}} = \left\lbrace (0,y): y \in [0.7,0.9] \right\rbrace$ and $\Gamma_{\text{out}} = \left\lbrace (x,0): x \in [0.7,0.9] \right\rbrace$ (see the left plot of \autoref{5_M2_Ex3}).	Set
	\begin{equation}
		\bm{g}_f = 
		\begin{cases}
			\left(1 - \left( 10y - 8 \right)^2, 0\right)^T & x = 0,\ y \in [0.7, 0.9], \\
			\left(0, 1 - \left( 10x - 8 \right)^2\right)^T & x \in [0.7, 0.9],\ y = 0, \\
			(0, 0)^T & \text{ otherwise on } \partial D.
		\end{cases}
	\end{equation}
    	\end{example}
        
The evolutions of shape, volume, and drag are shown in \autoref{Figure_M2E3_VandJ}. 
	
	\vspace*{2ex}
	\begin{example}
		\label{5_M2E4}
		(2D flow channel design under body force \cite{BorrvallPetersson:2003}) Set $D = (0,1)^2$, $\Gamma_{\text{in}} = \big\lbrace (0,y): y \in \big[ \dfrac{7}{12},\dfrac{3}{4}\big] \big\rbrace$ and $\Gamma_{\text{out}} = \big\lbrace (1,y): y \in \big[ \dfrac{7}{12},\dfrac{3}{4}\big] \big\rbrace$ (see the right plot of \autoref{5_M2_Ex3}).
	We choose different forces $\bm{f}_f = (-1125, 0)^T$, $(562.5, 0)^T$, and $(1687.5, 0)^T$ in the circle centered on $\big(0.5, \dfrac{1}{3} \big)$ with radius $\dfrac{1}{12}$, while 
	\begin{equation}
		\bm{g}_f = 
		\begin{cases}
			\Big(1 - \left( 12y - 8 \right)^2, 0\Big)^T, & x = 0 \text{ or } 1,\ y \in \Big[ \dfrac{7}{12}, \dfrac{3}{4} \Big], \\
			(0, 0)^T, & \text{ otherwise on } \partial D.
		\end{cases}
	\end{equation}
    	\end{example}
        
	Starting from the same initial design, \autoref{c} illustrates the final shapes agree well with those in \cite{BorrvallPetersson:2003}. \autoref{Figure_M2E4_VandJ} shows the convergence histories of drag and the approximate volume, which remains close to its target value.
	
	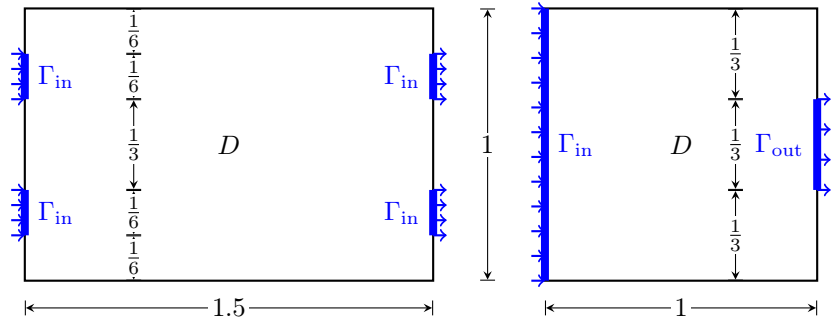
\begin{figure}[htbp] 
		\centering
		\begin{tikzpicture}[scale=3.6] 
			\draw[thick] (0,0) rectangle (1.5,1);
			\node at (0.75,0.5) {$D$};  
			
			\draw[line width=3pt, blue] (0, 1.0/6.0) -- (0, 1.0/3.0);
			\node[right=1.5pt, blue] at (0, 0.25) {$\Gamma_{\text{in}}$};
			\def\N{3}
			\def\dx{0.05}
			\def\dy{0}
			\foreach \i in {0,1,...,\N}
			{
				\pgfmathsetmacro{\ypos}{1.0/3.0 - \i/(6*\N)}
				\draw[->, thick, blue] (-\dx, \ypos) -- ++(\dx, 0);
			}
			\draw[line width=3pt, blue] (0, 2.0/3.0) -- (0, 5.0/6.0);
			\node[right=1.5pt, blue] at (0, 0.75) {$\Gamma_{\text{in}}$};
			\def\N{3}
			\def\dx{0.05}
			\def\dy{0}
			\foreach \i in {0,1,...,\N}
			{
				\pgfmathsetmacro{\ypos}{5.0/6.0 - \i/(6*\N)}
				\draw[->, thick, blue] (-\dx, \ypos) -- ++(\dx, 0);
			}
			\draw[line width=3pt, blue] (1.5, 1.0/6.0) -- (1.5, 1.0/3.0);
			\node[left=1.5pt, blue] at (1.5, 0.25) {$\Gamma_{\text{in}}$};
			\def\N{3}
			\def\dx{0.05}
			\def\dy{0}
			\foreach \i in {0,1,...,\N}
			{
				\pgfmathsetmacro{\ypos}{1.0/3.0 - \i/(6*\N)}
				\draw[->, thick, blue] (1.5, \ypos) -- ++(\dx, 0);
			}
			\draw[line width=3pt, blue] (1.5, 2.0/3.0) -- (1.5, 5.0/6.0);
			\node[left=1.5pt, blue] at (1.5, 0.75) {$\Gamma_{\text{in}}$};
			\def\N{3}
			\def\dx{0.05}
			\def\dy{0}
			\foreach \i in {0,1,...,\N}
			{
				\pgfmathsetmacro{\ypos}{5.0/6.0 - \i/(6*\N)}
				\draw[->, thick, blue] (1.5, \ypos) -- ++(\dx, 0);
			}
			
			\tikzset{
				dim/.style={black, line width=0.4pt, |<->|, >=stealth},
				dim label/.style={midway, fill=white, inner sep=1.2pt}
			}
			
			\draw[<->, dim] (0,-0.1) -- (1.5,-0.1) 
			node[dim label] {1.5};
			
			\draw[<->, dim] (0.4, 0) -- (0.4, 1.0/6.0) 
			node[dim label] {$\frac{1}{6}$};
			\draw[<->, dim] (0.4, 1.0/6.0) -- (0.4, 1.0/3.0) 
			node[dim label] {$\frac{1}{6}$};
			\draw[<->, dim] (0.4, 1.0/3.0) -- (0.4, 2.0/3.0) 
			node[dim label] {$\frac{1}{3}$};
			\draw[<->, dim] (0.4, 2.0/3.0) -- (0.4, 5.0/6.0) 
			node[dim label] {$\frac{1}{6}$};
			\draw[<->, dim] (0.4, 5.0/6.0) -- (0.4, 1.0) 
			node[dim label] {$\frac{1}{6}$};
   		\draw[<->, dim] (1.7, 0) -- (1.7, 1)  node[dim label] {1};
		\end{tikzpicture}
        \hspace{0.2cm}
        \begin{tikzpicture}[scale=3.6] 
			\draw[thick] (0,0) rectangle (1,1);
			\node at (0.5,0.5) {$D$};  
			
			\draw[line width=3pt, blue] (0, 0) -- (0, 1.0);
			\node[right=1.5pt, blue] at (0, 0.5) {$\Gamma_{\text{in}}$};
			\def\N{11}
			\def\dx{0.05}
			\def\dy{0}
			\foreach \i in {0,1,...,\N}
			{
				\pgfmathsetmacro{\ypos}{1.0 - \i/\N}
				\draw[->, thick, blue] (-\dx, \ypos) -- ++(\dx, 0);
			}
			
			\draw[line width=3pt, blue] (1, 1.0/3.0) -- (1, 2.0/3.0);
			\node[left=1.5pt, blue] at (1, 0.5) {$\Gamma_{\text{out}}$};
			\def\N{3}
			\def\dx{0.05}
			\def\dy{0}
			\foreach \i in {0,1,...,\N}
			{
				\pgfmathsetmacro{\ypos}{2.0/3.0 - \i/(3.0*\N)}
				\draw[->, thick, blue] (1, \ypos) -- ++(\dx, 0);
			}
			
			\tikzset{
				dim/.style={black, line width=0.4pt, |<->|, >=stealth},
				dim label/.style={midway, fill=white, inner sep=1.2pt}
			}
			
			\draw[<->, dim] (0,-0.1) -- (1,-0.1) node[dim label] {1};
			
			\draw[<->, dim] (0.7, 0) -- (0.7, 1.0/3.0) 
			node[dim label] {$\frac{1}{3}$};
			\draw[<->, dim] (0.7, 1.0/3.0) -- (0.7, 2.0/3.0) 
			node[dim label] {$\frac{1}{3}$};
			\draw[<->, dim] (0.7, 2.0/3.0) -- (0.7, 1.0) 
			node[dim label] {$\frac{1}{3}$};
		\end{tikzpicture}
		\caption{The design domain $D$ and its boundary for Example \ref{5_M2E1} (left) and Example \ref{5_M2E2} (right).}
		\label{5_M2_Ex1}
	\end{figure}
	
	\begin{figure}[htbp]
		\centering
            \begin{subfigure}[b]{0.32\textwidth}			\includegraphics[width=\linewidth]{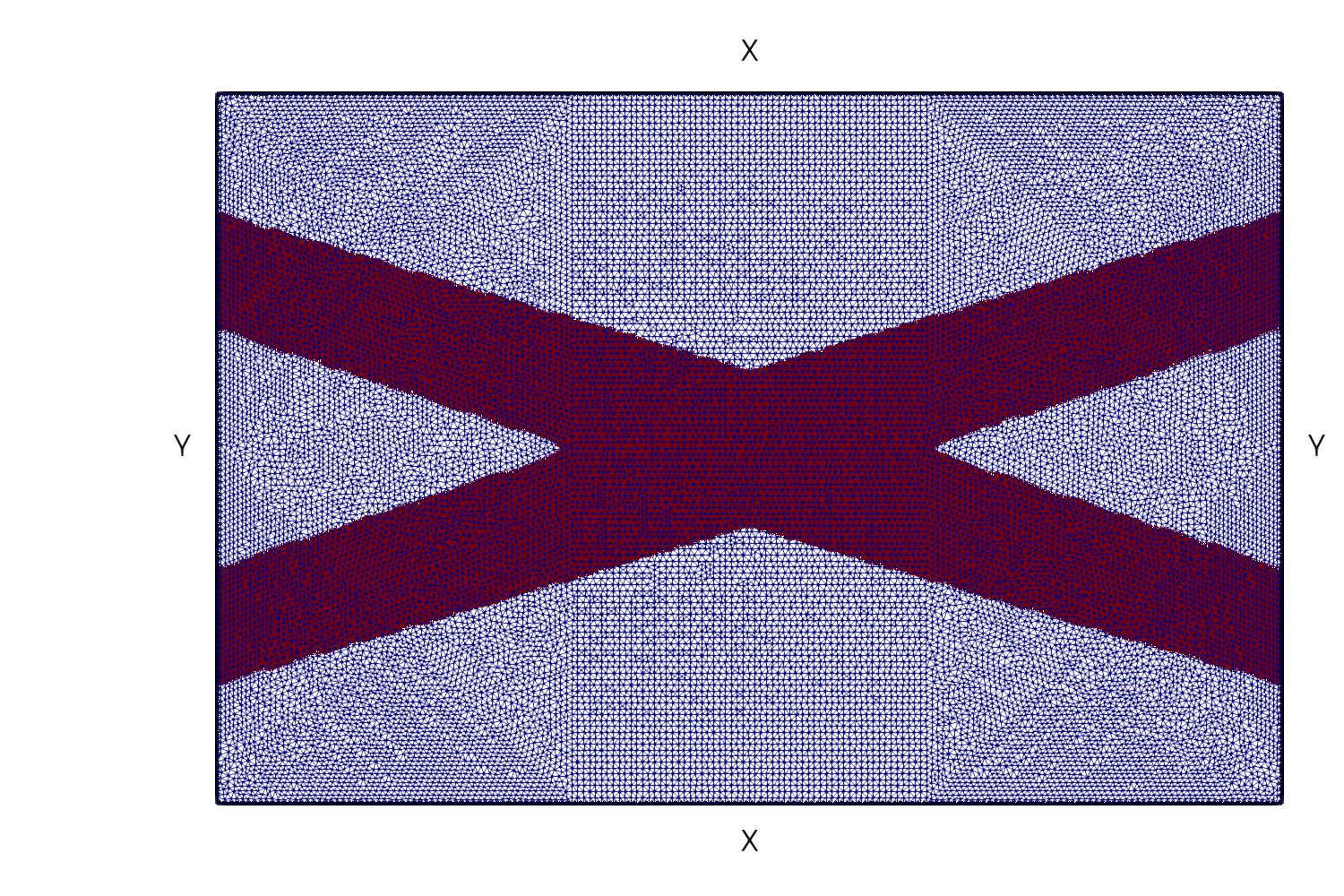} 
		\end{subfigure}       
            \begin{subfigure}[b]{0.32\textwidth}			\includegraphics[width=\linewidth]{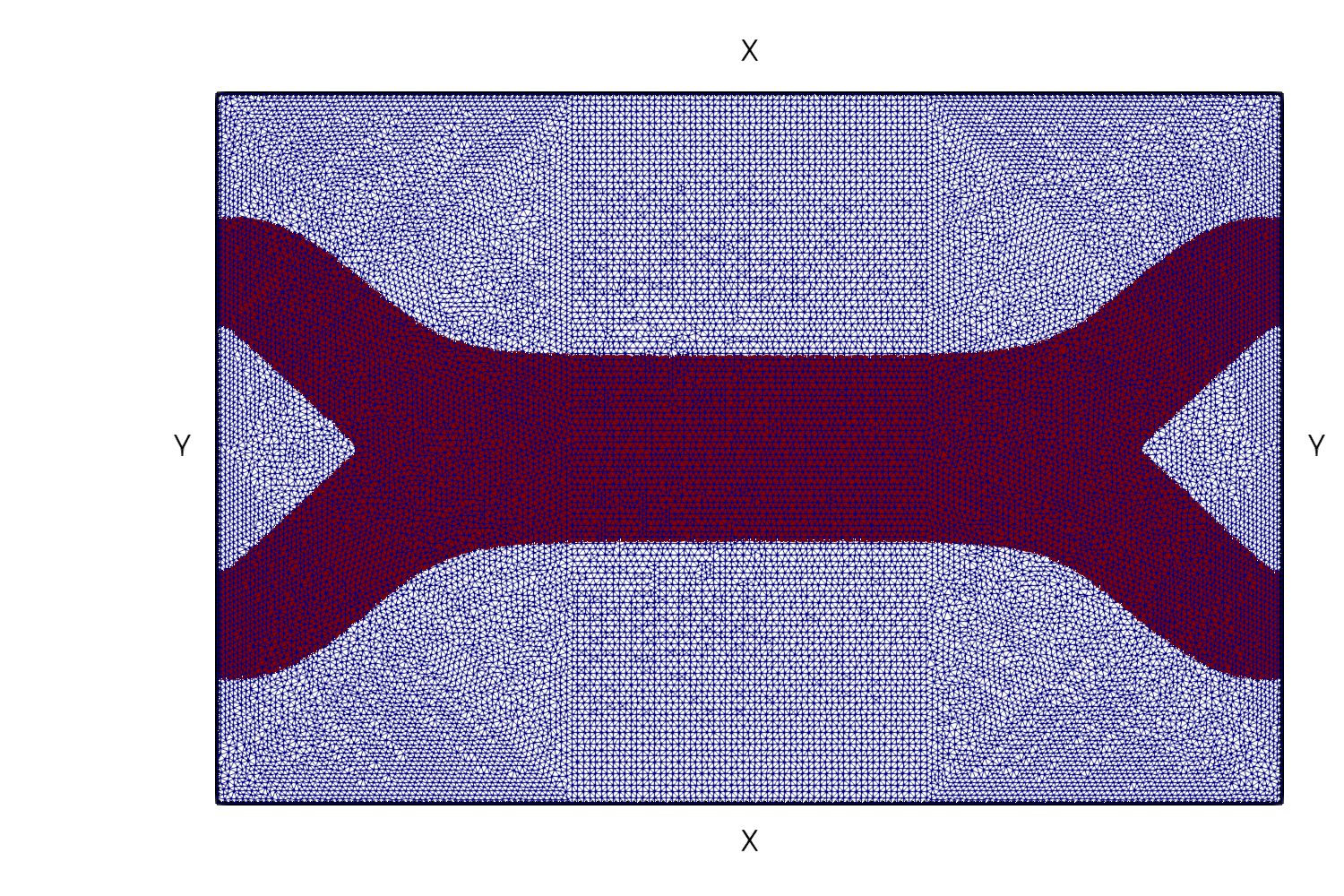}
		\end{subfigure}
            \begin{subfigure}[b]{0.32\textwidth}			\includegraphics[width=\linewidth]{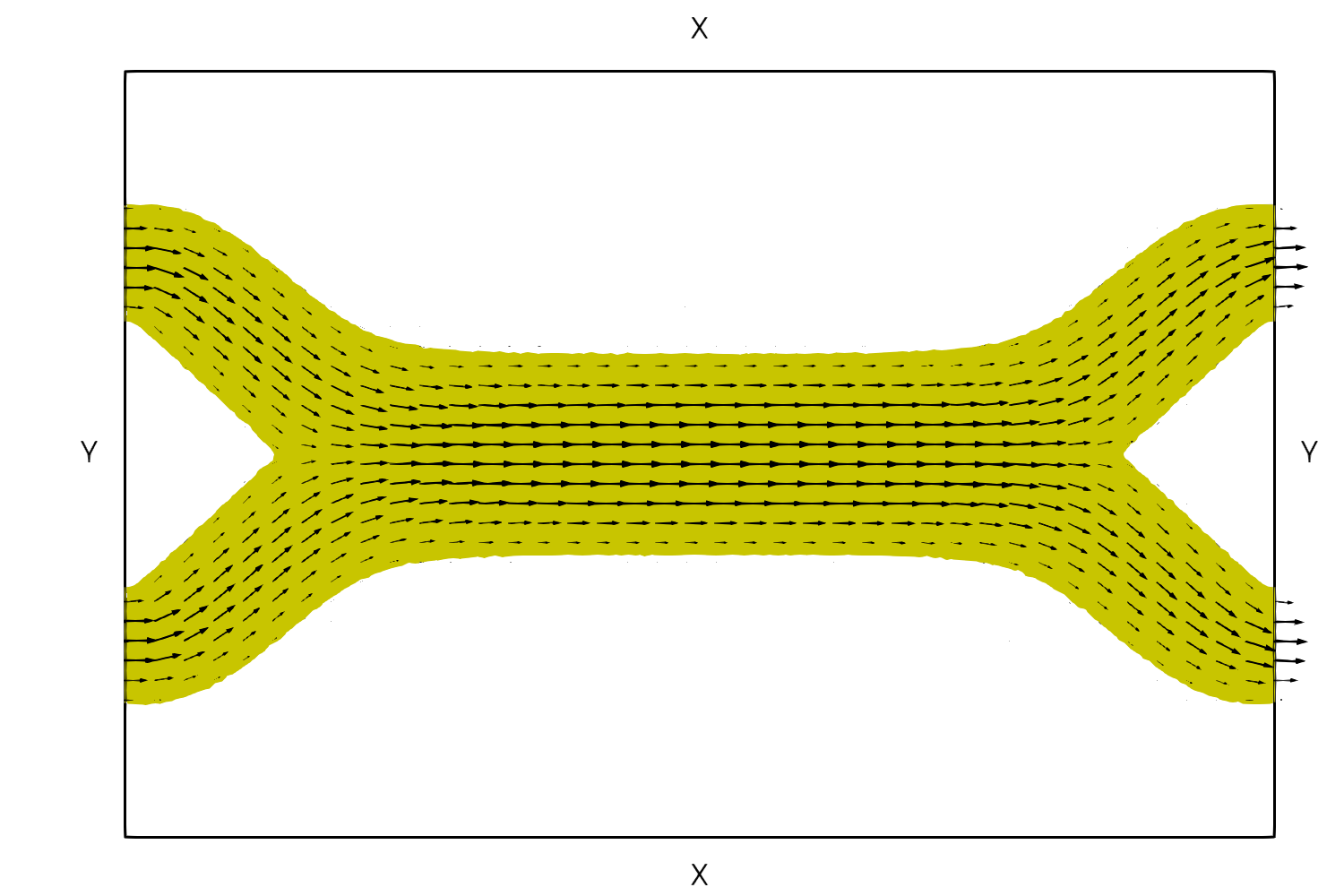}
		\end{subfigure}
        
        \centering
		\begin{subfigure}[b]                  {0.342\textwidth}
        \includegraphics[width=\linewidth]    {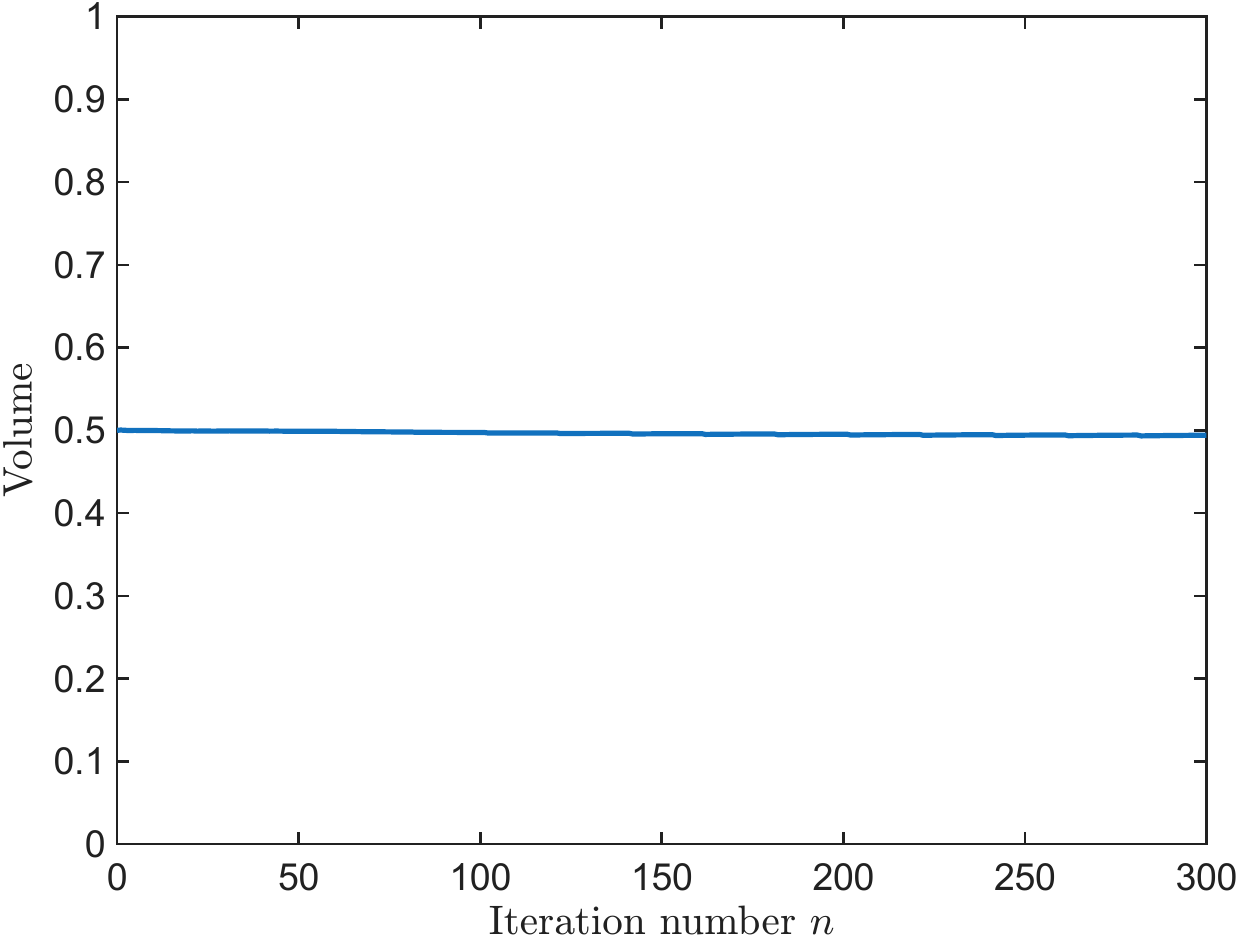} 
		\end{subfigure}
		\begin{subfigure}[b]{0.338\textwidth}		\includegraphics[width=\linewidth]{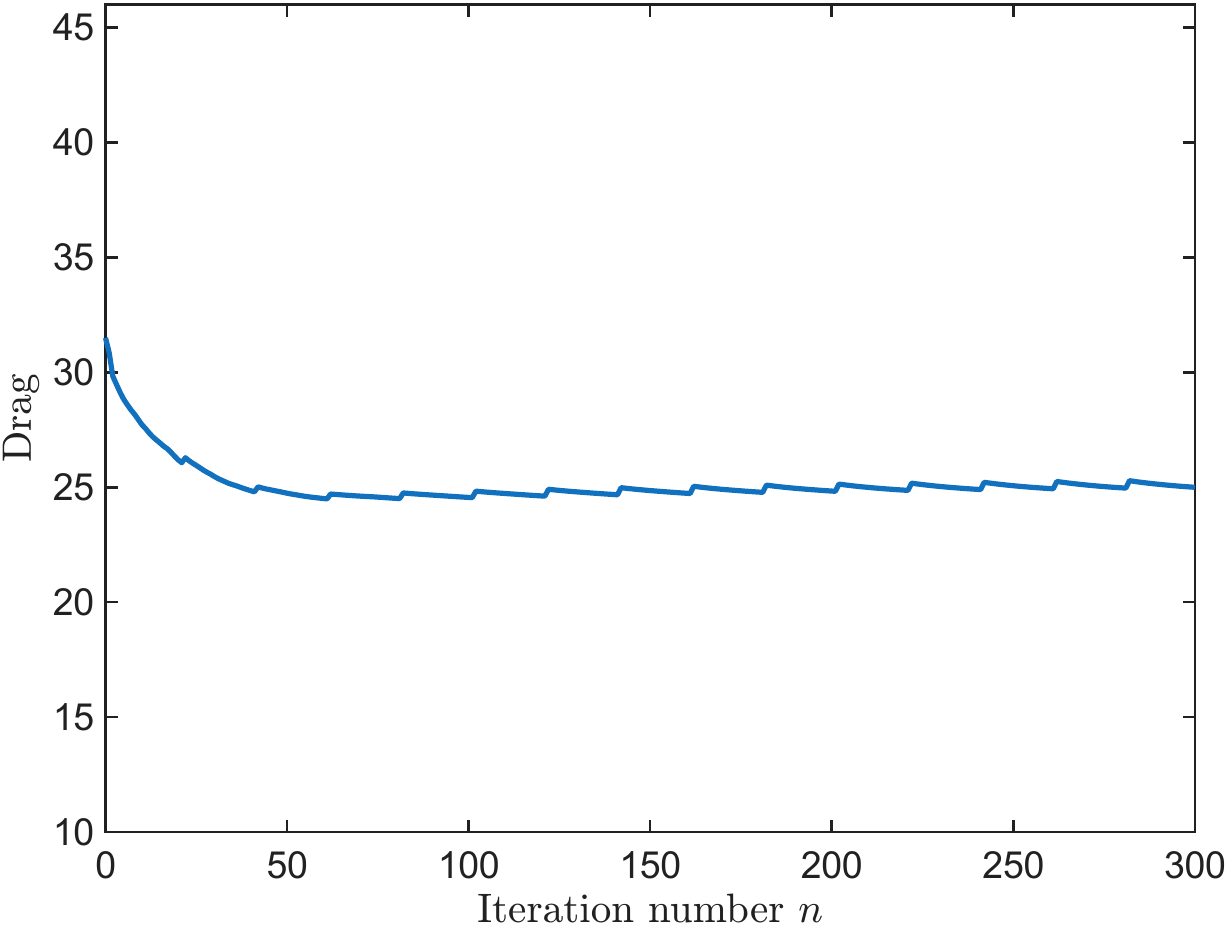}
		\end{subfigure}
		\caption{The initial design (top left), final design (top mid), stream lines of $\bm{u}$ (top right) and convergence histories of the volume (bottom left) and the energy (bottom right) for Example \ref{5_M2E1}.}
		\label{Figure_M2E1_VandJ}
	\end{figure}
	
	\begin{figure}[htbp]
		\centering
		\begin{subfigure}[b]{0.22\textwidth}
	    \includegraphics[width=\linewidth]{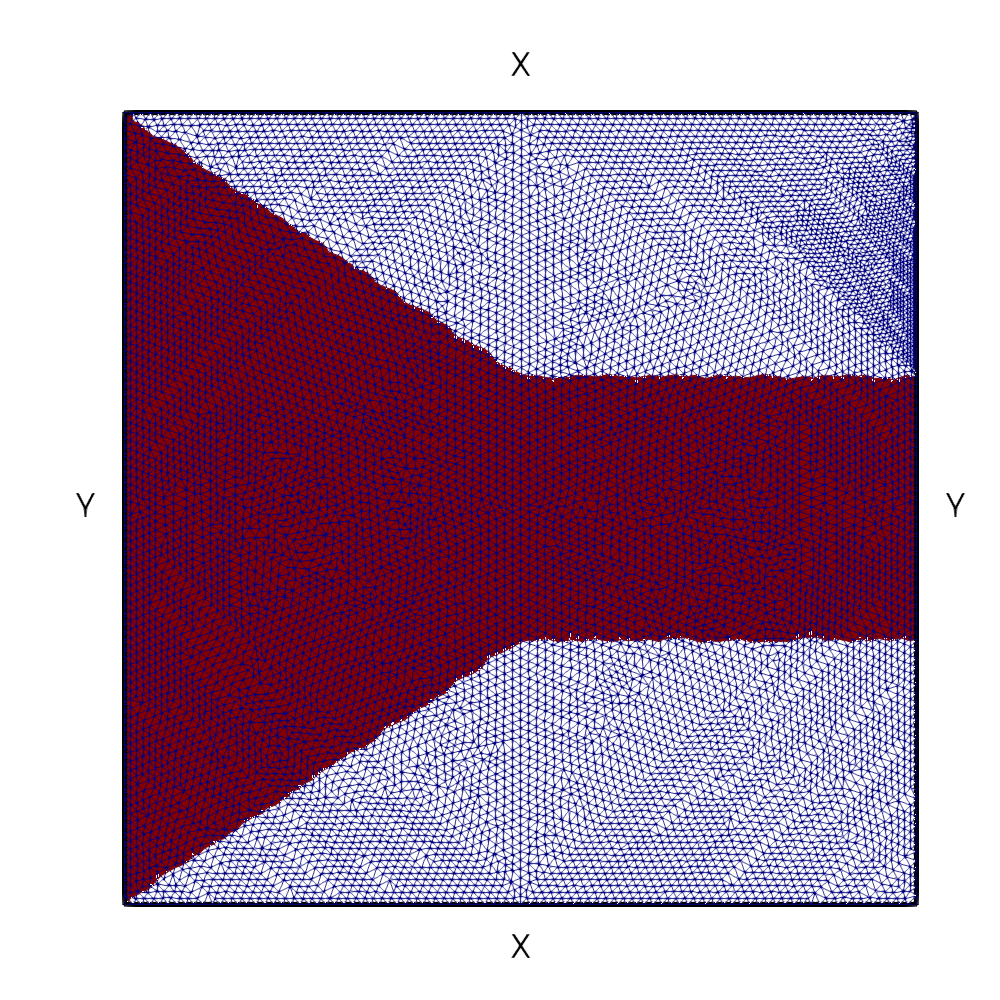} 
		\end{subfigure}
        \centering
        \begin{subfigure}[b]{0.22\textwidth}
		\includegraphics[width=\linewidth]{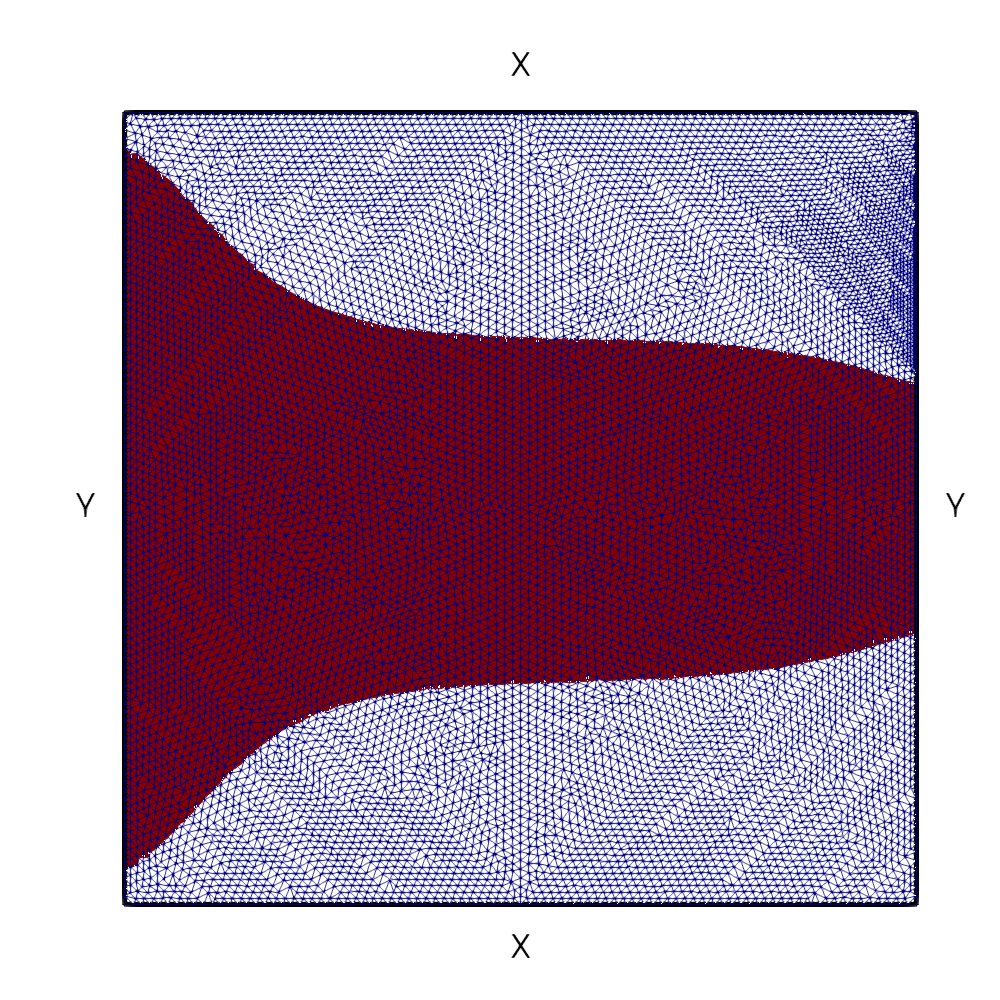}
		\end{subfigure}
        \begin{subfigure}[b]{0.22\textwidth}
		\includegraphics[width=\linewidth]{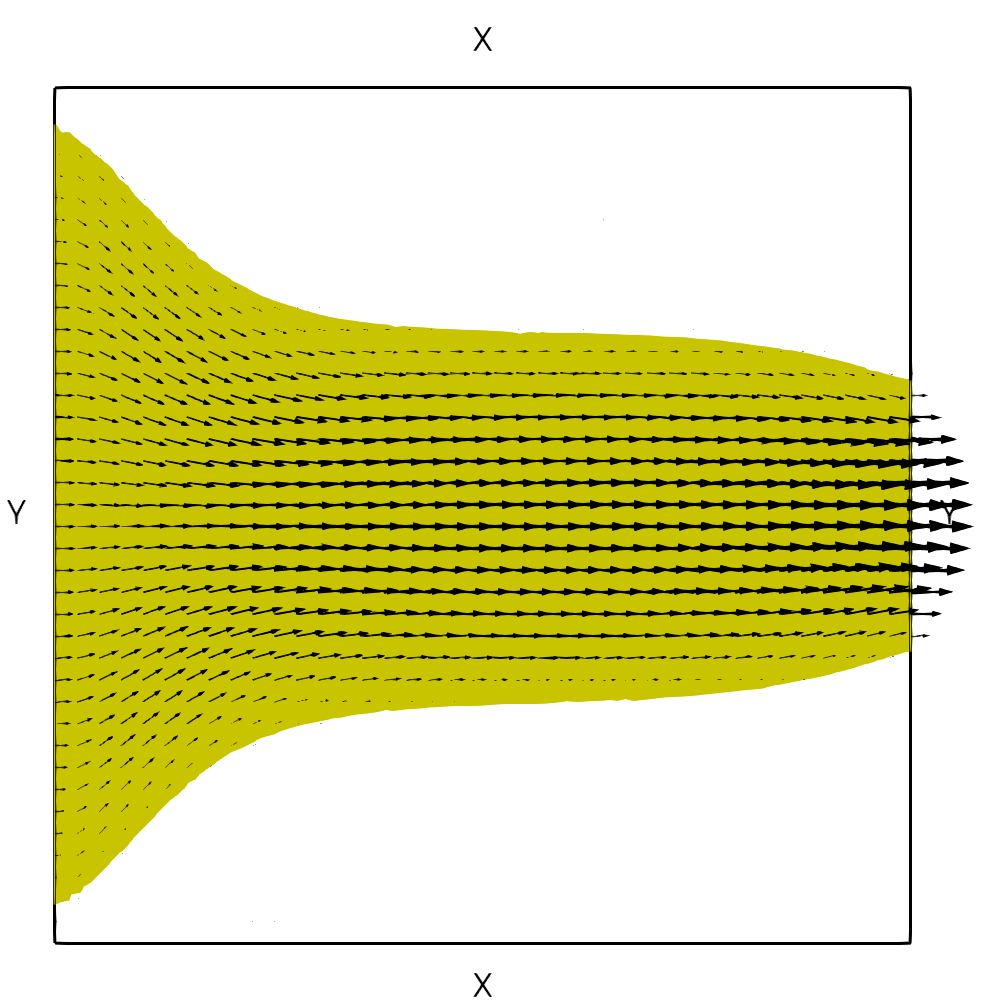}
		\end{subfigure}
        \\
		\begin{subfigure}[b]{0.332\textwidth}			\includegraphics[width=\linewidth]{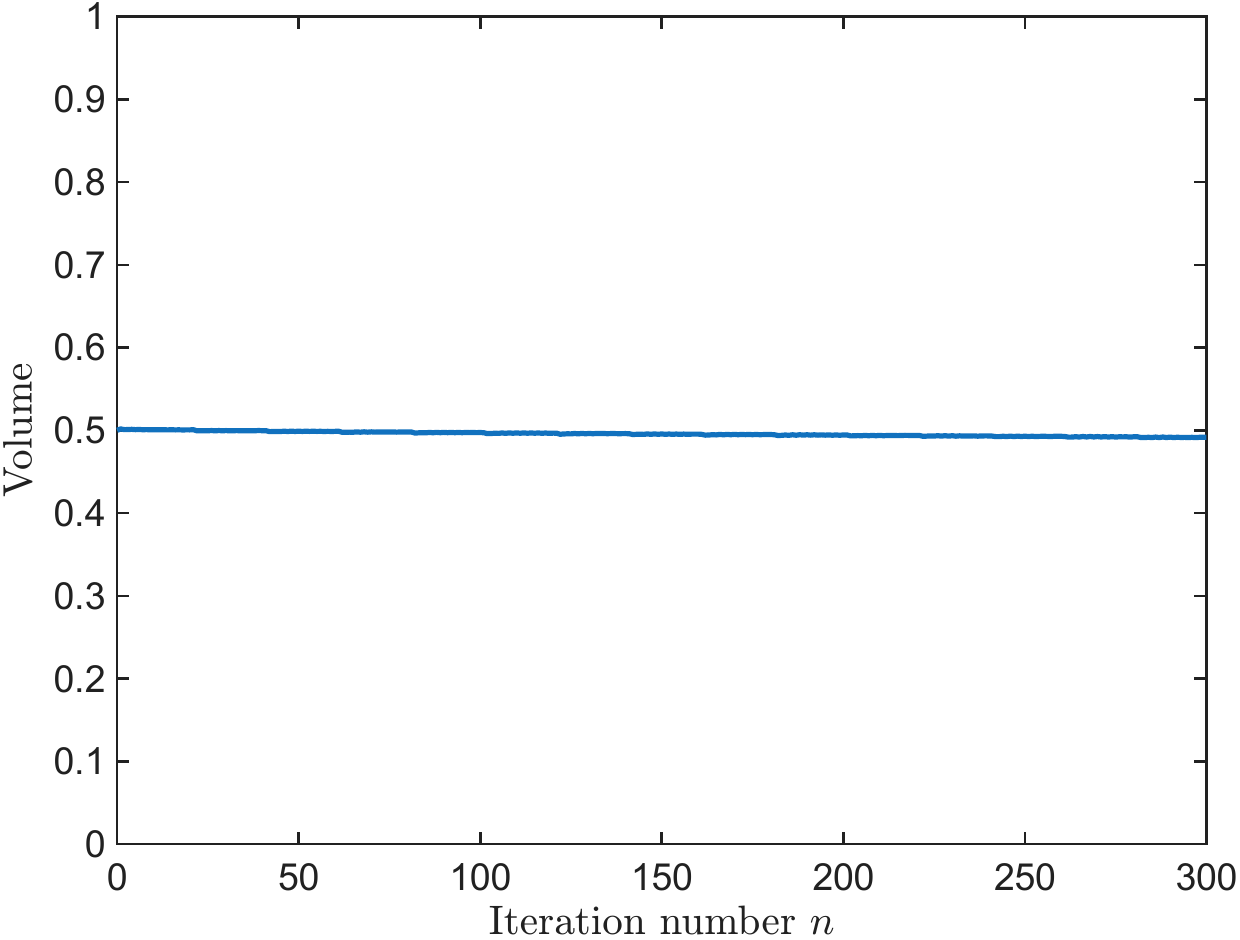} 
		\end{subfigure}
		\begin{subfigure}[b]{0.328\textwidth}		\includegraphics[width=\linewidth]{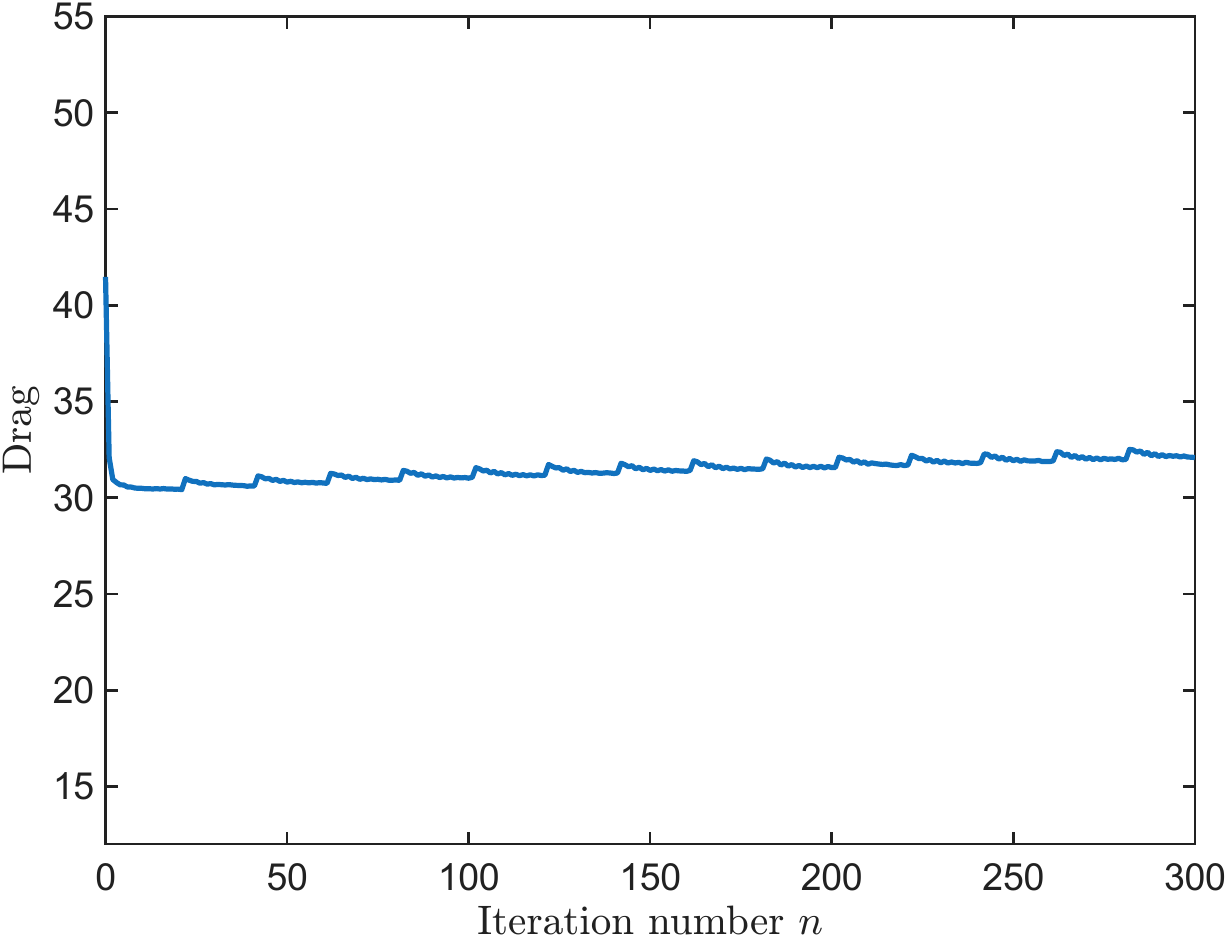}
		\end{subfigure}
		\caption{The initial design (top left), final design (top mid), stream lines of $\bm{u}$ (top right), and convergence histories of the volume (bottom left) and the energy (bottom right) for Example \ref{5_M2E2}.}
		\label{Figure_M2E2_VandJ}
	\end{figure}

	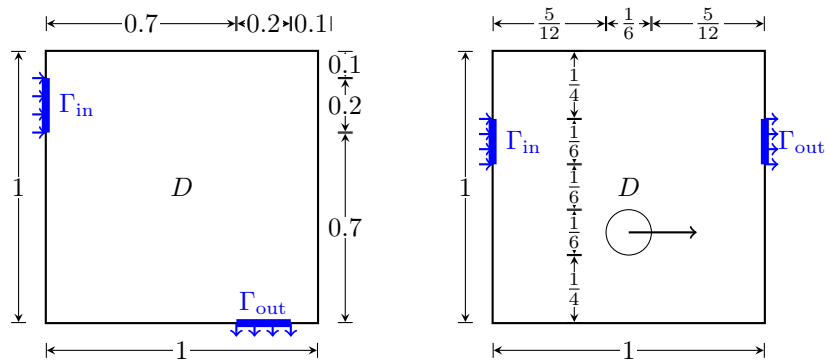
\begin{figure}[htbp] 
		\centering
		\begin{tikzpicture}[scale=3.6] 
			\draw[thick] (0,0) rectangle (1,1);
			\node at (0.5,0.5) {$D$};  
			
			\draw[line width=3pt, blue] (0, 0.7) -- (0, 0.9);
			\node[right=1.5pt, blue] at (0, 0.8) {$\Gamma_{\text{in}}$};
			\def\N{3}
			\def\dx{0.05}
			\def\dy{0}
			\foreach \i in {0,1,...,\N}
			{
				\pgfmathsetmacro{\ypos}{0.9 - \i/(5*\N)}
				\draw[->, thick, blue] (-\dx, \ypos) -- ++(\dx, 0);
			}
			
			\draw[line width=3pt, blue] (0.7, 0) -- (0.9, 0);
			\node[blue] at (0.8, 0.075) {$\Gamma_{\text{out}}$};
			\def\N{3}
			\def\dx{0}
			\def\dy{0.05}
			\foreach \i in {0,1,...,\N}
			{
				\pgfmathsetmacro{\xpos}{0.9 - \i/(5*\N)}
				\draw[->, thick, blue] (\xpos, 0) -- ++(0, -\dy);
			}
			
			\tikzset{
				dim/.style={black, line width=0.4pt, |<->|, >=stealth},
				dim label/.style={midway, fill=white, inner sep=0.6pt}
			}
			
			\draw[<->, dim] (0.9, 1.1) -- (1.05, 1.1) 
			node[dim label] {$0.1$};
			\draw[<->, dim] (0.7, 1.1) -- (0.9, 1.1) 
			node[dim label] {$0.2$};
			\draw[<->, dim] (0.0, 1.1) -- (0.7, 1.1) 
			node[dim label] {$0.7$};
			\draw[<->, dim] (0,-0.1) -- (1,-0.1) 
			node[dim label] {1};
			
			\draw[<->, dim] (1.1, 0.9) -- (1.1, 1.0) 
			node[dim label] {$0.1$};
			\draw[<->, dim] (1.1, 0.7) -- (1.1, 0.9) 
			node[dim label] {$0.2$};
			\draw[<->, dim] (1.1, 0.0) -- (1.1, 0.7) 
			node[dim label] {$0.7$};
			\draw[<->, dim] (-0.1, 0) -- (-0.1, 1) 
			node[dim label] {1};
		\end{tikzpicture}
        \hspace{1cm}
        \begin{tikzpicture}[scale=3.6] 
			\draw[thick] (0,0) rectangle (1,1);
			\node at (0.5,0.5) {$D$};  
			
			\draw[line width=3pt, blue] (0, 7.0/12.0) -- (0, 0.75);
			\node[right=1.5pt, blue] at (0, 2.0/3.0) {$\Gamma_{\text{in}}$};
			\def\N{3}
			\def\dx{0.05}
			\def\dy{0}
			\foreach \i in {0,1,...,\N}
			{
				\pgfmathsetmacro{\ypos}{0.75 - \i/(6*\N)}
				\draw[->, thick, blue] (-\dx, \ypos) -- ++(\dx, 0);
			}
			
			\draw[line width=3pt, blue] (1, 7.0/12.0) -- (1, 0.75);
			\node[right=1.5pt, blue] at (1, 2.0/3.0) {$\Gamma_{\text{out}}$};
			\def\N{3}
			\def\dx{0.05}
			\def\dy{0}
			\foreach \i in {0,1,...,\N}
			{
				\pgfmathsetmacro{\ypos}{0.75 - \i/(6*\N)}
				\draw[->, thick, blue] (1, \ypos) -- ++(\dx, 0);
			}
			
			\draw[black] (0.5, 1.0/3.0) circle (1.0/12.0);
			\draw[->, thick, black] (0.5, 1.0/3.0) -- ++(0.25, 0);
			\tikzset{
				dim/.style={black, line width=0.4pt, |<->|, >=stealth},
				dim label/.style={midway, fill=white, inner sep=0.6pt}
			}
			
			\draw[<->, dim] (0.0, 1.1) -- (5.0/12.0, 1.1) 
			node[dim label] {$\frac{5}{12}$};
			\draw[<->, dim] (5.0/12.0, 1.1) -- (7.0/12.0, 1.1) 
			node[dim label] {$\frac{1}{6}$};
			\draw[<->, dim] (7.0/12.0, 1.1) -- (1.0, 1.1) 
			node[dim label] {$\frac{5}{12}$};
			\draw[<->, dim] (0,-0.1) -- (1,-0.1) 
			node[dim label] {1};
			
			\draw[<->, dim] (0.3, 0.0) -- (0.3, 0.25) 
			node[dim label] {$\frac{1}{4}$};
			\draw[<->, dim] (0.3, 0.25) -- (0.3, 5.0/12.0) 
			node[dim label] {$\frac{1}{6}$};
			\draw[<->, dim] (0.3, 5.0/12.0) -- (0.3, 7.0/12.0) 
			node[dim label] {$\frac{1}{6}$};
			\draw[<->, dim] (0.3, 7.0/12.0) -- (0.3, 0.75) 
			node[dim label] {$\frac{1}{6}$};
			\draw[<->, dim] (0.3, 0.75) -- (0.3, 1.0) 
			node[dim label] {$\frac{1}{4}$};
			\draw[<->, dim] (-0.1, 0) -- (-0.1, 1) 
			node[dim label] {1};
		\end{tikzpicture}
		\caption{The design domain $D$ and its boundary for Example \ref{5_M2E3} (left) and Example \ref{5_M2E4} (right).}
		\label{5_M2_Ex3}
	\end{figure}
	
	\begin{figure}[htbp]
		\centering
		\begin{subfigure}[b]{0.32\textwidth}
			\centering
			\includegraphics[width=0.75\linewidth]{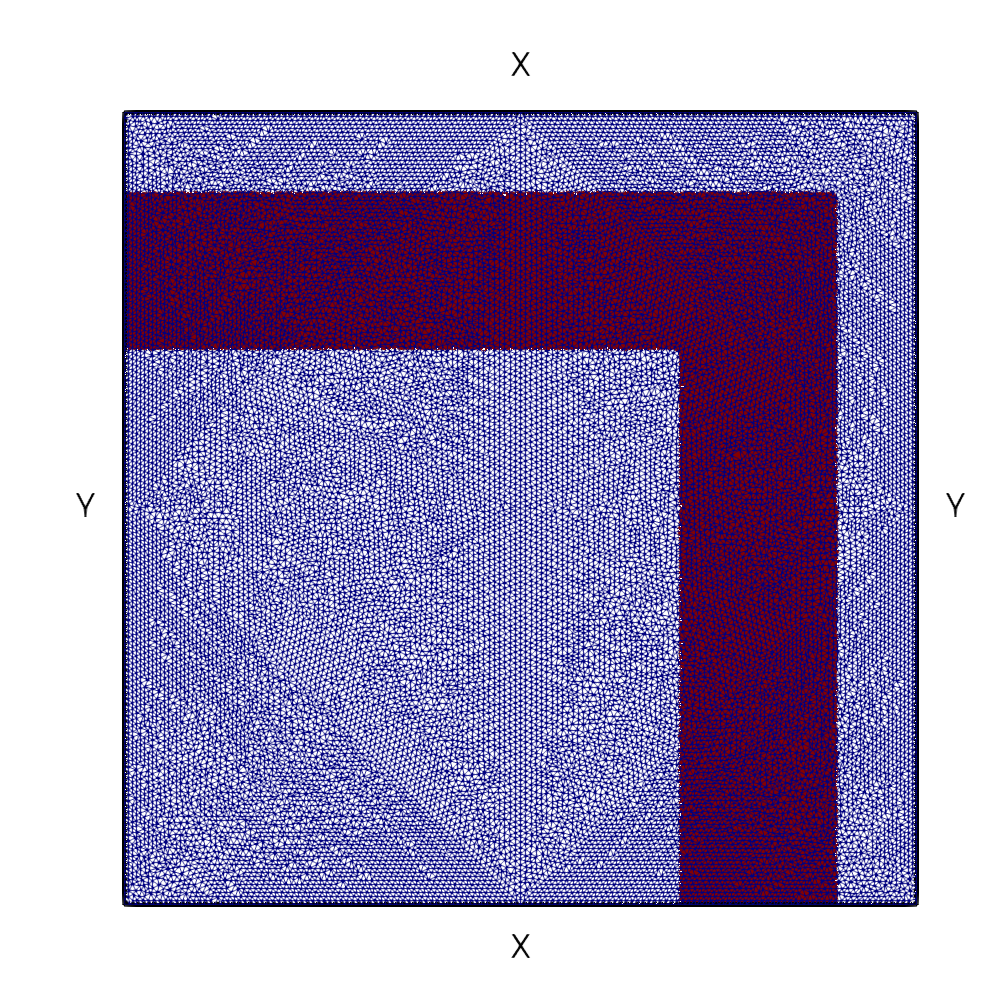} 
		\end{subfigure}
		\begin{subfigure}[b]{0.32\textwidth}
			\centering
			\includegraphics[width=0.75\linewidth]{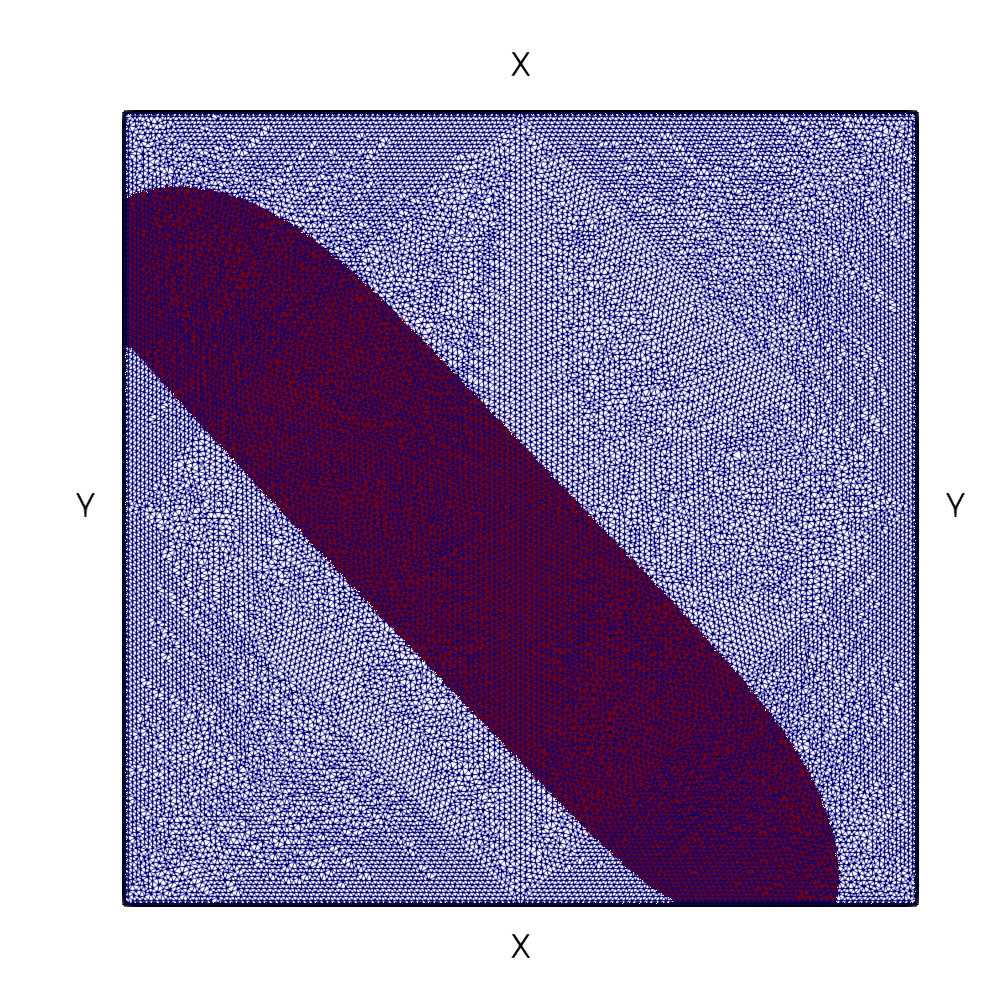}
		\end{subfigure}
            \begin{subfigure}[b]{0.32\textwidth}
			\centering
			\includegraphics[width=0.75\linewidth]{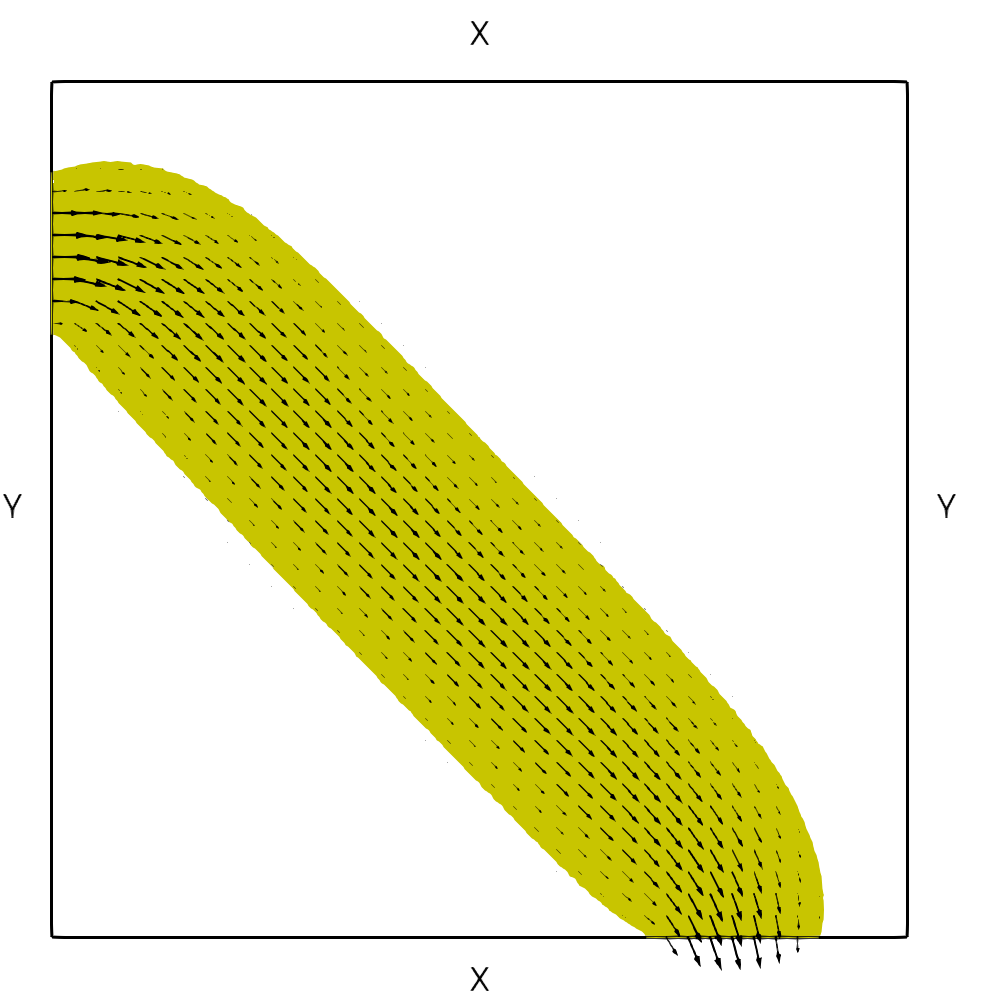}
		\end{subfigure}
        
		\begin{subfigure}[b]{0.352\textwidth}			\includegraphics[width=\linewidth]{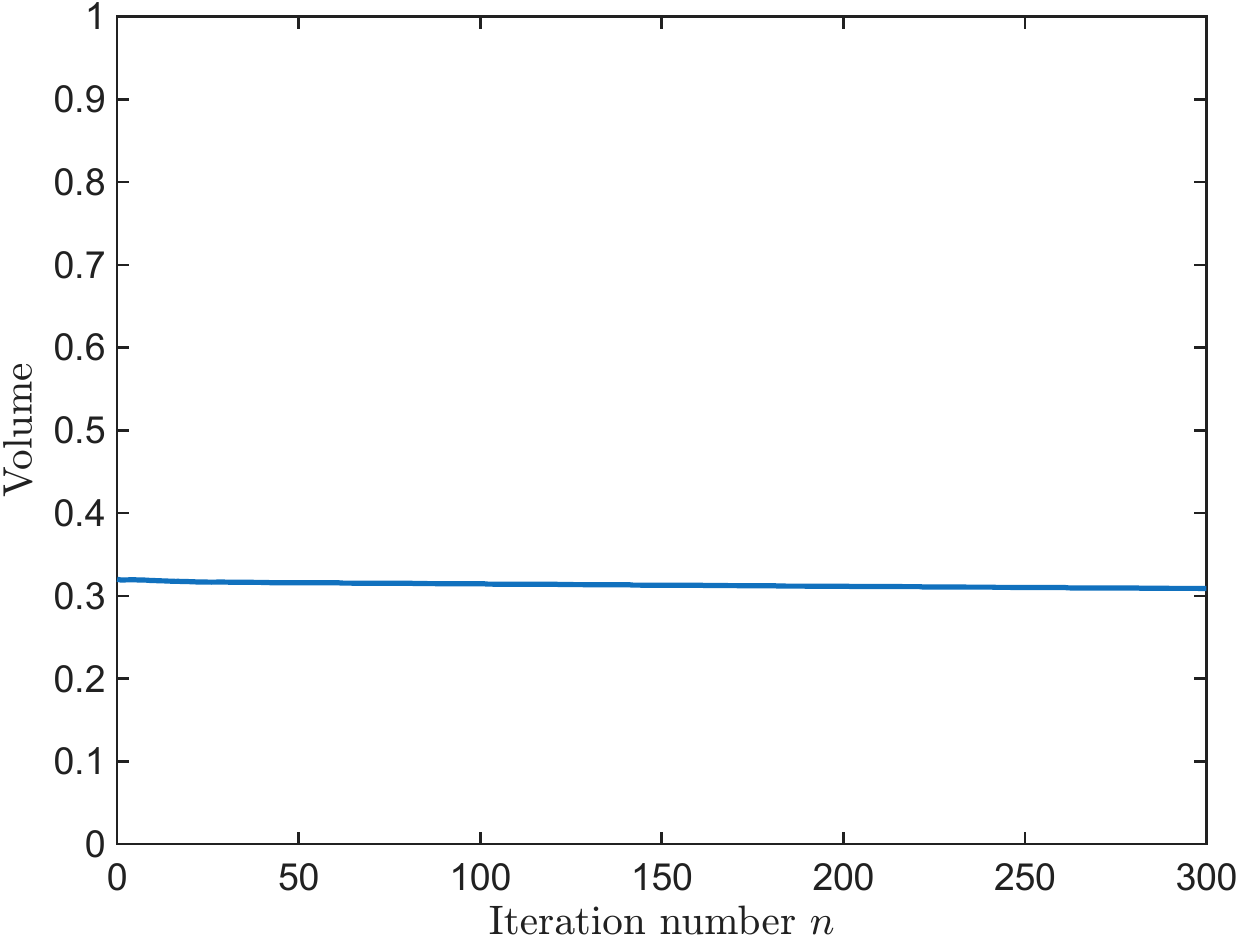} 
		\end{subfigure}
		\begin{subfigure}[b]{0.348\textwidth}			\includegraphics[width=\linewidth]{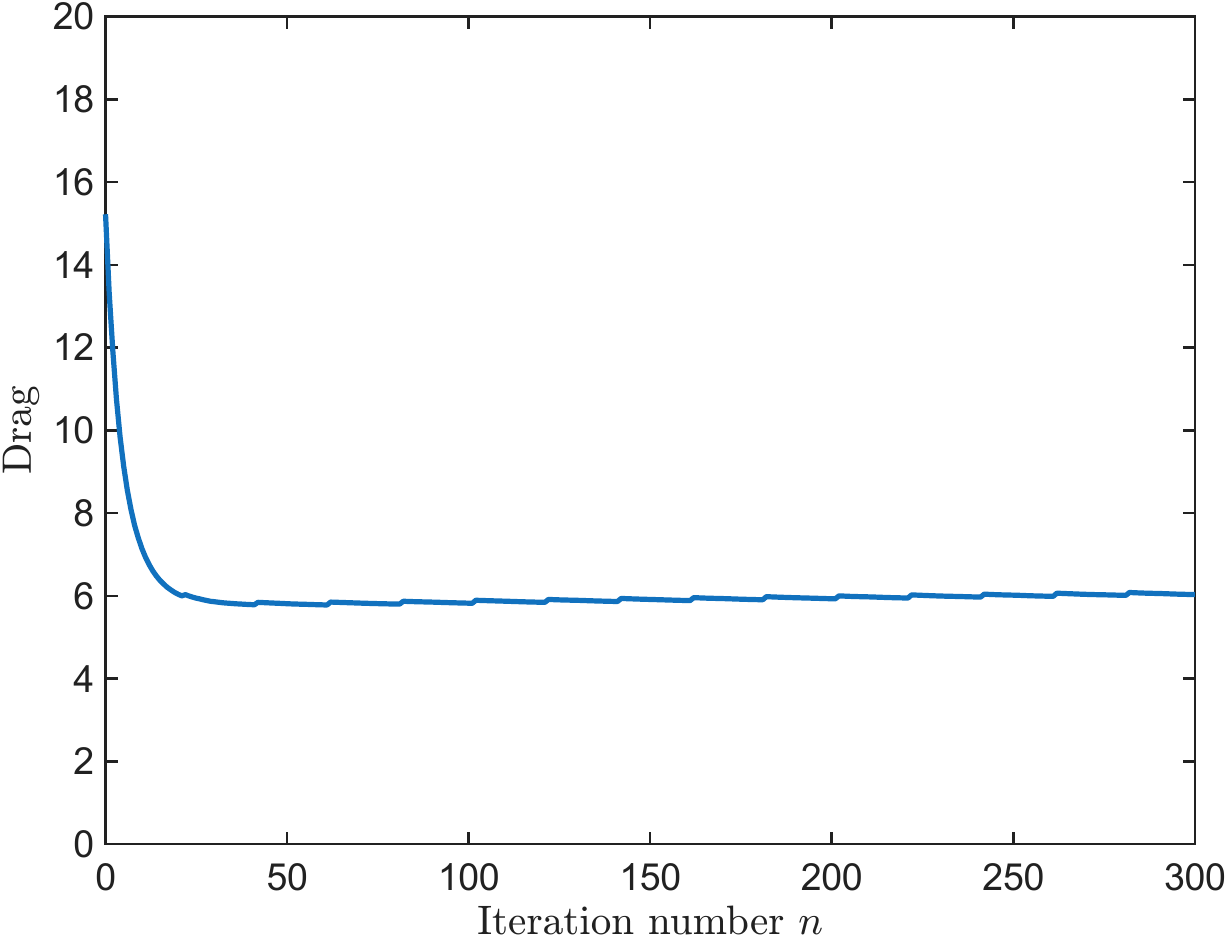}
		\end{subfigure}
		\caption{The initial design (top left), final design (top mid), stream lines of $\bm{u}$ (top right), and convergence histories of the volume (bottom left) and the energy (bottom right) for Example \ref{5_M2E3}.}
		\label{Figure_M2E3_VandJ}
	\end{figure}
		
	\begin{figure}[htbp]
		\centering
            \captionsetup[subfigure]{font=scriptsize}
		\begin{subfigure}[b]{0.24\textwidth}
			\includegraphics[width=\linewidth]{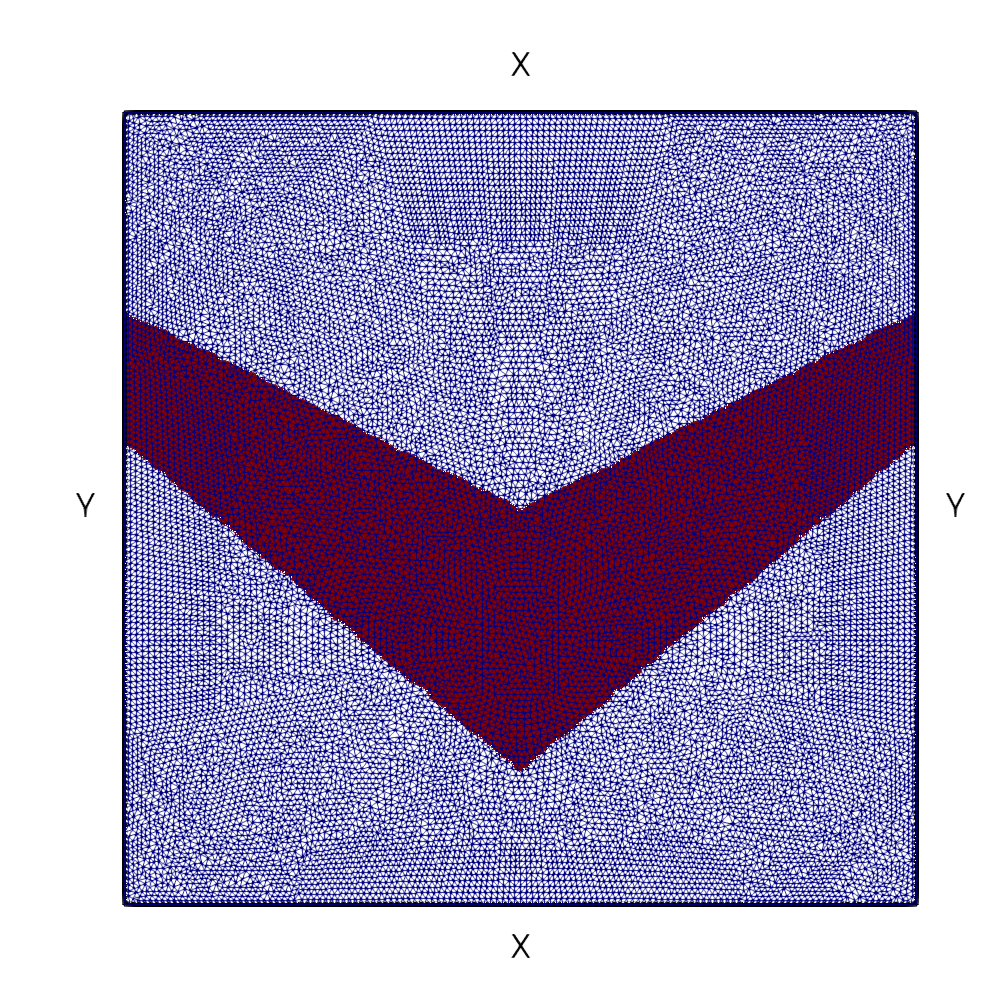} 
			\caption{Initial design}
			\label{Figure_M2E4_Rho_a}
		\end{subfigure}
		\begin{subfigure}[b]{0.24\textwidth}
			\includegraphics[width=\linewidth]{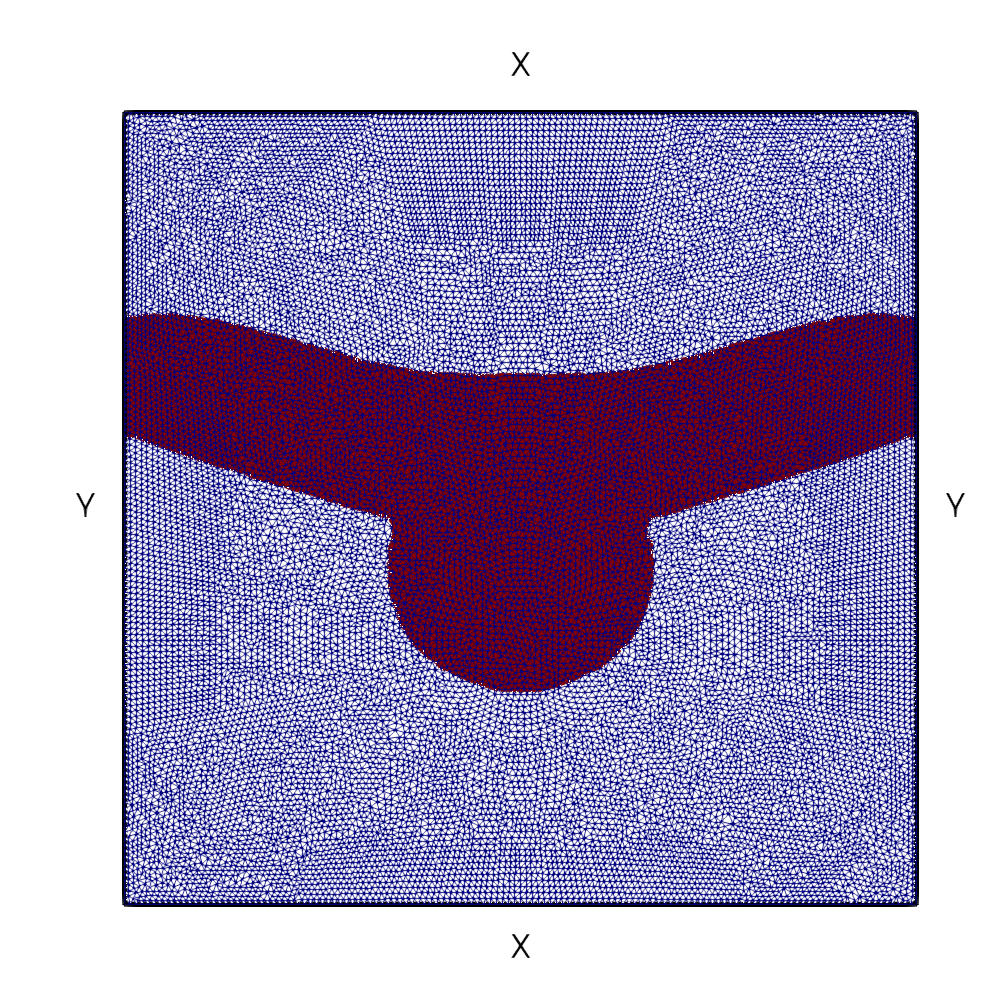}
			\caption{Final design: $\bm{f} = (-1125, 0)^T$}
			\label{Figure_M2E4P1_Rho_d}
		\end{subfigure}
		\begin{subfigure}[b]{0.24\textwidth}
			\includegraphics[width=\linewidth]{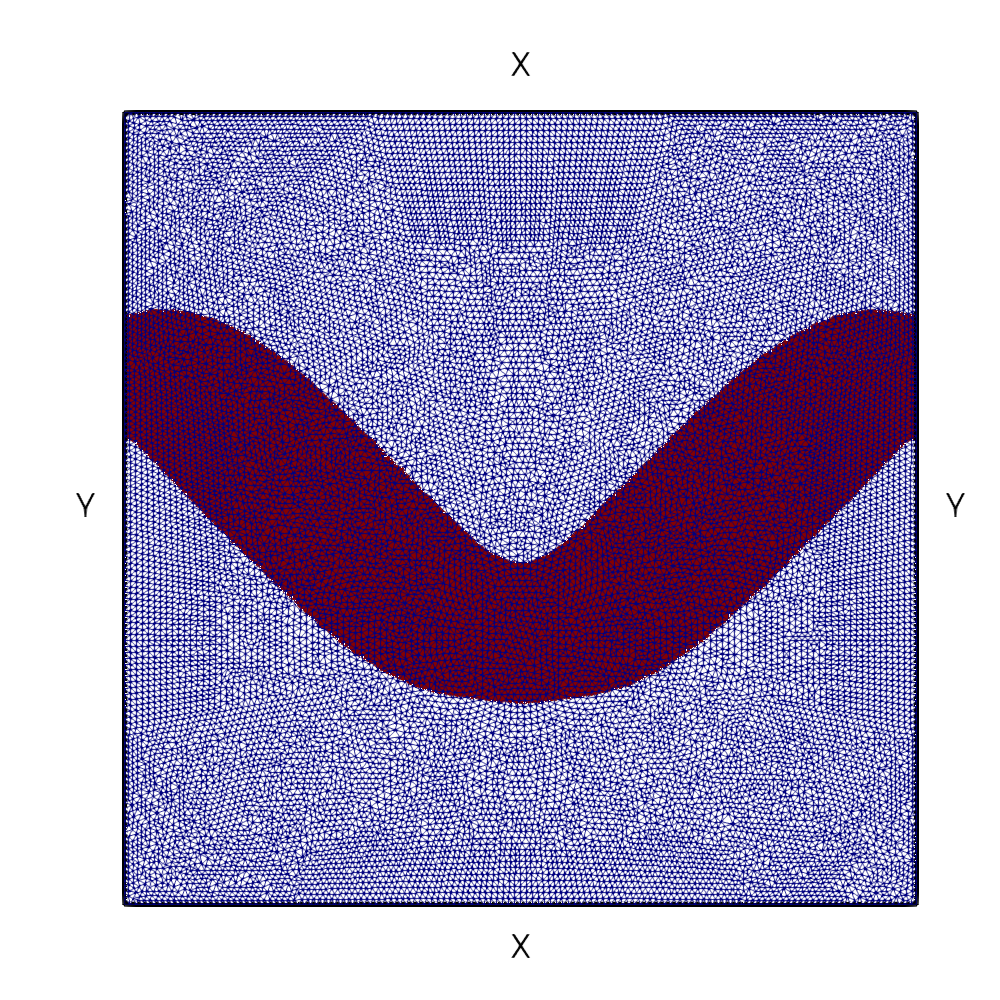}
			\caption{Final design: $\bm{f} = (562.5, 0)^T$}
			\label{Figure_M2E4P2_Rho_d}
		\end{subfigure}
		\begin{subfigure}[b]{0.24\textwidth}
			\includegraphics[width=\linewidth]{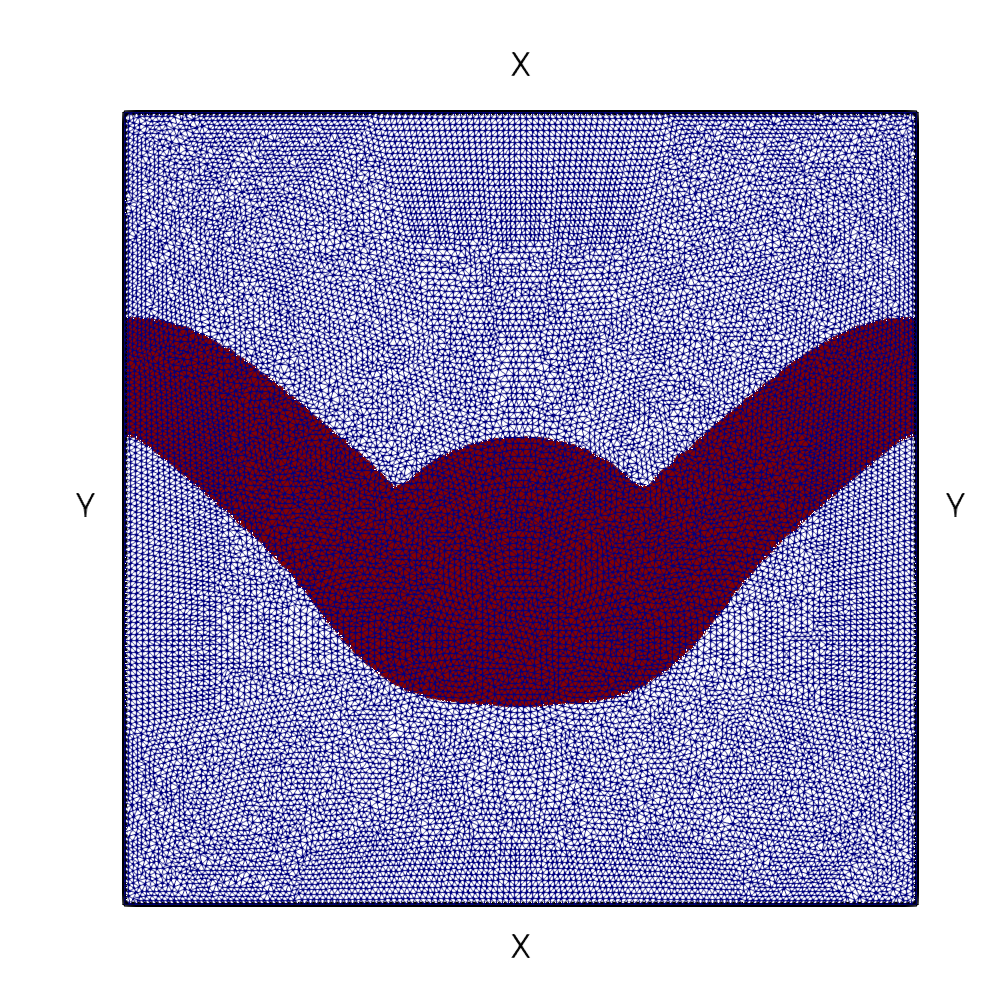}
			\caption{Final design: $\bm{f} = (1687.5, 0)^T$}
			\label{Figure_M2E4P3_Rho_d}
		\end{subfigure}
		\caption{Three cases for Example \ref{5_M2E4}.}
		\label{c}
	\end{figure}
	
	\begin{figure}[htbp]
		\centering
		\begin{subfigure}[b]{0.317\textwidth}			\includegraphics[width=\linewidth]{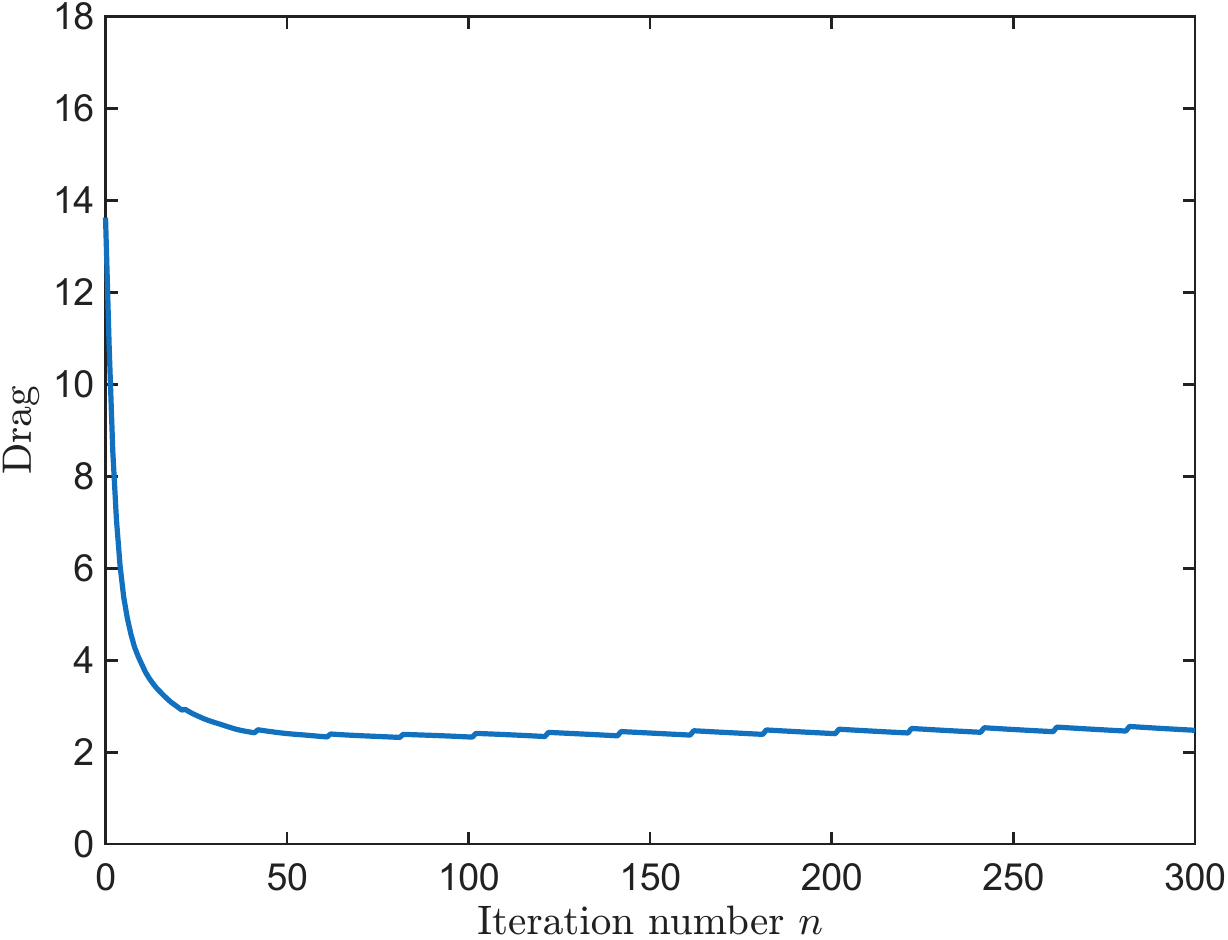} 
		\end{subfigure}
		\hfill 
		\begin{subfigure}[b]{0.311\textwidth}			\includegraphics[width=\linewidth]{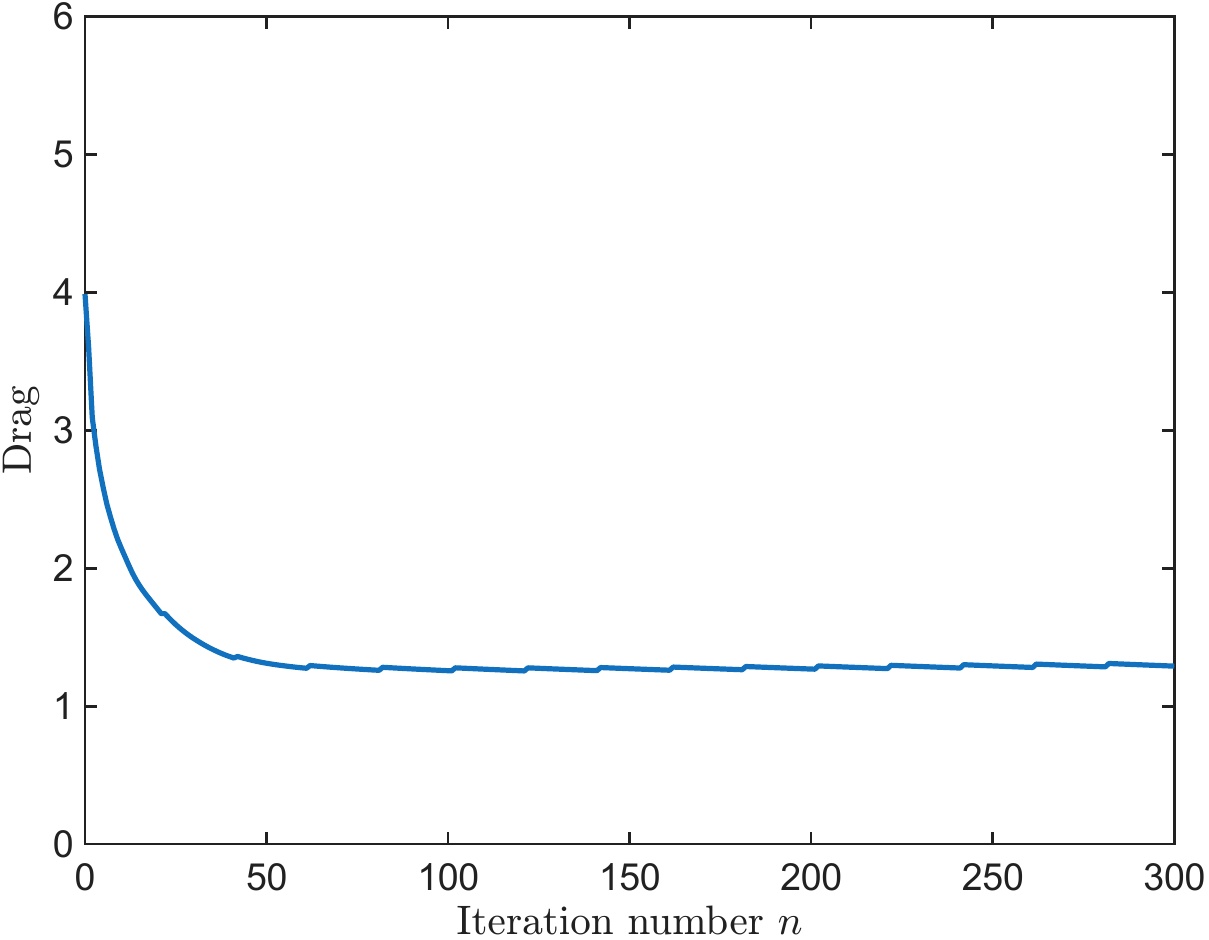}
		\end{subfigure}
		\hfill
		\begin{subfigure}[b]{0.32\textwidth}
			\includegraphics[width=\linewidth]{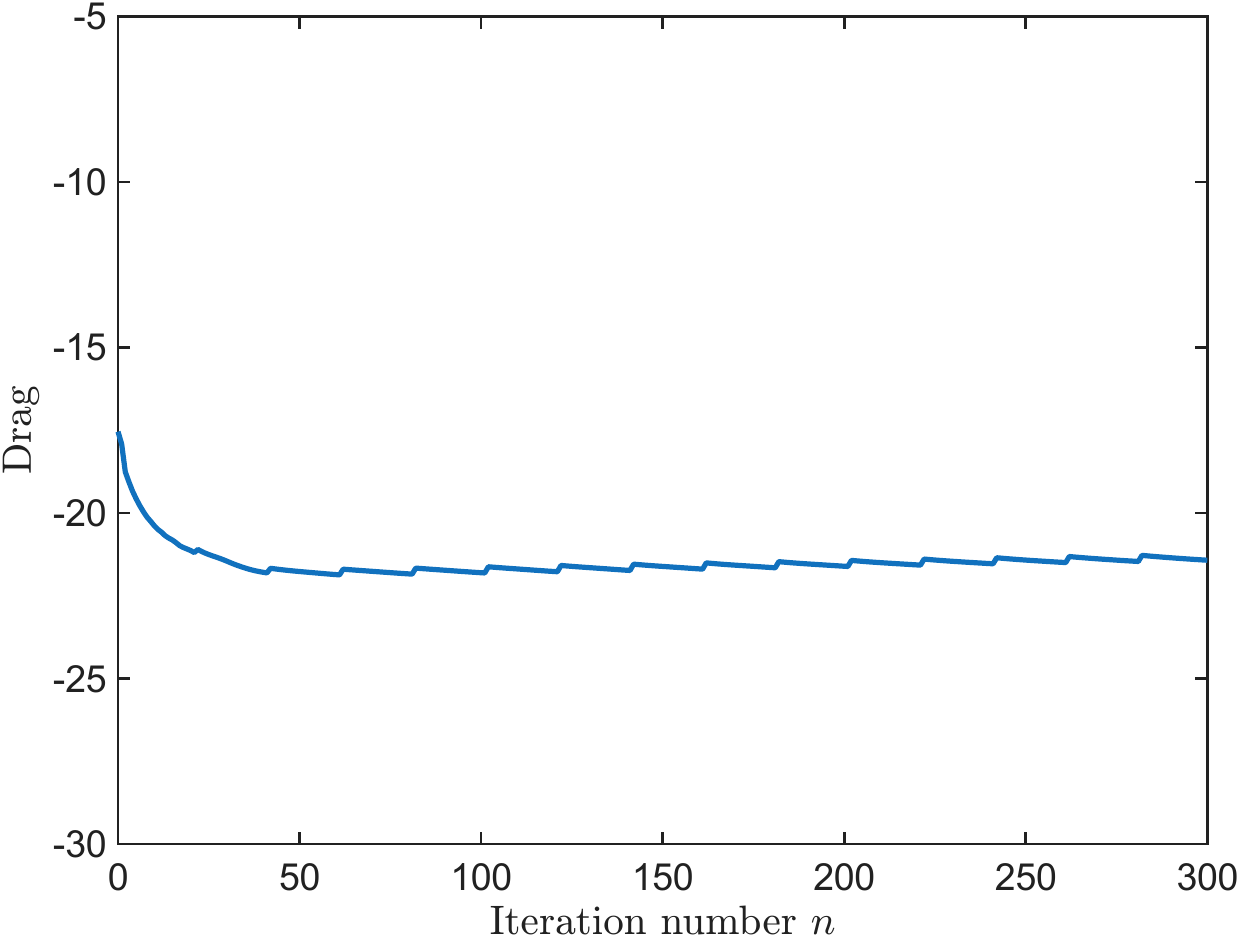}
		\end{subfigure}
		\centering
		\begin{subfigure}[b]{0.32\textwidth}
			\includegraphics[width=\linewidth]{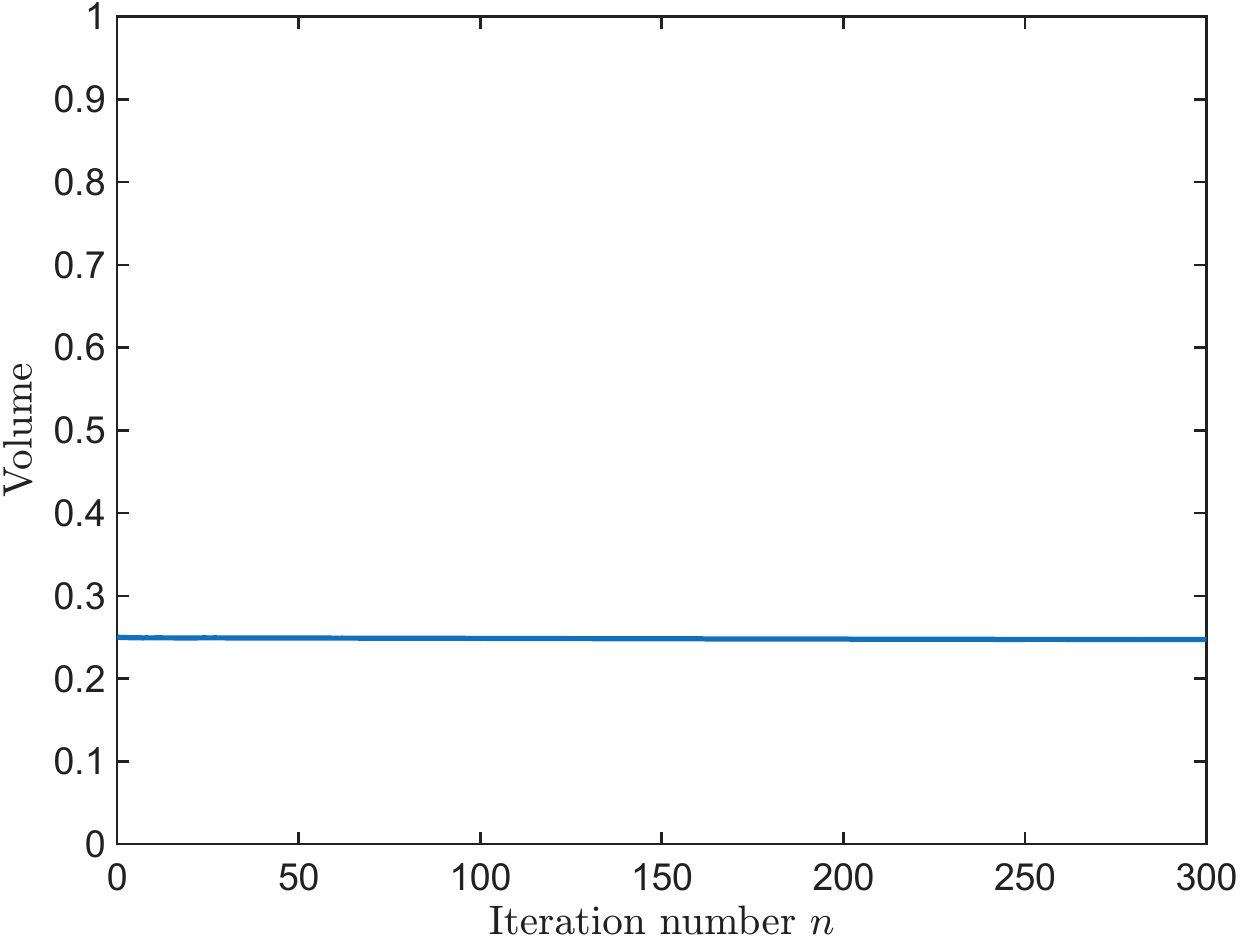} 
		\end{subfigure}
		\hfill 
		\begin{subfigure}[b]{0.32\textwidth}
			\includegraphics[width=\linewidth]{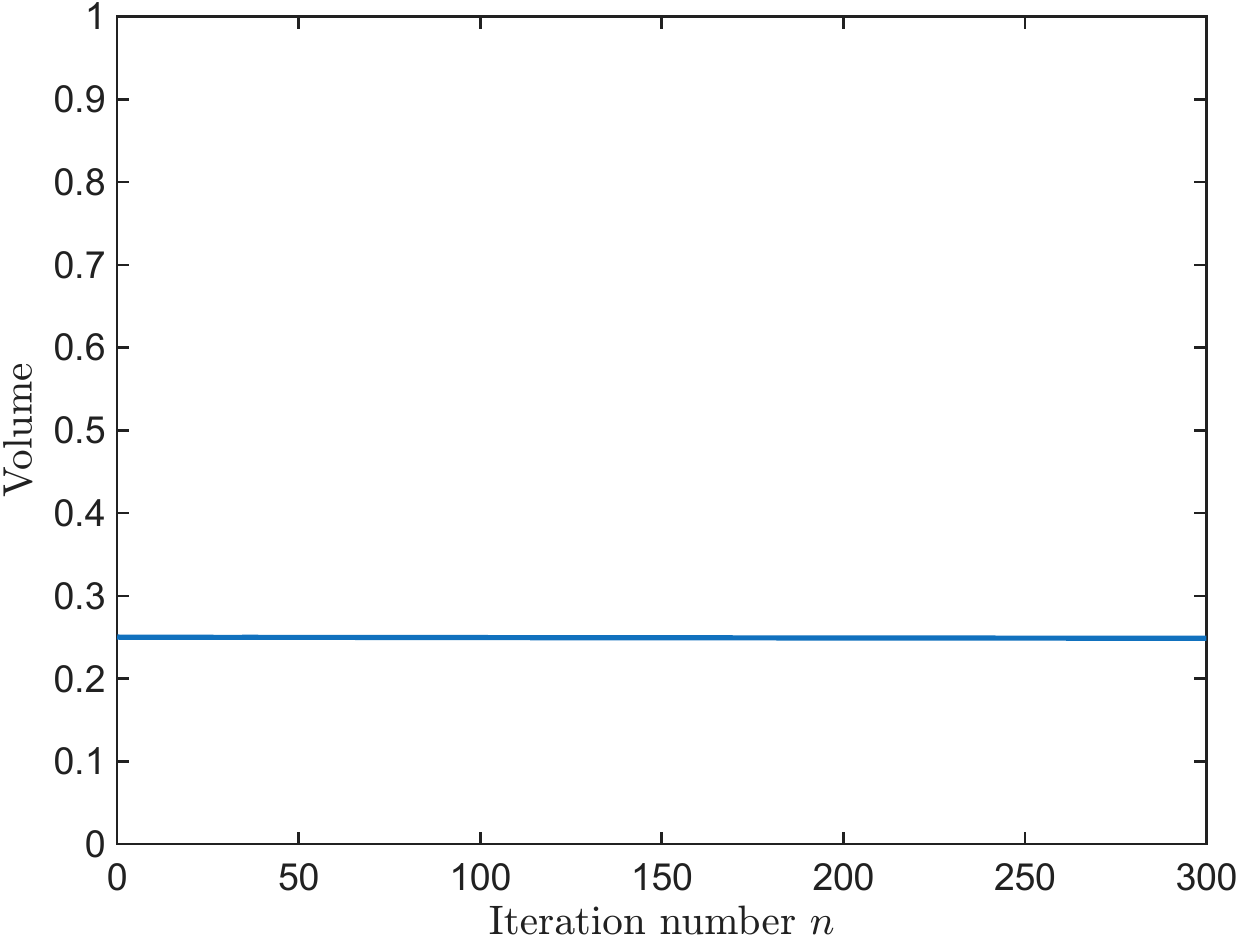}
		\end{subfigure}
		\hfill
		\begin{subfigure}[b]{0.32\textwidth}
			\includegraphics[width=\linewidth]{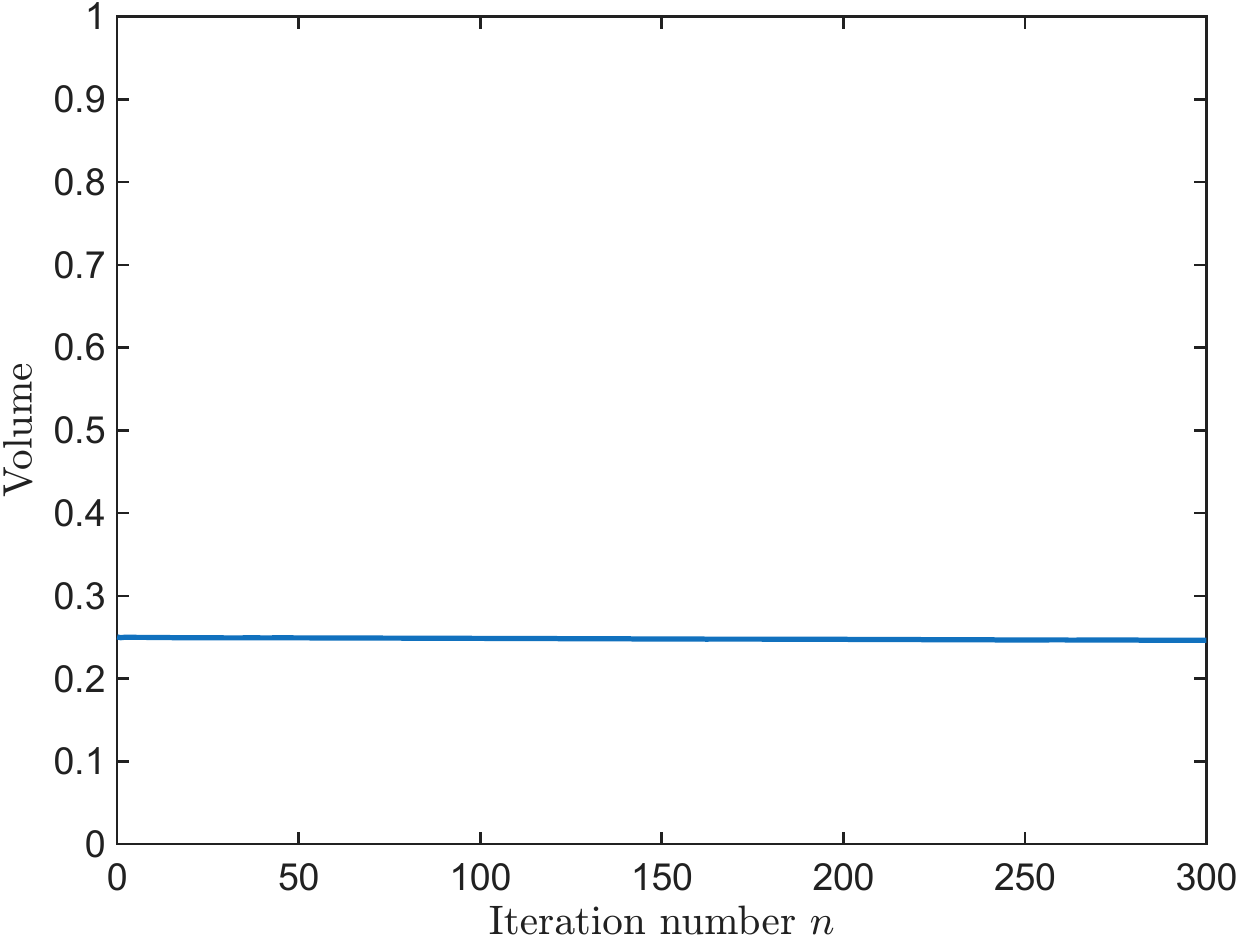}
		\end{subfigure}
		\caption{Convergence histories of the objective functional (top) and the volume (bottom) for Example \ref{5_M2E4}: Case 1, Case 2, and Case 3 from left to right.}
		\label{Figure_M2E4_VandJ}
	\end{figure}

    \begin{example}
        \label{5_M2E5}
        (2D flows past an obstacle \cite{pironneau1973optimum_1})  Set $D = (0,2) \times (0,1)$, $\Gamma_{\rm{in}} = \lbrace (0,y): 0 \leqslant y \leqslant 1 \rbrace$ and  $\Gamma_{\rm{out}} = \left\lbrace (2,y): 0 \leqslant y \leqslant 1 \right\rbrace$ (see the right plot of \autoref{5_M4_Ex3}). Set
	\begin{equation}
		\bm{g}_f = \left\{
		\begin{aligned}
			\Big(4y(1-y), 0\Big)^T,\ & x = 0 \text{ or } 2, 0 \leqslant y \leqslant 1, \\
			(0, 0)^T,\ & \text{ otherwise on } \partial D.
		\end{aligned}
        \right.
	\end{equation}
        \end{example}
        
The initial shape with a circular obstacle is shown in \autoref{5_M4_Ex3}, where the final shape can also be found. \autoref{Figure_M2E5_VandJ} shows the variations of the approximate volume and the drag, respectively.

    \begin{figure}[htbp]
    \centering
        \begin{subfigure}[b]{0.32\textwidth}
			\centering
			\includegraphics[width=2.5in]{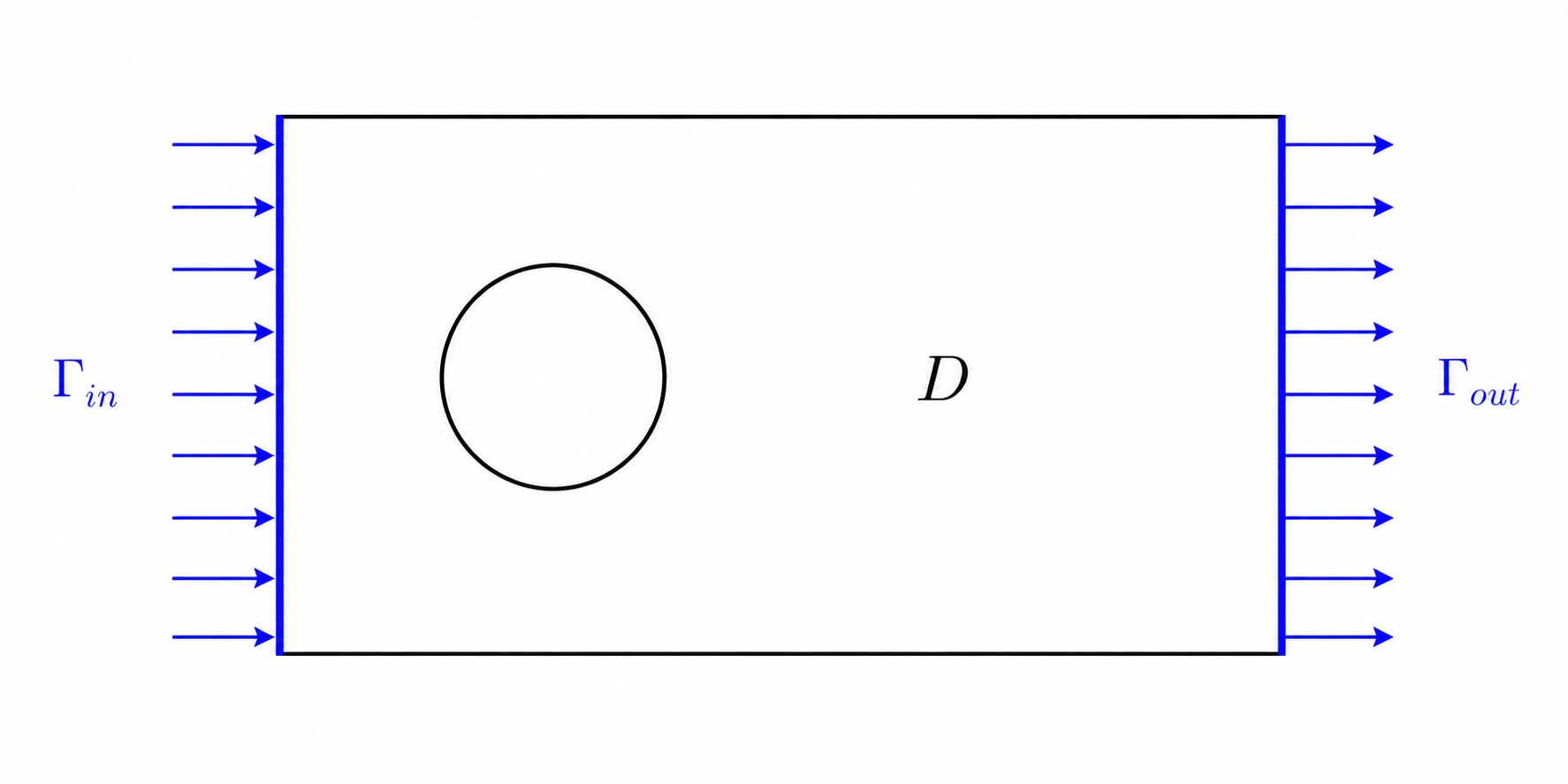} 
	\end{subfigure}
        \begin{subfigure}[b]{0.32\textwidth}
			\centering
			\includegraphics[width=2.3in]{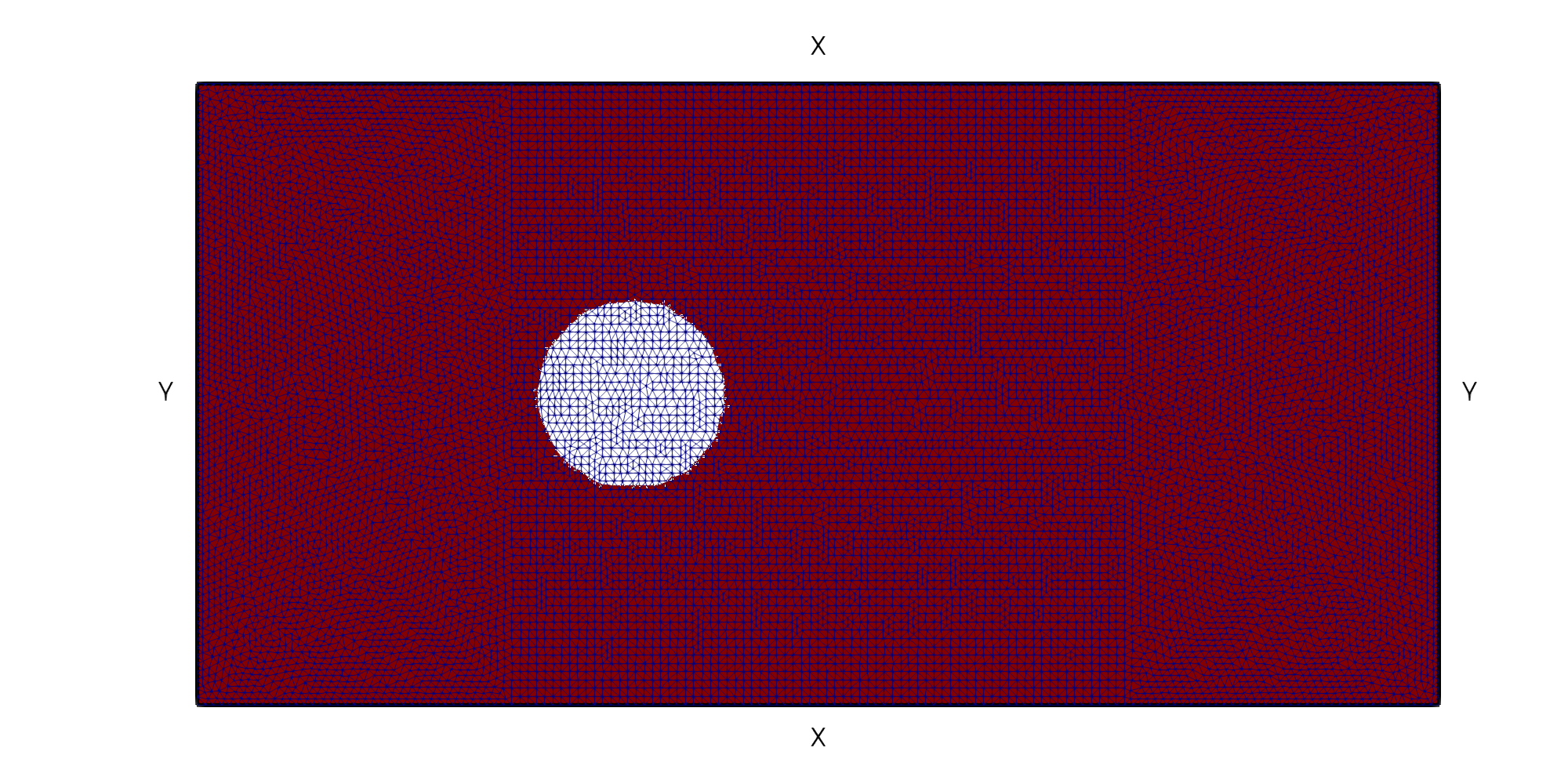} 
	\end{subfigure}
        \begin{subfigure}[b]{0.32\textwidth}
			\centering
			\includegraphics[width=2.5in]{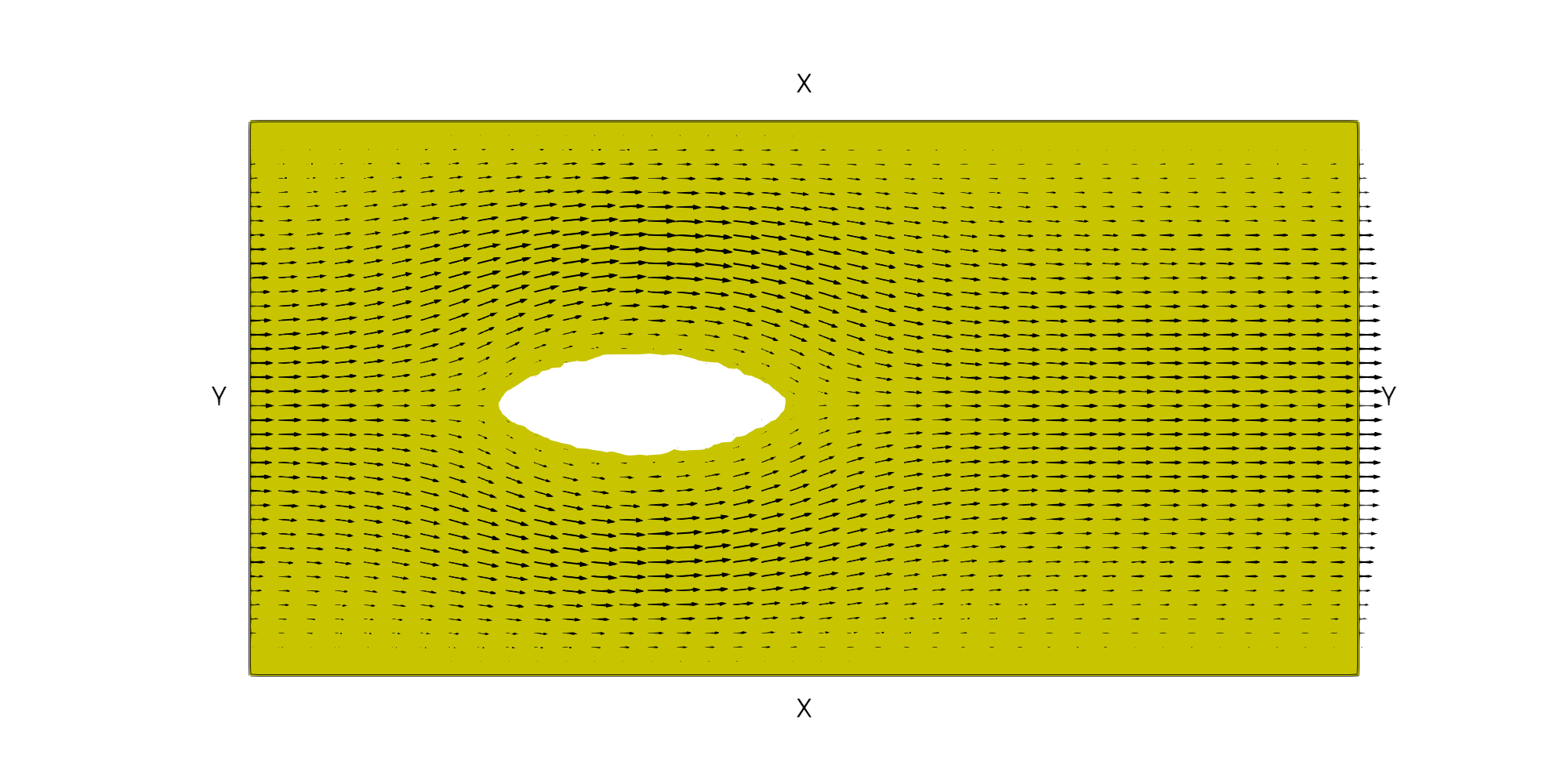}
	\end{subfigure}
    \caption{Problem setting, initial design, final design and streamlines of $\bm{u}$ (from left to right) for Example \ref{5_M2E5}.}
    \label{5_M4_Ex3}
    \end{figure}
	
    \begin{figure}[htbp]
		\centering
		\begin{subfigure}[b]{0.343\textwidth}			\includegraphics[width=\linewidth]{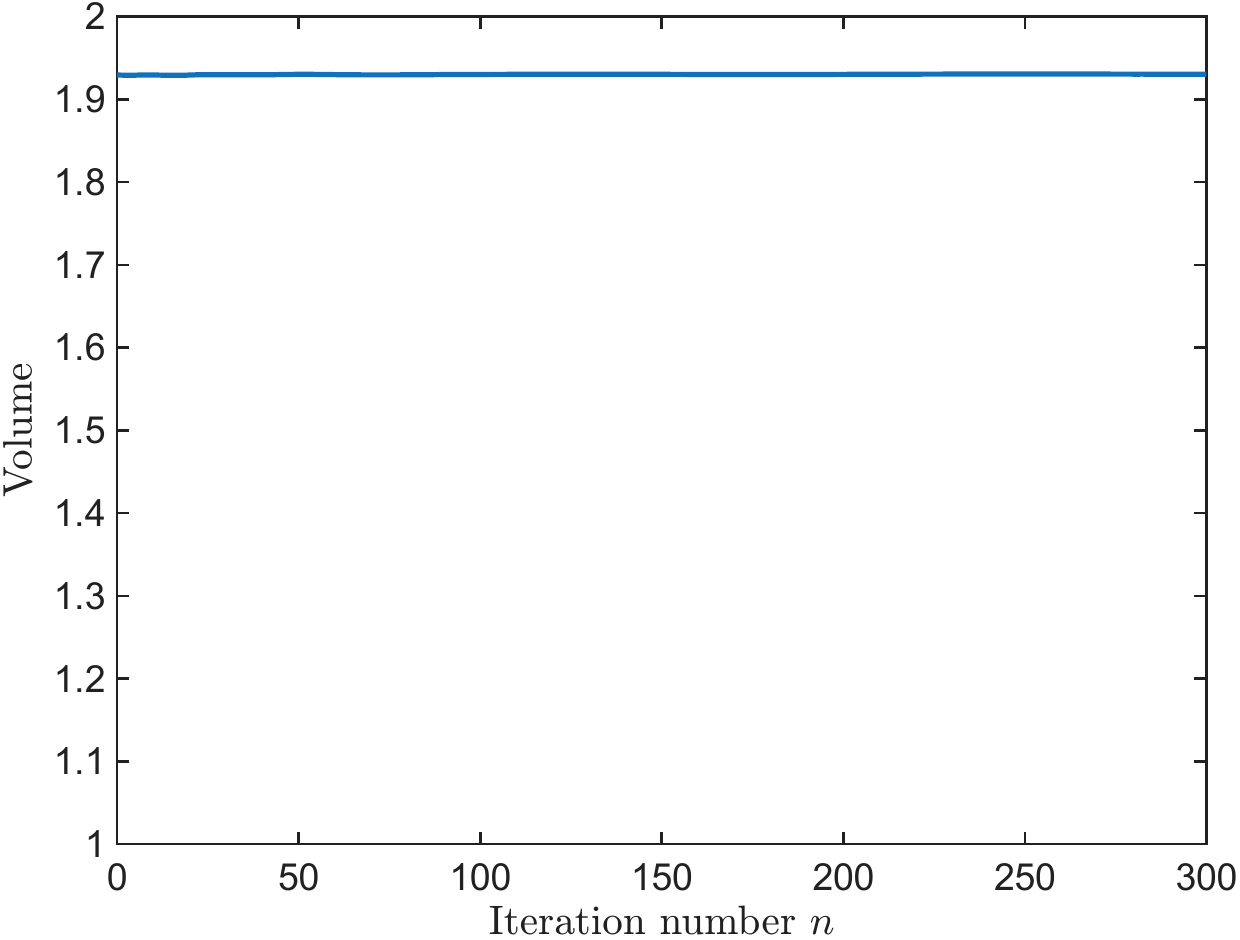}
		\end{subfigure}
		\begin{subfigure}[b]{0.337\textwidth}		\includegraphics[width=\linewidth]{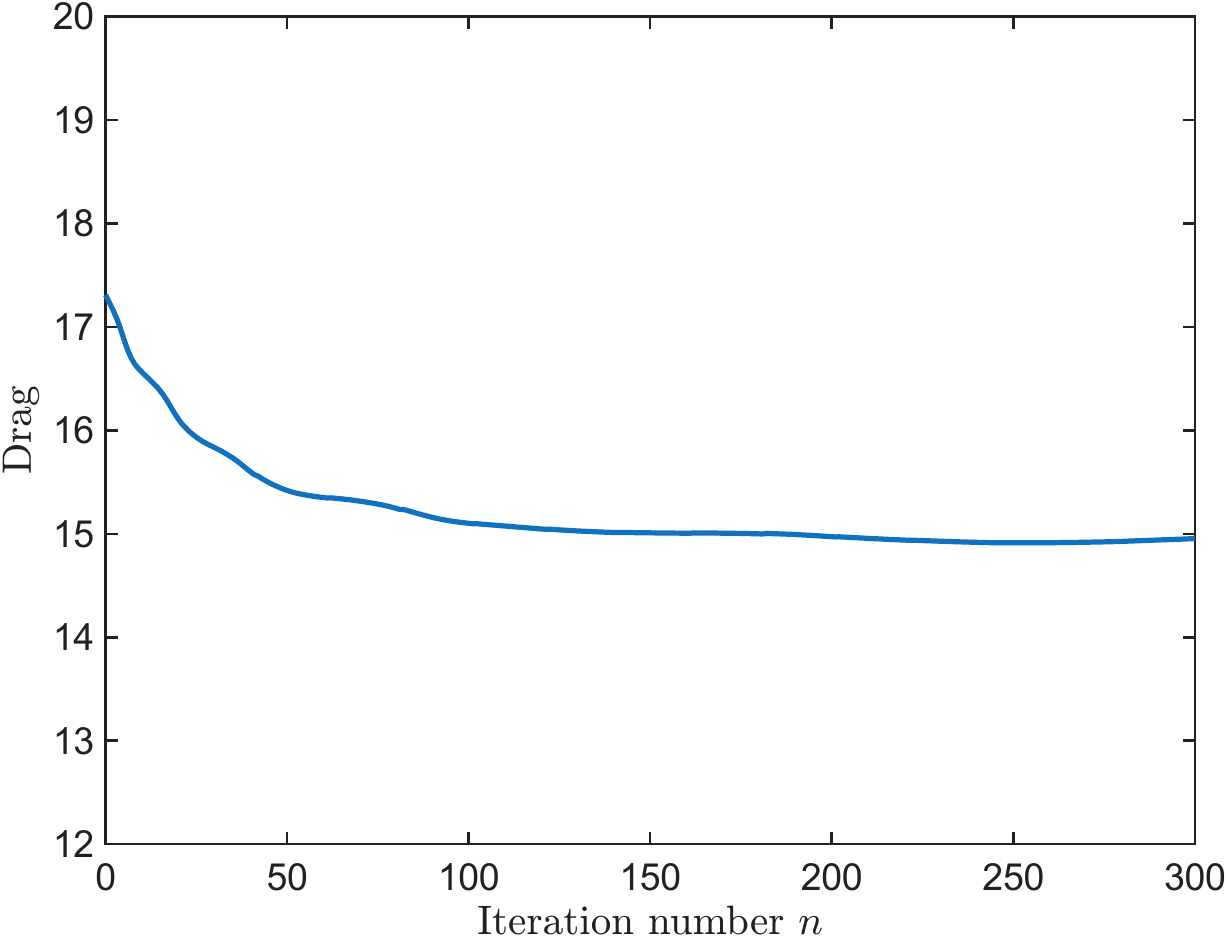}
		\end{subfigure}
		\caption{Convergence histories of the volume (left) and the energy (right) for Example \ref{5_M2E5}.}
		\label{Figure_M2E5_VandJ}
    \end{figure}
  
	\begin{example}
		\label{5_M3E1}
		(3D flow channel design \cite{BorrvallPetersson:2003}) Set $D = (0,1)^3$, $\Gamma_{\rm{in}} = \left\lbrace (x,0,z): (x,0,z) \in D \right\rbrace$, $\bm{g}_f = (0,1,0)^T$, and
        $\Gamma_{\rm{out}} = \left\lbrace (x,1,z): (x-0.5)^2+(z-0.5)^2 \leqslant 0.04 \right\rbrace$.
	\end{example}
	
\autoref{Figure_M3E1_VandJ} shows topological changes, volume-control behavior, and decrease in drag, respectively.
	
	\begin{example}
		\label{5_M3E2}
		(3D multi-inlet flow channel design \cite{BorrvallPetersson:2003})  Set $D = (0,1)^3$, 
        \[
        \Gamma_{\rm{in}} = \left\lbrace (x,0,z),(x,1,z): (x-0.5)^2+(z-0.7)^2
        \leqslant 0.04 \right\rbrace \bigcup \left\lbrace (0,y,z),(1,y,z): (y-0.5)^2+(z-0.7)^2  \leqslant 0.04 \right\rbrace
        \]\noindent
        and $\Gamma_{\rm{out}} = \left\lbrace (x,y,0): (x-0.5)^2+(y-0.5)^2 \leqslant 0.04  \right\rbrace$. 
Set	\begin{equation}
		\bm{g}_f = 
		\begin{cases}
			(0,1,0)^T, & y = 0, \\
			(0,-1,0)^T, & y = 1, \\
			(1,0,0)^T, & x = 0, \\
			(-1,0,0)^T, & x = 1.
		\end{cases}
	\end{equation} 
    	\end{example}
        
	\autoref{Figure_M3E2_VandJ} shows shape changes, volume-control behavior, and a monotonic decrease in drag.

    \begin{example}
        \label{5_M3E3}
        (3D flows past an obstacle \cite{pironneau1973optimum_1})  Set $D = (0,1) \times (0,2) \times (0,1)$,  $\Gamma_{\rm{in}} = \lbrace (x,0,z): 0 \leqslant x,z \leqslant 1 \rbrace$ and  $\Gamma_{\rm{out}} = \left\lbrace (x,2,z): 0 \leqslant x,z \leqslant 1 \right\rbrace$ (see the left plot of \autoref{5_M4_Ex3}).
   Set
	\begin{equation}
		\bm{g}_f = \left\{
		\begin{aligned}
			\Big(0,48x(1-x)z(1-z), 0\Big)^T, \ & y = 0\ {\rm or }\ 2, 0 \leqslant x,z \leqslant 1, \\
			(0,0, 0)^T, \ & {\rm otherwise\ on\ } \partial D.
		\end{aligned}
        \right.
	\end{equation}
        \end{example}
        
 \autoref{Figure_M3E3_VandJ} shows the design process, the evolution of the approximate volume, and the convergence history of drag.
	
	\begin{figure}[htbp]
        \centering
		\begin{subfigure}[b]{0.24\textwidth}
			\includegraphics[width=\linewidth]{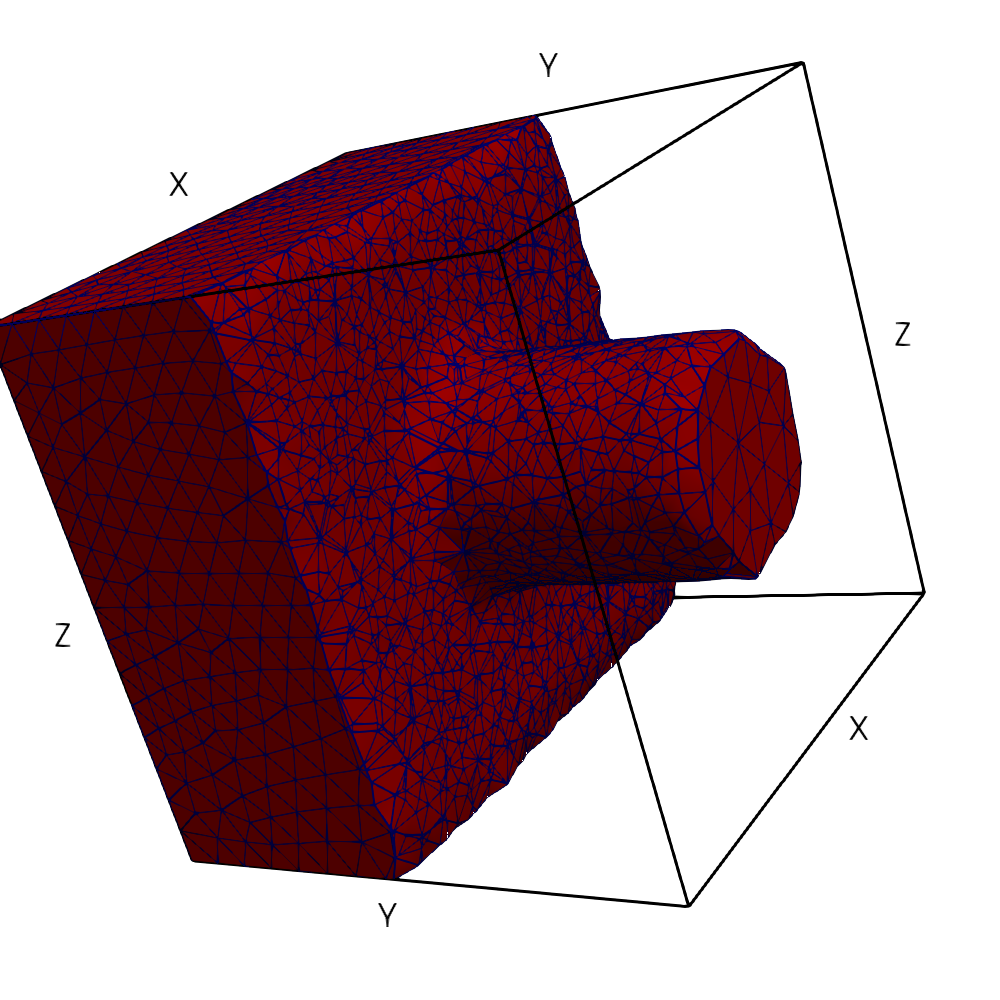} 
		\end{subfigure}
		\begin{subfigure}[b]{0.24\textwidth}
			\includegraphics[width=\linewidth]{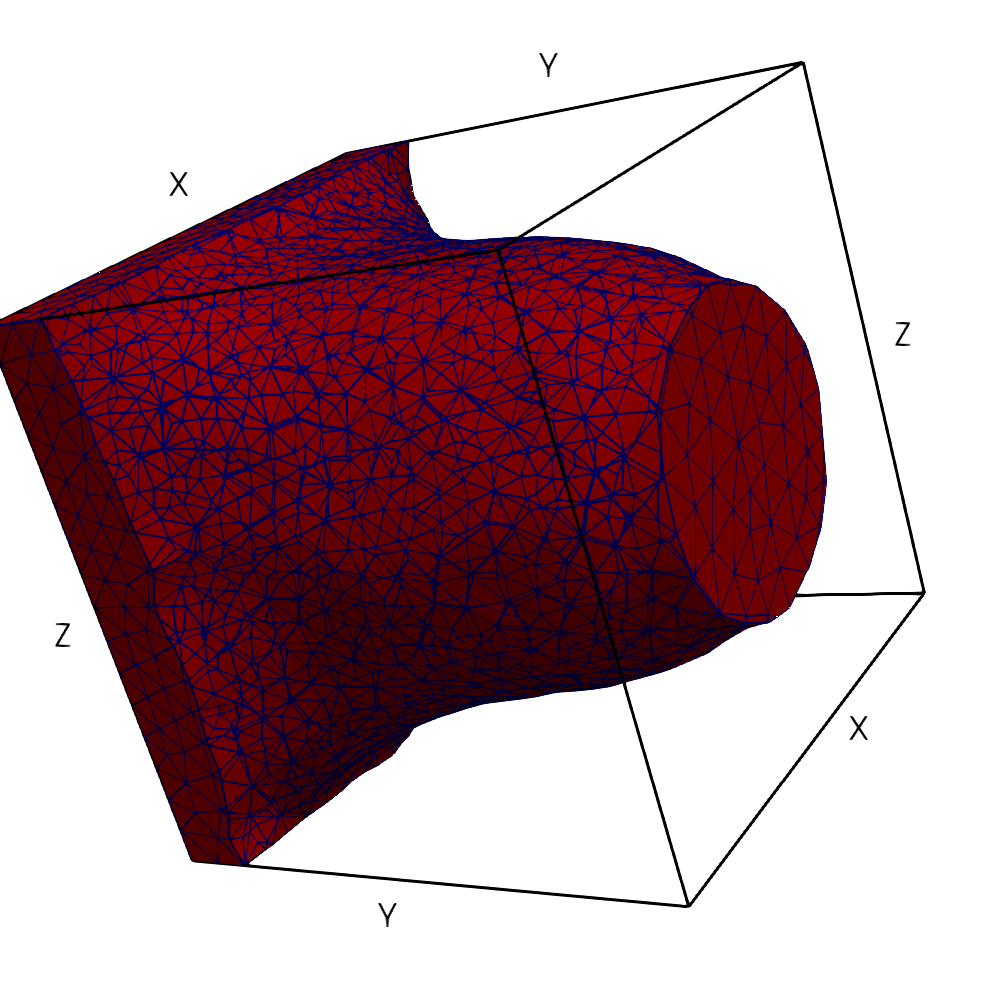}
		\end{subfigure}
        \begin{subfigure}[b]{0.24\textwidth}
			\includegraphics[width=\linewidth]{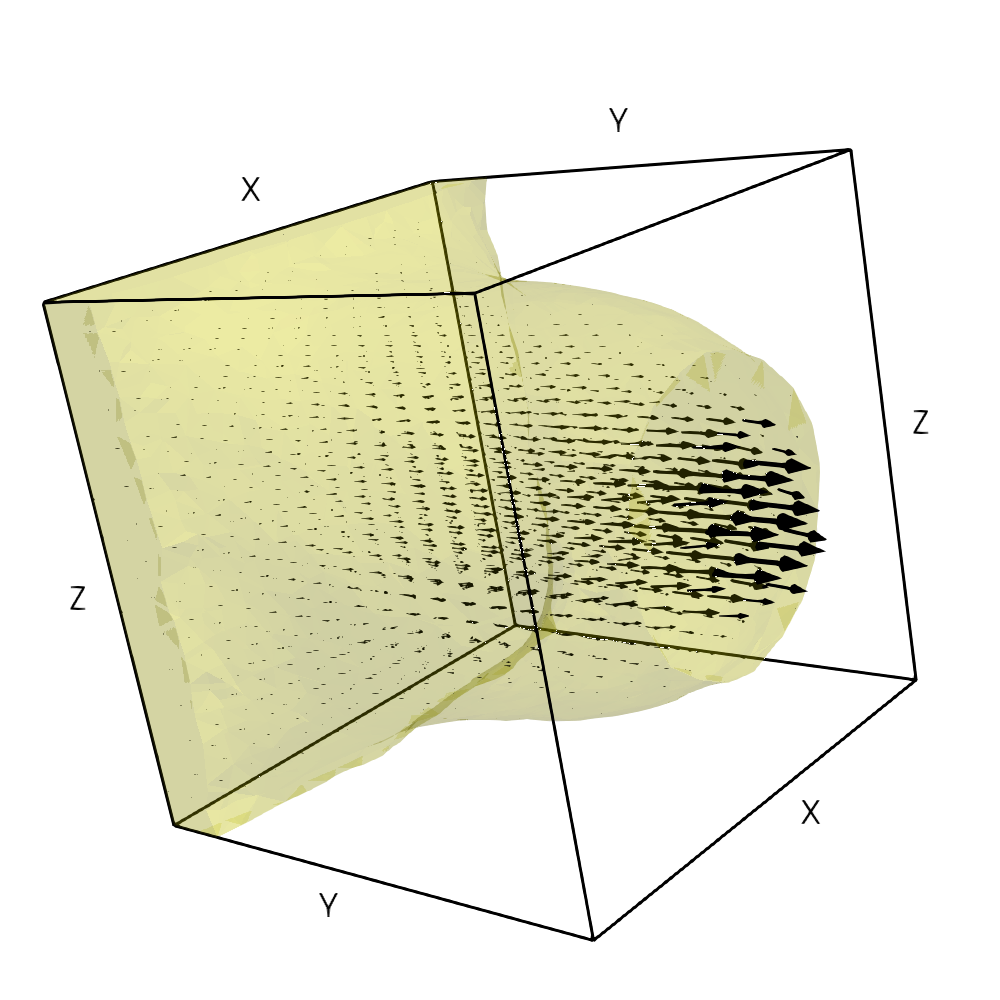}
		\end{subfigure}
		\begin{subfigure}[b]{0.348\textwidth}		\includegraphics[width=\linewidth]{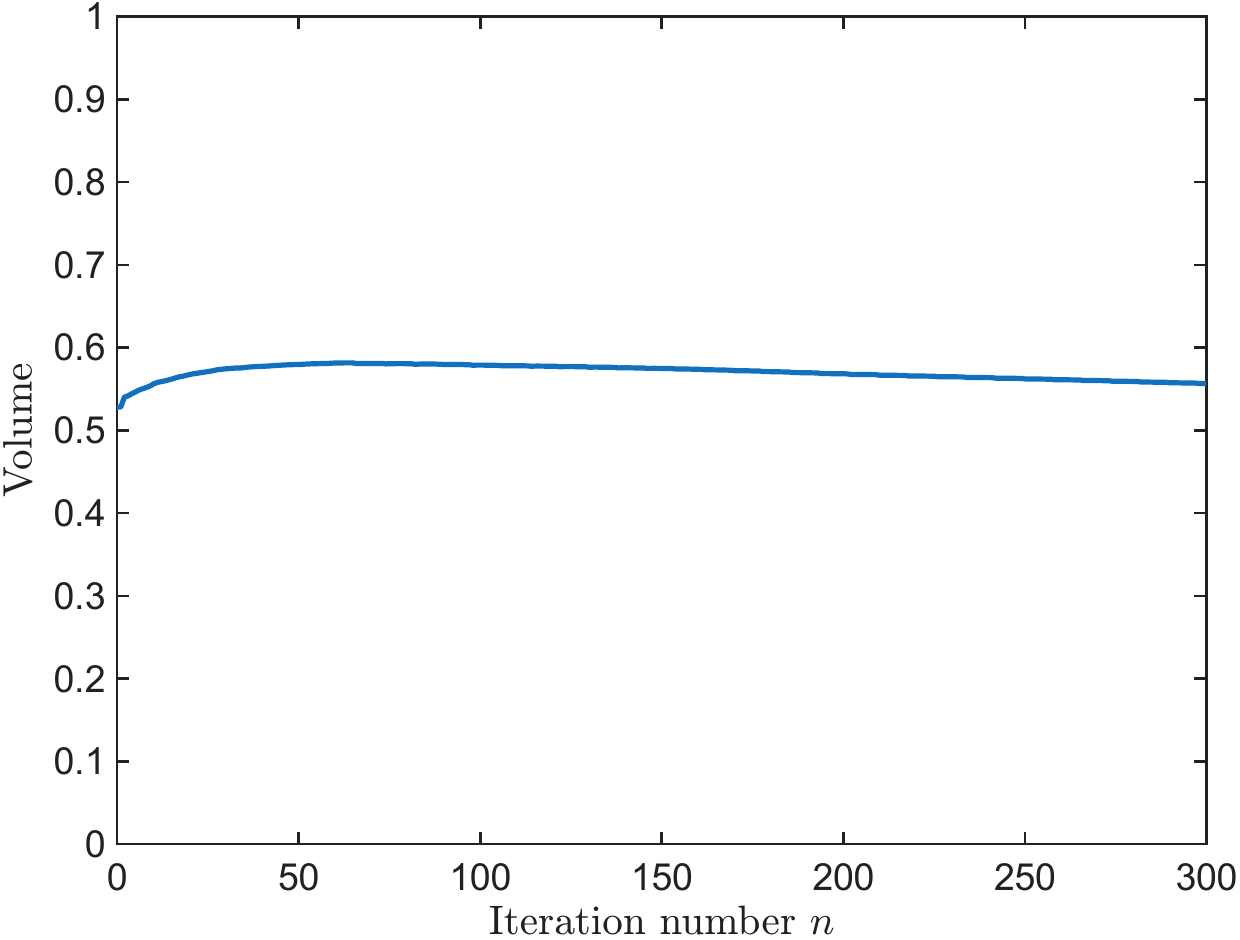} 
		\end{subfigure}
		\begin{subfigure}[b]{0.351\textwidth}		\includegraphics[width=\linewidth]{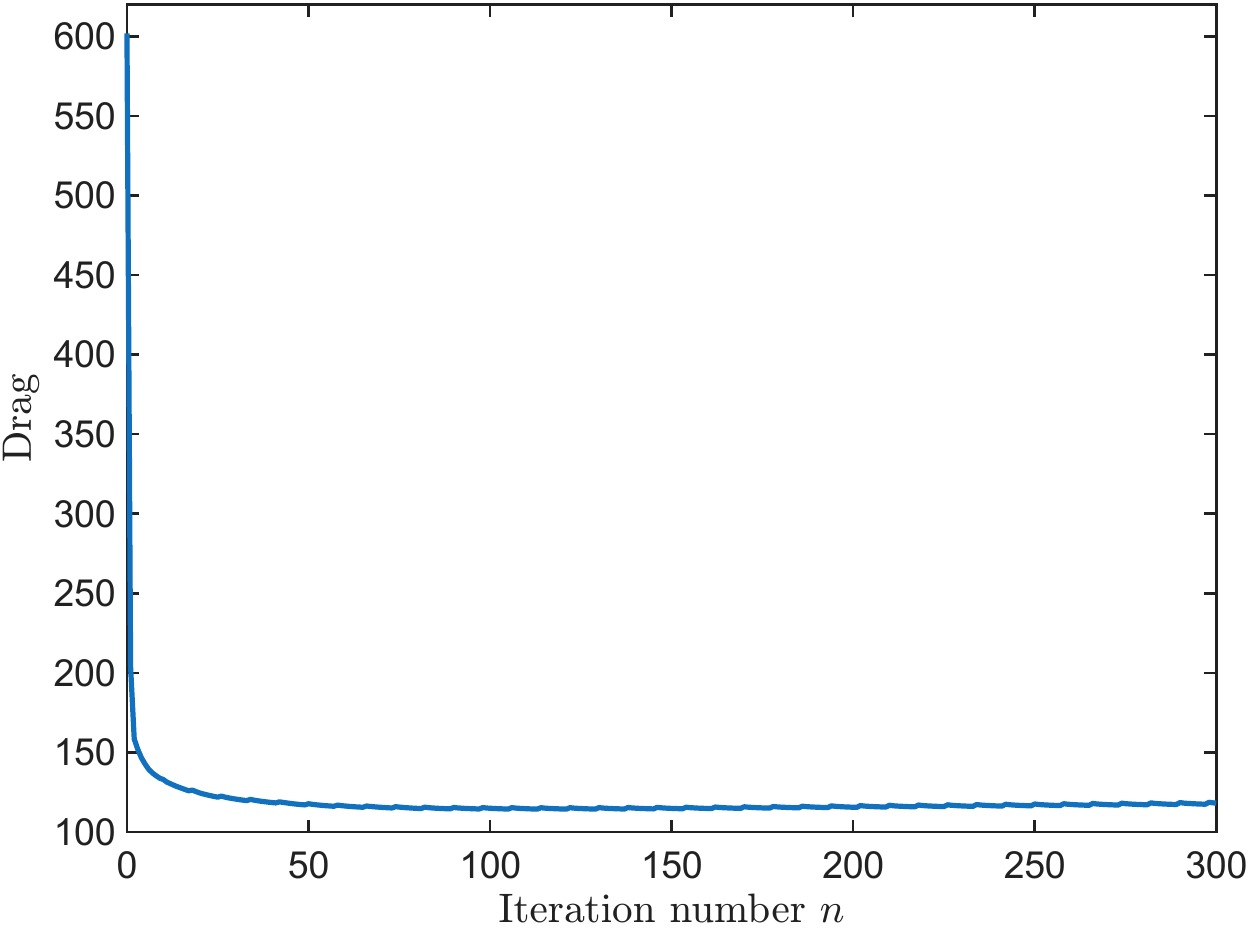}
		\end{subfigure}
		\caption{Evolutions of the design ($n = 0,10,300$) from left to right (first row), streamlines of $\bm{u}$ (top right) and convergence histories of the volume (bottom left) and energy (bottom right) for Example \ref{5_M3E1}.}
		\label{Figure_M3E1_VandJ}
	\end{figure}
	
	\begin{figure}[htbp]
		\centering
        		\begin{subfigure}[b]{0.24\textwidth}
			\includegraphics[width=\linewidth]{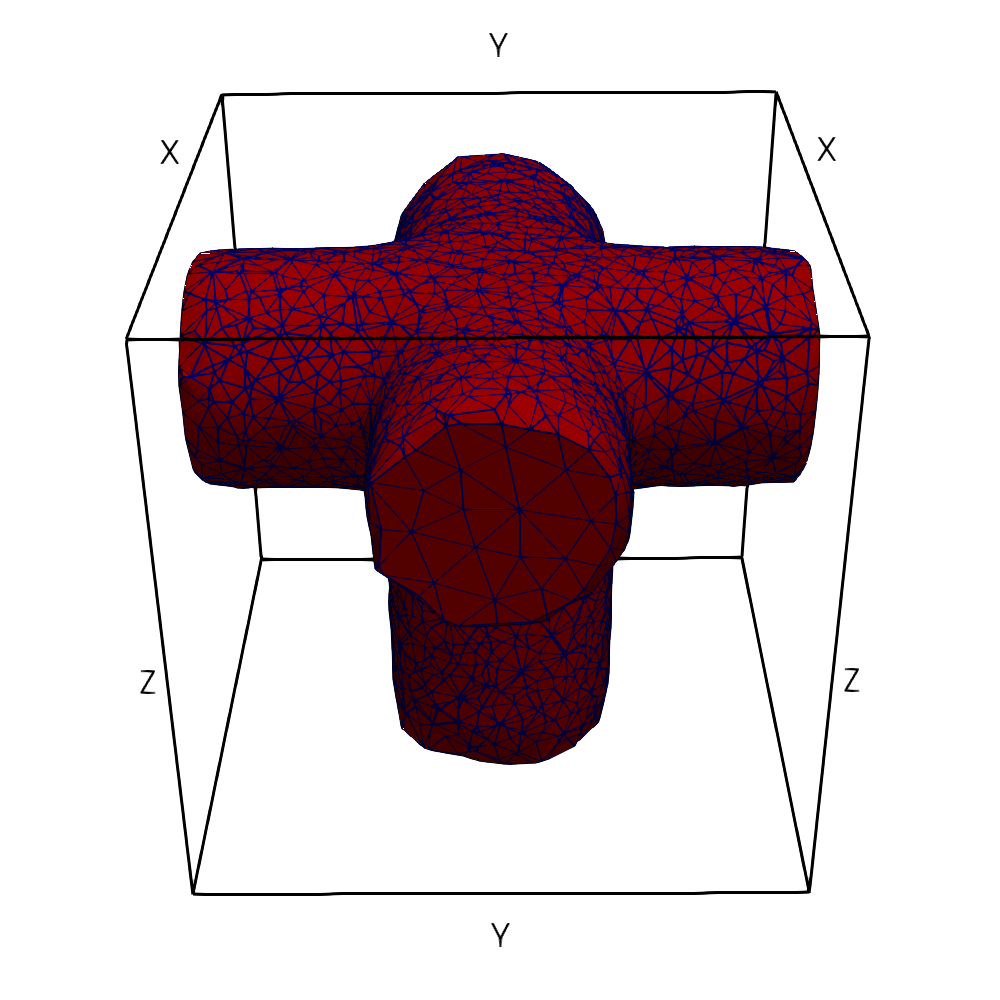} 
		\end{subfigure}
        \begin{subfigure}[b]{0.24\textwidth}
			\includegraphics[width=\linewidth]{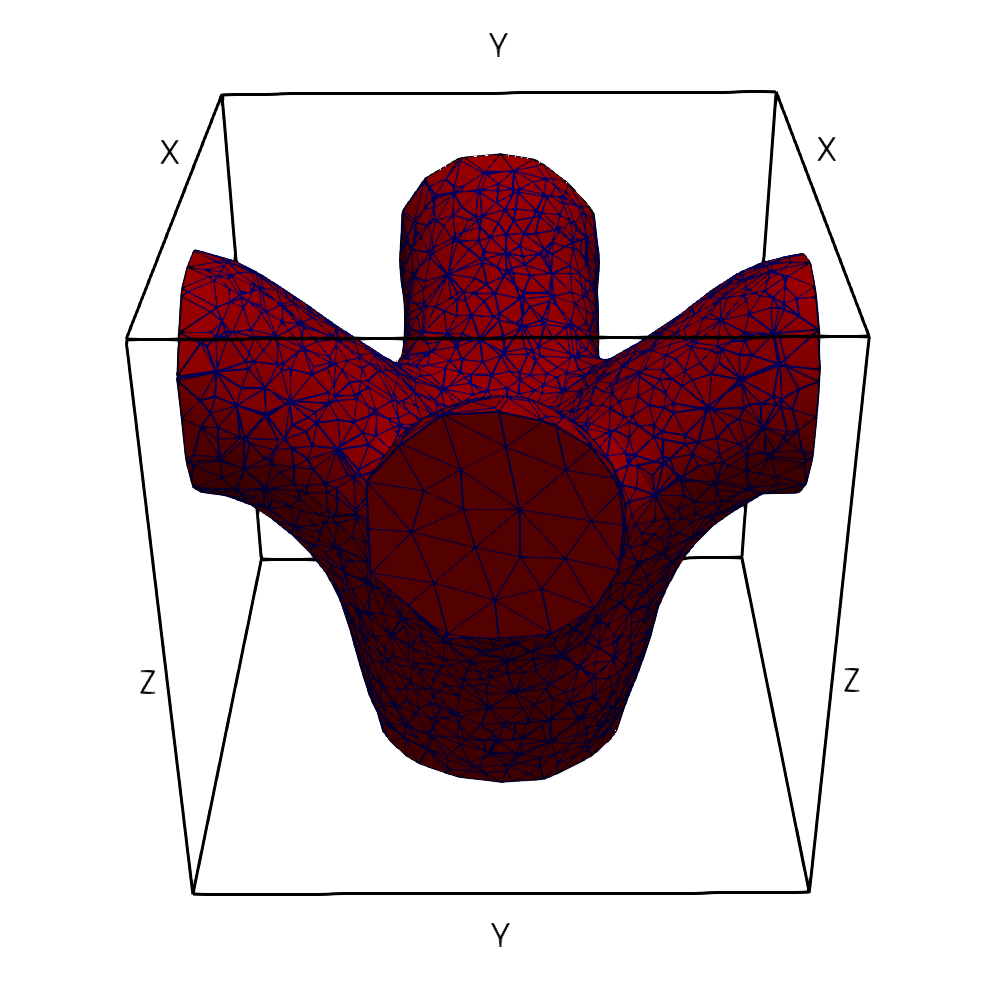}
		\end{subfigure}
        \begin{subfigure}[b]{0.24\textwidth}
			\includegraphics[width=\linewidth]{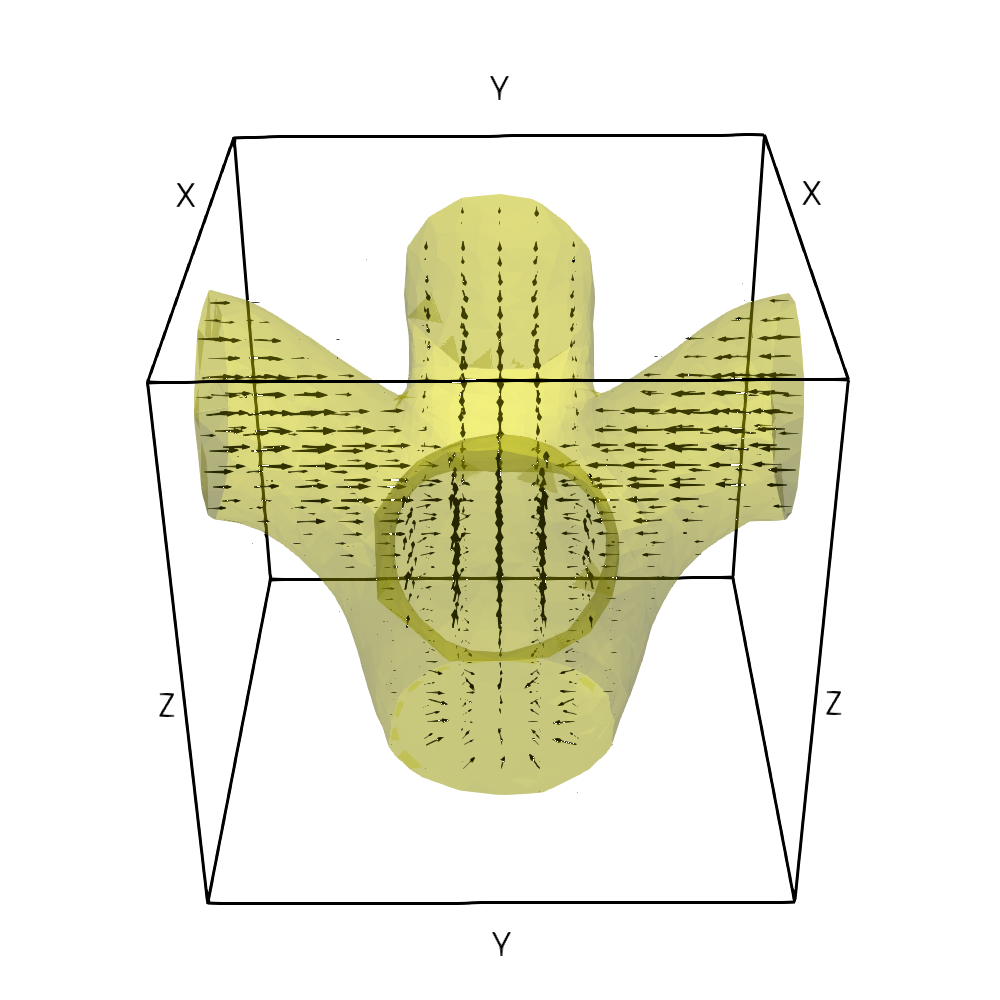}
		\end{subfigure}
        
		\begin{subfigure}[b]{0.33\textwidth}			\includegraphics[width=\linewidth]{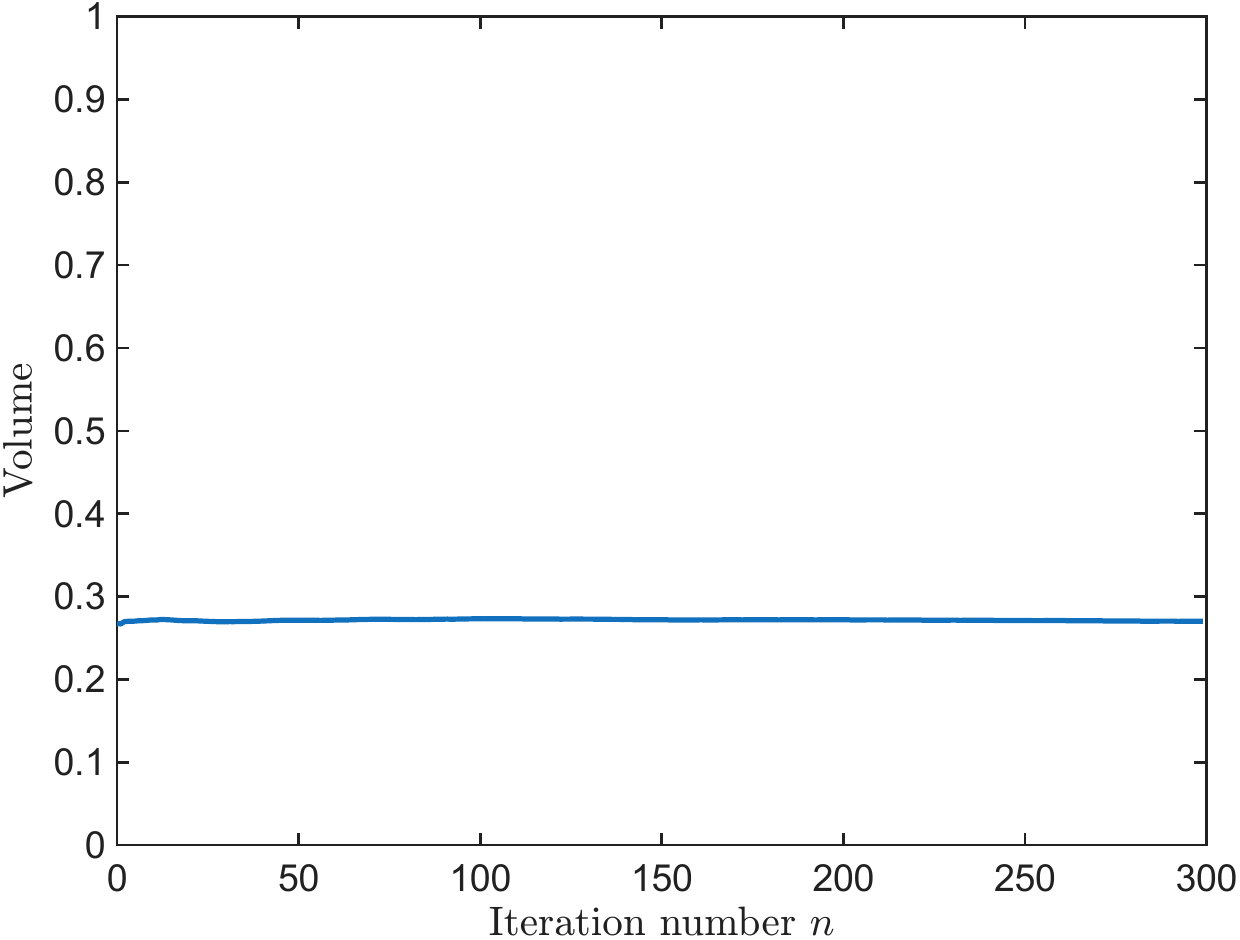} 
		\end{subfigure}
		\begin{subfigure}[b]{0.33\textwidth}			\includegraphics[width=\linewidth]{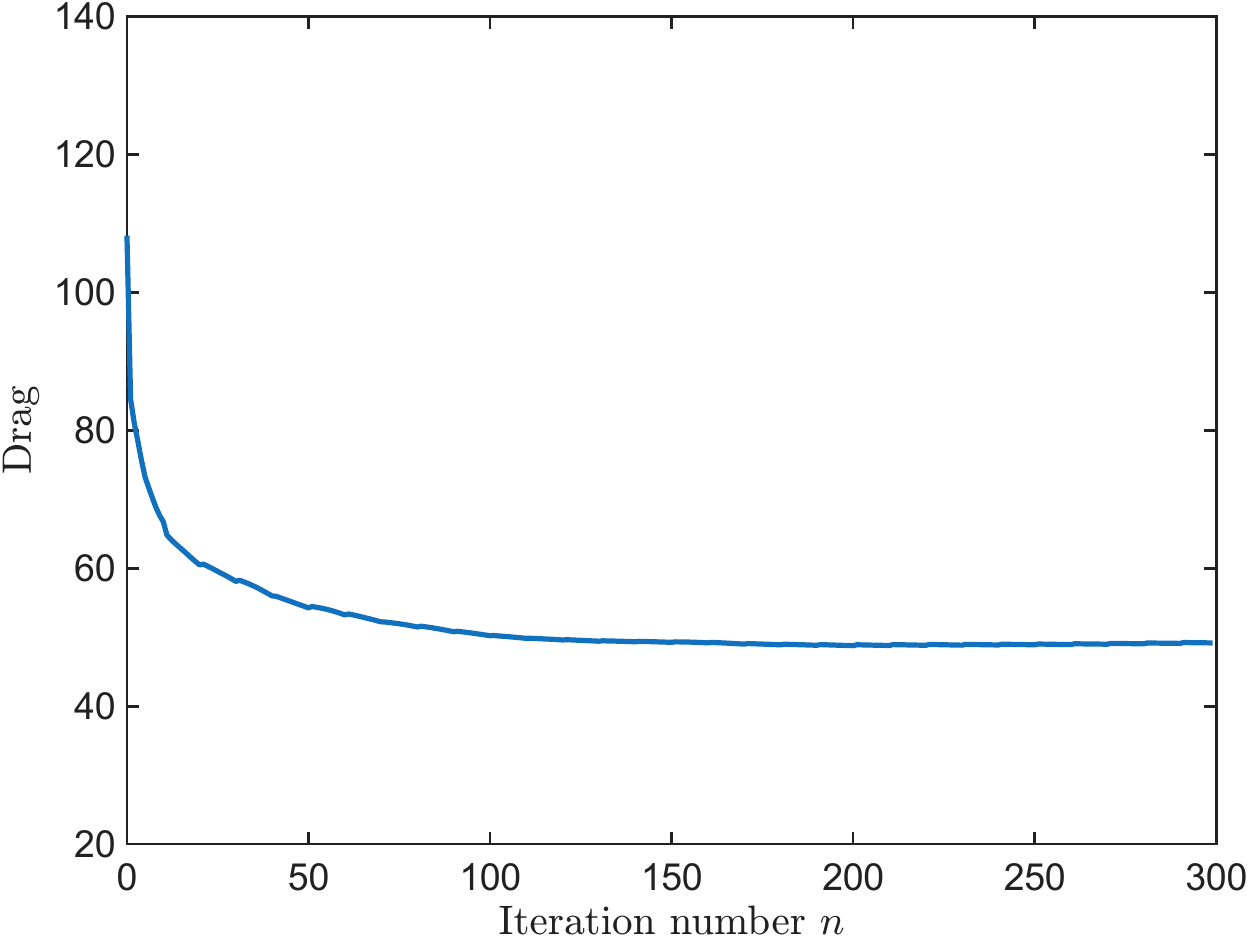}
		\end{subfigure}
		\caption{The initial design (top left), final design (top mid), streamlines of $\bm{u}$ (top right), and convergence histories of the volume (bottom left) and energy (bottom right) for Example \ref{5_M3E2}.}
		\label{Figure_M3E2_VandJ}
	\end{figure}

\begin{figure}[htbp]
    \centering

    \makebox[\textwidth][c]{%
    \begin{subfigure}[b]{0.315\textwidth}
        \centering
        \includegraphics[
            width=\linewidth,
            trim=390bp 90bp 390bp 130bp,
            clip
        ]{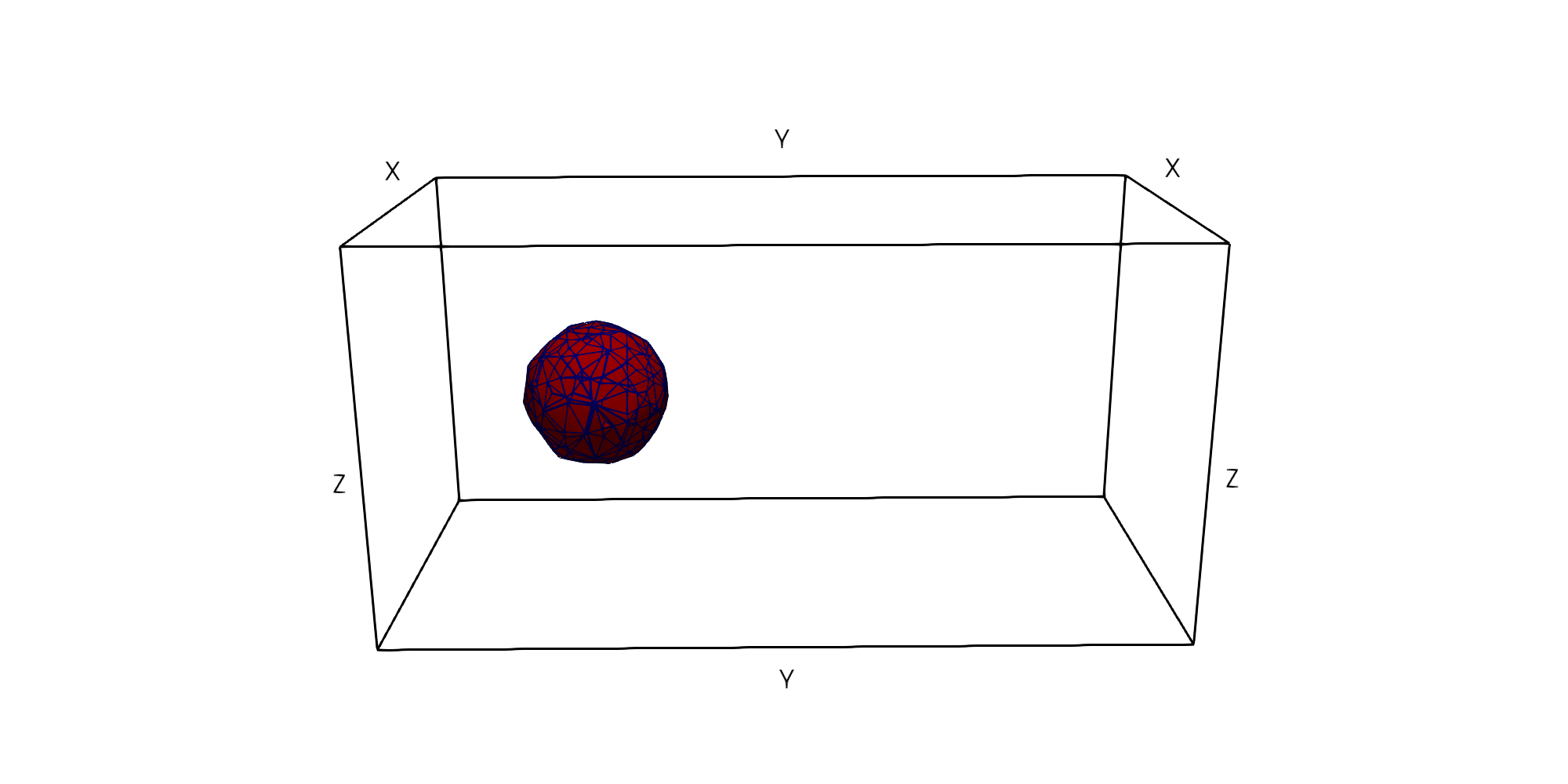}
    \end{subfigure}%
    \hspace{0.005\textwidth}%
    \begin{subfigure}[b]{0.315\textwidth}
        \centering
        \includegraphics[
            width=\linewidth,
            trim=390bp 90bp 390bp 130bp,
            clip
        ]{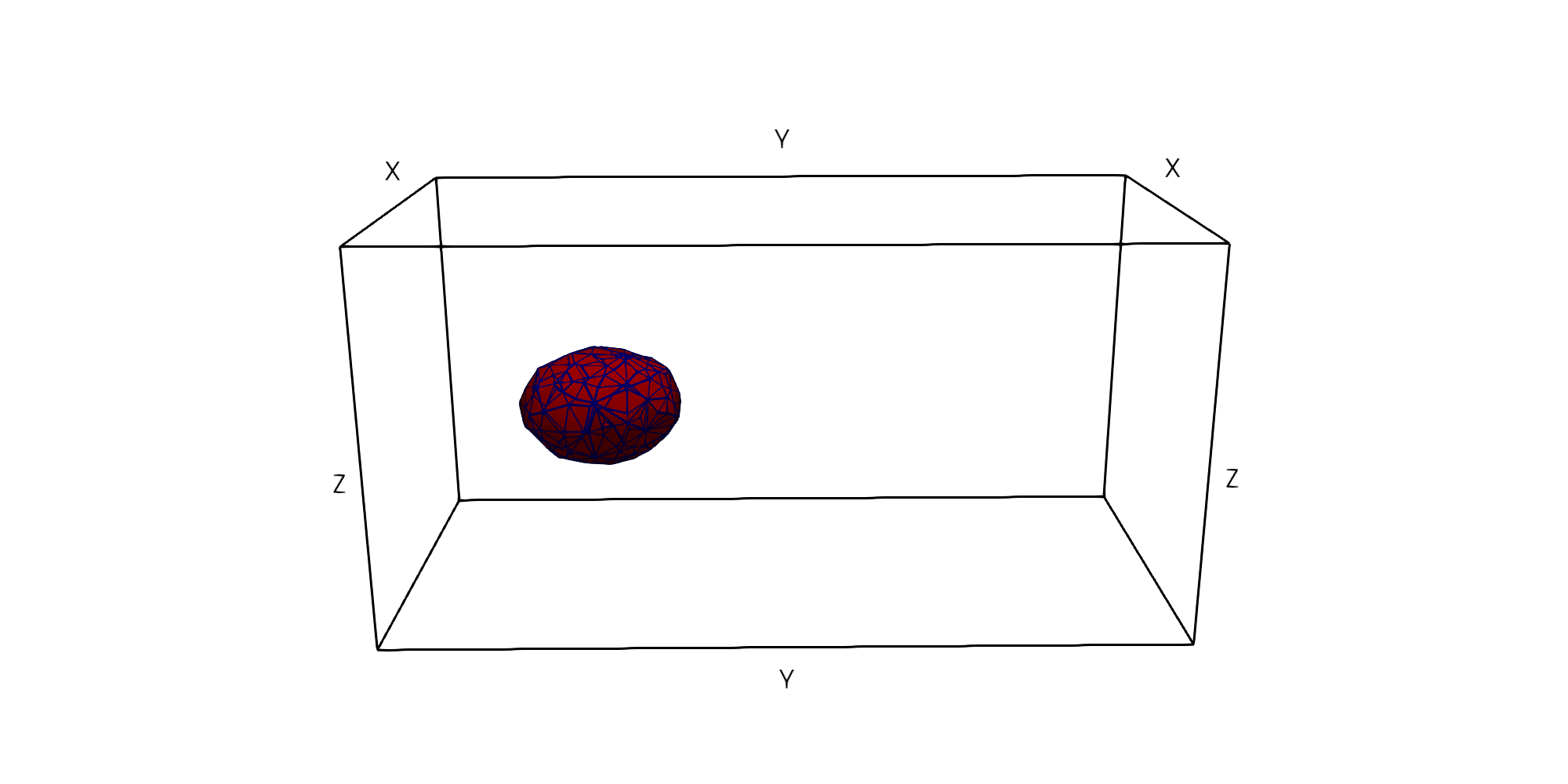}
    \end{subfigure}%
    \hspace{0.005\textwidth}%
    \begin{subfigure}[b]{0.315\textwidth}
        \centering
        \includegraphics[
            width=\linewidth,
            trim=390bp 90bp 390bp 130bp,
            clip
        ]{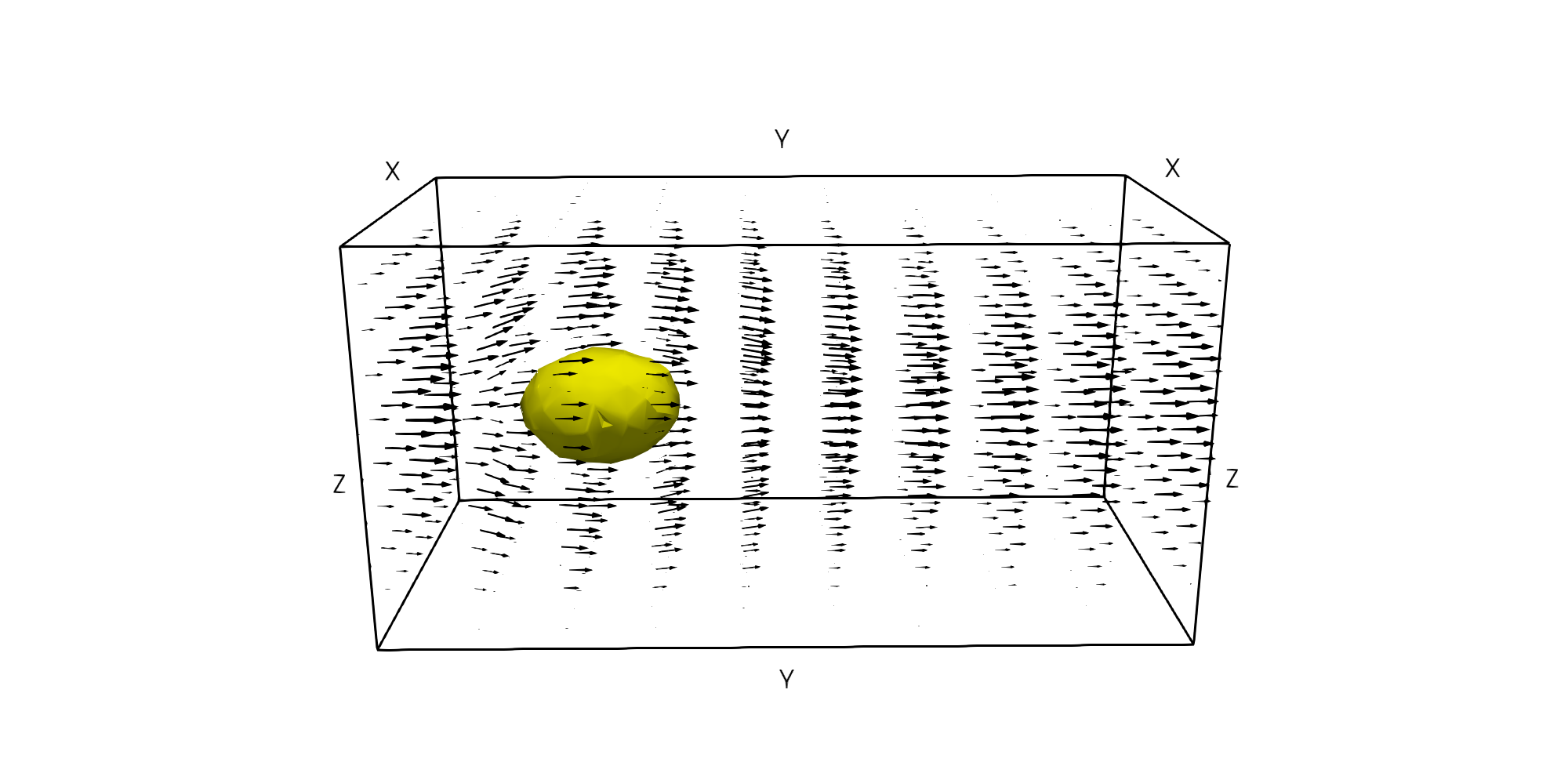}
    \end{subfigure}%
}

\vspace{0.3em}

    \makebox[\textwidth][c]{%
        \begin{subfigure}[b]{0.365\textwidth}
            \centering
            \includegraphics[width=\linewidth]{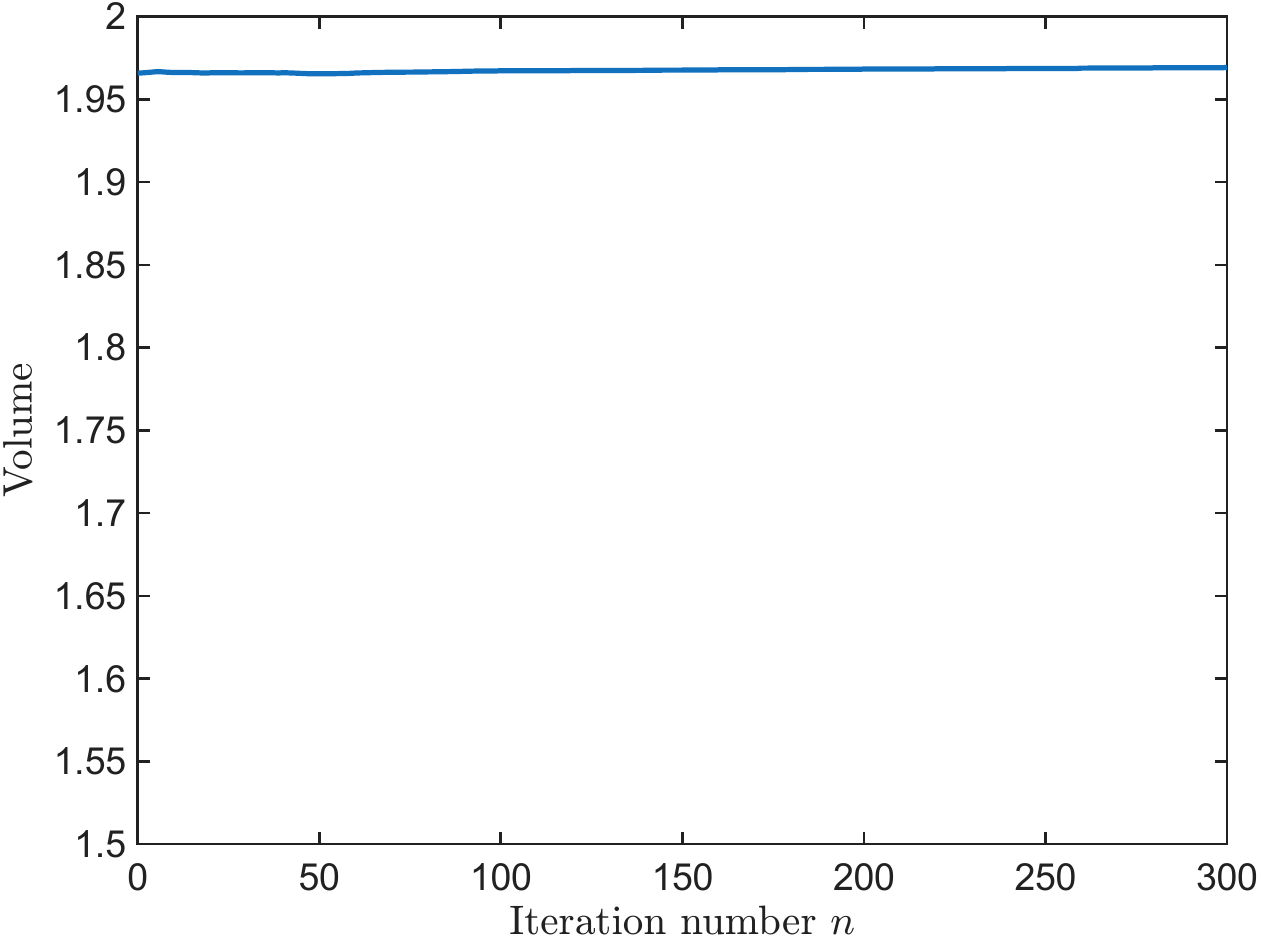}
        \end{subfigure}%
        \hspace{0.04\textwidth}%
        \begin{subfigure}[b]{0.355\textwidth}
            \centering
            \includegraphics[width=\linewidth]{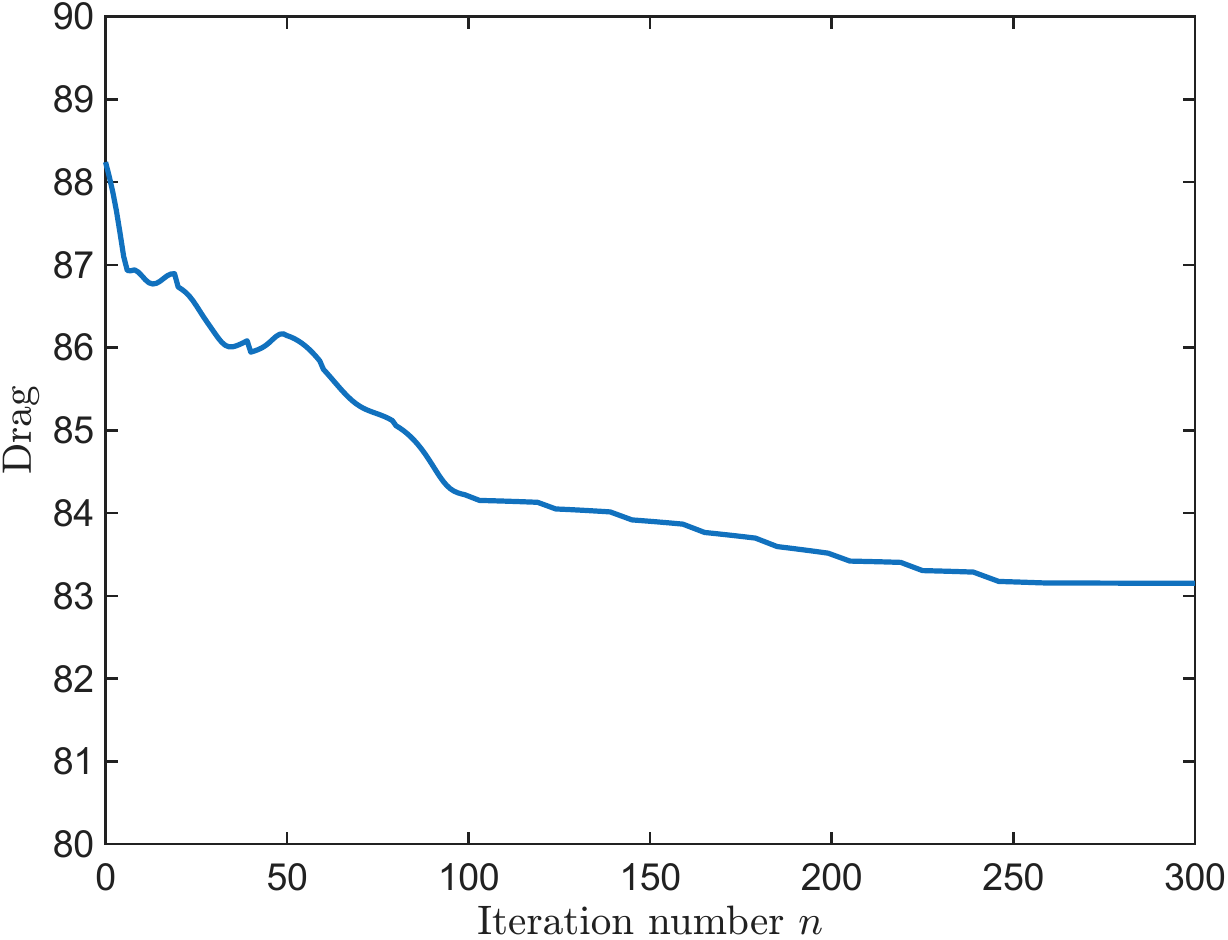}
        \end{subfigure}%
    }
    \caption{The initial design (top left), final design (top mid), streamlines of $\bm{u}$ (top right), and convergence histories of the volume (bottom left) and energy (bottom right) for Example \ref{5_M3E3}.}
    \label{Figure_M3E3_VandJ}
\end{figure}

\begin{table}[htbp]
  \centering
  \caption{Numerical parameters and volume-control errors for Examples~\ref{5_M1E1}--\ref{5_M3E3}.}
  \label{tab:1}
    \resizebox{\textwidth}{!}{
    \begin{tabular}{llccccccc}
      \toprule
      Problems & Examples & $d$ & $h$ & $N$ & $T=5\tau$ & $V_0$ & $|V_0-V_N|/V_0$ & $V^e$ \\
      \midrule
      \multirow{5}{*}{Structural compliance}
        & \ref{5_M1E1} & $2$ & $0.025$ & $300$  & $0.075$ & $2+(\eta-1)0.3\pi$ & $0.0409$ & $0.0454$ \\
        & \ref{5_M1E2} & $2$ & $0.020$ & $300$  & $0.050$ & $2+0.36\pi(\eta-1)$ & $0.0259$ & $0.0225$ \\
        & \ref{5_M1E3} & $2$ & $0.023$ & $1000$ & $0.008$ & $2+(\eta-1)0.289\pi$ & $0.0430$ & $0.0345$ \\
        & \ref{5_M4E1} & $3$ & $0.060$ & $300$  & $\min\{0.0025,0.4h/\lVert\bm{V}\rVert\}$ & $0.308,\,0.272$ & $0.0295,\,0.0051$ & $0.0082,\,0.0060$ \\
        & \ref{5_M4E2} & $3$ & $0.060$ & $300$  & $\min\{0.0025,0.4h/\lVert\bm{V}\rVert\}$ & $0.382,\,0.253$ & $0.0976,\,0.1055$ & $0.0334,\,0.0205$ \\
      \midrule
      \multirow{3}{*}{Structural eigenfrequency}
        & \ref{5_M5E1} & $2$ & $0.025$ & $700$ & $0.060$ & $2+0.3\pi(\eta_1-1)$ & $0.0126$ & $0.0105$ \\
        & \ref{5_M5E2} & $2$ & $0.029$ & $700$ & $0.030$ & $\frac{64}{25}+(1-\eta_2)\frac{252\pi}{625}$ & $0.0151$ & $0.0482$ \\
        & \ref{5_M5E3} & $3$ & $0.097$ & $300$ & $0.200$ & $\approx1.000$ & $0.0240$ & $0.023$ \\
      \midrule
      \multirow{8}{*}{Fluid flows}
        & \ref{5_M2E1} & $2$ & $0.017$ & $300$ & $10^{-3}$ & $0.5$ & $0.0120$ & $0.0038$ \\
        & \ref{5_M2E2} & $2$ & $0.018$ & $300$ & $0.005$ & $0.5$ & $0.0168$ & $0.0045$ \\
        & \ref{5_M2E3} & $2$ & $0.014$ & $300$ & $0.005$ & $0.32$ & $0.0344$ & $0.0069$ \\
        & \ref{5_M2E4} & $2$ & $0.015$ & $300$ & $5\times10^{-4},\,1.5\times10^{-3},\,1.5\times10^{-3}$ & $0.25$ & $0.0128,\,0.0064,\,0.0160$ & $0.0022,\,0.0010,\,0.0024$ \\
        & \ref{5_M2E5} & $2$ & $0.020$ & $300$ & $2.5\times10^{-4}$ & $1.93$ & $0.0012$ & $0.0007$ \\
        & \ref{5_M3E1} & $3$ & $0.108$ & $300$ & $3.5\times10^{-4}$ & $\approx0.528$ & $0.0549$ & $0.042$ \\
        & \ref{5_M3E2} & $3$ & $0.113$ & $300$ & $10^{-4}$ & $\approx0.267$ & $0.0105$ & $0.0042$ \\
        & \ref{5_M3E3} & $3$ & $0.100$ & $300$ & $4\times10^{-5}$ & $1.966$ & $0.0018$ & $0.0017$ \\
      \bottomrule
    \end{tabular}
  }
\end{table}

\clearpage
\section{Conclusion} \label{sect:Conclusion}
We have proposed a level-set framework for topology optimization in which volume control is incorporated into the velocity extension-regularization step through a scalar weighted compatibility constraint. For each frozen continuous transport interval, the constraint cancels the initial rate of change of the regularized volume, and the remaining drift is shown to start at order $T^2$. Under a narrow-band assumption, the corresponding constant is uniform with respect to the regularization width. For the spatially semi-discrete DG transport problem, we derive a quantitative regularized-volume drift estimate that identifies the effects of interface smoothing, spatial approximation, transport time, and the discrete velocity. Time discretization and reinitialization introduce additional numerical errors, which are monitored separately in the computations. The refinement diagnostics show approximately first-order decay of the transport drift with $T$. Numerical experiments for structural compliance, structural vibration, and Stokes-flow topology optimization show that the proposed strategy keeps the volume drift small while retaining stable objective evolution.

	\appendix
\section{Derivations of the shape derivatives}\label{sect:appendix}
	Here we derive \eqref{3_Modified_dJ1},  \eqref{3_Modified_dJ3} and \eqref{3_Modified_dJ2}  for the sake of completeness and self-containment of the paper.
	{
		\begin{proof}[Proof of \eqref{3_Modified_dJ1}]
			Let $\phi_t(\bm{x}) = \phi(t, \bm{x})$ and $\bm{u}_t = \bm{u}(\phi_t)$ be the solution of the following problem:
			\begin{equation} \label{3_Model1_timeeq}
				\begin{cases}
					-\bm{\nabla} \cdot (M_{\eta,\gamma}(\phi_t) A\varepsilon(\bm{u}_t)) = \bm{f}_s & \texttt{in} \ D, \\
					\bm{u}_t = \bm{0} & \texttt{on} \ \partial D_D, \\
					M_{\eta,\gamma}(\phi_t) A\varepsilon(\bm{u}_t)\bm{n} = \bm{g}_s & \texttt{on} \ \partial D_N, \\
					M_{\eta,\gamma}(\phi_t) A\varepsilon(\bm{u}_t)\bm{n} = \bm{0} & \texttt{on} \ \partial D_0.
				\end{cases}
			\end{equation}
			We obtain
			\begin{equation} \label{3_dJModel1_firststep}
				d\mathcal{J}(\phi; \bm{V}) = \int_D \frac{\partial M_{\eta,\gamma}(\phi_t)}{\partial t} \Big|_{t = 0} A \varepsilon(\bm{u}):\varepsilon(\bm{u}) {\rm d}x + \int_D 2M_{\eta,\gamma}(\phi) A \varepsilon \Big( \frac{\partial \bm{u}_t}{\partial t} \Big|_{t = 0} \Big) : \varepsilon(\bm{u}) {\rm d}x.
			\end{equation}
			Consider the weak formulation of equation \eqref{3_Model1_timeeq}
			\begin{equation}
				\int_D M_{\eta,\gamma}(\phi_t) A \varepsilon(\bm{u}_t):\varepsilon(\bm{v}) {\rm d}x = \int_D \bm{f}_s \cdot \bm{v} {\rm d}x + \int_{\partial D_N} \bm{g}_s \cdot \bm{v} \ds \quad \forall \bm{v} \in [H_{\partial D_D}^1(D)]^d.
			\end{equation}
			Taking derivatives with respect to $t$ on both sides of the above equation yields
			\begin{equation} \label{3_Model1_dteq}
				\int_D \frac{\partial M_{\eta,\gamma}(\phi_t)}{\partial t} A \varepsilon(\bm{u}_t):\varepsilon(\bm{v}) {\rm d}x + \int_D M_{\eta,\gamma}(\phi) A \varepsilon \Big( \frac{\partial \bm{u}_t}{\partial t} \Big) : \varepsilon(\bm{v}) {\rm d}x = 0 \quad \forall \bm{v} \in [H_{\partial D_D}^1(D)]^d.
			\end{equation}
			Setting $t = 0$ and $\bm{v} = \bm{u}$ in \eqref{3_Model1_dteq}, from \eqref{3_dJModel1_firststep} and \eqref{3_levelset_eq} it follows that
			\begin{align}
				d\mathcal{J}(\phi; \bm{V}) &= -\int_D \frac{\partial M_{\eta,\gamma}(\phi_t)}{\partial t} \Big|_{t = 0} A \varepsilon(\bm{u}):\varepsilon(\bm{u}) {\rm d}x \notag \\
				&= \int_D -M_{\eta,\gamma}'(\phi_0) \frac{\partial \phi}{\partial t}(0, \bm{x}) A \varepsilon(\bm{u}):\varepsilon(\bm{u}) {\rm d}x \notag \\
				&= \int_D M_{\eta,\gamma}'(\phi_0) A \varepsilon(\bm{u}):\varepsilon(\bm{u}) \bm{V} \cdot \nabla \phi_0(\bm{x}) {\rm d}x \notag 
			\end{align}
			This finishes the proof. \qed
		\end{proof} 
	}
            {
		\begin{proof}[Proof of \eqref{3_Modified_dJ3}]
			Let $\phi_t(\bm{x}) = \phi(t, \bm{x})$ and $(\lambda_{1,t},\bm{u}_{1,t})$ with $ \bm{u}_{1,t}= \bm{u}_1(\phi_t)$ be the first eigenpair of 
			\begin{equation} \label{3_Model3_timeeq}
				\begin{cases}
					-\bm{\nabla} \cdot (M_{\eta_1,\gamma}(\phi_t) A\varepsilon(\bm{u}_t)) = \lambda_{t} \rho_M \bm{u}_t & \texttt{in} \ D, \\
					\bm{u}_t = \bm{0} & \texttt{on} \ \partial D_D, \\
					M_{\eta_1,\gamma}(\phi_t) A\varepsilon(\bm{u}_t) = \bm{0} & \texttt{on} \ \partial D_N. 
				\end{cases}
			\end{equation}
			The first eigenpair satisfies \eqref{3_Model3_timeeq}:
			\begin{equation} \label{3_Model3_weakform}
				\int_D M_{\eta_1,\gamma}(\phi_t) A \varepsilon(\bm{u}_{1,t}):\varepsilon(\bm{v}) {\rm d}x = \lambda_{1,t} \int_D \rho_M \bm{u}_{1,t} \cdot \bm{v} {\rm d}x  \quad \forall \bm{v} \in [H_{\partial D_D}^1(D)]^d.
			\end{equation}
            
            Taking derivatives with respect to $t$ on both sides of the above equation yields
			\begin{align} \label{3_Model3_dteq}
				& \int_D \frac{\partial M_{\eta_1,\gamma}(\phi_t)}{\partial t} A \varepsilon(\bm{u}_{1,t}):\varepsilon(\bm{v}) {\rm d}x + \int_D M_{\eta_1,\gamma}(\phi) A \varepsilon \Big( \frac{\partial \bm{u}_{1,t}}{\partial t} \Big) : \varepsilon(\bm{v}) {\rm d}x \notag \\
                = &\frac{\partial \lambda_{1,t}}{\partial t} \int_D \rho_M \bm{u}_{1,t} \cdot \bm{v} {\rm d}x + \lambda_{1,t} \int_D \frac{\partial \rho_M}{\partial t} \bm{u}_{1,t} \cdot \bm{v} {\rm d}x + \lambda_{1,t} \int_D \rho_M \frac{\partial \bm{u}_{1,t}}{\partial t} \cdot \bm{v} {\rm d}x \quad \forall \bm{v} \in [H_{\partial D_D}^1(D)]^d.
			\end{align}
			Because the small region is not subject to optimization, the mass derivative satisfies
			\begin{equation}
				\left.\frac{\partial\rho_M}{\partial t}\right|_{t=0}=-(1-\eta_2)(1-\chi_{\texttt{block}})M'_\gamma(\phi_0)\,\bm V\cdot\nabla\phi_0.
			\end{equation}
			Setting $t = 0$ and $\bm{v} = \bm{u}_{1,t}$ in \eqref{3_Model3_dteq}, $\bm{v} = \partial_t \bm{u}_{1,t}$ in \eqref{3_Model3_weakform}, then from \eqref{3_levelset_eq} it follows that
			\begin{align}
				\frac{\partial \lambda_{1,t}}{\partial t} \Big|_{t = 0} \notag &= \int_D \frac{\partial M_{\eta_1,\gamma}(\phi_t)}{\partial t}\Big|_{t = 0} A \varepsilon(\bm{u}_1):\varepsilon(\bm{u}_1) {\rm d}x - \lambda_{1} \int_D \frac{\partial \rho_M}{\partial t} \Big|_{t = 0}|\bm{u}_1|^2 {\rm d}x\notag \\
				&= -\int_D \left[ A \varepsilon(\bm{u}_1):\varepsilon(\bm{u}_1) - \lambda_1 \frac{1-\eta_2}{1-\eta_1}(1-\chi_{\texttt{block}}) |\bm{u}_1|^2\right] \ M_{\eta_1,\gamma}^\prime(\phi_0) \ \bm{V} \cdot \nabla \phi_0 {\rm d}x.                
			\end{align}
			This finishes the proof. \qed
		\end{proof} 
	}

	{
		\begin{proof}[Proof of \eqref{3_Modified_dJ2}]
			Let $\phi_t(\bm{x}) = \phi(t, \bm{x})$ and $(\bm{u}_t, p_t) = (\bm{u}(\Omega_t), p_f(\Omega_t))$ be the solution of the following problem
			\begin{equation} \label{3_Model2_timeeq}
				\begin{cases}
					-\nu_f \bm{\Delta} \bm{u}_t + \alpha(M_{\gamma}(\phi_t)) \bm{u}_t + \bm{\nabla} p_t = \bm{f}_f & \texttt{ in } D, \\
					\bm{\nabla} \cdot \bm{u}_t = 0 & \texttt{ in } D, \\
					\bm{u}_t = \bm{g}_f & \texttt{ on } \Gamma_{\texttt{in}}, \\
					\bm{u}_t = \bm{0} & \texttt{ on } \partial D_{\texttt{wall}}, \\
					\nu_f \bm{\nabla} \bm{u}_t \bm{n} - p_t \bm{n} = \bm{0} & \texttt{ on } \Gamma_{\texttt{out}},
				\end{cases}
			\end{equation}
			whose weak formulation is given by
            \begin{equation}\label{3_Model2_weakeq}
			\begin{aligned} 
				\int_D \alpha(M_{\gamma}(\phi_t)) \bm{u}_t \cdot \bm{v} {\rm d}x + \int_D \nu_f \bm{\nabla u_t} : \bm{\nabla v} {\rm d}x - \int_D \bm{\nabla} \cdot \bm{v} \, p_t {\rm d}x = \int_D \bm{f}_f \cdot \bm{v} {\rm d}x & \quad \forall \bm{v} \in \left[ H^1_{\Gamma_{\texttt{in}} \cup \partial D_{\texttt{wall}}}(D) \right]^d, \\
				\int_D \bm{\nabla} \cdot \bm{u}_t \, q_f {\rm d}x = 0 & \quad \forall q_f \in L_0^2(D).
			\end{aligned}
            \end{equation}
			From \eqref{3_Model2_timeeq} we can obtain
			\begin{equation}
				\begin{dcases}
					-\nu_f \bm{\Delta} \dfrac{\partial \bm{u}_t}{\partial t} + (\alpha \circ M_{\gamma})'(\phi_t) \dfrac{\partial \phi_t}{\partial t} \bm{u}_t + \alpha(M_{\gamma}(\phi_t)) \dfrac{\partial \bm{u}_t}{\partial t} + \bm{\nabla} \dfrac{\partial p_t}{\partial t} = \bm{0} & \texttt{ in } D, \\
					\bm{\nabla} \cdot \dfrac{\partial \bm{u}_t}{\partial t} = 0 & \texttt{ in } D, \\
					\dfrac{\partial \bm{u}_t}{\partial t} = \bm{0} & \texttt{ on } \Gamma_{\texttt{in}}, \\
					 \dfrac{\partial \bm{u}_t}{\partial t} = \bm{0} & \texttt{ on } \partial D_{\texttt{wall}}, \\
					\nu_f \bm{\nabla} \dfrac{\partial \bm{u}_t}{\partial t} \bm{n} - \dfrac{\partial p_t}{\partial t} \bm{n} = 0 & \texttt{ on } \Gamma_{\texttt{out}},
				\end{dcases}
			\end{equation}
			so that $\dfrac{\partial \bm{u}_t}{\partial t} \in \left[ H^1_{\Gamma_{\texttt{in}} \cup \partial D_{\texttt{wall}}}(D) \right]^d$.
			We can obtain
			\begin{align} \label{3_dJModel2_firststep}
				d\mathcal{J}(\phi; \bm{V}) =& \frac{1}{2} \int_D (\alpha \circ M_{\gamma})'(\phi_0) \frac{\partial \phi_t}{\partial t}(0, \bm{x}) \bm{u} \cdot \bm{u} {\rm d}x + \int_D \alpha(M_{\gamma}(\phi_0)) \frac{\partial \bm{u}_t}{\partial t} \Big|_{t = 0} \cdot \bm{u} {\rm d}x \notag \\
				&+ \int_D \nu_f \bm{\nabla} \Big( \frac{\partial \bm{u}_t}{\partial t} \Big|_{t = 0} \Big) : \bm{\nabla u} {\rm d}x - \int_D \bm{f}_f \cdot \frac{\partial \bm{u}_t}{\partial t} \Big|_{t = 0} {\rm d}x.
			\end{align}
			Taking $t = 0$ and $(\bm{v}, q_f) = \Big( \dfrac{\partial \bm{u}_t}{\partial t}, p_f \Big)$ in \eqref{3_Model2_weakeq} yields
			\begin{align}
				\int_D \alpha(M_{\gamma}(\phi_0)) \frac{\partial \bm{u}_t}{\partial t} \Big|_{t = 0} \cdot \bm{u} {\rm d}x + \int_D \nu_f \bm{\nabla} \Big( \frac{\partial \bm{u}_t}{\partial t} \Big|_{t = 0} \Big) : \bm{\nabla u} {\rm d}x = \int_D \bm{f}_f \cdot \frac{\partial \bm{u}_t}{\partial t} \Big|_{t = 0} {\rm d}x.
			\end{align}
			Thus
			\begin{align}
				d\mathcal{J}(\phi; \bm{V}) = \frac{1}{2} \int_D (\alpha \circ M_{\gamma})'(\phi_0) \frac{\partial \phi_t}{\partial t}(0, \bm{x}) \bm{u} \cdot \bm{u} {\rm d}x 
				= -\frac{1}{2} \int_D \alpha^\prime(M_{\gamma}(\phi_0)) \left| \bm{u} \right|^2 \ M'_{\gamma}(\phi_0) \ \bm{V} \cdot \nabla \phi_0 {\rm d}x.
			\end{align}
			This gives the result. \qed
		\end{proof} 
	}

\section*{Declaration of generative AI use}
During the preparation of this work, the authors used ChatGPT 5.6-sol in order to improve the language and readability of the manuscript, to assist with the typesetting of its tables, and to suggest the proof idea for Theorem~\ref{thm:volume-drift}. After using this tool, the authors reviewed and edited the content as needed and take full responsibility for the content of the published article. 

\bibliographystyle{unsrtnat}
\bibliography{cas-refs}

\end{document}